%% file: mapmaking_final.tex
\date{}
\title{An axiomatic characterization of the Brownian map}
\author{Jason Miller and Scott Sheffield}
\documentclass[12pt,naturalnames]{article}
\usepackage{amsmath}
\usepackage{amssymb}
\usepackage{amsthm}
\usepackage{amsfonts}
\usepackage{graphicx}
\usepackage{color}
\usepackage{tabularx}
\usepackage{subfig}
\usepackage{microtype}
\usepackage{enumerate}
\usepackage[usenames,dvipsnames]{xcolor}
\usepackage[colorinlistoftodos,textsize=footnotesize]{todonotes}

\usepackage[usenames,dvipsnames]{xcolor}

\def\@rst #1 #2other{#1}
\newcommand\MR[1]{\relax\ifhmode\unskip\spacefactor3000 \space\fi
  \MRhref{\expandafter\@rst #1 other}{#1}}
\newcommand{\MRhref}[2]{\href{http://www.ams.org/mathscinet-getitem?mr=#1}{MR#2}}

\usepackage[margin=1.19in]{geometry}

\newcommand{\CA}{{\mathcal A}}

\newcommand{\CF}{{\mathcal F}}
\newcommand{\CS}{{\mathcal S}}

\allowdisplaybreaks

\newif\ifdraft
\drafttrue
\long\def\comment#1{}
\numberwithin{equation}{section}
\numberwithin{figure}{section}

\newtheorem{theorem}{Theorem}
\numberwithin{theorem}{section}

\newtheorem{lemma}[theorem]{Lemma}

\newtheorem{proposition}[theorem]{Proposition}

\theoremstyle{remark}\newtheorem{definition}[theorem]{Definition}
\theoremstyle{remark}\newtheorem{remark}[theorem]{Remark}

\newcommand{\R}{\mathbf{R}}
\newcommand{\C}{\mathbf{C}}
\newcommand{\D}{\mathbf{D}}
\newcommand{\Z}{\mathbf{Z}}
\newcommand{\T}{\mathbf{T}}
\newcommand{\N}{\mathbf{N}}
\newcommand{\snake}{{\mathcal S}}

\newcommand{\HH}{\mathbf{H}}
\newcommand{\h}{\HH}

\newcommand{\cadlag}{c\`adl\`ag}

\renewcommand{\S}{{\mathbf S}}

\definecolor{purple}{rgb}{0.7,0,0.7}
\definecolor{gray}{rgb}{0.6,0.6,0.6}
\definecolor{dgreen}{rgb}{0.0,0.4,0.0}
\definecolor{dblue}{rgb}{0.0,0.0,0.5}

\newcommand{\Q}{{\mathbf Q}}

\newcommand{\one}{{\mathbf 1}}

\newcommand{\CC}{{\mathcal C}}
\newcommand{\ol}{\overline}
\newcommand{\ul}{\underline}
\newcommand{\wh}{\widehat}
\newcommand{\wt}{\widetilde}

\newcommand{\CU}{{\mathcal U}}
\newcommand{\CG}{{\mathcal G}}

\newcommand{\CJ}{{\mathcal J}}

\newcommand{\CL}{{\mathcal L}}

\newcommand{\giv}{\,|\,}

\newcommand{\excursion}{{\mathbf n}}

\newcommand{\mustwo}{\mu_{\mathrm{SPH}}^2}
\newcommand{\mustwoplusone}{\mu_{\mathrm{SPH}}^{2+1}}
\newcommand{\musk}{\mu_{\mathrm{SPH}}^k}
\newcommand{\musa}{{\mu_{\mathrm{SPH}}^{A=1}}}

\newcommand{\mudonel}{\mu_{ \mathrm{DISK}}^{1, L}}
\newcommand{\mudl}{\mu_{\mathrm{DISK}}^L}
\newcommand{\muml}{\mu_{\mathrm{MET}}^L}

\def\mumlset#1{\mu_{\mathrm{MET}}^{#1}}

\newcommand{\tmustwo}{\wt{\mu}_{\mathrm{SPH}}^2}
\newcommand{\tmudonel}{\wt{\mu}_{ \mathrm{DISK}}^{1, L}}
\newcommand{\tmudl}{\wt{\mu}_{\mathrm{DISK}}^L}

\def\tmudlset#1{\wt{\mu}_{\mathrm{DISK}}^{#1}}
\def\tmudonelset#1{\wt{\mu}_{ \mathrm{DISK}}^{1, #1}}

\def\gmsspace{{\mathcal M}_{\mathrm{SPH}}}
\def\gmsspacer{{\mathcal M}_{\mathrm{SPH},r}}

\def\mmspace{{\mathcal M}}
\def\mmsigma{{\mathcal F}}

\newcommand{\mmsigmaspho}{{\mathcal F}_{\mathrm{SPH}}^{2,O}}

\newcommand{\matrixmspace}{\mathcal N}
\newcommand{\compactmatrix}{\mathcal C}

\newcommand{\fb}[2]{B^\bullet(#1,#2)}

\def\Cov{\mathop{\mathrm{Cov}}}

\def\diam{\mathop{\mathrm{diam}}}

\def\dist{\mathop{\mathrm{dist}}}

\newcommand{\SLE}{{\rm SLE}}
\newcommand{\QLE}{{\rm QLE}}

\newcommand{\var}{{\mathrm{var}}}

\def\Ito/{It\^o}

\def \P {{\mathbf P}}

\def \p {{\P}}
\def \E {{\mathbf E}}

\newcommand{\CT}{{\mathcal T}}

\newcommand{\snakeexcursionmeasure}{{\Pi}}

\newcommand{\gh}{{\mathcal X}}
\newcommand{\dgh}{d_{\mathrm {GH}}}
\renewcommand{\dh}{d_{\mathrm{H}}}
\newcommand{\spheresgeo}{{\mathcal Y}}
\newcommand{\spheresgeoclose}{\ol{\spheresgeo}}

\newcommand{\treespace}{\CT}
\newcommand{\treeequivspace}{\CA}

\begin{document}

\maketitle

\begin{abstract}
The Brownian map is a random sphere-homeomorphic metric measure space obtained by ``gluing together'' the continuum trees described by the $x$ and $y$ coordinates of the Brownian snake.  We present an alternative ``breadth-first'' construction of the Brownian map, which produces a surface from a certain decorated branching process.  It is closely related to the peeling process, the hull process, and the Brownian cactus.

Using these ideas, we prove that the Brownian map is the only random sphere-homeomorphic metric measure space with certain properties: namely, scale invariance and the conditional independence of the inside and outside of certain ``slices'' bounded by geodesics and metric ball boundaries.  We also formulate a characterization in terms of the so-called L\'evy net produced by a metric exploration from one measure-typical point to another.  This characterization is part of a program for proving the equivalence of the Brownian map and the Liouville quantum gravity sphere with parameter $\gamma= \sqrt{8/3}$.
\end{abstract}

\newpage
\tableofcontents
\newpage

\parindent 0 pt
\setlength{\parskip}{0.25cm plus1mm minus1mm}

\medbreak {\noindent\bf Acknowledgements.} 

\input{acknowledgements.tex}

We also acknowledge two anonymous referees, whose helpful feedback led to many improvements to the article.

\input{support_acknowledgements.tex}

\section{Introduction}
\label{sec::introduction}

\subsection{Overview}

In recent years, numerous works have studied a random measure-endowed metric space called the {\em Brownian map}, which can be understood as the $n \to \infty$ scaling limit of the uniformly random quadrangulation (or triangulation) of the sphere with~$n$ quadrilaterals (or triangles).  We will not attempt a detailed historical account here. Miermont's recent St.\ Flour lecture notes are a good place to start for a general overview and a list of additional references \cite{miermontstflour}.\footnote{To give an extremely incomplete sampling of other papers relevant to this work, let us mention the early planar map enumerations of Tutte and Mullin \cite{tutte1962census,mullintrees,MR0218276}, a few early works on tree bijections by Schaeffer and others \cite{cori1981planar, jacquardschaeffer, schaeffermaps, bousquetmelouschaeffer, chassaingschaeffer},  early works on path-decorated surfaces by Duplantier and others \cite{dk1998rwks, ds1989fk, dup1998rw_qg}, the pioneering works by Watabiki and by Angel and Schramm on triangulations and the so-called peeling process \cite{watabikipeeling, angelUIPTpercolation, angelschrammUIPT}, Krikun's work on reversed branching processes \cite{krikun2005uniform}, the early Brownian map definitions of Marckert and Mokkadem \cite{marckert2006limit} and Le Gall and Paulin \cite{le2008scaling} (see also the work \cite{MR2399286} of Miermont), various relevant works by Duquesne and Le Gall on L\'evy trees and related topics \cite{dlg2002trees_levy,  duquesnelegall2005levytrees,duquesnelegallhausdorff, duquesnelegallrerooting}, the Brownian cactus of Curien, Le Gall, and Miermont \cite{browniancactus}, the stable looptrees of Curien and Kortchemski \cite{ck2013looptrees}, and several recent breakthroughs by Le Gall and Miermont \cite{legallgeodesics, legalluniqueanduniversal, miermontlimit, legalluniversallimit}.}

This paper will assemble a number of ideas from the literature and use them to derive some additional fundamental facts about the Brownian map: specifically, we explain how the Brownian map can be constructed from a certain branching ``breadth-first'' exploration.  This in turn will allow us to characterize the Brownian map as the only random metric measure space with certain properties.

Roughly speaking, in addition to some sort of scale invariance, the main property we require is the conditional independence of the inside and the outside of certain sets (namely, filled metric balls and ``slices'' of filled metric balls bounded between pairs of geodesics from the center to the boundary) given an associated boundary length parameter.  Section~\ref{subsec::discreteintuition} explains that certain discrete models satisfy discrete analogs of this conditional independence; so it is natural to expect their limits to satisfy a continuum version. Our characterization result is in some sense analogous to the characterization of the Schramm-Loewner evolutions ($\SLE$s) as the only random paths satisfying conformal invariance and the so-called domain Markov property~\cite{S0}, or the characterization of conformal loop ensembles (CLEs) as the only random collections of loops with a certain Markov property~\cite{MR2979861}.

The reader is probably familiar with the fact that in many random planar map models, when the total number of faces is of order  $n^4$, the length of a macroscopic geodesic path has order $n$, while the length of the outer boundary of a macroscopic metric ball has order $n^2$.  Similarly, if one rescales an instance of the Brownian map so that distance is multiplied by a factor of $C$, the area measure is multiplied by $C^4$, and the length of the outer boundary of a metric ball (when suitably defined) is multiplied by $C^2$ (see Section~\ref{sec::brownianmap}). One might wonder whether there are other continuum random surface models with other scaling exponents in place of the $4$ and the $2$ mentioned above, perhaps arising from other different types of discrete models. However, in this paper the exponents $4$ and $2$ are shown to be determined by the axioms we impose; thus a consequence of this paper is that any continuum random surface model with different exponents must fail to satisfy at least one of these axioms.

One reason for our interest in this characterization is that it plays a role in a larger program for proving the equivalence of the Brownian map and the Liouville quantum gravity (LQG) sphere with parameter $\gamma = \sqrt{8/3}$.  Both $\sqrt{8/3}$-LQG and the Brownian map describe random measure-endowed surfaces, but the former comes naturally equipped with a conformal structure, while the latter comes naturally equipped with the structure of a geodesic metric space.  The program provides a bridge between these objects, effectively endowing each one with the {\em other's} structure, and showing that once this is done, the laws of the objects agree with each other.

The rest of this program is carried out in  \cite{qlebm,quantum_spheres,qle_continuity, qle_determined}, all of which build on \cite{she2010zipper, ms2012imag1, ms2012imag2, ms2012imag3, ms2013imag4, ms2013qle, matingtrees} --- see also Curien's related work on this question \cite{cur2013glimpse}.  After using a quantum Loewner evolution (QLE) exploration to impose a metric structure on the LQG sphere, the papers \cite{qlebm, qle_continuity} together prove that the law of this metric has the properties that characterize the law of the Brownian map, and hence is equivalent to the law of the Brownian map.

\subsection{Relation with other work}

There are several independent works which were posted to the arXiv shortly after the present work that complement and partially overlap the work done here in interesting ways. Bertoin, Curien, and Kortchemski \cite{bck} have independently constructed a breadth-first exploration of the Brownian map, which may also lead to an independent proof that the Brownian map is uniquely determined by the information encoding this exploration. They draw from the theory of fragmentation processes to describe the evolution of the whole countable collection of unexplored component boundaries. They also explore the relationship to discrete breadth-first searches in some detail. Abraham and Le Gall \cite{alg} have studied an infinite measure on Brownian snake excursions in the positive half-line (with the individual Brownian snake paths stopped when they return to $0$). These excursions correspond to disks cut out by a metric exploration of the Brownian map, and play a role in this work as well.  Finally, Bettinelli and Miermont \cite{bettinelli_miermont_disks} have constructed and studied properties of {\em Brownian disks} with an interior marked point and a given boundary length $L$ (corresponding to the measure we call $\mudonel$; see Section~\ref{subsec::mapsdisksnets}) including a decomposition of these disks into geodesic slices, which is related to the decomposition employed here for metric balls of a given boundary length (chosen from the measure we call $\muml$). They show that as a point moves around the boundary of the Brownian disk, its distance to the marked point evolves as a type of Brownian bridge. In particular, this implies that the object they call the Brownian disk has finite diameter a.s.

We also highlight two more recent works. First, Le Gall in \cite{legallsubordination} provides an alternative approach to constructing the object we call the L\'evy net in this paper and explores a number of related ideas. The L\'evy net as defined in this paper is (in some sense) the set of points in the Brownian map observed by a metric exploration (``continuum peeling'') process from a point $x$ to a point $y$.  Roughly speaking, the approach in Le Gall's paper is to start with the continuum random tree used in the construction of the Brownian map (which encodes a space-filling path on the Brownian map) and then take the quotient w.r.t.\ an equivalence relation that makes two points the same if they belong to the closure of the same excursion into the complement of the L\'evy net (such an excursion always leaves and re-enters the L\'evy net at the same point). This equivalence relation is easy to describe directly using the Brownian snake, which makes the L\'evy net construction very direct. We also make note of a recent work by Bertoin, Budd, Curien, and Kortchemski \cite{betroinbuddcurienkortchemski} that studies (among other things) the fragmentation processes that appear in variants of the Brownian map that arise as scaling limits of surfaces with ``very large'' faces.

\subsection{Theorem statement}
\label{subsec::theoremstatement}

In this subsection, we give a quick statement of our main theorem. However, we stress that several of the objects involved in this statement (leftmost geodesics, the Brownian map, the various $\sigma$-algebras, etc.) will not be formally defined until later in the paper. Let $\gmsspace$ be the space of geodesic metric spheres that come equipped with a {\em good} measure (i.e., a finite measure that has no atoms and assigns positive mass to each open set). In other words, $\gmsspace$ is the space of (measure-preserving isometry classes of) triples $(S, d, \nu)$, where $d \colon S\times S \to [0,\infty)$ is a distance function on a set $S$ such that $(S,d)$ is topologically a sphere, and $\nu$ is a good measure on the Borel $\sigma$-algebra of $S$.

Denote by $\musa$ the standard unit area (sphere homeomorphic) Brownian map, which is a random variable that lives on the space $\gmsspace$. We will also discuss a closely related {\em doubly marked Brownian map} measure $\mustwo$ on the space $\gmsspace^2$ of elements of $\gmsspace$ that come equipped with two distinguished marked points $x$ and $y$. This $\mustwo$ is an {\em infinite} measure on the space of {\em finite} volume surfaces. The quickest way to describe it is to say that sampling from $\mustwo$ amounts to\begin{enumerate}
 \item letting $A$ be a positive real number whose law is the infinite measure $A^{-3/2}dA$,
\item  letting $(S, d, \nu)$ be an independent measure-endowed surface from the law $\musa$,
\item then letting $x$ and $y$ be two marked points on $S$ chosen independently from $\nu$,
\item then ``rescaling'' the doubly marked surface $(S, d, \nu, x, y)$ so that its area is $A$ (scaling area by $A$ and distances by $A^{1/4}$).
\end{enumerate} The measure  $\mustwo$ turns out to describe the natural ``grand canonical ensemble'' on doubly marked surfaces. We formulate our main theorems in terms of $\mustwo$ (although they can indirectly be interpreted as theorems about $\musa$ as well).

Given an element $(S, d, \nu, x, y) \in \gmsspace^2$, and some $r \geq 0$, let $B(x,r)$ denote the open metric ball with radius $r$ and center $x$. Let $\fb{x}{r}$ denote the {\em filled} metric ball of radius $r$ centered at $x$, as viewed from $y$. That is, $\fb{x}{r}$ is the complement of the $y$-containing component of the complement of $\ol{B(x,r)}$.  One can also understand $S \setminus \fb{x}{r}$ as the set of points $z$ such that there exists a path from $z$ to $y$ along which the function $d(x, \cdot)$ stays strictly larger than $r$. Note that if $0 < r < d(x,y)$ then $\fb{x}{r}$ is a closed set whose complement contains $y$ and is topologically a disk. In fact, one can show (see Proposition~\ref{prop::boundariesarecircles}) that the boundary $\partial \fb{x}{r}$ is topologically a circle, so that $\fb{x}{r}$ is topologically a closed disk.  We will sometimes interpret~$\fb{x}{r}$ as being itself a metric measure space with one marked point (the point~$x$) and a measure obtained by restricting~$\nu$ to~$\fb{x}{r}$. For this purpose, the metric we use on $\fb{x}{r}$ is the {\em interior-internal metric} on $\fb{x}{r}$ that it inherits from $(S,d)$ as follows: the distance between two points is the infimum of the~$d$ lengths of paths between them that (aside from possibly their endpoints) stay in the interior of~ $\fb{x}{r}$.  In most situations, one would expect this distance to be the same as the ordinary {\em interior metric}, in which the infimum is taken over all paths contained in~$\fb{x}{r}$, with no requirement that these paths stay in the interior. However, one can construct examples where this is not the case, i.e., where paths that hit the boundary on some (possibly fractal) set of times are shorter than the shortest paths that do not. In general, the interior-internal metric is less informative than the internal metric; given either metric, one can compute the~$d$ lengths of paths that remain in the interior; however the interior-internal metric does not determine the~$d$ lengths of curves that hit the boundary an uncountable number of times.  Whenever we make reference to metric balls (as in the statement of Theorem~\ref{thm::markovmapcharacterization} below) we understand them as marked metric measure spaces, endowed with the interior-internal metric induced by~$d$, and the restriction of~$\nu$.  (When we discuss ``slices'' bounded between two geodesics, it is more natural to use the ordinary internal metric. A minimal path between two points $x$ and $y$ can be constructed so that if it hits one of the two geodesic boundary arcs in two locations, then it traces the entire arc between those locations.)

We will later recall that in the doubly marked Brownian map, if we fix $r>0$, then on the event that $d(x,y) > r$, the circle $\partial \fb{x}{r}$ a.s.\ comes endowed with a certain ``boundary length measure'' (which scales like the square root of the area measure).  This is not too surprising given that the Brownian map is a scaling limit of random triangulations, and the discrete analog of a filled metric ball clearly comes with a notion of boundary length. We review this idea, along with more of the discrete intuition behind Theorem~\ref{thm::markovmapcharacterization}, in Section~\ref{subsec::discreteintuition}.

We will also see in Section~\ref{sec::preliminaries} that there is a certain $\sigma$-algebra on the space of doubly marked metric measure spaces (which induces a $\sigma$-algebra $ \mmsigma^2$ on $\gmsspace^2$) that is in some sense the ``weakest reasonable'' $\sigma$-algebra to use. We formulate Theorem~\ref{thm::markovmapcharacterization} in terms of that $\sigma$-algebra.  (In some sense, a weaker $\sigma$-algebra corresponds to a stronger theorem in this context, since if one has a measure defined on a stronger $\sigma$-algebra, one can always restrict it to a weaker $\sigma$-algebra. Theorem~\ref{thm::markovmapcharacterization} is a general characterization theorem for these restrictions.) We will also explain in Section~\ref{sec::preliminaries} why  $ \mmsigma^2$ is strong enough for practical purposes --- strong enough so that certain natural events and functions are measurable.

Specifically, we will explain in Section~\ref{sec::preliminaries} why the hypotheses in the theorem statement are meaningful (e.g., why objects like $\fb{x}{r}$, viewed as a metric measure space as described above, are measurable random variables), and we will explain the term ``leftmost'' (which makes sense once one of the two orientations of the sphere has been fixed). However, let us clarify one point upfront: whenever we discuss geodesics in this paper, we will refer to paths between two endpoints that have minimal length among {\em all} paths between those endpoints (i.e., they do not just have this property in some local sense).

\begin{theorem}
\label{thm::markovmapcharacterization}
The (infinite) doubly marked Brownian map measure $\mustwo$ is the only measure on $(\gmsspace^2, \mmsigma^2)$ with the following properties. (Here a sample from the measure is denoted by $(S,d,\nu,x,y)$.)
\begin{enumerate}
\item The law is invariant under the Markov operation that corresponds to forgetting $x$ (or $y$) and then resampling it from the measure $\nu$ (multiplied by a constant to make it a probability measure). In other words, given $(S,d,\nu)$, the points $x$ and $y$ are conditionally i.i.d.\ samples from the probability measure $\nu/\nu(S)$.
\item Fix $r > 0$ and let $\mathcal E_r$ be the event that $d(x,y) > r$.  Then $\mustwo(\mathcal E_r) \in (0, \infty)$, so that $\mustwo(\mathcal E_r)^{-1}$ times the restriction of $\mustwo$ to $\mathcal E_r$ is a probability measure.  Suppose that we have chosen an orientation of $S$ by tossing an independent fair coin.  Under this probability measure, the following are true for $s = r$ and also for $s = d(x,y) - r$.
\begin{enumerate}
\item There is an $\mmsigma^2$-measurable random variable that we denote by $L_s$ (which we interpret as a ``boundary length'' of $\partial \fb{x}{s}$) such that {\em given} $L_s$ and the orientation of $S$, the random oriented metric measure spaces $\fb{x}{s}$ and $S \setminus \fb{x}{s}$ are conditionally independent of each other.  In the case $s = r$, the conditional law of $S \setminus \fb{x}{s}$ depends only on the quantity $L_s$, and does so in a scale invariant way; i.e., there exists some fixed $a$ and $b$ such that the law given $L_s = C$ is the same as the law given $L_s = 1$ except that areas and distances are respectively scaled by $C^a$ and $C^b$. The same holds for the conditional law of $\fb{x}{s}$ in the case $s = d(x,y)-r$.
\item In the case that $s=d(x,y) - r$, there is a measurable function that takes $(S,d,\nu,x,y)$ as input and outputs $(S,d,\pi,x,y)$ where $\pi$ is a.s.\ a good measure (which we interpret as a {\em boundary length} measure) on $\partial \fb{x}{s}$ (which is necessarily homeomorphic to a circle) that has the following properties:
\begin{enumerate}[(i)]
\item The total mass of $\pi$ is a.s.\ equal to~$L_s$.
\item Suppose we first sample $(S,d,\nu,x,y)$, then produce $\pi$, then sample $z_1$ from $\pi$, and then position $z_2, z_3, \ldots, z_n$ so that $z_1, z_2, z_3, \ldots, z_n$ are evenly spaced around $\partial \fb{x}{s}$ according to $\pi$ following the orientation of $\partial \fb{x}{s}$. Then the $n$ ``slices'' produced by cutting $\fb{x}{s}$ along the leftmost geodesics from $z_i$ to $x$ are (given $L_s$) conditionally i.i.d.\ (as suggested by Figure~\ref{fig::shards} and Figure~\ref{fig::slicedecomposition}) and the law of each slice depends only on $L_s/n$, and does so in a scale invariant way (with the same exponents $a$ and $b$ as above).
\end{enumerate}
\end{enumerate}
\end{enumerate}
\end{theorem}

As we will explain in more detail in Section~\ref{sec::preliminaries}, every doubly marked geodesic metric sphere comes with two possible ``orientations'' and one way to specify one of these orientations is to specify an ordered list of three additional distinct marked points on boundary of a filled metric ball (which effectively determine a ``clockwise'' direction). Two quintuply marked spheres defined this way can be said to be ``equivalent'' if they are equivalent as doubly marked spheres and the extra triples of points encode the same orientation --- and one can then limit attention to events that consist of unions of equivalence classes.  We will also show that both $\fb{x}{s}$ and $S \setminus \fb{x}{s}$ are topological disks.  One can specify an orientation of either by specifying three additional distinct marked points on the boundary and this is what we mean in the statement of Theorem~\ref{thm::markovmapcharacterization} when we say that these spaces come with an orientation.

We remark that one can formulate a version of Theorem~\ref{thm::markovmapcharacterization} in which one assumes that the space {\em comes} with an orientation (not necessarily chosen by a fair coin toss). As we explain in  Section~\ref{sec::preliminaries}, formulating statements about random oriented spheres requires us to extend the $\sigma$-algebra slightly to account for the extra bit of information that encodes the orientation. The reader may recall that the Brownian map can be interpreted as a random oriented metric sphere (since the Brownian snake construction produces a directed Peano curve that traces the boundary of a geodesic tree in what we can define to be the ``clockwise'' direction.) Although in the Brownian map the geodesics to the root are a.s.\ unique (from a.a.\ points) we are interested in random metric spheres for which this is not assumed to be the case \emph{a priori} --- and in these settings we will use the fact that one can always define \emph{leftmost geodesics} in a unique way.

We also remark that the statement that we have a way to assign a {\em boundary length} measure to $\partial \fb{x}{s}$ can be reformulated as the statement that we have a way to randomly assign a marked boundary point~$z$ to $\partial \fb{x}{s}$.  The boundary length measure is then~$L_s$ times the conditional law of~$z$ given $(S,d,\nu,x,y)$.

Among other things, the conditions of Theorem~\ref{thm::markovmapcharacterization} will ultimately imply that $L_r$ can be viewed as a process indexed by $r \in [0,d(x,y)]$, and that both $L_r$ and its time-reversal can be understood as excursions derived from Markov processes. We will see {\em a posteriori} that the time-reversal of $L_r$ is given by a certain time change of a $3/2$-stable L\'evy excursion with only positive jumps. One can also see {\em a posteriori} (when one samples from a measure which satisfies the axioms in the theorem --- i.e., from the Brownian map measure $\mustwo$) that the definition of the ``slices'' above is not changed if one replaces ``leftmost'' with ``rightmost'' because, in fact, from almost all points on $\partial \fb{x}{s}$ the geodesic to $x$ is unique. We remark that the last condition in Theorem~\ref{thm::markovmapcharacterization} can be understood as a sort of ``infinite divisibility'' assumption for the law of a certain filled metric ball, given its boundary length.

Before we prove Theorem~\ref{thm::markovmapcharacterization}, we will actually first formulate and prove another closely related result:  Theorem~\ref{thm::levynetbasedcharacterization}.  To explain roughly what Theorem~\ref{thm::levynetbasedcharacterization} says, note that for any element of $\gmsspace^2$, one can consider the union of the boundaries $\partial \fb{x}{r}$ taken over all $r \in [0,d(x,y)]$. This union is called the {\em metric net} from $x$ to $y$ and it comes equipped with certain structure (e.g.,\ there is a distinguished leftmost geodesic from any point on the net back to $x$).  Roughly speaking, Theorem~\ref{thm::levynetbasedcharacterization} states that $\mustwo$ is the only measure on $(\gmsspace^2, \mmsigma^2)$ with certain basic symmetries and the property that the infinite measure it induces on the space of metric nets corresponds to a special object called the {\em $\alpha$-(stable) L\'evy net} that we will define in Section~\ref{sec::surfacesfromtrees}.

\subsection{Outline}
\label{subsec::outline}

In Section~\ref{sec::preliminaries} we discuss some measure theoretic and geometric preliminaries. We begin by defining a {\em metric measure space} (a.k.a.\ {\em $mm$-space)} to be a triple $(S, d, \nu)$ where $(S,d)$ is a complete separable metric space, $\nu$ is a measure defined on its Borel $\sigma$-algebra, and $\nu(S) \in (0, \infty)$.\footnote{Elsewhere in the literature, e.g., in \cite{grevenpfaffelhuberwinter}, the definition of a metric measure space also requires that the measure be a {\em probability} measure, i.e., that $\nu(S) = 1$.  It is convenient for us to relax this assumption so that the definition includes area-measure-endowed surfaces whose total area is different from one. Practically speaking, the distinction does not matter much because one can always recover a probability measure by dividing the area measure by the total area. It simply means that we have one extra real parameter --- total mass --- to consider. Any topology or $\sigma$-algebra on the space of metric {\em probability}-measure spaces can be extended to the larger space we consider by taking its product with the standard Euclidean topology (and Borel-$\sigma$-algebra) on $\R$.} Let $\mmspace$ denote the space of all metric measure spaces. Let $\mmspace^k$ denote the set of metric measure spaces that come with an ordered set of $k$ marked points.

As mentioned above, before we can formally make a statement like ``The doubly marked Brownian map is the only measure on $\mmspace^2$ with certain properties'' we have to specify what we mean by a ``measure on $\mmspace^2$,'' i.e., what $\sigma$-algebra a measure is required to be defined on. The weaker the $\sigma$-algebra, the stronger the theorem, so we would ideally like to consider the weakest ``reasonable'' $\sigma$-algebra on $\mmspace$ and its marked variants. We argue in Section~\ref{sec::preliminaries} that the weakest reasonable $\sigma$-algebra on $\mmspace$ is the $\sigma$-algebra $\mmsigma$ generated by the so-called Gromov-weak topology. We recall that this topology can be generated by various natural metrics that make $\mmspace$ a complete separable metric space, including the so-called Gromov-Prohorov metric and the Gromov-$\underline \Box_1$ metric \cite{grevenpfaffelhuberwinter, gromovprohorovandgromovbox}.

We then argue that this $\sigma$-algebra is at least strong enough so that the statement of our characterization theorem makes sense: for example, since our characterization involves surfaces cut into pieces by ball boundaries and geodesics, we need to explain why certain simple functions of these pieces can be understood as measurable functions of the original surface. All of this requires a bit of a detour into metric geometry and measure theory, a detour that occupies the whole of Section~\ref{sec::preliminaries}. The reader who is not interested in the details may skip or skim most of this section.

In Section~\ref{sec::surfacesfromtrees}, we recall the tree gluing results from \cite{matingtrees}. In \cite{matingtrees} we proposed using the term {\em peanosphere}\footnote{The term emerged in a discussion with Kenyon. On the question of whether to capitalize ({\it \`a la} Laplacian, Lagrangian, Hamiltonian, Jacobian, Bucky Ball) or not ({\em \`a la} boson, fermion, newton, hertz, pascal, ohm, einsteinium, algorithm, buckminsterfullerene) the authors express no strong opinion.} to describe a space, topologically homeomorphic to the sphere, that comes endowed with a good measure and a distinguished space-filling loop (parameterized so that a unit of area measure is filled in a unit of time) that represents an interface between a continuum ``tree'' and ``dual tree'' pair. Several of the constructions in \cite{matingtrees} describe natural measures on the space of peanospheres, and we note that the Brownian map also fits into this framework.

Some of the constructions in \cite{matingtrees} also involve the {\em $\alpha$-stable looptrees} introduced by Curien and Kortchemski in \cite{ck2013looptrees}, which are in turn closely related to the L\'evy stable random trees explored by Duquesne and Le Gall \cite{dlg2002trees_levy, duquesnelegall2005levytrees, duquesnelegallhausdorff,  duquesnelegallrerooting}. For $\alpha \in (1,2)$ we show how to glue an $\alpha$-stable looptree ``to itself'' in order to produce an object that we call the $\alpha$-{\em stable L\'evy net}, or simply the  $\alpha$-{\em L\'evy net} for short.  The L\'evy net is a random variable which takes values in the space which consists of a planar real tree together with an equivalence relation which encodes how the tree is glued to itself.  It can be understood as something like a Peano {\em carpet}.  It is a space homeomorphic to a closed subset of the sphere (obtained by removing countably many disjoint open disks from the sphere) that comes with a natural measure and a path that fills the entire space; this path represents an interface between a geodesic tree (whose branches also have well-defined length parameterizations) and its dual (where in this case the dual object is the $\alpha$-stable looptree itself).

We then show how to explore the L\'evy net in a breadth-first way, providing an equivalent construction of the L\'evy net that makes sense for all $\alpha \in (1,2)$.  Our results about the L\'evy net apply for general $\alpha$ and can be derived independently of their relationship to the Brownian map. Indeed, the Brownian map is not explicitly mentioned at all in Section~\ref{sec::surfacesfromtrees}.

In Section~\ref{sec::brownianmap} we make the connection to the Brownian map.  To explain roughly what is done there, let us first recall recent works by Curien and Le Gall \cite{cl2014peeling, 2014arXiv1409.4026C} about the so-called {\em Brownian plane}, which is an infinite volume Brownian map that comes with a distinguished origin.  They consider the {\em hull process} $L_r$, where $L_r$ denotes an appropriately defined ``length'' of the outer boundary of the metric ball of radius $r$ centered at the origin, and show that $L_r$ can be understood in a certain sense as the time-reversal of a continuous state branching process (which is in turn a time change of a $3/2$-stable L\'evy process).  See also the earlier work by Krikun on reversed branching processes associated to an infinite planar map \cite{krikun2005uniform}.

Section~\ref{sec::brownianmap} will make use of {\em finite-volume} versions of the relationship between the Brownian map and $3/2$-stable L\'evy processes.  In these settings, one has two marked points~$x$ and~$y$ on a finite-diameter surface, and the process~$L_r$ indicates an appropriately defined ``length'' of $\partial \fb{x}{r}$.  The restriction of the Brownian map to the union of these boundary components is itself a random metric space (using the shortest path distance within the set itself).  In Section~\ref{subsec::mmsigma} we will show that one can view this space as corresponding to a real tree (which describes the leftmost geodesics from points on the filled metric ball boundaries $\partial \fb{x}{r}$ to $x$) together with an equivalence relation which describes which points in the leftmost geodesic tree are identified with each other.  We will show that this structure agrees in law with the $3/2$-L\'evy net.

Given a single instance of the Brownian map, and a single fixed point $x$, one may let the point $y$ vary over some countable dense set of points chosen i.i.d.\ from the associated area measure; then for each $y$ one obtains a different instance of the L\'evy net.  We will observe that, given this collection of coupled L\'evy net instances, it is possible to reconstruct the entire Brownian map.
Indeed, this perspective leads us to the ``breadth-first'' construction of the Brownian map.  (As we recall in Section~\ref{sec::brownianmap}, the conventional construction of the Brownian map from the Brownian snake involves a ``depth-first'' exploration of the geodesic tree associated to the Brownian map.)

The characterization will then essentially follow from the fact that $\alpha$-stable L\'evy processes (and the corresponding continuous state branching processes) are themselves characterized by certain symmetries (such as the Markov property and scale invariance; see Proposition~\ref{prop::strongstablecsbp}) and these correspond to geometric properties of the random metric space. An additional calculation will be required to prove that $\alpha = 3/2$ is the only value consistent with the axioms that we impose, and to show that this determines the other scaling exponents of the Brownian map.

\subsection{Discrete intuition}
\label{subsec::discreteintuition}

This paper does not address discrete models directly.  All of our theorems here are formulated and stated in the continuum.  However, it will be useful for intuition and motivation if we recall and sketch a few basic facts about discrete models. We will not include any detailed proofs in this subsection.

\subsubsection{Infinite measures on singly and doubly marked surfaces}
The literature on planar map enumeration begins with Mullin and Tutte in the 1960's \cite{tutte1962census,mullintrees,MR0218276}. The study of geodesics and the metric structure of random planar maps has roots in an influential bijection discovered by Schaeffer \cite{MR1465581}, and earlier by Cori and Vauquelin \cite{cori1981planar}.

The Cori-Vauquelin-Schaeffer construction is a way to encode a planar map by a pair of trees: the map $M$ is a quadrangulation, and a ``tree and dual tree'' pair on $M$ are produced from $M$ in a deterministic way.  One of the trees is a breadth-first search tree of $M$ consisting of geodesics; the other is a type of dual tree.\footnote{It is slightly different from the usual dual tree definition.  As in the usual case, paths in the dual tree never ``cross'' paths in the tree; however, the dual tree is defined on the same vertices as the tree itself; it has some edges that cross quadrilaterals diagonally and others that overlap the tree edges.}  In this setting, as one traces the boundary between the geodesic tree and the dual tree, one may keep track of the distance from the root in the dual tree, and the distance in the geodesic tree itself; Chassaing and Schaeffer showed that the scaling limit of this random two-parameter process is the continuum random path in $\R^2$ traced by the head of a Brownian snake \cite{chassaingschaeffer}, whose definition we recall in Section~\ref{sec::brownianmap}.
The Brownian map\footnote{The Brownian map was introduced in works by Marckert and Mokkadem and by Le Gall and Paulin \cite{marckert2006limit,le2008scaling}. For a few years, the term ``Brownian map'' was used to refer to any one of the subsequential Gromov-Hausdorff scaling limits of certain random planar maps. Works by Le Gall and by Miermont established the uniqueness of this limit, and proved its equivalence to the metric space constructed directly from the Brownian snake \cite{legalluniqueanduniversal, miermontlimit, legalluniversallimit}.}
is a random metric space produced directly from this continuum random path; see Section~\ref{sec::brownianmap}.

Let us remark that tracing the boundary of a tree counterclockwise can be intuitively understood as performing a ``depth-first search'' of the tree, where one chooses which branches to explore in a left-to-right order.   In a sense, the Brownian snake is associated to a {\em depth-first} search {\em of the tree of geodesics} associated to the Brownian map. We mention this in order to contrast it with the {\em breadth-first search} of the same geodesic tree that we will introduce later.

The scaling limit results mentioned above have been established for a number of types of random planar maps, but for concreteness, let us now focus our attention on triangulations. According to \cite[Theorem~2.1]{angelschrammUIPT} (applied with $m=0$, see also \cite{angelUIPTpercolation}), the number of triangulations (with no loops allowed, but multiple edges allowed) of a sphere with $n$ triangles and a distinguished oriented edge is given by
\begin{equation}
\label{eqn::probdecay}
\frac{2^{n+1} (3n)!}{n! (2n+2)!} \approx C (27/2)^n n^{-5/2}
\end{equation}
where $C > 0$ is a constant.  Let $\mu^1_{\mathrm{TRI}}$ be the probability measure on triangulations such that the probability of each specific $n$-triangle triangulation (with a distinguished oriented edge --- whose location one may treat as a ``marked point'') is proportional to $(27/2)^{-n}$.  Then~\eqref{eqn::probdecay} implies that the $\mu^1_{\mathrm{TRI}}$ probability of obtaining a triangulation with $n$ triangles decays asymptotically like a constant times $n^{-5/2}$.  One can define a new (non-probability) measure on random metric measure spaces $\mu^1_{\mathrm{TRI},k}$, where the area of each triangle is $1/k$ (instead of constant) but the measure is multiplied by a constant to ensure that the $\mu^1_{\mathrm{TRI},k}$ measure of the set of triangulations with area in the interval $(1,2)$ is given by $\int_1^2 x^{-5/2} dx$, and distances are scaled by $k^{-1/4}$.  As $k \to \infty$ the vague limit (as defined w.r.t.\ the Gromov-Hausdorff topology on metric spaces) is an infinite measure on the set of measure-endowed metric spaces. Note that we can represent any instance of one of these scaled triangulations as $(M, A)$ where $A$ is the total area of the triangulation and $M$ is the measure-endowed metric space obtained by rescaling the area of each triangle by a constant so that the total becomes $1$ (and rescaling all distances by the fourth root of that constant).

\begin{figure}[ht!]
\begin{center}
\includegraphics[scale=0.7]{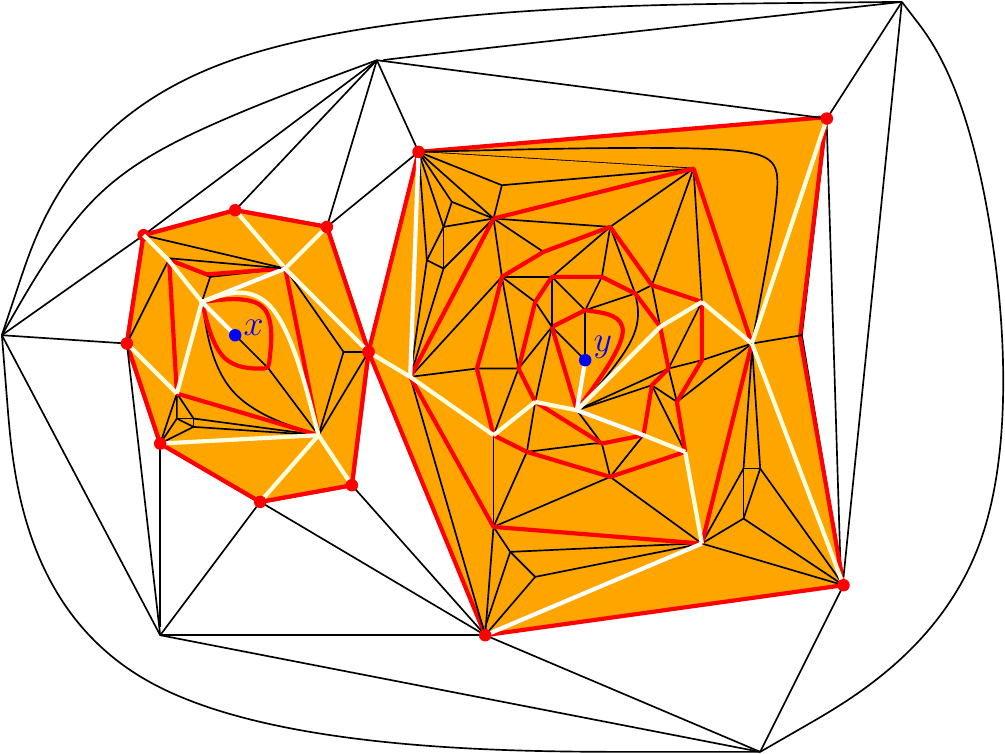}
\caption{\label{fig::shards}  Shown is a triangulation of the sphere (the outer three edges form one triangle) with two marked points: the blue dots labeled $x$ and $y$.  For a vertex $z$ and $r \in \N$, the metric ball $B(z,r)$ of radius $r$ consists of the union of all faces which contain a vertex whose distance to $z$ is at most $r-1$.  The red cycles are outer boundaries of metric balls centered at $x$ as viewed from $y$ (of radii $1$, $2$, $3$) and at $y$ as viewed from $x$ (of radii $1$, $2$, $3$, $4$, $5$).  From each point on the outer boundary of $\fb{x}{3}$ (resp.\ $\fb{y}{5}$) a geodesic toward $x$ (resp.\ $y$) is drawn in white.  The geodesic drawn is the ``leftmost possible'' one; i.e., to get from a point on the circle of radius $k$ to the circle of radius $k-1$, one always takes the leftmost edge (as viewed from the center point).  ``Cutting'' along white edges divides each of $\fb{x}{3}$ and $\fb{y}{5}$ into a collection of triangulated surfaces (one for each boundary edge) with left and right boundaries given by geodesic paths of the same length.  Within $\fb{x}{3}$ (resp.\ $\fb{y}{5}$), there happens to be a single longest slice of length $3$ (resp.\ $5$) reaching all the way from the boundary to $x$ (resp.\ $y$). Parts of the left and right boundaries of these longest slices are identified with each other when the slice is embedded in the sphere. This is related to the fact that all of the geodesics shown in white have ``merged'' by their final step. Between $\fb{x}{3}$ and $\fb{y}{5}$, there are $8+5 =13$ slices in total, one for each boundary edge. The white triangles outside of $\fb{x}{3} \cup \fb{y}{5}$ form a triangulated disk of boundary length $13$.}
\end{center}
\end{figure}

As $k \to \infty$ the measures $\mu^1_{\mathrm{TRI},k}$ converge vaguely to the measure $dM \otimes A^{-5/2} dA$, where $dM$ is the standard {\em unit volume} Brownian map measure (see \cite{legalluniqueanduniversal} for the case of triangulations and $2p$-angulations for $p \geq 2$ and \cite{miermontlimit} for the case of quadrangulations); a sample from $dM$ comes equipped with a single marked point.  
The measure  $dM \otimes A^{-5/2} dA$ can be understood as type of grand canonical or Boltzmann measure on the space of (singly marked) Brownian map instances.

Now suppose we consider the set of {\em doubly marked} triangulations such that in addition to the root vertex (the first point on the distinguished oriented edge), there is an additional distinguished or ``marked'' vertex somewhere on the triangulation off the root edge.  Since, given an $n$-triangle triangulation, there are (by Euler's formula) $n/2$ other vertices one could ``mark,'' we find that the number of these doubly marked triangulations is (up to constant factor) given by $n$ times the expression in~\eqref{eqn::probdecay}, i.e.
\begin{equation}
\label{eqn::probdecay2}
\frac{n 2^{n+1} (3n)!}{2n! (2n+2)!} \approx C (27/2)^n n^{-3/2}.
\end{equation}
Let $\mu^2_{\mathrm{TRI}}$ denote this probability measure on doubly marked surfaces (and let $\mu^2_{\mathrm{TRI},k}$ be the obvious the doubly marked analog of $\mu^1_{\mathrm{TRI},k}$).  Then the scaling limit of $\mu^2_{\mathrm{TRI},k}$ is an infinite measure of the form $dM \otimes A^{-3/2} dA$, where $M$ now represents a unit area {\em doubly marked} surface with distinguished points $x$ and $y$.  Note that if one ignores the point $y$, then the law $dM$ in this context is exactly the same as in the one marked point context.

Generalizing the above analysis to $k$ marked points, we will write $\musk$ to denote the natural limiting infinite measure on $k$-marked spheres, which can be understood (up to a constant factor) as the $k$-marked point version of the Brownian map.  To sample from $\musk$, one may
\begin{enumerate}
\item Choose $A$ from the infinite measure $A^{-7/2 + k}dA$.
\item Choose $M$ as an instance of the standard unit area Brownian map.
\item Sample $k$ points independently from the measure with which $M$ is endowed.
\item Rescale the resulting $k$-marked sphere so that it has area $A$.
\end{enumerate}
Of the measures $\musk$, we mainly deal with $\mustwo$ in this paper. As mentioned earlier, we also sometimes use the notation $\musa$ to describe the standard unit-area Brownian map measure, i.e., the measure described as $dM$ above.

\subsubsection{Properties of the doubly marked Brownian map}
\label{subsubsec::doublymarkedproperties}

In this section, we consider what properties of the measure $\mustwo$ on doubly marked measure-endowed metric spaces (as described above) can be readily deduced from considerations of the discrete models and the fact that $\mustwo$ is a scaling limit of such models. These will include the properties contained in the statement of Theorem~\ref {thm::markovmapcharacterization}. Although we will not provide fully detailed arguments here, we note that together with Theorem~\ref {thm::markovmapcharacterization}, this subsection can be understood as a justification of the fact that~$\mustwo$ is the only measure one can reasonably expect to see as a scaling limit of discrete measures such as~$\mu^2_{\mathrm{TRI}}$ (or more precisely as the vague limit of the rescaled measures~$\mu^2_{\mathrm{TRI},k}$). In principle it might be possible to use the arguments of this subsection along with Theorem~\ref {thm::markovmapcharacterization} (or the variant Theorem~\ref{thm::levynetbasedcharacterization}) to give an alternate proof of the fact that the measures~$\mu^2_{\mathrm{TRI}}$ have~$\mustwo$ as their scaling limit. To do this, one would have to show that any subsequential limit of the measures~$\mu^2_{\mathrm{TRI}}$ satisfies the hypotheses of Theorem~\ref {thm::markovmapcharacterization} (or Theorem~\ref{thm::levynetbasedcharacterization}).

We also remind the reader that one well known oddity of this subject is that to date there is no direct proof that the Brownian map (as constructed directly from the Brownian snake) satisfies root invariance. Rather, the existing proofs by Le Gall and Miermont derive root invariance as a consequence of discrete model convergence results \cite{legallgeodesics, legalluniqueanduniversal, miermontlimit, legalluniversallimit}.  Thus the fact that~$\mustwo$ itself satisfies the hypotheses of  Theorem~\ref {thm::markovmapcharacterization} (or Theorem~\ref{thm::levynetbasedcharacterization}) is a result whose existing proofs rely on planar map models.  Very roughly speaking, it is easy to see root invariance in the planar map models, and as one proves that the discrete models converge to the Brownian map (as in \cite{legalluniqueanduniversal, miermontlimit, legalluniversallimit}) one obtains that the Brownian map must be root invariant as well. In this paper, we cite this known fact (the root invariance of the Brownian map) and do not give an independent proof of that. Rather, the main results of this paper are in the other direction: we show that no measure {\em other than}~$\mustwo$ can  satisfy the either the hypotheses of  Theorem~\ref {thm::markovmapcharacterization} or the hypotheses of Theorem~\ref{thm::levynetbasedcharacterization}.

\begin{remark}
\label{rem::alternaterootinvarianceproof}
Although this is not needed for the current paper, we remark that in combination with later works by the authors, Theorem~\ref{thm::levynetbasedcharacterization} can be used to give to give a {\em purely continuum} (non-planar-map-based) proof of Brownian map root invariance based on the Liouville quantum gravity sphere (an object whose root invariance is easy to see directly \cite{matingtrees}).
In other words, the LQG sphere can be made to play the role that the random planar map plays in the earlier (and simpler) arguments by Le Gall and Miermont: it is an obviously-root-invariant object whose connection to the Brownian map can be used to prove the root invariance of the Brownian map itself.

More precisely, root invariance follows from Theorem~\ref{thm::levynetbasedcharacterization} of this paper and the main result of \cite{qle_continuity} because: \begin{enumerate}
\item Theorem~\ref{thm::levynetbasedcharacterization} states that no measure other than~$\mustwo$ satisfies certain hypotheses. 
\item \cite{qle_continuity} constructs a root-invariant measure (from the LQG sphere) and proves that it satisfies those hypotheses.
\item Ergo that measure is $\mustwo$ and $\mustwo$ is root-invariant.
\end{enumerate}

\end{remark}

Let us stress again that all of the properties discussed in this subsection can be proved rigorously for the doubly marked Brownian map measure~$\mustwo$. But for now we are simply using discrete intuition to argue (somewhat heuristically) that these are properties that any scaling limit of the measures~$\mu^2_{\mathrm{TRI}}$ should have.

Although $\mustwo$ is an infinite measure, we have that $\mustwo[ A > c ]$ is finite whenever $c > 0$.  Based on what we know about the discrete models, what other properties would we expect $\mustwo$ to have?  One such property is obvious; namely, the law $\mustwo$ should be invariant under the operation of resampling one (or both) of the two marked points from the (unit) measure on $M$.  This is a property that $\mu^2_{\mathrm{TRI}}$ clearly has.  If we fix $x$ (with its directed edge) and resample $y$ uniformly, or vice-versa, the overall measure is preserved.  Another way to say this is the following: to sample from $dM$, one may first sample $M$ as an unmarked unit-measure-endowed metric space (this space has no non-trivial automorphisms, a.s.) and then choose $x$ and $y$ uniformly from the measure on $M$.

\begin{figure}[ht!]
\begin{center}
\includegraphics[scale=0.85]{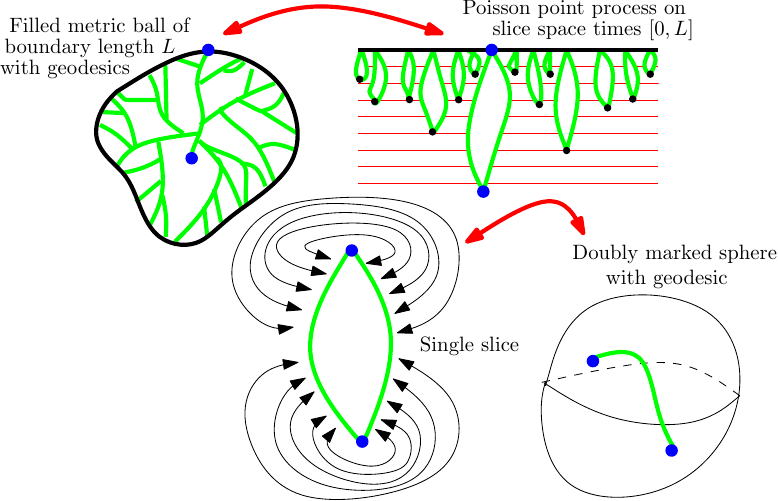}
\caption{\label{fig::slicedecomposition} {\bf Upper left:} A filled metric ball of the Brownian map (with boundary length $L$) can be decomposed into ``slices'' by drawing geodesics from the center to the boundary.  {\bf Upper right:} the slices are embedded in the plane so that along the boundary of each slice, the geodesic distance from the black outer boundary (in the left figure) corresponds to the Euclidean distance below the black line (in the right figure). We may glue the slices back together by identifying points on same horizontal segment (leftmost and rightmost points on a given horizontal level are also identified) to recover the filled metric ball. {\bf Bottom:} Lower figures explain the equivalence of the slice measure and $\mustwo$.}
\end{center}
\end{figure}

Before describing the next properties we expect $\mustwo$ to have, let us define 
$\fb{x}{r}$ to be the set of vertices $z$ with the property that every path from $z$ to $y$ includes a point whose distance from $x$ is less than or equal to $r$. This is the obvious discrete analog of the definition of $\fb{x}{r}$ given earlier. Informally, $\fb{x}{r}$ includes the radius $r$ metric ball centered at $x$ together with all of the components ``cut off'' from $y$ by the metric ball. It is not hard to see that vertices on the boundary of such a ball, together with the edges between them, form a cycle; examples of such boundaries are shown as the red cycles in Figure~\ref{fig::shards}.

Observe that if we condition on $\fb{x}{r}$, and on the event that $d(x,y) > r$ (so that $y \not \in \fb{x}{r}$), then the $\mu^2_{\mathrm{TRI},k}$ conditional law of the remainder of the surface depends only on the boundary length of $\fb{x}{r}$, which we denote by $L_r(x,y)$, or simply $L_r$ when the choice of $x$ and $y$ is understood.  This conditional law can be understood as the standard Boltzmann measure on singly marked triangulations of the disk with boundary length $L_r$, where the probability of each triangulation of the disk with $n$ triangles is proportional to $(27/2)^{-n}$. From this we conclude in particular that $L_r$ evolves as a Markovian process, terminating when $y$ is reached at step $d(x,y)$.  This leads us to a couple more properties one would expect the Brownian map to have, based on discrete considerations.

\begin{enumerate}
\item Fix a constant $r > 0$ and consider the restriction of $\mustwo$ to the event $d(x,y) > r$.  (We expect the total $\mustwo$ measure of this event to be finite.)  Then once $\fb{x}{r}$ is given, the conditional law of the singly marked surface comprising the complement of $\fb{x}{r}$ is a law that depends only a single real number, a ``boundary length'' parameter associated to $\fb{x}{r}$, that we call $L_r$.
\item This law depends on $L_r$ in a {\em scale invariant} way---that is, the random singly marked surface of boundary length $L$ and the random singly marked surface of boundary length $CL$ differ only in that distances and areas in the latter are each multiplied by some power of $C$.  (We do not specify for now what power that is.)  
\item The above properties also imply that the process $L_r$ (or at least its restriction to a countable dense set) evolves as a Markov process, terminating at time $d(x,y)$, and that the $\mustwo$ law of $L_r$ is that of the (infinite) excursion measure associated to this Markov process.
\end{enumerate}

The scale invariance assumptions described above do not specify the law of $L_r$.  They suggest that $\log L_r$ should be a time change of a L\'evy process, but this still leaves an infinite dimensional family of possibilities. In order to draw further conclusions about this law, let us consider the time-reversal of $L_r$, which should also be an excursion of a Markov process. (This is easy to see on a discrete level; suppose we do not decide in advance the value of $T =d(x,y)$, but we observe $L_{T-1}, L_{T-2}, \ldots$ as a process that terminates after $T$ steps.  Then the conditional law of $L_{T-k-1}$ given $L_{T-k},L_{T-k+1},\ldots,L_T$ is easily seen to depend only the value of $L_{T-k}$.) Given this reverse process up to a stopping time, what is the conditional law of the filled ball centered at $y$ with the corresponding radius?

On the discrete level, this conditional law is clearly the uniform measure (weighted by $(27/2)^{-n}$, where $n$ is the number of triangles, as usual) on triangulations of the boundary-length-$L$ disk in which there is a single fixed root and all points on the boundary are equidistant from that root. A sample from this law can be obtained by choosing $L$ independent ``slices'' and gluing them together, see Figure~\ref{fig::shards}. As illustrated in Figure~\ref{fig::slicedecomposition}, we expect to see a similar property in the continuum.  Namely, that given a boundary length parameter $L$, and a set of points along the boundary, the evolution of the lengths within each of the corresponding slices should be an independent process.

This suggests that the time-reversal of an $L_r$ excursion should be an excursion of a so-called {\em continuous state branching process}, as we will discuss in Section~\ref{subsec::csbp}. This property and scale invariance will determine the law of the $L_r$ process up to a single parameter that we will call $\alpha$.

In addition to the spherical-surface measures $\musk$ and $\musa$ discussed earlier, we will in the coming sections consider a few additional measures on disk-homeomorphic measure-endowed metric spaces with a given fixed ``boundary length'' value $L$. (For now we give only informal definitions; see Section~\ref{subsec::mapsdisksnets} for details.)
\begin{enumerate}
\item A probability measure $\mudl$ on boundary length $L$ surfaces that in some sense represents a ``uniform'' measure on all such surfaces --- just as $\musk$ in some sense represents a uniform measure on spheres with $k$ marked points.  It will be enough to define this for $L=1$, as the other values can be obtained by rescaling.
This $L=1$ measure is expected to be an $m \to \infty$ scaling limit of the probability measure on discrete disk-homeomorphic triangulations with boundary length $m$, where the probability of an $n$-triangle triangulation is proportional to $(27/2)^{-n}$.  (Note that for a given large $m$ value, one may divide area, boundary length, and distance by factors of $m^2$, $m$, and $m^{1/2}$ respectively to obtain an approximation of $\mudl$ with $L=1$.) (We remark that another construction of the measure we all $\mudl$ appears in the work by Abrams and Le Gall \cite{alg}.)
\item A measure $\mudonel$ on marked disks obtained by weighting $\mudl$ by area and then choosing an interior marked point uniformly from that area. In the context of  Theorem~\ref{thm::markovmapcharacterization}, this is the measure that should correspond to the conditional law of $S\setminus \fb{x}{r}$ given that the boundary length of $\fb{x}{r}$ is $L$.
\item A measure $\muml$ on disk-homeomorphic measure-endowed metric spaces with a given boundary length $L$ and an interior ``center point'' such that all vertices on the boundary are equidistant from that point. In other words, $\muml$ is a probability measure on the sort of surfaces that arises as a filled metric ball. Again, it should correspond to a scaling limit of a uniform measure (except that as usual the probability of an $n$-triangle triangulation is proportional to $(27/2)^{-n}$) on the set of all marked triangulations of a disk with a given boundary length and the property that all points on the boundary are equidistant from that marked point. This is the measure that satisfies the ``slice independence'' described at the end of the statement of Theorem~\ref{thm::markovmapcharacterization}.
\end{enumerate}

Suppose we fix $r >0$ and restrict the measure $\mustwo$ to the event that $d(x,y) > r$, so that $\mustwo$ becomes a finite measure. Then one expects that given the filled metric ball of radius $r$ centered at $x$, the conditional law of the component containing $y$ is a sample from $\mudonel$, where $L$ is a boundary length measure.  Similarly, suppose one conditions on the {\em outside} of the filled metric ball of radius $d(x,y) - r$ centered at $x$. Then the conditional law of the filled metric ball itself should be $\muml$. This is the measure that one expects (based on the intuition derived from Figures~\ref{fig::shards} and~\ref{fig::slicedecomposition} above) to have the ``slice independence'' property.

\section{Preliminaries}
\label{sec::preliminaries}

\subsection{Metric measure spaces}

A triple $(S, d, \nu)$ is called a {\bf metric measure space} (or {\bf mm-space}) if $(S,d)$ is a complete separable metric space and $\nu$ is a measure on the Borel $\sigma$-algebra generated by the topology generated by $d$, with $\nu(S) \in (0, \infty)$. We remark that one can represent the same space by the quadruple $(S, d, \wt \nu, m)$, where $m = \nu(S)$ and $\wt \nu = m^{-1} \nu$ is a probability measure.  This remark is important mainly because some of the literature on metric measure spaces requires $\nu$ to be a probability measure. Relaxing this requirement amounts to adding an additional parameter $m \in (0, \infty)$.

Two metric measure spaces are considered equivalent if there is a measure-preserving isometry from a full measure subset of one to a full measure subset of the other.  Let~$\mmspace$ be the space of equivalence classes of this form. Note that when we are given an element of~$\mmspace$, we have no information about the behavior of~$S$ away from the support of~$\nu$.

Next, recall that a measure on the Borel $\sigma$-algebra of a topological space is called {\bf good} if it has no atoms and it assigns positive measure to every open set. Let~$\gmsspace$ be the space of geodesic metric measure spaces that can be represented by a triple $(S,d,\nu)$ where $(S,d)$ is a geodesic metric space homeomorphic to the sphere and~$\nu$ is a good measure on~$S$.

Note that if $(S_1, d_1, \nu_1)$ and $(S_2, d_2, \nu_2)$ are two such representatives, then the a.e.\ defined measure-preserving isometry $\phi \colon S_1 \to S_2$ is necessarily defined on a dense set, and hence can be extended to the completion of its support in a unique way so as to yield a continuous function defined on all of $S_1$ (similarly for $\phi^{-1}$). Thus $\phi$ can be uniquely extended to an {\em everywhere} defined measure-preserving isometry.  In other words, the metric space corresponding to an element of $\gmsspace$ is uniquely defined, up to measure-preserving isometry.

As we are ultimately interested in probability measures on $\mmspace$, we will need to describe a $\sigma$-algebra on $\mmspace$. We will also show that $\gmsspace$ belongs to that $\sigma$-algebra, so that in particular it makes sense to talk about measures on $\mmspace$ that are supported on $\gmsspace$. We would like to have a $\sigma$-algebra that can be generated by a complete separable metric, since this would allow us to define regular conditional probabilities for random variables. We will introduce such a $\sigma$-algebra in Section~\ref{subsec::mmsigma}. We first discuss some basic facts about metric spheres in Section~\ref{subsec::metricsphereobservations}.

\subsection{Observations about metric spheres}
\label{subsec::metricsphereobservations}

Let $\gmsspace^k$ be the space of elements of $\gmsspace$ that come endowed with an ordered set of $k$ marked points $z_1, z_2, \ldots, z_k$.  When $j \leq k$ there is an obvious projection map from $\gmsspace^k$ to $\gmsspace^j$ that corresponds to ``forgetting'' the last $k-j$ coordinates. We will be particularly interested in the set $\gmsspace^2$ in this paper, and we often represent an element of $\gmsspace^2$ by $(S, d, \nu, x,y)$ where $x$ and $y$ are the two marked points. The following is a simple deterministic statement about geodesic metric spheres (i.e., it does not involve the measure $\nu$).

\begin{proposition}
\label{prop::boundariesarecircles}
 Suppose that $(S, d)$ is a geodesic metric space which is homeomorphic to $\S^2$ and that $x \in S$.  Then the following hold:
\begin{enumerate} 
\item Each of the components of $S \setminus \ol{B(x,r)}$ has a boundary that is a simple closed curve in $S$, homeomorphic to the circle $\S^1$.
\item Suppose that $\Lambda$ is a connected component of $\partial B(x,r)$.  Then the component of $S \setminus \Lambda$ that contains $x$ is homeomorphic to a disk. Moreover, there exists a homeomorphism from the unit disk to this component that extends  continuously to its boundary. (The restriction of the map to the boundary gives a map from $\S^1$ onto $\Lambda$, which can be interpreted as a closed curve in $S$. This curve may hit or retrace itself but --- since it is the boundary of a disk --- it does not cross itself.)
\end{enumerate}

\end{proposition}
\begin{proof}
We begin with the first item. Let~$U$ be one component of $S \setminus \ol{B(x,r)}$ and consider the boundary set~$\Gamma = \partial U$. We aim to show that~$\Gamma$ is homeomorphic to $\S^1$. Note that every point in~$\Gamma$ is of distance~$r$ from~$x$.

Since~$U$ is connected and has connected complement, it must be homeomorphic to~$\D$.  We claim that the set~$S \setminus \Gamma$ contains only two components: the component~$U$ and another component that is also homeomorphic to~$\D$.  To see this, let us define~$\wt U$ to be the component of~$S \setminus \Gamma$ containing~$x$. By construction,~$\partial \wt U \subseteq \Gamma$, so every point on~$\partial \wt{U}$ has distance~$r$ from~$x$. A geodesic from any {\em other} point in~$\Gamma$ to $x$ would have to pass through~$\partial \wt U$, and hence such a point would have to have distance greater than~$r$ from~$x$. Since all points in~$\Gamma$ have distance~$r$ from~$x$, we conclude that~$\partial \wt{U} = \Gamma$. Note that~$\wt U$ has connected complement, and hence is also homeomorphic to~$\D$.

The fact that~$\Gamma$ is the common boundary of two disjoint disks is not by itself enough to imply that~$\Gamma$ is homeomorphic to $\S^1$. There are still some strange counterexamples (e.g., topologist's sine curves, exotic prime ends, etc). To begin to rule out such things, our next step is to show that~$\Gamma$ is locally connected.

Suppose for contradiction that~$\Gamma$ is not locally connected.  By definition, this means that there exists a $z \in \Gamma$ such that $\Gamma$ is not locally connected at $z$, which in turn means that there exists an~$s > 0$ such that for every sub-neighborhood~$V \subseteq B(z,s)$ containing~$z$ the set~$V \cap \Gamma$ is disconnected.  Note that since~$\Gamma$ is connected the closure of every component of~$\Gamma \cap B(z,s)$ has non-empty intersection with~$\partial B(z,s)$, see Figure~\ref{fig::incursions}.  Since these components are closed within~$B(z,s)$, all but one of them must have positive distance from~$z$.   Moreover, for each~$\epsilon \in (0,s)$, the number of such components which intersect~$B(z,\epsilon)$ must be infinite since otherwise one could take $V$ to be the open set given by $B(z,s)$ minus the union of the components of $\Gamma \cap B(z,s)$ that do not hit $z$, and $V \cap \Gamma$ would be connected by construction, contradicting our non-local-connectedness assumption.

\begin{figure}[ht!]
\begin{center}
\includegraphics[scale=.8]{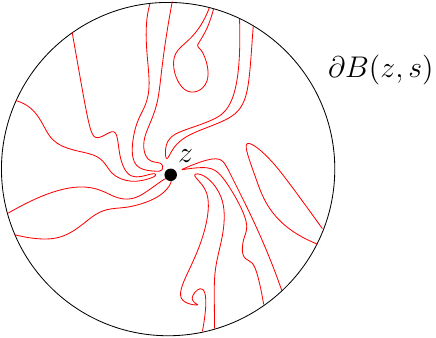}
\caption{\label{fig::incursions} Schematic drawing of $B(z,s)$ (which does not actually have to be ``round'' in a Euclidean sense, or even simply connected) together with $z$ and some possible components of $\Gamma \cap B(z,s)$ colored in red. In light of the non-local-connectedness assumption (assumed for purpose of deriving a contradiction) we have (for some $z$ and $s$) infinitely many red components intersecting $B(z,\epsilon)$ for each $\epsilon<s$. Note that $\Gamma$ is the common boundary of $U$ and $\wt{U}$, each of which is homeomorphic to a disk. Any point on a red component is incident to both $U$ and $\wt{U}$.}
\end{center}
\end{figure}

Now (still assuming that~$\Gamma$ is not locally connected), the above discussion implies that there must be an annulus~$A$ (i.e., a difference between the disk-homeomorphic complements of two concentric filled metric balls) centered at~$z$ such that~$A \cap \Gamma$ contains infinitely many connected components crossing it.  Let $\delta$ be equal to the width of $A$ (i.e., the distance between the inside and outside boundaries of $A$).  It is not hard to see from this that both~$A \cap U$ and~$A \cap \wt{U}$ contain infinitely many distinct components crossing~$A$, each of diameter at least $\delta$.

Let $A_I$ be the inner boundary of $A$ and let $A_M$ be the image of a simple loop $\phi$ in $A$ which has positive distance from $\partial A$ and surrounds $A_I$.  Fix $\epsilon > 0$.  We claim that the above implies that we can find $w \in A_I \cap B(x,r)$ and points $z_1,z_2 \in A_M \cap \partial \wt{U}$ with $d(z_1,z_2) < \epsilon$ such that a given geodesic $\gamma$ which connects $w$ and $x$ necessarily crosses a given geodesic $\eta$ which connects $z_1$ and $z_2$, see Figure~\ref{fig::annuluscrossing}. Indeed, let $(s_j,t_j)$ be the (pairwise disjoint) collection of intervals of time so that each $\phi((s_j,t_j))$ is a component of $A_M \cap \wt{U}$ which disconnects part of the inner boundary (i.e., in $A_I$) of a component of $A \cap \wt{U}$ from its outer boundary (i.e., in the outer boundary of $A$).  We note that for each of the infinitely many components $\wt{V}$ of $A \cap \wt{U}$, there exists at least one $j$ so that $\phi((s_j,t_j)) \subseteq \wt{V}$.  In particular, we can find such a $j$ so that $d(\phi(s_j),\phi(t_j)) < \epsilon$ by the continuity of $\phi$.  Let $z_1 = \phi(s_j)$, $z_2 = \phi(t_j)$, and let $w$ be a point on $A_I$ so that $\phi((s_j,t_j))$ disconnects $w$ from the outer boundary of $A$ in some component of $A \cap \wt{U}$.

\begin{figure}[ht!]
\begin{center}
\includegraphics[scale=.8]{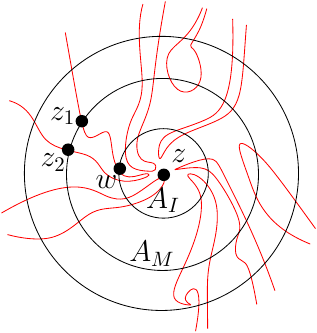}
\caption{\label{fig::annuluscrossing} Schematic drawing of $A_I$ and $A_M$ (which again do not actually have to be ``round'' in a Euclidean sense) together with $z$ and choices for $z_1$, $z_2$, and $w$. We assume that topological considerations imply that a geodesic from $w$ to $x$ (which cannot cross $\Gamma$ except at $x$) has to cross a geodesic connecting $z_1$ and $z_2$. Roughly speaking, one gets a contradiction by noting that this crossing point has to be close to $\Gamma$ (since $z_1$ and $z_2$ can be made arbitrarily close to each other) but also far from $\Gamma$ (since the geodesic starting at $w$ had to travel most of the distance from $A_I$ to $A_M$ before reaching the crossing point).}
\end{center}
\end{figure}

Since $w \in B(x,r)$, we have that $\gamma$ is contained in $B(x,r)$.  Let $v$ be a point on $\gamma \cap \eta$.  Then $d(x,w) = d(x,v) + d(v,w)$.  We claim that $d(v,w) < \epsilon$.  Indeed, if $d(v,w) \geq \epsilon$ then as $d(z_j,v) < \epsilon$ for $j=1,2$ we would have that
\[ d(x,z_j) \leq d(x,v) + d(v,z_j) < d(x,v) + \epsilon \leq d(x,v) + d(v,w) = d(x,w) < r.\]
This contradicts that $z_1,z_2 \notin B(x,r)$, which establishes the claim.  Since $d(v,w) < \epsilon$, we therefore have that 
\[ d(z_j,w) \leq d(z_j,v) + d(v,w) < 2\epsilon.\]
Since $\epsilon > 0$ was arbitrary and $A_I,A_M$ are closed, we therefore have that $A_M \cap A_I \neq \emptyset$.  This is a contradiction since we took $A_M$ to be disjoint from $A_I$.  Therefore $\Gamma$ is locally connected.

Note that the image of $\Gamma$ under a homeomorphism $S \to \S^2$ must be locally connected as well.  Moreover, there is a conformal map $\varphi$ from $\D$ to the image of $\wt{U}$, and a standard result from complex analysis (see e.g.\ \cite[Proposition~3.6]{LAW05}) states that since the image of $\Gamma$ is locally connected, the map $\varphi$ must extend continuously to its boundary. This tells us that $\Gamma$ is given by the image of a continuous curve $\psi \colon \S^1 \to S$.  It remains only to show $\psi(z_1) \neq \psi(z_2)$ for all $z_1,z_2 \in \S^1$.  This will complete the proof because then $\psi$ is a simple curve which parameterizes $\partial U$.

Assume for contradiction that there exists $z_1,z_2 \in \S^1$ distinct so that $\psi(z_1) = \psi(z_2)$.  We write $[z_1,z_2]$ for the counterclockwise segment of $\S^1$ which connects $z_1$ and $z_2$.  Then we have that $\psi$ restricted to each of $[z_1,z_2]$ and $\S^1 \setminus (z_1,z_2)$ is a loop and the two loops touch only at $\psi(z_1) = \psi(z_2)$ by the connectedness of $U$.  Therefore the loops are nested and only one of them separates $U$ from $x$.  We assume without loss of generality that $\psi|_{\S^1 \setminus (z_1,z_2)}$ separates~$U$ from~$x$.  Fix $w \in (z_1,z_2)$, let $\eta$ be a path from $x$ to $w$, and let $t_1$ (resp.\ $t_2$) be the first time that $\eta$ hits $\partial U$ (resp.\ $w$).  Then we have that $t_1 \neq t_2$.  Applying this to the particular case of a geodesic from $x$ to $w$, we see that the distance of $x$ to $w$ is strictly larger than the distance of $\partial U$ to $w$.  This a contradiction, which completes the proof of the first item in the theorem statement.  To prove the second item, we apply exactly the same argument above with $\Lambda$ in place of $\Gamma$ in order to show that $\Lambda$ is locally connected, which implies, as above, that the map from the unit disk to the $x$-containing component of $S \setminus \Lambda$ extends continuously to the boundary.
\end{proof}

As mentioned earlier, given a doubly marked geodesic metric space $(S, d, x, y)$ which is homeomorphic to~$\S^2$, we let $\fb{x}{r}$ denote the filled metric ball of radius~$r$ centered at~$x$, as viewed from~$y$. That is, $\fb{x}{r}$ is the complement of the $y$-containing component of $S \setminus \ol{B(x,r)}$.

Fix some~$r$ with $0 < r < d(x,y)$, and a point $z \in \partial \fb{x}{r}$.  Clearly, any geodesic from~$z$ to~$x$ is a path contained in~$\fb{x}{r}$. In general there may be more than one such geodesic, but the following proposition gives us a way to single out a unique geodesic.

\begin{proposition}
\label{prop::uniqueleftmost}
Suppose that $(S, d, x,y)$ is a doubly marked geodesic metric space which is homeomorphic to $\S^2$, that $0 < r < d(x,y)$, and that $\fb{x}{r}$ is the radius $r$ filled ball centered at $x$ and $z \in \partial \fb{x}{r}$. Assume that an orientation of $\partial \fb{x}{r}$ is fixed (so that one can distinguish the ``clockwise'' and ``counterclockwise'' directions). Then there exists a unique geodesic from $z$ to $x$ that is {\em leftmost} viewed from $x$  (i.e., furthest counterclockwise) when lifted and understood as an element of the universal cover of $\fb{x}{r} \setminus \{x\}$.
\end{proposition}
\begin{proof}
Proposition~\ref{prop::boundariesarecircles} implies that $\fb{x}{r}$ is homeomorphic to $\ol{\D}$.  Therefore $\fb{x}{r} \setminus \{x\}$ is homeomorphic to $\ol{\D} \setminus \{0\}$.  It thus follows that the universal cover of $\fb{x}{r} \setminus \{x\}$ is homeomorphic to $\ol{\h}$.  Let $\pi \colon \ol{\h} \to \fb{x}{r} \setminus \{x\}$ be the associated projection map.  Let $z$ be as in the statement of the proposition and let $z' \in \R$ be a preimage of $z$ with respect to $\pi$ (i.e., $\pi(z') = z$).  Note that for each $r' \in (0,r)$, the lifting of $\partial \fb{x}{r'}$ to the universal cover $\ol{\h}$ is homeomorphic to~$\R$ (by Proposition~\ref{prop::boundariesarecircles} and since~$\R$ is the lifting of the circle to its universal cover).  Let $z_r'$ be the leftmost (i.e., furthest counterclockwise) point in $\ol{\h}$ reachable by the lifting of any geodesic connecting $z$ to $x$ taken to start from $z'$.  We claim that $s \mapsto \pi(z_{r-s}')$ for $s \in [0,r]$ forms the desired leftmost geodesic.  By definition, it is to the left of any geodesic connecting $z$ to $x$ as in the statement of the proposition.  It therefore suffices to show that it is in fact a geodesic from $z$ to $x$.

Suppose that $\eta_1,\eta_2$ are geodesics from $z$ to $x$.  Then there exists a geodesic $\eta$ from $z$ to $x$ which is to the left of $\eta_1,\eta_2$.  Indeed, let $\eta_1',\eta_2'$ be the liftings of $\eta_1,\eta_2$ to $\ol{\h}$ starting from $z'$.  Let $I = \cup_j (s_j,t_j)$ be the set of times so that $\eta_1' \neq \eta_2'$ where the $(s_j,t_j)$ are pairwise disjoint.  In each such interval, we have that $\eta_1',\eta_2'$ do not intersect and therefore one of the paths is to the left of other in $\ol{\h}$.  We take $\eta'$ in $(s_j,t_j)$ to be the leftmost of these two paths.  Outside of $I$, we take $\eta'$ be equal to the common value of $\eta_1$ and $\eta_2$.  Then we take $\eta = \pi(\eta')$.  Then $\eta$ is a geodesic from $z$ to $x$ which is to the left of $\eta_1$ and $\eta_2$.

For each $s \in (0,r)$, there exists a sequence of geodesics $(\eta_n)$ from $z$ to $x$ such that if $\eta_n'$ is the lifting of $\eta_n$ to $\ol{\h}$ starting from $z'$ then $\eta_n'(s)$ converges to $z_{r-s}'$ as $n \to \infty$.  By the Arzel\'a-Ascoli theorem, by passing to a subsequence, we may assume without loss of generality that $(\eta_n)$ converges in the limit to a geodesic $\eta$ connecting $z$ to $x$ whose lifting $\eta'$ starting from $z'$ passes through $z_{r-s}'$ at time $s$.  By combining this with the statement proved in the previous paragraph, we see that there exists a geodesic $\eta$ so that its lifting $\eta'$ to $\ol{\h}$ starting from $z'$ passes through all of the $z_{r-s}'$, as desired. 

\end{proof}

We next establish some ``rigidity'' results for metric spaces.  Namely, we will first show that there is no non-trivial isometry of a geodesic closed-disk-homeomorphic metric space which fixes the boundary.  We will then show that the identity map is the only orientation-preserving isometry of a triply marked geodesic sphere that fixes all of the marked points.  (Note that there can be many automorphisms of the unit sphere that fix two marked points if those points are on opposite poles.) We will note that it suffices to fix two points if one also fixes a distinguished geodesic between them.

\begin{proposition}
\label{prop::boundaryfixesisometry}
Suppose that $(S,d)$ is a geodesic metric space such that there exists a homeomorphism $\varphi \colon \ol{\D} \to S$.  Suppose that $\phi \colon S \to S$ is an isometry which fixes $\partial S := \varphi(\partial \D)$.  Then $\phi(z) = z$ for all $z \in S$.
\end{proposition}
\begin{proof}
Fix $x_1,x_2,x_3 \in \partial S$ distinct.  Then $x_1,x_2,x_3$ determine an orientation of $\partial S$.  Thus for $x \in \partial S$ and $z \in S$, we have a well-defined leftmost geodesic $\gamma$ connecting~$z$ to~$x$ with respect to this orientation.  Since $\phi$ fixes $\partial S$, it preserves the orientation of $\partial S$.  In particular, if it is true that $\phi(z) = z$ then it follows that $\phi$ must fix $\gamma$ (for otherwise we would have more than one leftmost geodesic from $z$ to $x$).  We conclude that $\{z : \phi(z) = z\}$ is connected and connected to the boundary, and hence its complement must have only simply connected components.  Moreover, if $U$ is such a component then we have that $\phi(U) = U$.  Brouwer's fixed point theorem implies that none of these components can be non-empty, since there would necessarily be a fixed point inside.  This implies that $\phi(z) = z$ for all $z \in S$.
\end{proof}

\begin{proposition}
\label{prop::threepointsfixisometry}
Suppose that $(S,d,x_1,x_2,x_3)$ is a triply marked geodesic metric space with $x_1,x_2,x_3$ distinct which is topologically equivalent to $\S^2$.  We assume that~$S$ is oriented so that we can distinguish the clockwise and counterclockwise directions of simple loops.  Suppose that $\phi \colon S \to S$ is an orientation-preserving isometry with $\phi(x_j) = x_j$ for $j=1,2,3$.  Then $\phi(z) = z$ for all $z \in S$.  Similarly, if $(S,d,x_1, x_2)$ is a doubly marked space with $x_1,x_2$ distinct and $\gamma$ is a geodesic from $x_1$ to $x_2$, then the identity is the only orientation-preserving isometry that fixes $x_1$, $x_2$, and $\gamma$.
\end{proposition}
\begin{proof}
The latter statement is immediate from Proposition~\ref{prop::boundaryfixesisometry} applied to the disk obtained by cutting the sphere along $\gamma$.  To prove the former statement, we assume without loss of generality that $R = d(x_1,x_2) \leq d(x_1,x_3)$.

We first consider the case that $x_2$ is on a geodesic from $x_1$ to $x_3$.  Consider the filled metric ball $\fb{x_1}{R}$ (relative to $x_3$) so that $x_2 \in \partial \fb{x_1}{R}$.  Since we have assumed that~$S$ is oriented, we have that $\partial \fb{x_1}{R}$ is oriented, hence Proposition~\ref{prop::uniqueleftmost} implies that there exists a unique leftmost geodesic~$\gamma$ from~$x_1$ to~$x_2$.  Since~$\phi$ fixes $x_1,x_3$ and~$\phi$ is an isometry, it follows that~$\phi$ fixes $\partial \fb{x_1}{R}$.  Moreover, $\phi(\gamma)$ is a geodesic from $\phi(x_1) = x_1$ to $\phi(x_2) = x_2$.  As $\phi$ is orientation preserving, we must in fact have that $\phi(\gamma) = \gamma$.  Therefore the latter part of the proposition statement implies that, in this case, $\phi$ fixes all of~$S$.

We next consider the case that $x_2$ is not on a geodesic from $x_1$ to $x_3$.  Let $A$ be the union of all of the geodesics from $x_1$ to $x_3$ and note that $A$ is closed.  Moreover, the boundary of the component $U$ of $S \setminus A$ containing $x_2$ consists of two geodesics: one ($\gamma_L$) which passes to the left of $x_2$ and one ($\gamma_R$) which passes to the right of $x_2$.  Since $\phi$ fixes $x_1$, $x_2$, and $x_3$ it follows that $\phi$ fixes $U$.  As $\phi$ is orientation preserving, it also fixes both $\gamma_L$ and $\gamma_R$.  Therefore the latter part of the proposition statement implies that, in this case, $\phi$ fixes all of~$S$.
\end{proof}

We remark that the above argument implies that the identity is the only map that fixes $x$ and the restriction of $\gamma$ to {\em any} neighborhood about $x$. In other words, the identity is the only map that fixes $x$ and the equivalence class of geodesics $\gamma$ that end at $x$, where two geodesics considered equivalent if they agree in a neighborhood of $x$. This is analogous to the statement that a planar map on the sphere has no non-trivial automorphisms (as a map) once one fixes a single oriented edge. We next observe that Proposition~\ref{prop::boundaryfixesisometry} can be further strengthened.

\begin{proposition}
\label{prop::boundaryfixesisometry2}
In the context of Proposition~\ref{prop::boundaryfixesisometry}, if the isometry $\phi \colon S \to S$ is orientation preserving and fixes one point $x \in \partial S$ it must be the identity.
\end{proposition}
\begin{proof}
By Proposition~\ref{prop::boundaryfixesisometry}, it suffices to check that~$\phi$ fixes the circle $\partial S$ pointwise (since~$\phi$ is a homeomorphism, it clearly fixes $\partial S$ as a set).  Note that the set $\{y \in \partial S : \phi(y) = y\}$ is closed and non-empty.  Suppose for contradiction that $\{ y \in \partial S : \phi(y) = y\}$ is not equal to all of $\partial S$.  Then there exists $I \subseteq \partial S$ connected which is relatively open in $\partial S$ such that $\phi$ fixes the endpoints $z_1,z_2$ of $I$ but does not fix any point in $I$ itself.  Fix $\epsilon > 0$ small so that there exists $z \in I$ with $d(z,z_1) = \epsilon$.  Then there is a well-defined first point $w \in I$ starting from $z_1$ with $d(z_1,w) = \epsilon/2$.  Since $\phi$ fixes $I$ as a set, it must be that $\phi(w)=w$.  This is a contradiction, which gives the result.
\end{proof}

We now return to our study of leftmost geodesics.

\begin{proposition}
\label{prop::leftmostconvergence}
Suppose that we are in the setting of Proposition~\ref{prop::uniqueleftmost}.  Suppose that $a \in  \partial \fb{x}{r}$ and that $(a_j)$ is a sequence of points in $\partial \fb{x}{r}$ which approach $a$ from the left.  For each $j$, we let $\gamma_j$ be the leftmost geodesic from $a_j$ to $x$ and $\gamma$ the leftmost geodesic from $a$ to $x$.  Then we have that $\gamma_j \to \gamma$ uniformly as $j \to \infty$.  Moreover, for all but countably many values of $a$ (which we will call {\bf jump values}) the same is true when the $a_j$ approach $a$ from the right.  If $a$ is one of these jump values, then the limit of the geodesics from $a_j$, as the $a_j$ approach $a$ from the right, is a non-leftmost geodesic from $a$ to $x$.
\end{proposition}
\begin{proof}
Suppose that the $(a_j)$ in $\partial \fb{x}{r}$ approach $a \in \partial \fb{x}{r}$ from the left and~$(\gamma_j)$, $\gamma$ are as in the statement.  Suppose that $(\gamma_{j_k})$ is a subsequence of $(\gamma_j)$.  It suffices to show that $(\gamma_{j_k})$ has a subsequence which converges uniformly to~$\gamma$.  The Arzel\'a-Ascoli theorem implies that $(\gamma_{j_k})$ has a subsequence which converges uniformly to some limiting path $\wt{\gamma}$ connecting $a$ to $x$.  This path is easily seen to be a geodesic connecting $a$ to $x$ which is non-strictly to the left of $\gamma$.  Since $\gamma$ is leftmost, we conclude that $\gamma = \wt{\gamma}$.  This proves the first part of the proposition.

Suppose now that the $(a_j)$ approach $a$ from the right and let $\gamma_j,\gamma$ be as in the previous paragraph.  The Arzel\'a-Ascoli theorem implies that every subsequence of $(\gamma_j)$ has a further subsequence which converges uniformly to a geodesic connecting~$a$ to~$x$.  That the limit does not depend on the subsequence follows by monotonicity.

To prove the second part of the proposition, note that each jump value $a$ is associated with the non-empty open set $J_a \subseteq \fb{x}{r}$ which is between the leftmost geodesic from~$a$ to~$x$ and the uniform limit of leftmost geodesics along any sequence $(a_j)$ approaching~$a$ from the right.  Moreover, for distinct jump values $a,a'$ we must have that $J_a \cap J_{a'} = \emptyset$.  Therefore the set of jump values is countable.
\end{proof}

As in the proof of Proposition~\ref{prop::leftmostconvergence}, if~$a$ is a jump value, we let $J_a$ denote the open set bounded between the (distinct) left and right limits described in Proposition~\ref{prop::leftmostconvergence}, both of which are geodesics from $a$ to $x$.  Recall that if $a,a'$ are distinct jump values then~$J_a$, $J_{a'}$ are disjoint.  Moreover, observe that the union of the~$J_a$ (over all jump values~$a$) is the complement of the closure of the union of all leftmost geodesics.  As the point~$a$ moves around the circle, the leftmost geodesic from $a$ to $x$ may vary continuously (as it does when $(S,d)$ is a Euclidean sphere) but it may also have countably many times when it ``jumps'' over an open set~$J_a$ (as is a.s.\ the case when $(S,d, \nu)$ is an instance of  the Brownian map, see Section~\ref{sec::brownianmap}).

We next need to say a few words about ``cutting'' geodesic metric spheres along curves and/or ``welding'' closed geodesic metric disks together. Before we do this, let us consider the general question of what it means to take a quotient of a metric space w.r.t.\ an equivalence relation (see \cite[Chapter 3]{courseinmetricgeometry} for more discussion on this point). Given any metric space $(S,d)$ and any equivalence relation $\cong$ on $S$, one may define a distance function $\ol{d}$ between equivalence classes of $\cong$ as follows: if $a$ and $b$ are representatives of distinct equivalence classes, take $\ol{d}(a,b)$ to be the infimum, over even-length sequences $a = x_0, x_1, x_2, \ldots, x_{2k} = b$ with the property that $x_m \cong x_{m+1}$ for odd $m$, of the sum
\[ \sum_{m=0}^{k-1} d(x_{2m}, x_{2m+1}).\]
This $\ol{d}$ is {\em a priori} only a pseudometric on the set of equivalence classes of $\cong$ (i.e., it may be zero for some distinct $a$ and $b$). However, it defines a metric on the set of equivalence classes of $\cong^*$ where $a \cong^* b$ whenever $\ol{d}(a,b) = 0$. It is not hard to see that $\ol{d}$ is the largest pseudometric such that $\ol{d}(a,b) \leq d(a,b)$ for all $a,b$ and $d(a,b) = 0$ when $a \cong b$. The procedure described above is what we generally have in mind when we speaking of taking a quotient of a metric space w.r.t.\ an equivalence relation.

Now let us ask what happens if a geodesic metric sphere is {\em cut} along a simple loop~$\Gamma$, to produce two disks. Note that on each disk, there is an {\em interior-internal metric}, where the distance between points~$a$ and~$b$ is defined to be the length of the shortest path that stays entirely within the given disk. This distance is clearly finite when~$a$ and~$b$ are in the interior of the disk. (This can be deduced by taking a closed path from~$a$ to~$b$ bounded away from the disk boundary, covering it with open metric balls bounded away from the disk boundary, and taking a finite subcover.) However, when either~$a$ or~$b$ is on the boundary of the disk, it is not hard to see that (if the simple curve is windy enough) it could be infinite.

Let us now ask a converse question. What happens when we take the two metric disks and try to ``glue them together'' to recover the sphere? We can clearly recover the sphere as a topological space, but what about the metric? Before we address that point, note there is always {\em one} way to glue the disks back together to create a new metric space: namely, we may consider the disjoint union of the pair of disks to be a common metric space (with the distance between points on distinct disks formally set to be infinity) and then take a metric quotient (in the sense discussed above) w.r.t.\ the equivalence relation that identifies the boundary arcs.  This can be understood as the largest metric compatible with the boundary identification. In this metric, the distance between~$a$ and~$b$ is the infimum of the lengths (in the original metric) of paths from~$a$ to~$b$ that only cross~$\Gamma$ finitely many times. However, one can actually construct a geodesic metric sphere with a closed curve~$\Gamma$ and points~$a$ and~$b$ such that the shortest path from~$a$ to~$b$ that crosses~$\Gamma$ finitely many times is {\em longer} than the shortest path overall.\footnote{For example, consider the ordinary Euclidean metric sphere and let $\Gamma$ be the equator curve.  The equator comes with a Lebesgue length measure; let $A$ be a closed positive-Lebesgue-measure subset of the equator whose complement is dense within the equator. Let $d$ be the largest metric on $\C$ for which the $d$ length of any rectifiable path is the Euclidean length of the portion of that path that does not lie in $A$. (Informally, $d$ is the ordinary Euclidean metric modified so that there is no cost for travel within $A$.) Then $d$ is topologically equivalent to the original metric; but any path between points $a$ and $b$ that intersects $\Gamma$ only finitely many times will have the same length in both metrics, despite the fact that the distance between $a$ and $b$ may be different in the two metrics.} In other word, there are situations where cutting a metric sphere into two disks and gluing the disks back together (using the quotient procedure described above) does not reproduce the original sphere.

On the other hand, it is easy to see that this type of pathology does not arise if~$\Gamma$ is a curve comprised of a finite number of geodesic arcs, since one can easily find a geodesic~$\gamma$ between any points $a$ and $b$ that crosses no geodesic arc of~$\Gamma$ more than once. (If it crosses an arc multiple times, one may replace the portion of~$\gamma$ between the first and last hitting times by a portion of the arc itself.) The same applies if one has a disk cut into two pieces using a finite sequence of geodesic arcs.  This is an important point, since in this paper we will frequently need to glue together disk-homeomorphic ``slices'' whose boundaries are geodesic curves. The following proposition formalizes one example of such a statement.

\begin{proposition}
\label{prop::recover_sphere}
Suppose that $(S,d,x,y)$ is a doubly marked geodesic metric space which is homeomorphic to $\S^2$.  Suppose that $\gamma_1,\gamma_2$ are distinct geodesics which connect $x$ to $y$ and that $S \setminus (\gamma_1 \cup \gamma_2)$ has two components $U_1,U_2$.  For $j=1,2$, let $x_j$ (resp.\ $y_j$) be the first (resp.\ last) point on $\partial U_j$ visited by $\gamma_1$ (or equivalently by $\gamma_2$).  We then let $(U_j,d_j,x_j,y_j)$ be the doubly marked metric space where $d_j$ is given by the interior-internal metric induced by $d$ on $U_j$.  Let $\wt{S}$ be given by the disjoint union of $\ol{U}_1$ and $\ol{U}_2$ and let $\wt{d}$ be the distance on $\wt{S}$ which is defined by $\wt{d}(a,b) = d_j(a,b)$ if $a,b \in \ol{U}_j$ for some $j=1,2$, otherwise $\wt{d}(a,b) = \infty$.  We then define an equivalence relation $\cong$ on $\wt{S}$ by declaring that $a \cong b$ if either $a=b$ or if $a \in \partial U_1$ corresponds to the same point $b \in \partial U_2$ in $S$.  Let $\ol{d}$ be the largest metric compatible with $\wt{S}/\cong$.  Then $\ol{d} = d$.  That is, the metric gluing of the $(U_j,d_j,x_j,y_j)$ along their boundaries gives $(S,d,x,y)$.
\end{proposition}

For future reference, let us remark that another instance where this pathology will not arise is when $(S,d,x)$ is an instance of a Brownian map with a marked point $x$ and $\Gamma$ is the boundary of a filled metric ball centered at $x$. In that case, the definition of~$d$ given in Section~\ref{subsec::spheresassymetric} will imply that the length of the shortest path between points $a$ and $b$ is the infimum over the lengths of paths comprised of finitely many arcs, each of which is a segment of a geodesic from some point to $x$.  By definition, such a path clearly only crosses $\Gamma$ finitely many times. Note that the two situations discussed above (cutting along geodesics and along boundaries of filled metric balls) are precisely those that are needed to make sense of the statements in Theorem~\ref{thm::markovmapcharacterization}.

In this article we will not rule out the possibility that the interior-internal metric associated with $S \setminus \fb{x}{r}$ defines an infinite diameter metric space. ({\it Update:} This has been subsequently ruled out in the works \cite{bettinelli_miermont_disks,lg2019disksnake}.  Indeed, \cite{bettinelli_miermont_disks} gives a Brownian snake construction of the Brownian disk from which it is immediate that the Brownian disk has finite diameter and \cite{lg2019disksnake} shows that this definition is equivalent to the filled metric ball complement considered here.) Let us note, however, that one can recover the entire collection of geodesics back to~$x$ (hence~$d$) from the interior-internal metrics associated with $S \setminus \fb{x}{r}$ and~$\fb{x}{r}$.  In particular, if $z \in S \setminus \fb{x}{r}$ then by the very definition of $\fb{x}{r}$ we have that the distance between~$z$ and $\partial \fb{x}{r}$ is finite and given by $d(x,z)-r$.  Moreover, the shortest paths from $z$ to $\partial \fb{x}{r}$ in $S \setminus \fb{x}{r}$ comprise of the initial $(d(x,z)-r)$-length segments of the geodesics from~$z$ to~$x$.  It is clearly the case that the remaining $r$-length segments of the geodesics from~$z$ to~$x$ are contained in~$\fb{x}{r}$.  

{\it Update:} Pathologies of the aforementioned type were ruled out in other settings for natural gluing operations one can perform for Brownian and $\sqrt{8/3}$-LQG surfaces in \cite{gwynne-miller:gluing}, which together with \cite{gwynne-miller:saw,gwynne-miller:uihpq} has led to a proof that the self-avoiding walk on random quadrangulations converges to $\SLE_{8/3}$ on $\sqrt{8/3}$-LQG.

\subsection{A consequence of slice independence/scale invariance}
\label{subsec::slice_independence}

At the end of Section~\ref{subsec::discreteintuition}, the measure $\muml$ is informally described, along with a notion of ``slice independence'' one might expect such a measure to satisfy. Although we have not given a formal description of $\muml$ yet, we can observe now some properties we would expect this measure to have. For concreteness, let us assume that $L=1$ and that a point on the boundary is fixed, so that the boundary of a sample from $\muml$ can be identified with the interval $[0,1]$.  We ``cut'' along the leftmost geodesic from $0$ to $x$ and view a sample from $\muml$ as a ``triangular slice'' with one side identified with $[0,1]$ and the other two sides forming geodesics of the same length (one from $0$ to $x$ and one from $1$ to $x$).

We define $\wt{d}(a,b)$ to be the distance from the boundary at which the leftmost geodesic from $a$ to $x$ and the leftmost geodesic from $b$ to $x$ merge.  Now, no matter what space and $\sigma$-algebra $\muml$ is defined on, we would expect that if we restrict to rational values of $a$ and $b$, then the $\wt{d}(a,b)$ should be a countable collection of real-valued random variables. Before we even think about $\sigma$-algebras on $\mmspace$ or $\gmsspace$, we can answer a more basic question. What would ``slice independence'' and ``scale invariance'' assumptions tell us about the joint law of these random variables $\wt{d}(a,b)$? The following proposition formalizes what we mean by scale invariance and slice independence, and shows that in fact these properties characterize the joint law of the random variables $\wt{d}(a,b)$ up to a single real parameter.  As we will see in the proof of Theorem~\ref{thm::markovmapcharacterization}, this will allow us to deduce that the metric net associated with a space which satisfies the hypotheses of Theorem~\ref{thm::markovmapcharacterization} is related to the so-called L\'evy net introduced in Section~\ref{sec::surfacesfromtrees} below.

\begin{proposition} 
\label{prop::poissonslicestructure}
Consider a random function $\wt{d}$ defined on all pairs $(a,b) \in (\Q \cap [0,1])^2$ such that
\begin{enumerate}
\item $\wt{d}(a,b) = \wt{d}(b,a)$ for all $a,b \in \Q \cap [0,1]$
\item If $a,b,c,d \in \Q \cap [0,1]$ with $a<b$ and $c<d$ then $\wt{d}(a,b)$ and $\wt{d}(c,d)$ are independent provided that $(a,b)$ and $(c,d)$ are disjoint.
\item $\wt{d}(a,a) = 0$ a.s.\ for all $a \in \Q \cap [0,1]$
\item If $a<b<c$ are in $\Q \cap [0,1]$ then $\wt{d}(a,c) = \max\bigl( \wt{d}(a,b), \wt{d}(b,c) \bigr)$.
\item The law of $\wt{d}(a,b)$ depends only on $|b-a|$.  In fact, there is some $\beta$ so that for any $a$ and $b$ the law of $\wt{d}(a,b)$ is equivalent to the law of $|a-b|^{\beta} \wt{d}(0,1)$.
\end{enumerate}
Then the law of $\wt{d}(a,b)$ has a particular form. Precisely, one can construct a sample from this law as follows.  First choose a collection of pairs $(s,x)$ as a Poisson point process on $[0,1] \times \R_+$ with intensity $ds \otimes x^{\alpha}dx$ where $\alpha = -1/\beta - 1$ and $ds$ (resp.\ $dx$) denotes Lebesgue measure on $[0,1]$ (resp.\ $\R_+$). Then define $\wt{d}(a,b)$ to be the largest value of $x$ such that $(s,x)$ is a point in this point process for some $s \in (a,b)$.
\end{proposition}
\begin{proof}
The lemma statement describes two ways of choosing a random $\wt{d}$ and asserts that the two laws agree.  It is immediate from Lemma~\ref{lem::stabilityundersuprema} (stated and proved just below) that the laws agree when one restricts attention to $\{0,1/k,2/k,\ldots, 1\}^2$, for any $k \in \N$.  Since this holds for all $k$, the result follows. 
\end{proof}

\begin{lemma}
\label{lem::stabilityundersuprema}
Suppose for some $\beta>0$, a real-valued random variable $A$ has the following property.  When $A_1, A_2, \ldots, A_k$ are i.i.d.\ copies of $A$, the law of $k^{-\beta} \max_{1 \leq i \leq k} A_i$ is the same as the law of $A$.  Then $A$ agrees in law (up to some multiplicative constant) with the size of the maximum element of a Poisson point process chosen from the infinite measure $x^{\alpha} dx$, where $\alpha=  -1/\beta-1$ and $dx$ denotes Lebesgue measure on $\R_+$.
\end{lemma}
\begin{proof}
Let $F$ be the cumulative distribution function of $A$, so that $F(s) = \p[A \leq s]$.  Then
\[ F(s) = \p[ A \leq s] = \p[k^{-\beta} A \leq s]^k = F(k^\beta s)^k.\]
Thus $F(k^\beta s) = F(s)^{1/k}$.  Set $r =k^\beta$ so that $1/k = r^{-1/\beta}$.  Then when $r$ has this form we have $F(r s) = F(s)^{1/k} = F(s)^{r^{-1/\beta}}$.  Applying this twice allows us to draw the same conclusion when $r = k_1^\beta/ k_2^\beta$ for rational $k = k_1/k_2$, i.e., for all values $r$ which are a $\beta$th power of a rational.  Since this is a dense set, we can conclude that in general, if we set $e^t = F(1)$, we have
\begin{equation}
\label{eqn::a_dist_function}
F(r) = e^{t r^{-1/\beta}}.
\end{equation}
It is then straightforward to see that this implies that (up to a multiplicative constant) $A$ has the same law as the Poisson point process maximum described in the lemma statement.  (See, e.g., \cite[Exercise~22.4]{satolevyprocesses}.)
\end{proof}

\subsection{A $\sigma$-algebra on the space of metric measure spaces}
\label{subsec::mmsigma}

We present here a few general facts about measurability and metric spaces, following up on the discussion in Section~\ref{subsec::outline}. Most of the basic information we need about the Gromov-Prohorov metric and the Gromov-weak topology can be found in \cite{grevenpfaffelhuberwinter}. Other related material can be founded in the metric geometry text by Burago, Burago, and Ivanov \cite{courseinmetricgeometry}, as well as Villani's book \cite[Chapters 27-28]{villanibook}.

As in Section~\ref{subsec::outline}, let $\mmspace$ denote the space of metric measure spaces, defined modulo a.e.\ defined measure preserving isometry. Suppose that $(S, d, \nu)\in \mmspace$. If we choose points $x_1, x_2, \ldots, x_k$ i.i.d.\ from $\nu$, then we obtain a $k \times k$ matrix of distances $d_{ij} = d(x_i, x_j)$ indexed by $i,j \in \{1,2,\ldots,k\}$.  Denote this matrix by $M_k= M_k(S,d,\nu)$.

If $\psi$ is any fixed bounded continuous function on $\R^{k^2}$, then the map
\[ (S,d,\nu) \to \E_\nu [\psi(M_k)]\]
is a real-valued function on $\mmspace$.  The Gromov-weak topology is defined to be the weakest topology w.r.t.\ which the functions of this type are continuous.  In other words, a sequence of elements of $\mmspace$ converge in this topology if and only if the laws of the corresponding $M_k$ (understood as measures on $\R^{k^2}$) converge weakly for each $k$.  We denote by $\mmsigma$ the Borel $\sigma$-algebra generated by this topology.  Since we would like to be able to sample marked points from $\nu$ and understand their distances from each other, we feel comfortable saying that $\mmsigma$ is the weakest ``reasonable'' $\sigma$-algebra we could consider.  We will sometimes abuse notation and use $(\gmsspace, \mmsigma)$ to denote a measure space, where in this context $\mmsigma$ is understood to refer to the intersection of $\mmsigma$ with the set of subsets of $\gmsspace$. (We will apply a similar notational abuse to the ``marked'' analogs $\mmspace^k$, $\gmsspace^k$, and $\mmsigma^k$ introduced below.)
   
It turns out that the Gromov-weak topology can be generated by various natural metrics that make $\mmspace$ a complete separable metric space: the so-called Gromov-Prohorov metric and the Gromov-$\underline \Box_1$ metric \cite{grevenpfaffelhuberwinter, gromovprohorovandgromovbox}.  Thus, $(\mmspace, \mmsigma)$ is a {\em standard Borel space} (i.e., a measure space whose $\sigma$-algebra is the Borel $\sigma$-algebra of a topology generated by a metric that makes the space complete and metrizable).  We do not need to discuss the details of these metrics here. We bring them up in order to show that $(\mmspace, \mmsigma)$ is a standard Borel space. One useful consequence of the fact that $(\mmspace, \mmsigma)$ is a standard Borel space is that if $\CG$ is any sub-$\sigma$-algebra of $\mmsigma$, then the regular conditional probability of a random variable, conditioned on $\CG$, is well-defined \cite[Chapter~5.1.3]{durrettprobtext}. 

We can also consider {\em marked} spaces; one may let $\mmspace^k$ denote the set of tuples of the form $(S,d, \nu, x_1,x_2, \ldots, x_k) $ where $(S,d,\nu) \in \mmspace$ and $x_1, x_2, \ldots, x_k$ are elements (``marked points'') of $S$.  Given such a space, one may sample additional points $x_{k+1}, x_{k+2}, \ldots, x_m$ i.i.d.\ from $\nu$ and consider the random matrix $M_m$ of distances between the $x_i$. One may again define a Gromov-weak topology on the marked space to be the weakest topology w.r.t.\ which expectations of bounded continuous functions of $M_m$ are continuous. We let $\mmsigma^k$ denote the Borel $\sigma$-algebra of the marked space. Clearly for any $m > k$ one has a measurable map $\mmspace^m \to \mmspace^k$ that corresponds to ``forgetting'' the last $m-k$ points. One can similarly define $\mmspace^\infty$ to be the space of $(S, d, \nu, x_1, x_2, \ldots)$ with an $x_j$ defined for all positive integer $j$. The argument that these spaces are standard Borel is essentially the same as in the case without marked points. One immediate consequence of the definition of the Gromov-weak topology is the following:

\begin{proposition}
\label{prop::randomapproximatemetricsconverge}
Fix $(S,d,\nu) \in \mmspace$ with $\nu(S) = 1$. Let $x_1, x_2, \ldots$ be i.i.d.\ samples from $\nu$.  Let $(S_{m}, d_{m}, \nu_{m})$ be defined by taking $S_{m} = \{x_1, x_2, \ldots, x_{m} \}$, letting $d_{m}$ be the restriction of $d$ to this set, and letting $\nu_{m}$ assign mass $1/m$ to each element of $S_{m}$. Then $(S_{m}, d_{m}, \nu_{m})$ converges to $(S, d, \nu)$ a.s.\ in the Gromov-weak topology.  A similar statement holds for marked spaces.  If $k < m$ and $(S,d,\nu, x_1, x_2, \ldots, x_{k}) \in \mmspace^{k}$ then one may choose $x_{k+1}, x_{k+2}, \ldots, x_{m}$ i.i.d.\ and consider the discrete metric on $\{x_1, \ldots, x_{m} \}$ with uniform measure, and $x_1,\ldots, x_{k}$ marked. Then these approximations converge a.s.\ to $(S,d,\nu,x_1,\ldots,x_{k})$ in the Gromov-weak topology on $\mmspace^{k}$.
\end{proposition}

Let $\matrixmspace$ be the space of all infinite-by-infinite matrices (entries indexed by $\N \times \N$) with the usual product $\sigma$-algebra and let $\widehat{\matrixmspace}$ be the subset of $\matrixmspace$ consisting of those matrices with the property that for each $k$, the initial $k \times k$ matrix of $\matrixmspace$ describes a distance function on $k$ elements, and the limit of the corresponding $k$-element metric spaces (endowed with the uniform probability measure on the $k$ elements) exists in $\mmspace$.  We refer to this limit as the {\em limit space} of the infinite-by-infinite matrix. It is a straightforward exercise to check that $\widehat {\matrixmspace}$ is a measurable subset of $\matrixmspace$.

\begin{proposition}
\label{prop::measurability_correspondence}
There is a one-to-one correspondence between
\begin{enumerate}
\item Real-valued $\mmsigma$-measurable functions $\phi$ on $\mmspace$, and
\item Real-valued measurable functions $\wt \phi$ on $\widehat {\matrixmspace}$ with the property that their value depends only on the limit space.
\end{enumerate}
The relationship between the functions is the obvious one:
\begin{enumerate}
\item If we know $\wt \phi$, then we define $\phi$ by setting $\phi\bigl( (S,d,\nu) \bigr)$ to be the a.s.\ value of $\wt \phi (M_\infty)$ when $M_\infty$ is chosen via $(S,d,\nu)$.
\item If we know $\phi$, then $\wt \phi(M_\infty)$ is $\phi$ of the limit space of $M_\infty$.
\end{enumerate}
Moreover, for each $k \in \N$ the analogous correspondence holds with $(\mmspace^k,\mmsigma^k)$ in place of $(\mmspace,\mmsigma$).
\end{proposition}
\begin{proof}
We will prove the result for $(\mmspace,\mmsigma)$; the case of $(\mmspace^k,\mmsigma^k)$ for general $k \in \N$ is analogous.

Suppose that $\wt{\phi}$ is a bounded, continuous function on $\matrixmspace$ which depends only on a finite number of coordinate entries.  Then we know that $(S,d,\nu) \mapsto \E_\nu[ \wt{\phi}(M_\infty)]$ is an $\mmsigma$-measurable function where $M_\infty$ is the infinite matrix of distances associated with an i.i.d.\ sequence $(x_i)$ chosen from $\nu$.  From this it is not difficult to see that if $\wt{\phi}$ is an indicator function of the form $\one_A$ where $A \subseteq \R^{k^2}$ is compact (i.e., $\one_A$ depends only on the initial $k \times k$ matrix) then $(S,d,\nu) \mapsto \E_\nu[ \wt{\phi}(M_\infty)]$ is $\mmsigma$-measurable.  We note that the collection of such sets $A$ is a $\pi$-system which generates the product $\sigma$-algebra on $\matrixmspace$.  We also note that the set of all functions $\wt{\phi}$ on $\matrixmspace$ for which $(S,d,\nu) \mapsto \E_\nu[ \wt{\phi}(M_\infty)]$ is $\mmsigma$-measurable is closed under taking finite linear combinations and non-negative monotone limits.  Therefore the monotone class theorem implies that $(S,d,\nu) \mapsto \E_\nu[\wt{\phi}(M_\infty)]$ is $\mmsigma$-measurable for any bounded, measurable function on $\matrixmspace$.  In particular, this holds if $\wt{\phi}$ is a bounded, measurable function on $\wh{\matrixmspace}$ which depends only on the limit space.  In this case, we note that the a.s.\ value of $\wt{\phi}(M_\infty)$ is the same as $\E_\nu[\wt{\phi}(M_\infty)]$.  This proves one part of the correspondence.

On the other hand, suppose that $\phi$ is an $\mmsigma$-measurable function of the form $(S,d,\nu) \mapsto \E_\nu[\psi(M_k)]$ where $\psi$ is a bounded, continuous function on $\R^{k^2}$ and $M_k$ is the matrix of distances associated with $x_1,\ldots,x_k$ chosen i.i.d.\ from $\nu$.  Suppose that $M_\infty \in \wh{\matrixmspace}$.  For each $j$, we let $(S_j,d_j,\nu_j)$ be the element of $\mmspace$ which corresponds to the $j \times j$ submatrix $M_j$ of $M_\infty$.  Then the map which associates $M_\infty$ with $\phi((S_j,d_j,\nu_j))$ is continuous on~$\wh{\matrixmspace}$.  Therefore the map which associates $M_\infty$ with $\phi((S,d,\nu))$ where $(S,d,\nu)$ is the limit space of $M_\infty$ is measurable as it is the limit of continuous maps.  The other part of the correspondence thus follows from the definition of $\CF$.
\end{proof}

We are now going to use Proposition~\ref{prop::measurability_correspondence} to show that certain subsets of $\mmspace$ are measurable.  We begin by showing that the set of compact metric spaces in $\mmspace$ is measurable.  Throughout, we let $\wh{\compactmatrix}$ consist of those elements of $\wh{\matrixmspace}$ whose limit space is compact.

\begin{proposition}
\label{prop::compact_measurable}
The set of compact metric spaces in $\mmspace$ is measurable.  More generally, for each $k \in \N$ we have that the set of compact metric spaces in $\mmspace^k$ with $k$ marked points is measurable.
\end{proposition}
\begin{proof}
We are going to prove the first assertion of the proposition (i.e., the case $k=0$).  The result for general values of $k$ is analogous.

For each $\epsilon > 0$ and $n \in \N$, we let $\wh{\matrixmspace}_{n,\epsilon}$ be those elements $(d_{ij})$ in $\wh{\matrixmspace}$ such that for every~$j$ there exists $1 \leq k \leq n$ such that $d_{jk} \leq \epsilon$.  That is, $(d_{ij})$ is in $\wh{\matrixmspace}_{n,\epsilon}$ provided the $\epsilon$-balls centered at points in the limit space which correspond to the first $n$ rows (or columns) in $(d_{ij})$ cover the entire space.  As $\wh{\matrixmspace}_{n,\epsilon}$ is measurable, we have that both $\wh{\matrixmspace}_\epsilon = \cup_n \wh{\matrixmspace}_{n,\epsilon}$ and $\cap_{\epsilon \in \Q_+} \wh{\matrixmspace}_\epsilon$ are measurable.  By Proposition~\ref{prop::measurability_correspondence}, it therefore suffices to show that $\cap_{\epsilon \in \Q_+} \wh{\matrixmspace}_\epsilon$ is equal to $\wh{\compactmatrix}$.  This, however, follows because a metric space is compact if and only if it is complete and totally bounded.

\end{proof}

To prove the measurability of certain sets in $\mmspace$, we will find it useful first to show that they are measurable with respect to the Gromov-Hausdorff topology and then use that there is a natural map from $\wh{\compactmatrix}$ into the Gromov-Hausdorff space which is measurable.  In order to remind the reader of the Gromov-Hausdorff distance, we first need to remind the reader of the definition of the Hausdorff distance.  Suppose that $K_1,K_2$ are closed subsets of a metric space $(S,d)$.  For each $\epsilon > 0$, we let $K_j^\epsilon$ be the $\epsilon$-neighborhood of $K_j$.  Recall that the {\bf Hausdorff distance} between $K_1,K_2$ is given by
\begin{equation}
\label{eqn::h_def}
\dh(K_1,K_2) = \inf\{ \epsilon >0 : K_1 \subseteq K_2^\epsilon,\ K_2 \subseteq K_1^\epsilon\}.
\end{equation}
Suppose that $(S_1,d_1)$, $(S_2,d_2)$ are compact metric spaces.  The {\bf Gromov-Hausdorff} distance between $(S_1,d_1)$ and $(S_2,d_2)$ is given by 
\begin{equation}
\label{eqn::dgh_def}
\dgh((S_1,d_1),(S_2,d_2)) = \inf\left\{ \dh(\varphi_1(S_1),\varphi_2(S_2)) \right\}
\end{equation}
where the infimum is over all metric spaces $(S,d)$ and isometries $\varphi_j \colon S_j \to S$.  We let $\gh$ be the set of all compact metric spaces equipped with the Gromov-Hausdorff distance~$\dgh$.  More generally, for each $k \in \N$, we let $\gh^k$ be the set of all compact metric spaces $(S,d)$ marked with $k$ points $x_1,\ldots,x_k \in S$.  We equip $\gh^k$ with the distance function
\begin{equation}
\label{eqn::dghk_def}
\begin{split}
&\dgh((S_1,d_1,x_{1,1},\ldots,x_{1,k}),(S_2,d_2,x_{2,1},\ldots,x_{2,k}))\\
=& \inf\left\{ \dh(\varphi_1(S_1),\varphi_2(S_2)) + \sum_{j=1}^k d(\varphi_1(x_{1,j}),\varphi_2(x_{2,j})) \right\}.
\end{split}
\end{equation}
where the infimum is as in~\eqref{eqn::dgh_def}.  We refer the reader to \cite[Chapter~27]{villanibook} as well as \cite[Chapter~7]{courseinmetricgeometry} for more on the Hausdorff and Gromov-Hausdorff distances.

We remark that in~\eqref{eqn::dgh_def}, one may always take the ambient metric space to be $\ell_\infty$.  Indeed, this follows because every compact metric space can be isometrically embedded into $\ell_\infty$.  We will use this fact several times in what follows.

We also note that there is a natural projection $\pi \colon \wh{\compactmatrix} \to \gh$.  Moreover, if we equip $\wh{\matrixmspace}$ with the $\ell_\infty$ topology (in place of the product topology), then the projection $\pi \colon \wh{\compactmatrix} \to \gh$ is $2$-Lipschitz.  Indeed, this can be seen by using the representation of $\dgh$ in terms of the distortion of a so-called correspondence between metric spaces; see \cite[Chapter~27]{villanibook}.  See, for example, the proof of \cite[Proposition~3.3.3]{miermontstflour} for a similar argument.  Since the product topology generates the same Borel $\sigma$-algebra as the $\ell_\infty$ topology on $\wh{\matrixmspace}$, it follows that $\pi$ is measurable.  This observation will be useful for us for proving that certain sets in $\wh{\matrixmspace}$ are measurable.  We record this fact in the following proposition.

\begin{proposition}
\label{prop::gm_measurable}
The projection $\pi \colon \wh{\compactmatrix} \to \gh$ is measurable.
\end{proposition}

In the following proposition, we will combine Proposition~\ref{prop::measurability_correspondence} and Proposition~\ref{prop::gm_measurable} to show that the set of compact, geodesic metric spaces in $\mmspace$ is measurable.

\begin{proposition}
\label{prop::geodesic_closed}
The set of compact, geodesic spaces is measurable in $\mmspace$.
\end{proposition}
\begin{proof}
That the set of geodesic spaces is closed hence measurable in $\gh$ follows from \cite[Theorem~27.9]{villanibook}; see also the discussion in \cite[Chapter~7.5]{courseinmetricgeometry}.  Therefore the result follows by combining Proposition~\ref{prop::measurability_correspondence} and Proposition~\ref{prop::gm_measurable}.
\end{proof}

We note that it is also possible to give a short proof of Proposition~\ref{prop::geodesic_closed} which does not rely on the measurability of the projection $\pi \colon \wh{\compactmatrix} \to \gh$. The following proposition will imply that the set of good measure endowed geodesic spheres is measurable in $\mmspace$.

\begin{proposition}
\label{prop::geodesic_spheres_measurable}
For each $k \in \N_0$ we have that $\gmsspace^k$ is measurable in $\mmspace^k$.
\end{proposition}

We will prove Proposition~\ref{prop::geodesic_spheres_measurable} in the case that $k = 0$ (i.e., we do not have any extra marked points).  The proof for general values of $k$ is analogous.  As in the proof of Proposition~\ref{prop::geodesic_closed}, it suffices to show that the set of geodesic metric spaces $(S,d)$ which are homeomorphic to $\S^2$ is measurable in $\gh$.  In order to prove this, we first need to prove the following lemma.

\begin{lemma}
\label{lem::sphere_contract}
Suppose that $(S,d)$ is a geodesic metric space homeomorphic to $\S^2$ and suppose that $\gamma$ is a non-space-filling loop on $S$.  Let $U$ be a connected component of $S \setminus \gamma$ and let $A = S \setminus U$.  For every $\epsilon > 0$, $\gamma$ is homotopic to a point inside of the $\epsilon$-neighborhood of $A$. 
\end{lemma}
\begin{proof}
Since $\gamma$ is a continuous curve, it follows that $U$ is topologically equivalent to $\D$.  Let $\varphi \colon \D \to U$ be a homeomorphism.  Then there exists $\delta > 0$ so that $\Gamma = \varphi(\partial (1-\delta) \D)$ is contained in the $\epsilon$-neighborhood of $A$.  Since $\Gamma$ is a simple curve, it follows that there exists a homeomorphism $\psi$ from $\D$ to the component $V$ of $S \setminus \Gamma$ which contains $\gamma$.  Let $\wt{\gamma} = \psi^{-1}(\gamma)$.  Then $\wt{\gamma}$ is clearly homotopic to $0$ in $\D$ hence $\gamma$ is homotopic to $\psi(0)$ in $V$, which implies the result.
\end{proof}

\begin{proof}[Proof of Proposition~\ref{prop::geodesic_spheres_measurable}]
For simplicity, we will prove the result in the case that $k=0$.  The case for general values of $k$ is established in an analogous manner.  We are going to prove the result by showing that the set $\spheresgeo$ of geodesic metric spaces in $\gh$ which are homeomorphic to $\S^2$ is measurable in $\gh$.  The result will then follow by invoking Proposition~\ref{prop::measurability_correspondence} and Proposition~\ref{prop::gm_measurable}.

Let $\spheresgeoclose$ be the closure of $\spheresgeo$ in $\gh$.  Suppose that $(S,d)$ is in $\gh$.  Let $\gamma$ be a path in $(S,d)$ and let $f(\gamma,(S,d))$ be the infimum of $\diam(A)$ over all $A \subseteq S$ in which $\gamma$ is homotopic in $A$ to a point in $S$ and $S \setminus A$ is connected.  Let $f(\delta,(S,d))$ be equal to the supremum of $f(\gamma,(S,d))$ over all paths $\gamma$ in $(S,d)$ with diameter at most $\delta$.  Let $\wt{\spheresgeo}$ consist of those $(S,d)$ in $\spheresgeoclose$ such that for every $\epsilon > 0$ there exists $\delta > 0$ such that $f(\delta,(S,d)) < \epsilon$.

We are first going to show that $\wt{\spheresgeo}=\spheresgeo$.  We clearly have that $\spheresgeo \subseteq \wt{\spheresgeo}$, so we just need to show that $\wt{\spheresgeo} \subseteq \spheresgeo$.  Suppose that $(S,d)$ is in $\wt{\spheresgeo}$.  We assume without loss of generality that $\diam(S) = 1$.  Then there exists a sequence $(S_n,d_n)$ in $\spheresgeo$ which converges to $(S,d)$ in $\gh$.  We note that we may assume without loss of generality that both $S$ and the $S_n$'s are subsets of $\ell_\infty$ such that $\dh(S_n,S) \to 0$ as $n  \to \infty$ and that $\diam(S_n) = 1$ for all~$n$.

Fix $\epsilon > 0$.  It suffices to show that there exists $\delta > 0$ such that $f(\delta,(S_n,d_n)) < \epsilon$ for all $n \in \N$.  Indeed, this implies that the $(S_n,d_n)$ converge to $(S,d)$ in $\gh$ \emph{regularly} which, by \cite{begle_regular}, implies that $(S,d)$ is in $\spheresgeo$.

Fix $\delta > 0$ such that $f(\delta,(S,d)) < \epsilon$ and assume that $\delta \leq \epsilon$.  We assume that $n_0 \in \N$ is sufficiently large so that
\begin{equation}
\label{eqn::ms_close}
\dh(S_n,S) \leq \frac{\delta}{16} \quad\text{for all}\quad n \geq n_0.
\end{equation}
We note that for each $1 \leq n \leq n_0$ there exists $\delta_n > 0$ such that $f(\delta_n,(S_n,d_n)) < \epsilon$.  We set $\delta_0 = \min_{1 \leq n \leq n_0} \delta_n$.  We are now going to show that there exists $\wh{\delta} > 0$ such that $f(\wh{\delta},(S_n,d_n)) \leq 43 \epsilon$ for all $n \geq n_0$.  Upon showing this, we will have that with $\wt{\delta} = \delta_0 \wedge \wh{\delta}$ we have $f(\wt{\delta},(S_n,d_n)) \leq 43\epsilon$ for all $n$.

Fix $n \geq n_0$ and suppose that $\gamma_n \colon \S^1 \to S_n$ is a path in $S_n$ with $\diam(\gamma_n) \leq \delta/4$.  Then we can construct a path $\gamma$ in $S$ as follows.  We pick times $0 \leq t_0^n < \cdots < t_j^n \leq 2\pi$ such that with $x_i^n = \gamma_n(t_i^n)$ we have
\begin{equation}
\label{eqn::max_diam}
\| x_{i-1}^n - x_i^n \|_{\ell_\infty} \leq \frac{\delta}{16} \quad\text{for all}\quad 1\leq i \leq j.
\end{equation}
By~\eqref{eqn::ms_close}, for each $1 \leq i \leq j$ there exists $x_i \in S \subseteq \ell_\infty$ such that $\| x_i^n - x_i \|_{\ell_\infty} \leq \delta/16$.  We then take $\gamma$ to be the path $\S^1 \to S$ which is given by successively concatenating geodesics from $x_{i-1}$ to $x_i$ for each $1 \leq i \leq j+1$ where we take $x_{j+1} = x_0$.  Suppose that $a,b \in \gamma$.  Then there exists $i_q$ such that $\| q - x_{i_q} \|_{\ell_\infty} \leq 3 \delta/16$ for $q \in \{a,b\}$ as $\|x_{i-1}-x_i\|_{\ell_\infty} \leq 3\delta/16$ for each $1 \leq i \leq j+1$.  Consequently, by~\eqref{eqn::ms_close} and~\eqref{eqn::max_diam} we have that
\begin{align*}
   \|a-b\|_{\ell_\infty}
&\leq \|a-x_{i_a}\|_{\ell_\infty} + \|x_{i_a} - x_{i_b}\|_{\ell_\infty} + \|x_{i_b} - b\|_{\ell_\infty}\\
&\leq \frac{3}{8}\delta + \|x_{i_a} - x_{i_a}^n\|_{\ell_\infty} + \|x_{i_a}^n - x_{i_b}^n\|_{\ell_\infty} + \|x_{i_b}^n - x_{i_b}\|_{\ell_\infty}\\
&\leq \frac{1}{2}\delta + \diam(\gamma_n) < \delta.
\end{align*}
This implies that $\diam(\gamma) < \delta$.  Moreover, we have that the $\dh$-distance between the ranges of $\gamma_n$ and $\gamma$ is at most $\delta/2$.

By assumption, we can contract $\gamma$ to a point in $S$ inside of a set $A \subseteq S$ of diameter at most~$\epsilon$ such that $B = S \setminus A$ is connected.  Pick $x \in B$ with $\dist(x,A) \geq (1-\epsilon)/2$.  Fix $x_n \in S_n$ with $\| x-x_n \|_{\ell_\infty} \leq \delta/16$.  Let~$B_n$ be the component of $S_n \setminus \gamma_n$ containing $x_n$ and let $A_n$ be the closure of $S_n \setminus B_n$.  By Lemma~\ref{lem::sphere_contract}, we have that $f(\gamma_n,(S_n,d_n)) \leq \diam(A_n)$.  It therefore suffices to bound $\diam(A_n)$.

Suppose that $u_n \in A_n$ is a point with distance at least $\epsilon+20\delta$ from $\gamma_n$.   Let $u \in S$ be such that $\|u_n - u\|_{\ell_\infty} \leq \delta/16$.  We will show that $u \in \ol{A}$.  This will imply that $\diam(A) \geq 10\delta$, a contradiction since we have assumed that $\delta \leq \epsilon$ and we have $\diam(A) \leq \epsilon$, hence 
\begin{equation}
\label{eqn:diam_an_bound}
\diam(A_n) \leq 2\epsilon + 40\delta + \diam(\gamma_n) \leq 2\epsilon + 41\delta.	
\end{equation}
Suppose that $u \notin \ol{A}$.  Then there exists a path $\eta$ from $u$ to $x$ which does not intersect $\gamma$.  Arguing as above, this implies that there exists a path $\eta_n$ in $S_n$ from $u_n$ to $x_n$ so that the $d_H$ distance of the ranges of $\eta$ and $\eta_n$ is at most $\delta/2$.  Since $\eta$ does not intersect $\{z \in A : \dist(z,\partial A) \geq \delta\}$, it follows that $\eta_n$ does not intersect $\{z \in A_n : \dist(z,\partial A_n) \geq 2\delta\}$.  This is a contradiction, which proves~\eqref{eqn:diam_an_bound}.

Since $\gamma_n$ was an arbitrary path in $(S_n,d_n)$ of diameter at most $\delta$, we have thus shown that $f(\delta,(S_n,d_n)) \leq 2\epsilon+41\delta \leq 43\epsilon$ for all $n \geq n_0$.  This finishes the proof that $\spheresgeo = \wt{\spheresgeo}$.

To finish proving the result, we will show that $\wt{\spheresgeo}$ (hence $\spheresgeo$) can be written as an intersection of sets which are relatively open in the closure of geodesic spheres in $\gh$, hence is measurable.  It follows from the argument given just above that, for each fixed $\delta > 0$, the map $(S,d) \mapsto f(\delta, (S,d))$ is uniformly continuous on $\spheresgeo$.  This implies that $(S,d) \mapsto f(\delta, (S,d))$ extends to a continuous map on $\spheresgeoclose$.  It therefore follows that, for each $\epsilon > 0$, we have that
\[ \ol{\spheresgeo}_{\epsilon,\delta} = \{ (S,d) \in \spheresgeoclose : f(\delta, (S,d)) < \epsilon\}\]
is relatively open in $\spheresgeoclose$.  Therefore with $\Q_+ = \Q \cap (0,\infty)$ we have that
\[ \bigcap_{\epsilon \in \Q_+} \bigcup_{\delta \in \Q_+} \ol{\spheresgeo}_{\epsilon,\delta}\]
is a Borel set in $\gh$.  The result follows since this set is equal to $\wt{\spheresgeo}$.
\end{proof}

In what follows, it will be useful to consider geodesic spheres which are marked by two points and also come with an orientation.  We note that if $(S,d,\nu,x,y) \in \gmsspace^2$ and $x \neq y$, then we can determine an orientation of $S$ by specifying three additional distinct marked points $\ul{x}=(x_1,x_2,x_3)$ with $x_i \in \partial \fb{x}{r}$ for $i=1,2,3$ for some $r \in (0,d(x,y))$ fixed.  The extra marked points $\ul{x}$ specify an orientation because they specify a continuous curve (modulo monotone parameterization) which parameterizes $\partial \fb{x}{r}$  and visits $x_1,x_2,x_3$ in order.  We observe that this also specifies for every $s \in (0,d(x,y))$ a continuous curve (modulo parameterization) which parameterizes $\partial \fb{x}{s}$.  Indeed, if $\epsilon > 0$ is sufficiently small and $s \in (r-\epsilon,r)$, then we can specify three distinct points on $\partial \fb{x}{s}$ by first specifying three points on $\partial \fb{x}{r}$ with distance at least $2 \epsilon$ from each other using the orientation and then taking points on $\partial \fb{x}{s}$ which each have distance $r-s$ from the given points on $\partial \fb{x}{r}$.  Continuing in this manner specifies a parameterization of $\partial \fb{x}{s}$ for each $s \in (0,r)$.  We can also specify a parameterization of $s \in (r,r+\epsilon)$ by first choosing three points on $\partial \fb{x}{s}$ with distance at least $2\epsilon$ from each other and then taking points on $\partial \fb{x}{r}$ which each have distance $s-r$ from the given points on $\partial \fb{x}{s}$ and then ordering the points on $\partial \fb{x}{s}$ in the same way as the points on $\partial \fb{x}{r}$.  Continuing in this manner specifies a parameterization of $\partial \fb{x}{s}$ for each $s \in (r,d(x,y))$.

We say that two spaces $(S^i,d^i,x^i,y^i,\ul{x}^i)$, $i=1,2$, in $\gmsspace^5$ with marked points as above are equivalent if there exists a measure preserving isometry $S^1 \to S^2$ which takes $x^1$ to $x^2$, $y^1$ to $y^2$, and is orientation preserving.  In what follows, we will be considering various measurable maps which will not depend on the specific choice of three points used to orient the sphere but rather just the equivalence class.  We note that the map $\gmsspace^5 \times \gmsspace^5 \to \{0,1\}$ that outputs $1$ if the two spaces have distinct marked points are equivalent and otherwise $0$ is measurable because if $\ul{x}^1 = (x_1^1,x_2^1,x_3^1)$, $\ul{x}^2 = (x_1^2,x_2^2,x_3^2)$ are two triples of distinct points in $S$ for $(S,d,\nu,x,y) \in \gmsspace^2$ with $x_i^1 \in \partial \fb{x}{r_1}$ and $x_i^2 \in \partial \fb{x}{r_2}$ with $r = r_1 = r_2$ then $\ul{x}^1$ and $\ul{x}^2$ determine the same orientation if and only if the following is true.  Suppose that $w_1=x$, $w_2 = y$, $w_3=x_1$, $w_4=x_2$, $w_5 = x_3$ and $w_6,w_7,\ldots$ is an i.i.d.\ sequence chosen from $\nu$ and suppose that $d_{ij} = d(w_i,w_j)$.  Then it is $\nu$-a.s.\ the case that for every $\epsilon > 0$ small enough there are points $u_1,\ldots,u_n$ among the $(w_i)$ such that $B(u_i,\epsilon) \cap \partial \fb{x}{r} \neq \emptyset$, $B(u_i,\epsilon) \cap B(u_{i+1},\epsilon) \neq \emptyset$ where we take $u_{n+1} = u_1$ and if $x_j^k \in B(u_{i_j^k},\epsilon)$ then $i_1^k,i_2^k,i_3^k$ are in cyclic order (viewed as elements of $\Z_n$) for both $k=1,2$.  In the case that $r_1 < r_2$, then $\ul{x}^1, \ul{x}^2$ determine the same orientation if and only if there exist paths $\gamma_1,\gamma_2,\gamma_3$ in $\fb{x}{r_2} \setminus \fb{x}{r_1}$ which respectively connect $x_i^1$ to $x_i^2$ and do not cross.  The existence of such paths can be similarly determined in a measurable manner using the matrix $(d_{ij})$ of pairwise distances.  Finally, we note that it is not difficult to construct for each $r > 0$ a measurable map $\gmsspace^{5} \to \gmsspace^{5}$ which takes as input $(S,d,\nu,x,y,\ul{x})$ and outputs $(S,d,\nu,x,y,\ul{x})$ if $d(x,y) \leq r$ or the marked points are not distinct or the $\ul{x}$ do not all lie on the same filled metric ball boundary and otherwise it outputs $(S,d,\nu,x,y,\ul{x}')$ where the points of $\ul{x}'$ are distinct, lie in $\partial \fb{x}{r}$, and induce an equivalent orientation.

We let $\gmsspace^{2,O}$ denote the subset of $\gmsspace^5$ with distinct marked points modulo the above equivalence relation.  We let $\mmsigmaspho$ be the associated Borel $\sigma$-algebra.  We will refer to an element of $\gmsspace^{2,O}$ using the notation $(S,d,\nu,x,y)$ and suppress the orientation unless we need to choose a representative of the equivalence class.

We will also need to consider orientations on planar metric measure spaces whose interior is a geodesic disk.  For our purposes, we do not need that the set of such spaces is measurable in the set of metric measure spaces.  So in this case we will introduce an equivalence relation on $\mmspace^4$ which has the property that if the two spaces are of the above type then they are equivalent if and only if they are equivalent as marked metric measure spaces and have the same orientation.  More specifically, we say that two elements $(S^i,d^i,x^i,\ul{x}^i)$, $i=1,2$, in $\mmspace^4$ with distinct marked points are equivalent if there exists a measure preserving isometry which takes $S^1$ to $S^2$, $x^1$ to $x^2$, and for each $\delta > 0$ there exists a simple path in $S^1 \setminus \ol{B(x^1,\delta)}$ which visits all of the marked points $\ul{x}^i$ for $i=1,2$ and between hitting any pair of these points hits $\partial B(x^1,\delta)$.  Arguing as above, it is easy to see that that the map which takes as input two spaces in $\mmspace^4$ and outputs $1$ (resp.\ $0$) if they are equivalent in this sense is Borel measurable since the existence of such a path can be described in terms of the infinite matrix of pairwise distances.  If $(S^i,d^i)$ are topological disks and the marked points $\ul{x}^i$ are in the boundary and distinct and $x^i$ is not in the boundary, then this equivalence relation is equivalent to $\ul{x}^i$ inducing the same orientation.  We let $\mmspace^{1,O}$ be given by $\mmspace^4$ modulo this equivalence relation.

\begin{proposition}
\label{prop::ball_interior_measurable}
Fix a constant $r > 0$ and let $\gmsspacer^{2,O}$ be the set of elements $(S,d,\nu,x,y) \in \gmsspace^{2,O}$ such that $R = d(x,y) - r > 0$ (and note that this is a measurable subset of $\gmsspace^2$).  Then the space which corresponds to $\fb{x}{R}$ (with its interior-internal metric) is in $\mmspace^1$. The function $\gmsspacer^{2,O} \to \mmspace^{1,O}$ given by associating $(S,d,\nu,x,y)$ to this space with the orientation induced from $(S,d,\nu,x,y)$ is Borel measurable.

Suppose that we have a measurable way of choosing $z_1, z_2, \ldots, z_k \in \partial \fb{x}{R}$ that only requires us to look at $S \setminus \fb{x}{R}$.  Then the map to the set of $k$ slices (i.e., the metric measure spaces which correspond to the regions between the leftmost geodesics from each $z_j$ to $x$) is measurable as a map $\gmsspacer^{2,O} \to (\mmspace^3)^k$.  (The three marked points in the $j$th slice are given by $z_j$, $z_{j+1}$, and the point where the leftmost geodesics from $z_j$ and $z_{j+1}$ to $x$ first meet.)
\end{proposition}
If there is a unique geodesic from $x$ to $y$, one example of a function which associates $S \setminus \fb{x}{R}$ with points $z_1,\ldots,z_k$ is as follows.  Assume that we have a measurable way of measuring ``boundary length'' on $\partial \fb{x}{R}$.  Then we take $z_1,\ldots,z_k \in \partial \fb{x}{R}$ to be equally spaced points according to boundary length with $z_1$ given by the point on $\partial \fb{x}{R}$ which is first visited by the geodesic from $x$ to $y$.
\begin{proof}[Proof of Proposition~\ref{prop::ball_interior_measurable}]
That the space which corresponds to $\fb{x}{R}$ is an element of $\mmspace^1$ is obvious.

We are now going to argue that the map which associates $(S,d,\nu,x,y) \in \gmsspacer^{2,O}$ with the metric measure space associated with $\fb{x}{R}$ and associated orientation is measurable.  To see this, we note that a point $w$ is in $S \setminus \fb{x}{R}$ if and only if there exists $\epsilon > 0$ and $y_1,\ldots,y_\ell \in S$ such that the following hold:
\begin{enumerate}
\item $d(y_j,x) \geq R+\epsilon$ for each $1 \leq j \leq \ell$,
\item $y \in B(y_1,\epsilon)$ and $w \in B(y_\ell,\epsilon)$, and
\item $B(y_j,\epsilon)$ has non-empty intersection with both $B(y_{j-1},\epsilon)$ and $B(y_{j+1},\epsilon)$ for each $2 \leq j \leq \ell-1$.
\end{enumerate}
Suppose that $w_1=x$, $w_2 = y$, $w_3=x_1$, $w_4=x_2$, $w_5 = x_3$ and $w_6,w_7,\ldots$ is an i.i.d.\ sequence chosen from $\nu$ and suppose that $d_{ij} = d(w_i,w_j)$.  The above tells us how to determine those indices $j$ such that $x_j \in S \setminus \fb{x}{R}$.  In particular, it is clear from the above that the event that $x_i \in \fb{x}{R}$ is a measurable function of $(d_{ij})$ viewed as an element of $\wh{\matrixmspace}$.  Suppose that we are on the event that $w_i,w_j \in \fb{x}{R}$ for $i,j$ distinct. Then the event that the interior-internal distance between $w_i$ and $w_j$ is at most $\delta$ is equivalent to the event that there exists $\epsilon > 0$ and indices $j_1 = i,j_2,\ldots,j_{k-1},j_k=j$ such that $d_{j_\ell j_{\ell+1}} < \epsilon$ for each $1 \leq \ell \leq k-1$, $(k-1) \epsilon < \delta$, and $B(x_{j_\ell},\epsilon) \subseteq \fb{x}{R}$ for each $1 \leq \ell \leq k$ (which we can determine using the recipe above).  Thus it is easy to see that the element of $\wh{\matrixmspace}$ which corresponds to the matrix of distances between the $(w_i)$ which are in $\fb{x}{R}$ with the interior-internal metric is measurable.  Thus the measurability of the metric measure space corresponding to $\fb{x}{R}$ viewed as an element of $\mmspace^1$ follows by applying Proposition~\ref{prop::measurability_correspondence}.  The same likewise holds for the orientation of $\fb{x}{R}$.

Suppose that $(S,d,\nu,x,y) \in \gmsspacer^{2,O}$ and let $(S,d,\nu,x,y,\ul{x})$ be a representative of the equivalence class of $(S,d,\nu,x,y)$ in $\gmsspace^5$.  To see the final claim of the proposition, we note that a point $w$ is in the slice between the leftmost geodesics from $z_i$ and $z_{i+1}$ to $x$ if and only if there exists $\delta > 0$ such that for every $\epsilon > 0$ small enough there exist points $y_1,\ldots,y_\ell$ with $y_\ell = w$ which satisfy the following properties:
\begin{enumerate}
\item\label{it:intersect_ball_boundary} $y_j \in \fb{x}{R}$ for each $1 \leq j \leq \ell$,
\item\label{it:away_from_endpoints} $d(y_j,z_i) \geq \delta$ and $d(y_j,z_{i+1}) \geq \delta$ for each $1 \leq j \leq \ell$,
\item\label{it:intersect_cw} $B(y_1,\epsilon)$ has non-empty intersection with the clockwise part of $\partial \fb{x}{R}$ between $z_i$ and $z_{i+1}$,
\item\label{it:intersect_each_other} $B(y_j,\epsilon)$ has non-empty intersection with $B(y_{j-1},\epsilon)$ and $B(y_{j+1},\epsilon)$ for each $2 \leq j \leq \ell-1$,
\item\label{it::no_geodesic_i_plus_1} No geodesic from $z_{i+1}$ to $x$ passes through the $B(y_j,\epsilon)$, and
\item\label{it::no_geodesic_i} No geodesic from a point on $\partial \fb{x}{R}$ which starts from a point on the counterclockwise segment of $\partial \fb{x}{R}$ from $z_i$ to $z_{i+1}$ to $x$ passes through the $B(y_j,\epsilon)$.
\end{enumerate}
It is obvious that properties~\ref{it:intersect_ball_boundary}, \ref{it:away_from_endpoints}, and~\ref{it:intersect_each_other} can be determined from the matrix $(d_{ij})$ in a measurable way.  We will now explain in further detail why the other properties can be measurably determined.  Let us first explain how to check property~\ref{it:intersect_cw}.  We note that $B(y_1,\epsilon)$ intersects the clockwise part of $\partial \fb{x}{R}$ if and only if the following is true for every $\epsilon > 0$ small enough.  There exist points $u_1,\ldots,u_n$ such that $B(u_i,\epsilon) \cap \partial \fb{x}{R} \neq \emptyset$, $B(u_i,\epsilon) \cap B(u_{i+1},\epsilon) \neq \emptyset$ where we take $u_{n+1} = u_1$, if $x_j \in B(u_{i_j},\epsilon)$ then $i_1,i_2,i_3$ are in cyclic order (viewed as elements of $\Z_n$) and the following is true.  If $k_1,k_2$ are such that $z_i \in B(u_{k_1},\epsilon)$ and $z_{i+1} \in B(u_{k_2},\epsilon)$ then $B(y_1,\epsilon)$ intersects $B(u_j,\epsilon)$ if and only if $k_1 \leq u_j \leq k_2$ (viewed as elements of $\Z_n$).
Property~\ref{it::no_geodesic_i_plus_1} holds if and only if
\[ \min_{1 \leq j \leq k} \inf \left\{ d(x,y) + d(y,z_{i+1}) : y \in B(y_j,\epsilon) \right\} > d(x,z_{i+1}).\]
Property~\ref{it::no_geodesic_i} can be checked by combining the ideas used to check properties~\ref{it:intersect_cw} and~\ref{it::no_geodesic_i_plus_1}.  Combining, the result thus follows in view of Proposition~\ref{prop::measurability_correspondence} and the argument described in the previous paragraph.
\end{proof}

\subsection{Measurability of the unembedded metric net}
\label{subsec:unembedded_measurability}

We will now develop some basic properties of metric nets and leftmost geodesic trees. There are many places to get an overview of real trees, plane trees, and contour functions; for example, \cite{ald1991crt1,ald1991crt2,ald1993crt3} uses these concepts to describe continuum random trees and Section 3 of~\cite{le2012scaling} reviews these concepts for the purpose of using them to construct the Brownian map. (They will also be further discussed in Section~\ref{sec::surfacesfromtrees}.)  Let $\T_1$ be the circle given by starting with $[0,1]$ and identifying $0$ and $1$.  We briefly recall that a {\em real planar tree} is a quotient of the type described in Figure~\ref{fig::lamination2}, where $X_t$ is any continuous non-negative function which is not constant in any interval indexed by $t \in \T_1$ with $\inf_{t \in \T_1} X_t = 0$.  Every real planar tree is compact, by definition.

Given any real tree embedded in the plane, one can construct a continuous  ``contour function'' $t \to X_t$ for $t \in \T_1$ by tracing the boundary of the tree continuously clockwise and keeping track of the distance from the root as a function of time. This $X_t$ can then be used to reconstruct the tree, as  Figure~\ref{fig::lamination2} illustrates. The contour function $t \to X_t$ is only determined up to monotone reparameterization: if $f$ is any increasing continuous function $\T_1 \to \T_1$ then $X_{f(t)}$ describes the same tree as $X_t$, via the procedure described in the caption to Figure~\ref{fig::lamination2}.  We call any such $f$ a {\em monotone reparameterization} of the circle.

Let $\treespace$ be the set of continuous functions $X \colon \T_1 \to \R_+$ with $\inf_{t \in \T_1} X_t = 0$ which are not constant in any interval modulo the equivalence relation $X \sim Y$ if and only if $X$ is a monotone reparameterization of $Y$.  For $X,Y \in \treespace$, we set
\[ d(X,Y) = \inf_{f} \|X - Y \circ f\|_\infty\]
where the infimum is over all monotone reparameterizations.  Then $d$ defines a metric on $\treespace$.  It is not hard to see that under this metric the space of compact real planar trees is complete and separable.  We equip $\treespace$ with the associated Borel $\sigma$-algebra.

\begin{figure}[ht!]
\begin{center}
\includegraphics[width=0.52\textwidth]{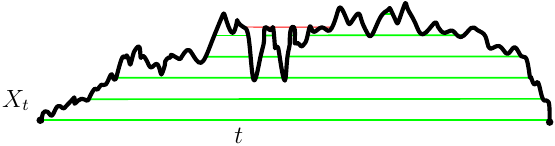}
\caption{\label{fig::lamination2}  Begin with the graph $X_t$ of a continuous excursion; then declare two points on that graph to be equivalent if they can be a connected by a horizontal chord that never goes above the graph of $X_t$ (these chords are shown as green lines). The equivalence classes form a real planar tree with a natural metric: the distance between points indexed by $s$ and by $t$ (with $s<t$) is $X_s + X_t - 2 \inf_{s< r < t} X_r$. It is not hard to see that the number of local maximum heights $a$ (i.e., values $a$ such that for some $s$, $X_s = a$ is a local maximum) is at most countably infinite (since any local maximum is the largest value obtained in {\em some} sufficiently small interval with rational endpoints). However, the horizontal red line illustrates that $X_t$ can have multiple local maxima (perhaps uncountably many) of the same height $a$.}
\end{center}
\end{figure}

It is not always the case that the leftmost geodesics of an oriented metric net form a real tree. For example, if $(S,d,x,y)$ is the Euclidean sphere (with $x$ and $y$ at opposite poles) then the metric net has no holes, and one can draw uncountably many disjoint leftmost geodesic arcs directed toward $x$. In this example, the ``tree of leftmost geodesics'' has uncountably many disjoint branches, each corresponding to a geodesic from~$y$ to~$x$ (with the endpoint $y$ itself not being included in these branches, since there is no single distinguished leftmost geodesic starting at $y$) and is clearly {\em not} a compact or precompact metric space when it is endowed with the natural tree metric. 

We say a doubly marked and oriented geodesic sphere $(S,d,x,y)$ is {\em strongly coalescent} if the leftmost geodesic tree of its metric net (endowed with the tree metric) {\em is} Cauchy-precompact, so that the completion of its leftmost geodesic tree (w.r.t.\ the natural tree metric) is a real planar tree.

\begin{proposition}
\label{prop::coalescentprecompact}
A doubly marked and oriented geodesic sphere $(S,d,x,y)$ is strongly coalescent if and only if for any $0<r<s <d(x,y)$ the number of disjoint leftmost-geodesic segments (toward $x$) one can draw from $\partial \fb{x}{s}$ to $\partial \fb{x}{r}$ is bounded above.
\end{proposition}
\begin{proof}
If there are infinitely many disjoint leftmost geodesic segments, for some~$r$ and~$s$, then the metric net is clearly not Cauchy-precompact. Conversely, if there are only finitely many for each $r$ and $s$, then this is true in particular if for some $\epsilon>0$ we have $r = n \epsilon$ and $s=(n+1)\epsilon$, and letting $n$ vary between $1$ and $d(x,y)/\epsilon$, we can show that it is possible to cover the tree with finitely many balls of diameter $\epsilon$. In other words, the metric net is totally bounded, which is equivalent to Cauchy-precompactness.
\end{proof}

Suppose $(S,d,x,y)$ is a strongly coalescent doubly marked and oriented geodesic sphere with metric net $N$. Let $T$ be the leftmost geodesic tree and $\wt T$ the completion of $T$.  The map from $T$ to $N \setminus \{y\}$ is one-to-one by construction, and can be extended continuously to a map from $\wt T$ to $N$. Two points on $T$ are called {\em equivalent} if they map to the same point on $N$. Let $X_t$ be the corresponding contour function.  We now present a few more definitions and quick observations about contour functions of metric nets:

\begin{enumerate}
\item $X_t$ necessarily assumes every value $r \in \bigl( 0, d(x,y) \bigr)$ an uncountable number of times.  This is because Proposition~\ref{prop::boundariesarecircles} shows that $\partial \fb{x}{r}$ is necessarily homeomorphic to a circle, and the map sending a point on the graph of $X_t$ to the corresponding point on the circle is onto.
\item In the proof of Proposition~\ref{prop::boundariesarecircles}, it was shown that $\Gamma = \partial \fb{x}{r}$ is locally connected and that therefore the homeomorphism from the unit disk to the interior of $\fb{x}{r}$ extends continuously to its boundary, so that the unit circle maps continuously onto $\Gamma$. Let $\Lambda_r \supseteq \Gamma$ be the component of $\partial B(x,r)$ that contains $\partial \fb{x}{r}$.  In other words, $\Lambda_r$ is the portion of $\partial B(x,r)$ that lies in the metric net. If the metric exploration ``pinches off holes'' exactly at time $r$ then $\Lambda_r$ could be strictly larger than $\Gamma$, as in Figure~\ref{fig::netandtree}. By the second part of Proposition~\ref{prop::boundariesarecircles}, the $x$-containing component of $\fb{x}{r} \setminus \Lambda_r$ is a topological disk, and a homeomorphism from the unit disk to that disk extends continuously to give a continuous map from the unit circle onto $\Lambda_r$ (which need not be one-to-one, since $\Lambda_r$ is not necessarily a topological circle). Thus it is natural to think of $\Lambda_r$ as being continuously parameterized by a circle (even if not in a strictly one to one way).
\item Any simple closed loop that can be drawn within $\Lambda_r$ is necessarily the boundary of one of the components of $S \setminus \partial B(x,r)$. This simply follows from the fact the loop divides the sphere into two pieces, and every point $z$ in the piece {\em not} containing $x$ must satisfy $d(x,z) > r$. Conversely, Proposition~\ref{prop::boundariesarecircles} (applied with $z$ in place of $y$) implies that every such boundary is a simple closed loop. In particular this implies that given the metric net $N$ (as a metric space with marked points $x$ and $y$) it is possibly to construct the (necessarily countable) collection of loops that form the boundaries of $N$ when $N$ is viewed as a subset of $S$. By gluing a topological disk into each of those loops, one obtains a topological space that is topologically equivalent to $S$, and which can be embedded in the sphere. In particular, this implies that the metric net $N$ determines its own embedding in the sphere (up to a topological homeomorphism of the sphere).
\begin{figure}[ht!]
\begin{center}
\includegraphics[width=0.32\textwidth]{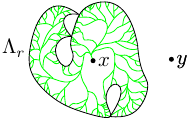}
\caption{\label{fig::netandtree}  Schematic drawing of the union of all of the leftmost geodesics from $\Lambda_r$ --- the component of $\partial B(x,r)$ containing $\partial \fb{x}{r}$ --- back to $x$. The completion of the union of these geodesics (endowed with the tree metric) is a closed subtree of $T$, continuously mapped to the metric net, with all of the set of leaves mapping onto $\Lambda_r$.}
\end{center}
\end{figure}
\item For $r \in \bigl( 0, d(x,y) \bigr)$, consider the closed set $A_r = \{t : X_t = r \}$. This is a subset of the circle $\T_1$, which corresponds to a closed subset of $\wt T$ and hence maps onto a closed subset of the metric net (the set $\Lambda_r$ described above). Let $\equiv$ be the (topological closure of) the smallest equivalence relation such that that two points are equivalent if they are at opposite endpoints of an open interval of $\T_1 \setminus A_r$.  It is not hard to see that if two points are equivalent in $\equiv$ then they must correspond to the same point in $\Lambda_r$, and that the topological quotient of $A_r$ w.r.t.\ $\equiv$ must be a topological circle, which is mapped onto $\Lambda_r$ (as in Figure~\ref{fig::netandtree}) in a continuous way. If this map is one-to-one, we say that that $\partial B(x,r)$ is a {\em simply traced loop}. If it is one-to-one except for two points that are mapped to the same place, then we say that $\partial B(x,r)$ is a {\em simply traced figure eight}. (The $\Lambda_r$ shown in  Figure~\ref{fig::netandtree} is neither of these; at least three pairs of points are ``pinched together'' in this image. There are other possibilities; for example, it is possible for a whole interval could get ``pinched'' to a single point.)
\item If $\partial B(x,r)$ is a simply traced figure eight, then $y$ lies in one of the two loops of the figure eight; all the points along the other loop correspond to local maxima of $X_t$, as the red line in Figure~\ref{fig::lamination2}. On the other hand, there must be a dense set of points along the loop containing $y$ that correspond to points that are {\em not} local maxima (since if there were a whole interval of points that were local maxima, then the points on either ends of that interval would have to be equivalent). Thus one can recover from the tree $\wt T$ where the two special points must be. This will be discussed later in the specific context of the L\'evy net (where the contour function for $\wt T$ is a so called L\'evy height function derived from a L\'evy excursion in a particular way). See Figure~\ref{fig::levynet2}.

\item In order to speak about a ``leftmost tree'' we have to have an orientation assigned to the metric net (so it is not quite enough to just have the metric space structure of the metric net).

\end{enumerate}

\begin{proposition}
\label{prop::leftmosttreeismeasurable}
Let $F \colon \gmsspace^{2,O} \to \treespace$ be the following map.  If the metric net of $X \in \gmsspace^{2,O}$ is not strongly coalescent, it outputs the $0$ function.  Otherwise, it outputs the contour function for the completion of the leftmost geodesic tree.  Then $F$ is Borel measurable.
\end{proposition}

This will follow from the propositions below.

The Gromov-Hausdorff distance on compact metric spaces marked by $k$ points discussed in~\eqref{eqn::dghk_def} can be extend to a metric on compact metric spaces $(S,A_1,A_2,\ldots,A_k)$ marked by $k$ distinguished closed subsets by setting
\[ \dgh\Bigl((S,A_1,A_2,\ldots,A_k), (\wt S,\wt A_1,\wt A_2,\ldots,\wt A_k)\Bigr) = \inf \Bigl( \dh(S, \wt S) + \sum_{i=1}^k \dh(A_i, \wt A_i) \Bigr).\]
It is a simple exercise to check the following:

\begin{proposition}
\label{prop::GHwithsets}
If a sequence of $k$-subset-marked compact metric spaces converges w.r.t.\ the metric $\dgh$ (on the $S$ components) then it has a subsequence that converges w.r.t.\ $\dgh$ (on all components). 
\end{proposition}

\begin{figure}[ht!]
\begin{center}
\includegraphics[scale=1]{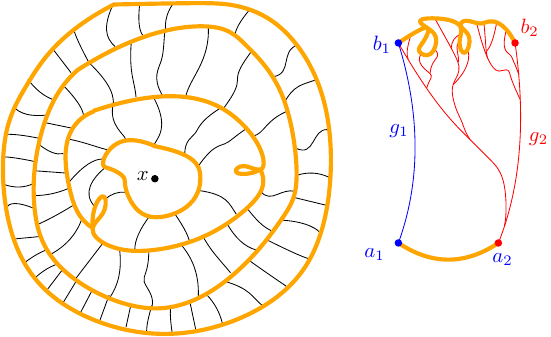}
\caption{\label{fig::quad} {\bf Left:} Shown in orange are the metric ball boundaries $\partial B(x,k\epsilon)$ for integer $k$.  {\bf Right:} The distance from $a_1$ to $b_1$ is exactly $\epsilon$. The distance from $a_2$ to {\em any} point on the black arc between $b_1$ and $b_2$ to is exactly $\epsilon$. The distance between any other pair of points --- with one on the $a_1$ to $a_2$ arc and one on the $b_1$ to $b_2$ arc --- is strictly larger than $\epsilon$. The metric sphere itself is a geodesic metric space, and it follows that the metric net is ``almost'' a geodesic space in the sense that if $a$ and $b$ are any points on the metric net then one can create a geodesic between them if one adds a countable collection of arcs, each of which connects two points on the same hole and has length given by the $d$ distance between those points.}

\end{center}
\end{figure}

Now given any metric net and small constant $\epsilon>0$, we can consider the ball boundaries $\partial B(x,k\epsilon)$ for positive integer $k$, as shown in Figure~\ref{fig::quad}.  A point on such a boundary $\partial B(x,k\epsilon)$ is called a ``coalescence point'' if it lies on a leftmost geodesic drawn from {\em some} point on $\partial B\bigl(x,(k+1)\epsilon\bigr)$ back to $x$.  Clearly, if the surface is strongly coalescent, the number of coalescence points on each $\partial B(x,k\epsilon)$ is finite (and each such point necessarily lies on $\partial \fb{x}{k\epsilon}$).  Now imagine we fix $k$ and number the coalescent points around $\partial \fb{x}{k\epsilon}$ clockwise as $a_0, a_1, \ldots a_{m-1}$. (It does not matter which one we designate as $a_0$.) For each $a_i$ we define a point $b_i$ to be the {\em rightmost} point on $\partial B\bigl(x,(k+1)\epsilon\bigr)$ with the property that the leftmost geodesic started at that point hits $a_i$. (``Rightmost'' can be interpreted as ``clockwisemost'' within the universal cover of the annulus.) Let $g_i$ be the leftmost geodesic connecting $b_i$ to $a_i$ and let $Q_i$ be the quadrilateral bounded between $g_i$ and $g_{i+1}$ (addition taken modulo $m$) as shown on the RHS of Figure~\ref{fig::quad}. Call such a $Q_i$ a {\bf quad}.

The caption of  Figure~\ref{fig::quad} explains some conditions that quads must satisfy. But the set of possible $Q_i$ satisfying these conditions is not a GH closed set: the GH limit of a sequence of $Q_i$ can degenerate in at least two ways: first, the paths $g_i$ may collide in the limit, so that they no longer correspond to disjoint and distinct leftmost geodesics.  Second, the distance between a pair of points on the upper and lower arcs may approach $\epsilon$ in the limit---so that perhaps the limiting quad has {\em additional} leftmost geodesics.

If $Q$ is the quad shown on the RHS of Figure~\ref{fig::quad}, then the portion of the clockwise arc of $\partial \fb{x}{k\epsilon}$ that lies between $b_1$ and $b_2$  (not counting $b_1$) is called the {\bf entrance arc} of $Q$ while the point $a_2$ is called the {\bf exit point} from $Q$.  Note that if $z$ is any point on the entrance arc of $Q$, then the leftmost geodesic from $z$ to $x$ necessarily passes through the exit point of $Q$. We say that a quad $Q'$ is a {\bf child} of $Q$ if the exit point of $Q'$ lies on the entrance arc of $Q$. We say $Q'$ is a {\bf boundary child} of $Q$ if the exit point of $Q'$ lies on the right boundary of the entrance arc of $Q$.

The collection of all quads $Q$, together with the child-parent relationship, forms a tree rooted at $B(x, \epsilon)$ (which one may interpret as a ``root quad''), in which some of the child-parent edges are designated ``boundary.''  Let $T$ be the labeled tree defined this way.

\begin{proposition}
\label{prop::treemeasurable}
The map $\gmsspace^{2,O} \to \treespace$ which outputs the contour function of the tree $T$ defined above is Borel measurable.
\end{proposition}
\begin{proof}
Given a labeled tree $T$ and positive $r,s,\delta$ (with $r<\delta/100$ and $s<\delta/100$) we let $A(T,r,s,\delta)$ be the set of geodesic quintuply marked (hence oriented) metric spheres $(S,d,x,y,\ul{x})$ where the marked points are distinct and $\ul{x} = (x_1,x_2,x_3)$ consists of three points in $\partial \fb{x}{r}$, $r = d(x,y)/2$, whose metric net can be sliced along (not necessarily leftmost) geodesics into ``approximate quads'' in the manner of Figure~\ref{fig::quad} in such a way that the following conditions hold (where, as illustrated, the upper and lower boundaries of an approximate quad need not be simple curves, but they are connected sets):
\begin{enumerate}
\item The distance between the two geodesics forming the left and right sides of any given quad (the curves $g_1$ and $g_2$ in Figure~\ref{fig::quad}) is at least $\delta$. Moreover, the distance between any quad $Q$ and any non-neighboring quad (i.e., any quad that would not be distance zero from $Q$ for a system of true quads corresponding to the labeled tree $T$) is at least $\delta$.
\item For each quad (labeled as in  Figure~\ref{fig::quad}) the distance between any point on the upper arc and any point on the lower arc is at least $\epsilon +s$ {\em unless} unless the lower point is in $B(a_2, s)$ or we have both that the lower point is in $B(a_1,s)$ and the upper point is in $B(b_1,s)$.
\item The left boundary of each quad coincides exactly with the right boundary of the quad to its left.
\item If $z_1$ is any point in an approximate quad $Q$, and $z_2$ is any other point outside of $Q$, then one can find a $z_3$ such that $z_3$ lies on the left, right, upper or lower boundary of $Q$ and $d(z_1, z_3) = d(z_1,z_3)+d(z_3,z_2)$. (This is also automatic from the construction and the fact that the overall metric sphere is a length space.)
\end{enumerate}

We stress that the ``approximate quads'' whose existence defines membership in  $A(T,r,s,\delta)$ are not ``true quads'' in the sense of satisfying all of the conditions of the quad shown in Figure~\ref{fig::quad}. Rather, they satisfy an approximation of those conditions.

The proof of the proposition proceeds in two parts:
\begin{enumerate}
\item First, we observe that $A(T,r,s,\delta)$ is a Gromov-Hausdorff closed set. (More precisely, every element in $A(T,r,s,\delta)$ that is obtainable as the metric net of a doubly marked oriented and geodesic sphere is again in $A(T,r,s,\delta)$.  So $A(T,r,s,\delta)$ is closed {\em within the space of quintuply marked geodesic metric spheres as above (i.e., oriented)}.)
\item 
Then we show that 
\[ A(T) = \bigcup_\delta \bigcap_r \bigcup_s A(T,r,s,\delta) \]
(where $\delta$, $r$, and $s$ are all restricted to powers of two) contains the set $M(T)$ of metric nets corresponding to the labeled tree $T$  --- which implies that the latter can be produced from countable unions/intersections of closed sets and is hence measurable. This implies that $T$ is a measurable function of the oriented metric net.
\end{enumerate}

The first part is a straightforward application of Proposition~\ref{prop::GHwithsets}. Given any sequence of elements in $A(T,r,s,\delta)$, we can form a marked sequence by decorating each sequence element with a set of quads satisfying the given conditions; for each quad there are five compact sets (the left, right, upper and lower boundaries and the whole quad itself) and by Proposition~\ref{prop::GHwithsets} one can find a subsequence along which the whole collection of sets converges; next one just observes that the properties that characterize $A(T,r,s,\delta)$ are all evidently preserved by limits of this form.

For the second part, note that since the sets $A(T,r,s,\delta)$ are decreasing in $s$, and we are taking a union over $s$, it is enough to consider very small $s$ (say $s$ smaller than any fixed threshold).  Similarly, since one is taking an intersection over $r$, it is enough to limit attention to $r$ below any fixed threshold.  Similarly, the sets $ \bigcap_r \bigcup_s A(T,r,s,\delta)$ are decreasing in $\delta$.  So, we find that the definition of $A(T)$ does not change if we require that $\delta < \epsilon/1024$ and that furthermore $s < \delta /1024$ and $r < \delta / 1024$.

Next, note that the conditions on the quads in $A(T,r,s,\delta)$ guarantee that no leftmost geodesic (starting at least $s$ distance away from $b_1$) terminates more than $s$ units from $a_2$.  This implies that either $a_2$ or a point slightly (at most $r$ units) to its left is a merge point.  So the number of true quads is at least the number of vertices in $T$. At this point, there could in principle be {\em other} true quads since there could be other merge points within $r$ units of $a_2$. However, any such true quad would have to have some positive width (some corresponding $\delta'$) which would have to be less than $r$---so the number of true quads with width {\em greater} than $r$ has to be at most the corresponding number in $T$.  Because we can take $r$ arbitrarily small, this implies that the number of {\em true quads} is exactly the number of vertices in $T$, taking the $r \to 0$ limit, it is not hard to see that the tree structure must agree with $T$.
\end{proof}

\begin{proof}[Proof of Proposition~\ref{prop::leftmosttreeismeasurable}]
For each $\epsilon > 0$, we let $T_\epsilon$ be the tree produced from Proposition~\ref{prop::treemeasurable} and let $X^\epsilon$ be its contour function.  We also let $T$ be the tree of leftmost geodesics.  If $(S,d,\nu,x,y)$ has a metric net which is strongly coalescent, by Proposition~\ref{prop::coalescentprecompact} we know that $T$ is precompact when equipped with the tree metric.  In particular, $T$ is totally bounded.  This implies that for every $\delta > 0$ there exists $\epsilon_0 > 0$ so that for every $\epsilon \in (0,\epsilon_0)$ every point $T$ has distance at most $\delta > 0$ from a point in $T_\epsilon$, when equipped with the tree metric.  From this, it is not difficult to see that for all $\epsilon,\epsilon' \in (0,\epsilon_0)$ the uniform distance between $X^\epsilon$ and $X^{\epsilon'}$ modulo monotone reparameterization is at most $\delta > 0$.  Therefore $X^\epsilon$ converges as $\epsilon \to 0$ modulo monotone reparameterization to the contour function $X$ for $T$.  This proves the desired measurability of $X$ as a function of $(S,d,\nu,x,y)$ since we have exhibited it as a limit of measurable maps.
\end{proof}

Combining the previous two propositions implies that the set of doubly-marked and oriented geodesic spheres whose metric net from $x$ to $y$ is strongly coalescent is measurable.  Indeed, we know that the contour function $X^\epsilon$ of $T_\epsilon$ is measurable for each $\epsilon > 0$.  As explained in the proof of Proposition~\ref{prop::coalescentprecompact}, whether $(S,d,\nu,x,y) \in \gmsspace^{2,O}$ has metric net from $x$ to $y$ which is strongly coalescent is determined by whether there exists rational $0 < r_1 < r_2$ such that the number of crossings made by the graph of $X^\epsilon$ across the lines with heights $r_1,r_2$ is unbounded as $\epsilon \to 0$.

We are now going to upgrade the statement of Proposition~\ref{prop::leftmosttreeismeasurable} to obtain that the map which takes as input an element of $\gmsspace^{2,O}$ and outputs the pair consisting of the contour function of the leftmost geodesic tree and the equivalence relation which encodes how the tree is glued to itself is Borel measurable.  In order to formalize this statement, we need to introduce an appropriate space and $\sigma$-algebra.

We let $\treeequivspace$ be the set of pairs consisting of a continuous function $X \colon \T_1 \to \R_+$ which is not constant in any interval and with $\inf_{t \in \T_1} X_t = 0$ and a compact set $K \subseteq \T_2$, $\T_2 = \T_1 \times \T_1$, where we consider pairs $(X,A)$, $(Y,K)$ in $\treeequivspace$ to be equivalent if there exists an increasing homeomorphism $f \colon \T_1 \to \T_1$ so that $X = Y \circ f$ and $A = f^{-1}(K)$ where we abuse notation and write $f^{-1}(K) = \{(f^{-1}(x),f^{-1}(y)) : (x,y) \in K\}$.  We define a metric $d$ on $\treeequivspace$ by setting
\[ d((X,A),(Y,K)) = \inf_f \bigg( \| X - Y \circ f\|_\infty + \dh(A,f^{-1}(K))  \bigg)\]
where the infimum is over all $f$ as above.

\begin{proposition}
\label{prop:equiv_measurable}
Consider the map $F \colon \gmsspace^{2,O} \to \treeequivspace$ which is defined as follows.  Suppose that $(S,d,\nu,x,y) \in \gmsspace^{2,O}$.  If $(S,d)$ is not an oriented geodesic sphere whose metric net from $x$ to $y$ is strongly coalescent, then it outputs $(0,\emptyset)$.  If $(S,d)$ is a geodesic sphere whose metric net from $x$ to $y$ is strongly coalescent, then it outputs the pair consisting of the contour function $X \colon \T_1 \to \R_+$ of the leftmost geodesic tree and the set $K \subseteq \T_2$ which consists exactly of those pairs $(s,t) \in \T_2$ which correspond to the same points in $S$, both modulo monotone parameterization.  Then $F$ is Borel measurable.
\end{proposition}
\begin{proof}
The same argument used to prove Proposition~\ref{prop::leftmosttreeismeasurable} implies that the following is true.  Consider the map $G$ from $\gmsspace^{2,O}$ to the space which consists of an element $(S,d,\nu,x,y)$ of $\gmsspace^{2,O}$ and a continuous map $f \colon \T_1 \to S$, defined modulo monotone parameterization, which is defined as follows.  If $(S,d,\nu,x,y)$ is not a geodesic sphere with strongly coalescent metric net from $x$ to $y$, then $G$ outputs the pair consisting of $(S,d,\nu,x,y)$ and the function $f(t) = x$ for all $t \in \T_1$.  If $(S,d,\nu,x,y)$ is a geodesic sphere with strongly coalescent metric net from $x$ to $y$, then $G$ outputs the pair consisting of $(S,d,\nu,x,y)$ and the function $f \colon \T_1 \to S$ which is defined by setting $f(t)$ for $t \in \T_1$ to be equal to the point in $S$ on the leftmost geodesic tree which corresponds to $t \in \T_1$ using the contour function $X$.  Then this is a measurable function when we equip the target space with the obvious extension of the Gromov-weak topology.  We can construct $F$ from $G$ as follows.  If $(S,d,\nu,x,y)$ is a doubly marked and oriented geodesic sphere whose metric net from $x$ to $y$ is strongly coalescent and $((S,d,\nu,x,y),f) = G((S,d,\nu,x,y))$, then we set $F( (S,d,\nu,x,y))$ to consist of the contour function produced by Proposition~\ref{prop::leftmosttreeismeasurable} and the compact set $K$ of $\T_2$ consisting of those $0 \leq s,t \leq 1$ so that $d(f(s),f(t)) = 0$.  This function is measurable since the function which takes as input a pair consisting of an element $(S,d,\nu,x,y)$ of $\gmsspace^{2,O}$ and a continuous function $f \colon \T_1 \to S$ and outputs the pair $(S,d,\nu,x,y)$ and the set $K$ as described just above is a measurable map from $\gmsspace^{2,O}$ to the space of doubly marked oriented and geodesic spheres which are also marked by a compact set in $\T_2$.
\end{proof}

\section{Tree gluing and the L\'evy net}
\label{sec::surfacesfromtrees}

Section~\ref{subsec::spheressymetric} and Section~\ref{subsec::spheressymetricholes} briefly recall two tree-mating constructions developed in~\cite{matingtrees}, one involving a pair of continuum random trees, and the other involving a pair of $\alpha$-stable looptrees \cite{ck2013looptrees}.  These very brief sections are not strictly necessary for the current project, but we include them to highlight some relationships between this work and~\cite{matingtrees} (relationships that play a crucial role in the authors' works relating the Brownian map and pure Liouville quantum gravity). The real work of this section begins in Section~\ref{subsec::spheresassymetricholes}, which describes how to construct the $\alpha$-L\'evy net by gluing an $\alpha$-stable looptree to itself (or equivalently, by gluing an $\alpha$-stable looptree to a certain related real tree derived from the $\alpha$-stable looptree --- the geodesic tree of the L\'evy net). The reader may find it interesting to compare the construction in Section~\ref{subsec::spheresassymetricholes}, where a {\em single} $\alpha$-stable looptree is glued {\em to itself}, to the one in~Section~\ref{subsec::spheressymetricholes}, where {\em two} $\alpha$-stable looptrees are glued {\em to each other}.  In Section~\ref{subsec::levynet_second_construction} we present a different but (it turns out) equivalent way to understand and visualize the L\'evy net construction given in Section~\ref{subsec::spheresassymetricholes}.  We give a review of continuous state branching processes in Section~\ref{subsec::csbp}, then give a breadth-first construction of the L\'evy net in Section~\ref{subsec::levynetbreadthfirst}, and finally prove the topological equivalence of the L\'evy net constructions in Section~\ref{subsec::levy_net_top_equiv}.  We end this section by showing in Section~\ref{subsec::recovering_embedding} that the embedding of the L\'evy net into $\S^2$ is determined up to homeomorphism by the geodesic tree and its associated equivalence relation in the L\'evy net.

\subsection{Gluing together a pair of continuum random trees}
\label{subsec::spheressymetric}

There are various ways to ``glue together'' two continuum trees to produce a topological sphere decorated by a space-filling path (describing the ``interface'' between the two trees).  One approach, which is explained in \cite[Section~1.1]{matingtrees}, is the following: let~$X_t$ and~$Y_t$ be independent Brownian excursions, both indexed by $t \in [0,T]$.  Thus $X_0 = X_T = 0$ and $X_t > 0$ for $t \in (0,T)$ (and similarly for~$Y_t$).  Once~$X_t$ and~$Y_t$ are chosen, choose~$C$ large enough so that the graphs of $X_t$ and $C-Y_t$ do not intersect.  (The precise value of $C$ does not matter.)  Write $R = [0,T] \times [0,C]$, viewed as a Euclidean metric space.

Let~$\cong$ denote the smallest equivalence relation on~$R$ that makes two points equivalent if they lie on the same vertical line segment with endpoints on the graphs of~$X_t$ and $C-Y_t$, or they lie on the same horizontal line segment that never goes above the graph of~$X_t$ (or never goes below the graph of $C-Y_t$).  Maximal segments of this type are shown in Figure~\ref{fig::lamination_treemaking}.  As explained in \cite[Section~1.1]{matingtrees}, if one begins with the Euclidean rectangle and then takes the topological quotient w.r.t.\ this equivalence relation, one obtains a topological sphere, and the path obtained by going through the vertical lines in left-to-right order is a continuous space-filling path on the sphere, which intuitively describes the ``interface'' between the trees encoded by $X_t$ and $Y_t$ after quotienting by $\cong$.  In fact, this remains true more generally when~$X_t$ and~$Y_t$ are not independent, and the pair $(X_t, Y_t)$ is instead an excursion of a correlated two-dimensional Brownian motion into the positive quadrant (starting and ending at the origin), as explained in detail in \cite{matingtrees, quantum_spheres}.

\begin{figure}[ht!]
\begin{center}
\includegraphics[width=0.52\textwidth]{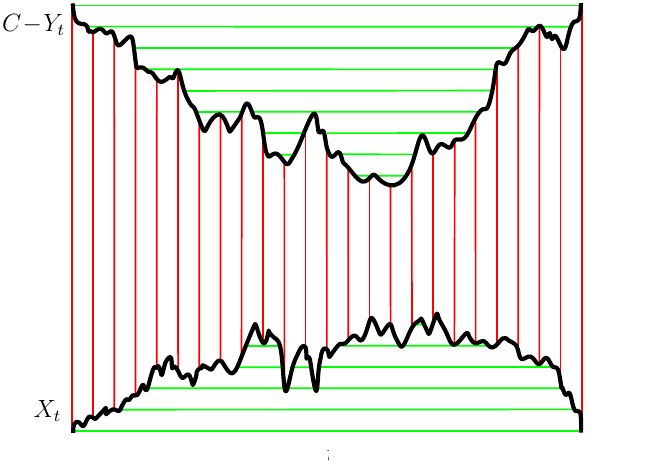}
\includegraphics[width=0.46\textwidth]{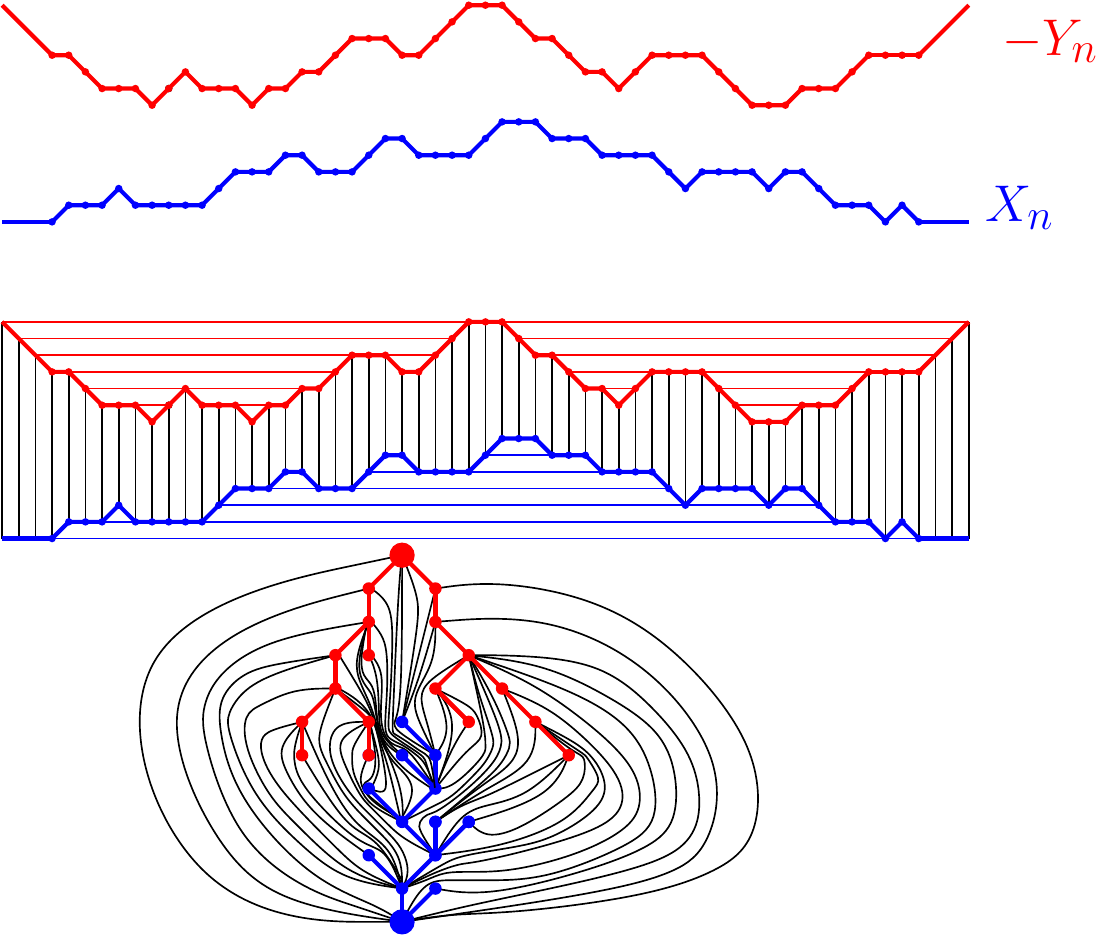}
\caption{\label{fig::lamination_treemaking}  {\bf Left:} {\it Gluing continuum random trees to each other.} Here $X_t$ and $Y_t$ are Brownian excursions and $C$ is a constant chosen so that the two graphs shown do not intersect.   Points on the same vertical (or horizontal) line segment are declared to be equivalent.  The space of equivalence classes (endowed with the quotient topology) can be shown to be homeomorphic to the sphere \cite[Section~1.1]{matingtrees}. {\bf Right:} {\it Gluing discrete trees to each other.}  There is a standard discrete analog of the construction shown in the left that produces a planar triangulation (with distinguished tree and dual tree) from a finite walk $(X_n,Y_n)$ in $\Z_+^2$ that starts and ends at $(0,0)$.   The bottom figure is obtained by collapsing the horizontal red and blue lines to produce two trees, connected to each other by black edges. See \cite{mullintrees, bernardibijection, sheffield2011quantum} for details. }
\end{center}
\end{figure}

\subsection{Gluing together a pair of stable looptrees}
\label{subsec::spheressymetricholes}

Also discussed in \cite[Section~1.3]{matingtrees} is a method of obtaining a sphere by gluing together two independent stable looptrees (with the disk in the interior of each loop included), as illustrated in Figure~\ref{fig::levytreegluing}.  Each stable looptree is encoded by the time-reversal of an $\alpha$-stable L\'evy excursion with only upward jumps with $\alpha \in (1,2)$.  In the setting discussed there, each of the grey disks surrounded by a loop is given a conformal structure (that of a ``quantum disk''), and this is shown to determine a conformal structure of the sphere obtained by gluing the trees together; given this structure, the interface between the trees in Figure~\ref{fig::levytreegluing} is shown to be an $\SLE_{\kappa'}$ process for $\kappa' = 16/\gamma^2 \in (4,8)$ where $\alpha = \kappa'/4 \in (1,2)$.  In a closely related construction, the interface between the trees in the left side of Figure~\ref{fig::lamination_treemaking} is shown to be a space-filling form of $\SLE_{\kappa'}$ in which the path ``goes inside and fills up'' each loop after it is created. As explained in \cite{matingtrees}, one obtains a range different values of $\kappa'$ by taking the trees to be correlated with each other and varying the correlation coefficient.

\begin{figure}[ht!]
\begin{center}
\includegraphics[scale=.8]{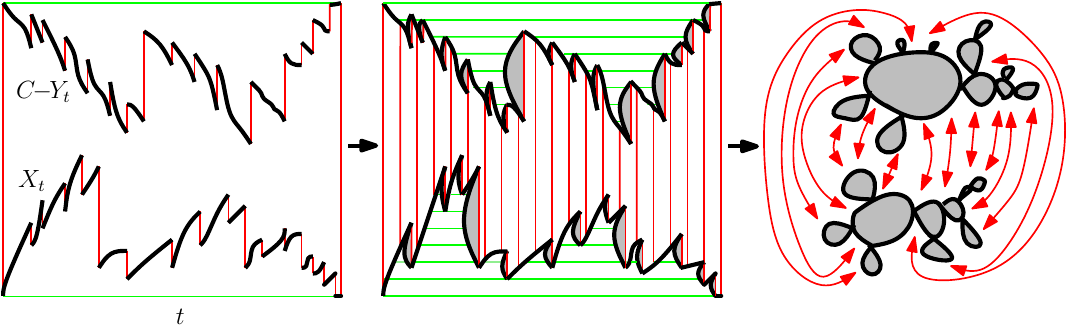}
\caption{\label{fig::levytreegluing} {\bf \it Gluing stable looptrees to each other.} {\bf Left:}  $X_t$ and $Y_t$ are independent and are each given by the time-reversal of an $\alpha$-stable L\'evy excursion, $\alpha \in (1,2)$, with only upward jumps (so that $X_t$, $Y_t$ have downward jumps).  Graphs of $X_t$ and $C-Y_t$ are sketched; red segments indicate jumps. {\bf Middle:} Add a black curve to the left of each jump, connecting its two endpoints; the precise form of the curve does not matter (as we care only about topology for now) but we insist that it intersect each horizontal line at most once and stay strictly below the graph of $X_t$ (or above the graph of $C-Y_t$) except at its endpoints.  (The reader may easily verify that it is a.s.\ possible to draw such a path for every jump discontinuity.) We also draw the vertical segments that connect one graph to another, as in the left side of Figure~\ref{fig::lamination_treemaking}, declaring two points equivalent if they lie on the same such segment (or on the same jump segment).  Shaded regions (one for each jump) are topological disks.  {\bf Right:} By collapsing green segments and red jump segments, one obtains two trees of disks with outer boundaries identified.}
\end{center}
\end{figure}

\subsection{Gluing a stable looptree to itself to obtain the L\'evy net}
\label{subsec::spheresassymetricholes}

Throughout, we fix $\alpha \in (1,2)$.  Figure~\ref{fig::levynet} illustrates a procedure for generating a sphere from a single stable looptree, which in turn is generated from the time-reversal of an $\alpha$-stable L\'evy excursion with only upward jumps. (See Definition~\ref{def::levynet} below for a more formal description.) Precisely, Proposition~\ref{prop::spherelevynet} below will show that the topological quotient of the rectangle, w.r.t.\ the equivalence relation illustrated, actually is a.s.\ homeomorphic to the sphere. The process~$Y_t$ illustrated there is sometimes known as the {\em height process} of the process $X_t$, which is the \cadlag\ modification of the time-reversal of an $\alpha$-stable L\'evy excursion with upward jumps (so $X_t$ has downward jumps). The fact that this~$Y_t$ is well-defined and a.s.\ has a continuous modification (along with H\"older continuity and the exact H\"older exponent) is established for example in \cite[Theorems~1.4.3 and~1.4.4]{duquesnelegall2005levytrees} (see also \cite{legalllejan1998levytrees}).

In this construction the upper tree in the figure is not independent of the lower tree (with holes); in fact, it is strictly determined by the L\'evy excursion below, as explained in the figure caption.  Note that every jump in the L\'evy excursion (corresponding to a bubble) comes with a ``height'' which is encoded in the upper tree.  If one removes from the constructed sphere the grey interiors of the disks shown, one obtains a closed subset of the sphere; this set, together with its topological structure, can also be obtained directly without reference to the sphere (simply take the quotient topology on the set of equivalence classes in the complement of the grey regions in Figure~\ref{fig::levynet}).  It is important to note that the set of \emph{record infima} achieved by $X|_{[t,T]}$ looks locally like the range of a stable subordinator with index $\alpha - 1$ \cite[Chapter~VIII, Lemma~1]{bertoinlevybook}, and that in particular it a.s.\ has a well-defined Minkowski measure \cite{ft1983local}, which also corresponds to the time parameter of the stable subordinator.\footnote{For an $\alpha$-stable process with no {\em negative} jumps ($\beta = 1$ in language of \cite{bertoinlevybook}) the statement in \cite[Chapter~VIII, Lemma~1]{bertoinlevybook} is that the set of record {\em maxima} (the range of the so-called ``ladder height'' process) has the law of the range of a stable subordinator of index $\alpha \rho$ where
\[ \rho = \frac12 + (\pi \alpha)^{-1} \arctan (\tan (\pi \alpha/2)) = \frac12 + (\pi\alpha)^{-1} (\pi  \alpha/2-\pi) = 1-1/\alpha.\]
(Recall that for $x \in (\pi/2, \pi)$ we have $\arctan (\tan (x)) = x - \pi$.)  Thus in this case the index of the stable subordinator is $\alpha \rho = \alpha - 1 $.  This value varies between $0$ and $1$ as $\alpha$ varies between $1$ and $2$.  The dimension of the range is given by the index $\alpha - 1$ (a special case of \cite[Chapter~III, Theorem~15]{bertoinlevybook}.}

\begin{figure}[ht!]
\begin{center}
\includegraphics [scale=.8]{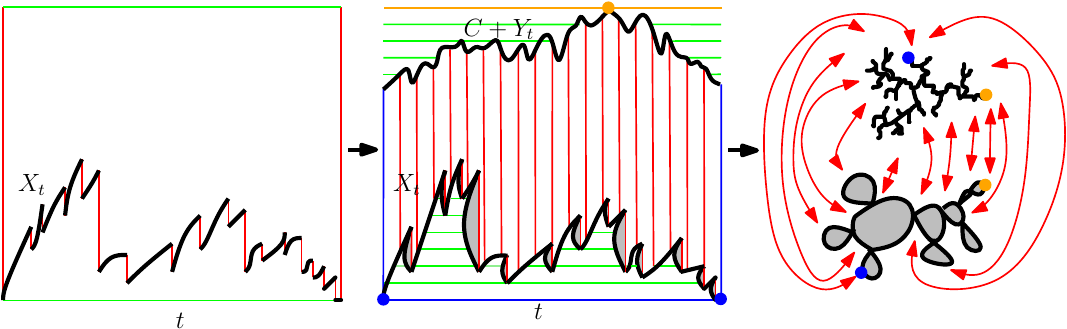}
\caption{\label{fig::levynet} {\bf \it Gluing a stable looptree to itself.}  Illustration of Definition~\ref{def::levynet}, the definition of the $\alpha$-stable L\'evy net. {\bf Left:}  $X_t$ is the \cadlag\ modification of the time-reversal of an $\alpha$-stable L\'evy excursion with only positive jumps.  {\bf Middle:} Extra arcs are added to the lower graph as in Figure~\ref{fig::levytreegluing}.  $Y_t$ is the Minkowski measure of the set of record infimum values obtained by $X|_{[t,T]}$.  (This quantity corresponds to a ``distance'' to the dual root, in the sense of \cite{dlg2002trees_levy}.) Red and green lines indicate equivalences.  Note that whenever the lower endpoints of two vertical red segments are connected to one another by a green segment, it must be the case that the upper endpoints have the same height (which may be hard to recognize from this hand-drawn figure).  {\bf Right:} Once the green lines are collapsed, one has a tree and a tree of loops (which we will refer to as either the dual tree or looptree).  The tree above is the geodesic tree.  The orange dot is the root of that tree.  The blue dot is a ``dual root'' (a second marked point).  The horizontal green lines above the graph of $Y_t$ ``wrap around'' from one side of the rectangle to the other; these lines correspond to the points on the geodesic tree arc from the orange dot to the blue dot.}
\end{center}
\end{figure}

We now give the formal definition of the L\'evy net, which is defined in terms of an $\alpha$-stable L\'evy excursion with only upward jumps.  Recall \cite[Chapter~VIII.4]{bertoinlevybook} that the standard infinite measure on $\alpha$-stable L\'evy excursions with only upward jumps is constructed as follows.  One first picks a lifetime~$T$ from the infinite measure $c_\alpha T^{-1/\alpha-1} dT$ where $c_\alpha > 0$ is a constant and $dT$ denotes Lebesgue measure on $\R_+$.  One then samples a normalized (unit length) excursion and then finally scales space and time respectively by the factors~$T^{1/\alpha}$ and~$T$.

\begin{definition}
\label{def::levynet}
Fix $\alpha \in (1,2)$ and suppose that $X_t$ is the \cadlag\ modification of the time-reversal of an $\alpha$-stable L\'evy excursion (as defined just above) and let $Y_t$ be its associated height process.  Fix $C > 0$ large enough so that the graphs of $X_t$ and $C+Y_t$ are disjoint and let~$R$ be the smallest Euclidean rectangle which contains both the graphs of~$X_t$ and $C+Y_t$.  We then define an equivalence relation on $R$ as follows.  We declare points of $R$ which lie above the graph of $C+Y_t$ to be equivalent if they lie on a horizontal chord which does not cross the graph of $C+Y_t$.  For each $t$, we declare the points of $R$ on the vertical line segment from $(t,X_t)$ to $(t,C+Y_t)$ to be equivalent.  Finally, we declare points of $R$ which lie below the graph of $X_t$ (extended as in Figure~\ref{fig::levytreegluing} and Figure~\ref{fig::levynet}) to be equivalent if they lie on a horizontal chord which does not cross the graph of~$X_t$.  The quotient space w.r.t.\ this equivalence relation is the doubly marked compact topological space that we call the {\bf ($\alpha$-stable) L\'evy net}. Let $\pi$ be the corresponding quotient map from $R$ to this space. As Figure~\ref{fig::levynet} illustrates the topological space can be understood as a gluing of a pair of trees: the geodesic tree $T_1$ (corresponding to $C-Y_t$) and a dual tree $T_2$ (corresponding to $X_t$). The roots of these two trees are respectively the \emph{root} and \emph{dual root} of the L\'evy net.

(The reason for these names for $T_1,T_2$ is that we will later find that the $3/2$-stable L\'evy net describes the metric net in the Brownian map where $T_1$ is the tree of geodesics.)

We view the L\'evy net as a random variable taking values in the space $\treeequivspace$ defined at the end of Section~\ref{subsec::mmsigma} (i.e., a continuous function on $\T_1 \to \R_+$ which defines a real tree together with a compact subset of $\T_2$ which defines a topologically closed equivalence relation on $\T_1$, both viewed modulo monotone reparameterization).  The tree is the one encoded by $-(Y_{t/T} - \sup_{t \in [0,T]} Y_t)$.  The equivalence relation on $\T_1$ is induced by the equivalence relation on $R$ defined above (we will prove in Proposition~\ref{prop::spherelevynet} that this equivalence relation is topologically closed).

Although {\em a priori} we do not put a full metric space structure on the L\'evy net, we define the \emph{distance to the root} of a point in the L\'evy net to be the distance inherited from the geodesic tree, i.e., the value of the function $\sup_s Y_s-Y_t$.  The image of a shortest path to the root in $T_1$ is called  a {\em geodesic to the root}. Also, it is not hard to see that every point in the L\'evy net corresponds to either one or two points in $T_1$, and hence has either one or two distinguished ``geodesics'' from itself to the root (see the proof of Proposition~\ref{prop::spherelevynet}). When there are two, we refer to them as a {\em leftmost geodesic} and a {\em rightmost geodesic}, depending on whether they correspond to the leftmost or rightmost path in $T_1$.
\end{definition}

The left and right geodesics arise in Definition~\ref{def::levynet} when two geodesics in the geodesic tree are identified together at some point.  Every point in the dual tree which is a child of such a point then has at least two geodesics in the L\'evy net which go back to the root.  Since the L\'evy net is defined by an equivalence relation on a Euclidean rectangle, there is a well-defined leftmost and rightmost geodesic from each point back to the root (there in fact can be many geodesics from a given point back to the root).  These are the left and right geodesics referred to in Definition~\ref{def::levynet} just above.

We now establish a few basic properties of the L\'evy net.

\begin{proposition}
\label{prop::nodoublesidedrecord}
Suppose that $Y_t$ is the height process associated with the time-reversal of an $\alpha$-stable L\'evy excursion with only upward jumps.  It is a.s.\ the case that $Y_t$ does not have a decrease time.  That is, it is a.s.\ the case that there does not exist a time $t_0$ and $h>0$ such that $Y_s \geq Y_{t_0}$ for all $s \in (t_0-h, t_0)$ and $Y_s \leq Y_{t_0}$ for $s \in (t_0, t_0+h)$.  Similarly, $Y_t$ a.s.\ does not have an increase time.
\end{proposition}

See Figure~\ref{fig::nodoublesidedrecord} for an illustration of the proof of Proposition~\ref{prop::nodoublesidedrecord}.  We will postpone the detailed proof to Section~\ref{subsec::levynetbreadthfirst}, at which point we will have collected some additional properties of the height process $Y_t$.  We emphasize that Proposition~\ref{prop::nodoublesidedrecord} will only be used in the proof of Proposition~\ref{prop::spherelevynet} stated and proved just below, so the argument is not circular.

\begin{proposition}
\label{prop::noisolatedmax}
Suppose that $Y_t$ is the height process associated with the time-reversal of an $\alpha$-stable L\'evy excursion with only upward jumps.  It is a.s.\ the case that $Y_t$ has countably many local maxima, and each of these local maxima occurs at a distinct height (and hence in particular each local maximum is isolated).
\end{proposition}
\begin{proof}
This is established in the first assertion in the proof of \cite[Theorem~4.4]{duquesnelegall2005levytrees}
See also \cite[Lemma~2.5.3]{dlg2002trees_levy} for a related result.
\end{proof}

\begin{proposition}
\label{prop::spherelevynet}
If one glues a topological disk into each of the loops of the looptree instance associated with an instance of the L\'evy net, then the topological space that one obtains is a.s.\ homeomorphic to~$\S^2$.
\end{proposition}

Proposition~\ref{prop::spherelevynet} implies that the quotient of the rectangle shown in Figure~\ref{fig::levynet}, w.r.t.\ the equivalence relation induced by the horizontal and vertical lines as illustrated is topologically equivalent to~$\S^2$.

We will prove Proposition~\ref{prop::spherelevynet} using Moore's theorem \cite{moore1925concerning}, which for the convenience of the reader we restate here.  Recall that an equivalence relation $\cong$ on $\S^2$ is said to be \emph{topologically closed} if and only if whenever $(x_n)$ and $(y_n)$ are two sequences in~$\S^2$ with $x_n \cong y_n$ for all $n$, $x_n \to x$ and $y_n \to y$ as $n \to \infty$, then $x \cong y$.  Equivalently, $\cong$ is topologically closed if the graph $\{(x,y): x \cong y \}$ is closed as a subset of $\S^2 \times \S^2$. The {\em topological closure} of a relation $\cong$ is the relation whose graph is the closure of the graph of $\cong$. (Note that it is not true in general that the topological closure of an equivalence relation is an equivalence relation.) The following statement of Moore's theorem is taken from \cite{milnor2004pasting}.

\begin{proposition}
\label{prop::moore}
Let~$\cong$ be any topologically closed equivalence relation on~$\S^2$.  Assume that each equivalence class is connected and not equal to all of~$\S^2$.  Then the quotient space $\S^2 / \cong$ is itself homeomorphic to~$\S^2$ if and only if no equivalence class separates~$\S^2$ into two or more connected components.
\end{proposition}

\begin{proof}[Proof of Proposition~\ref{prop::spherelevynet}]
We first claim that Proposition~\ref{prop::nodoublesidedrecord} implies that no vertical line segment corresponding to an equivalence class in Definition~\ref{def::levynet} (or Figure~\ref{fig::levynet}) has an endpoint on two distinct (non-zero-length) horizontal segments which correspond to an equivalence class in Definition~\ref{def::levynet}.  (The reader might find it helpful to look at Figure~\ref{fig::nodoublesidedrecord}, which illustrates the proof of Proposition~\ref{prop::nodoublesidedrecord}, to visualize the argument.)  Indeed, suppose that we have a vertical chord between the graphs of $X$ and $C+Y$ which connects to an endpoint of a horizontal chord, connecting $(a,Y_a+C)$ to $(b,Y_b+C)$ say, which lies above the graph of $C+Y$.  Then there cannot exist $t \in (a,b)$ so that the graph of $X$ in $(a,t]$ is strictly above $X_a$.  This follows because if there was such a $t \in (a,b)$ then the Minkowski measure of times at which $X|_{[t,T]}$ spends at its running infimum (i.e., the time parameter of the corresponding subordinator) would be larger than that of $X|_{[a,T]}$.  That is, $Y_t > Y_a$.  Thus if the vertical chord is from $(a,X_a)$ to $(a,Y_a+C)$, a horizontal chord below the graph of $X$ which contains $(a,X_a)$ must contain $(a,X_a)$ as its right endpoint.  This cannot happen because then $a$ would be a decrease time of $Y$, which is ruled out in Proposition~\ref{prop::nodoublesidedrecord}.  Alternatively, if the vertical chord is from $(b,X_b)$ to $(b,Y_b+C)$, then a horizontal chord below the graph of $X$ which contains $(b,X_b)$ must contain $(b,X_b)$ as its left endpoint.  Then $b$ would be an increase time of $Y$, which is again ruled out in Proposition~\ref{prop::nodoublesidedrecord}.  We conclude that no equivalence class contains a non-empty horizontal chord of both the upper and lower graphs.

The equivalence classes can thus be classified as:
\begin{enumerate}
\item[Type I:] Those containing neither upper nor lower horizontal chords.  These are isolated points in the interior of one of the topological disks glued into a loop of the looptree (e.g., on the interiors of the grey regions in Figure~\ref{fig::levynet}) or single vertical lines connecting one graph to the other.
\item[Type II:] Those containing an upper (but not lower) chord.  By Proposition~\ref{prop::noisolatedmax}, such a chord can hit the graph of $C+Y_t$ either two or three times, but not more. Thus these equivalence classes consist of a horizontal line segment attached to either two or three vertical chords.
\item[Type III:] Those containing a lower (but not upper) chord. Since stable L\'evy processes with only downward jumps have a countable collection of unique local minima, such a chord must hit the black curves in either two or three places. In the (a.s.\ countable) set of places where the latter occurs, it is not hard to see that the rightmost point is a.s.\ in the interior of one of the boundaries of the grey regions.  (One can see from this that the path tracing the boundary of the looptree hits no point more than twice.) Thus the number of vertical line segments is either one (if one of the two endpoints lies on the boundary of a grey region) or two (if neither endpoint lies on the boundary of a grey region).
\end{enumerate}

From this description, it is obvious that all equivalence classes are connected, fail to disconnect the space, and do not contain the entire space. It only remains to check that the equivalence relation is topologically closed. To do this we use essentially the same argument as the one given in \cite[Section~1.1]{matingtrees}.  Suppose that $x_i$ and $y_i$ are sequences with $x_i \to x$ and $y_i \to y$, and $x_i \cong y_i$ for all $i$. Then we can find a subsequence of $i$ values along which the equivalence classes of $x_i$ and $y_i$ all have the same type (of the types enumerated above). By compactness, we can then find a further subsequence and such that the collection of segment endpoints converges to a limit. It is not hard to see that the resulting limit is necessarily a collection of vertical chords and horizontal chords (each of which is an equivalence class) that are adjacent at endpoints; since $x$ and $y$ are both in this limit we must have $x \cong y$.
\end{proof}

We next briefly remark that the L\'evy net can be endowed with a metric space structure in various ways.  Recall from Definition~\ref{def::levynet} that each point in the L\'evy net has either one or two geodesics back to the root in the tree encoded by $C-Y$ and that in the case there are two geodesics there is always a distinguished left geodesic.  The approach that we use in Definition~\ref{def::levynet} is to use the distance inherited from the leftmost geodesics: given any two points $x$ and $y$, one may draw their leftmost geodesic until they merge at a point $z$ and define the distance to be the sum of geodesic arc lengths from $x$ to $z$ and from $y$ to $z$. Another is to consider the geodesic tree (as described by $Y_t$) with its intrinsic metric structure and then take the quotient (as in Section~\ref{subsec::metricsphereobservations}) w.r.t.\ the equivalence relation induced by the gluing with the looptree. Note that when two points in the upper tree are equivalent, their distance from the root is always the same; thus, the distance between any point and the root is the same in the quotient metric space as it is in the tree itself. This implies that the metric space quotient defined this way is not completely degenerate --- i.e., it is not the case that {\em all} points become identified with each other when one takes the metric space quotient in this way. It would be natural to try to prove a stronger form of non-degeneracy for this metric structure: namely, one would like to show that a.s.\ {\em no} two distinct points in the L\'evy net have distance zero from each other in this quotient metric. This is not something that we will prove for general $\alpha$ in this paper; however, in the case that $\alpha = 3/2$, it will be derived in Section~\ref{sec::brownianmap} as a consequence of the proof of our main theorem.

We will see in Section~\ref{subsec::recovering_embedding} that given the structure described in Definition~\ref{def::levynet}, one can recover additional structure: namely an embedding in the sphere (unique up to homeomorphism of the sphere), a cyclic ordering of the points around each metric ball boundary (which is homeomorphic to either a circle or a figure $8$) with a distinguished point where the geodesic from $x$ to $y$ intersects the metric ball boundary, and a boundary length measure on each such boundary.

\subsection{A second approach to the L\'evy net quotient}
\label{subsec::levynet_second_construction}

\begin{figure}[ht!]
\begin{center}
\includegraphics [scale=1]{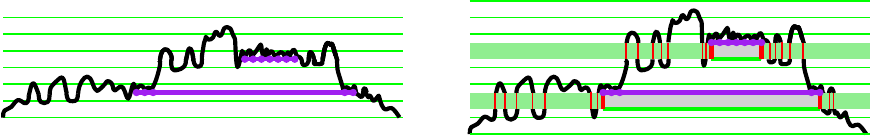}
\caption{\label{fig::levynet2}  {\bf Left:} Illustration of Definition~\ref{def::levynet2}, the second approach to the L\'evy net quotient.  Shown is the graph of $Y_t$ together with all horizontal lines, both above and below the graph, drawn as chords. The points on a horizontal chord that lies strictly above or below the graph ({\em except} for its two endpoints) are considered to be equivalent. The equivalence class corresponding to a given chord is either the chord itself or a pair of such chords above the graph with a common endpoint (a local maximum). The two horizontal purple segments correspond to sets of local minima of the same height each indicated with a purple dot, which in turn correspond to jumps of the L\'evy process. Only two such segments are drawn, but in fact there are infinitely many; the endpoints of such segments occupy a dense set of points on the graph of $Y_t$. Each such segment contains an uncountable collection of equivalence classes, including uncountably many single points (purple dots), countably many closed chords that lie strictly under the graph except at endpoints, and the pair of endpoints of the whole black segment (which is its own equivalence class). Each purple segment becomes a circle in the topological quotient.
{\bf Right:} Same graph with a horizontal stripe of ``extra space'' inserted at each purple segment. The height of the stripe can be chosen so that the sum of the heights of all of the (countably many) stripes is finite. At each of the (uncountably many) places where $Y_t$ intersects the purple segment, a corresponding red vertical ``bridge'' is added crossing the green stripe; points on the same bridge are considered equivalent.  Points on the closure of the same green rectangle (bounded between successive bridges) are also considered equivalent. The bottom, left, and right edges of each grey rectangle together constitute a single equivalence class, so that the topological quotient of each grey rectangle's boundary is a circle (as in the left figure).}
\end{center}
\end{figure}

We are now going to give another construction of a topological space with the height process $Y_t$ as the starting point which we will show just below is equivalent to the L\'evy net.  It is an arguably simpler way to understand Definition~\ref{def::levynet} (or Figure~\ref{fig::levynet}), which only involves the upper graph $C+Y_t$ (or equivalently just $Y_t$). The implications of this are discussed further in the caption to Figure~\ref{fig::levynet2}.

\begin{definition}[Second definition of the L\'evy net quotient]
\label{def::levynet2}
Let $R$ be the smallest rectangle which contains the graph of the height process $Y_t$.  We let $\cong$ be the smallest equivalence relation on $R$ in which points which lie on a horizontal chord which is strictly above or below the graph of $Y_t$ (except possibly at their endpoints) are equivalent and also points which are the left and right endpoints of the (uncountable) set of local minima of a given height corresponding to a jump time for $X_t$.
\end{definition}
See the left side of Figure~\ref{fig::levynet2} for an illustration of $\cong$ as in Definition~\ref{def::levynet2}.

\begin{proposition}
\label{prop::graphplushorizchords}
In the setting of Definition~\ref{def::levynet}, it is a.s.\ the case that two distinct points on the graph of $Y_t$ are equivalent in $\cong$ if and only if one of the following holds.
\begin{enumerate}
\item\label{it:equiv1} There is a horizontal chord above or below the graph of $Y_t$ that connects those two points and intersects the graph of $Y_t$ only at its endpoints.
\item\label{it:equiv2} There is a horizontal chord above the graph that intersects the graph of $Y_t$ at exactly one location, in addition to its two endpoints.
\item\label{it:equiv3} The two points are the left and right endpoints of the (uncountable) set of local minima of a given height corresponding to a jump time for $X_t$.
\end{enumerate}
Moreover, it is a.s.\ the case that two distinct points on the graph of $Y_t$ are equivalent under Definition~\ref{def::levynet} if and only if they are equivalent under Definition~\ref{def::levynet2}.
\end{proposition}
\begin{proof}

We begin by noting that a horizontal chord above the graph of $Y_t$ can intersect the graph of $Y_t$ in at most three places by Proposition~\ref{prop::noisolatedmax}.  We also note that a horizontal chord below the graph of $Y_t$ can only intersect the graph of $Y_t$ in two places or uncountably many places.  Indeed, suppose that the horizontal chord $[(a,Y_a)$, $(b,Y_b)]$ intersects the graph of $Y_t$ in at least $3$ places and let $(c,Y_c)$ be one of these points with $a < c < b$.  By the definition of $Y_t$, it follows that $X_c$ is a local minimum for $X_t$ which $X_t$ subsequently jumps below and therefore $b$ must be a jump time for $X_t$.  Therefore $[(a,Y_a),(b,Y_b)]$ necessarily intersects the graph of $Y_t$ uncountably many times, corresponding to the record minimum times of the time-reversal of $X|_{[a,b]}$.

We will now justify why the equivalence relations defined by Definition~\ref{def::levynet} and Definition~\ref{def::levynet2} are the same.  Suppose that $(a,Y_a)$, $(b,Y_b)$ are points on the graph of $Y_t$ with $a < b$.  Recalling the proof of Proposition~\ref{prop::spherelevynet}, they are equivalent using Definition~\ref{def::levynet} if and only if one of the following two possibilities hold (corresponding to Type II and Type III in the proof of Proposition~\ref{prop::spherelevynet}):
\begin{itemize}
\item $(a,Y_a)$, $(b,Y_b)$ are connected by a horizontal chord which lies above the graph of~$Y_t$.
\item $(a,X_a)$, $(b,X_b)$ are connected by a horizontal chord which lies below the graph of $X_t$.  By the definition of $Y_t$, this implies that $Y_a = Y_b$.  If $b$ is not a jump time for $X_t$, then this implies that $Y_r > Y_a = Y_b$ for all $r \in (a,b)$.  If $b$ is a jump time for $X_t$, then $(a,Y_a)$ and $(b,Y_b)$ are respectively the left and right endpoints of the set of local minima of $Y_t$ corresponding to the jump of $X_t$ at time $b$.
\end{itemize}
In each of these cases, $(a,Y_a)$ and $(b,Y_b)$ are equivalent under Definition~\ref{def::levynet2}.

Conversely, suppose that $(a,Y_a)$, $(b,Y_b)$ are equivalent under Definition~\ref{def::levynet2}.  Since $Y_t$ does not have increase or decrease times (Proposition~\ref{prop::nodoublesidedrecord}), it follows that the horizontal chord connecting $(a,Y_a)$, $(b,Y_b)$ cannot cross the graph of $Y_t$ (for otherwise there would be infinitely many intersections).  If the horizontal chord connecting $(a,Y_a)$, $(b,Y_b)$ lies non-strictly above the graph of $Y_t$, then it is obvious that $(a,Y_a)$, $(b,Y_b)$ are equivalent under Definition~\ref{def::levynet}.  If the horizontal chord connecting $(a,Y_a)$, $(b,Y_b)$ lies below the graph of $Y_t$ and intersects the graph of $Y_t$ only at its endpoints then we have that $X_a = X_b$ and the horizontal chord connecting $(a,X_a)$, $(b,X_b)$ lies below the graph of $X_t$.  Indeed, $X_t$ cannot have a downward jump at time $b$ because then the horizontal chord connecting $(a,Y_a)$, $(b,Y_b)$ would intersect the graph of $Y_t$ in infinitely many places.  Lastly, if $(a,Y_a)$, $(b,Y_b)$ correspond to the left and right endpoints of an uncountable set of local minima corresponding to a jump time for $X_t$, then $X_a = X_b$ and $X_r > X_a = X_b$ for $r \in (a,b)$ so that $(a,Y_a)$, $(b,Y_b)$ are equivalent under Definition~\ref{def::levynet}. 
\end{proof}

The right hand side of Figure~\ref{fig::levynet2} illustrates an alternate way to represent the topological sphere shown in Figure~\ref{fig::levynet}.  On the left hand side of Figure~\ref{fig::levynet2} (i.e., Definition~\ref{def::levynet2}), two distinct points are considered to be equivalent if and only if either:
\begin{enumerate}
\item[Case 1:] The line segment connecting them is horizontal and intersects the graph of~$Y_t$ in at most finitely many points. (Recall that it is a.s.\ the case that there can be at most three such intersection points, counting the endpoints themselves; and if one of these points is in the interior of the segment, it must be a local maximum of~$Y_t$.)
\item[Case 2:] They are the pair $\inf\{s : Y_s = m\}$ and $\sup\{s : Y_s = m\}$ where $m$ is the value of a local minimum for $Y_t$ (which in turn corresponds to a jump in the L\'evy process).
\end{enumerate}
It is interesting because at first glance it looks like any two points of the same horizontal line in the left side of Figure~\ref{fig::levynet2} should be equivalent. But of course, this is not the case if the segment between them intersects the graph of $Y_t$ infinitely often.\footnote{If one begins with the tree obtained by gluing along horizontal chords above the graph (the tree we call the geodesic tree) then each of the two types of equivalence classes described above produces an equivalence relation on this tree in which each equivalence class has exactly one or two elements. The smaller equivalence class obtained by focusing on either one of these two cases is a dense subset in the full equivalence relation; so the full relation can be understood as the topological closure of either of these two smaller relations.}

The quotient of the right side of Figure~\ref{fig::levynet2} is generated from the quotient of the left side of Figure~\ref{fig::levynet2} by gluing topological disks into each of the holes, which is the same procedure which generates the quotient in the middle image of Figure~\ref{fig::levynet} from the first definition of the L\'evy net.  Therefore the spaces defined in Definition~\ref{def::levynet} and Definition~\ref{def::levynet2} are equivalent.

We remark that one could also check directly that the relation on the right hand side of Figure~\ref{fig::levynet2} satisfies the conditions of Moore's theorem (Proposition~\ref{prop::moore}), since each of the equivalence classes is a single point, a single line segment (horizontal or vertical), a solid rectangle, or the union of the left, right, and lower sides of a grey rectangle.

\subsection{Characterizing continuous state branching processes}
\label{subsec::csbp}

To study the L\'evy net in more detail, we will need to recall some basic facts about {\em continuous state branching processes}, which were introduced by Ji\v{r}ina and Lamperti several decades ago \cite{miloslav, lamperticsbp, lampertibranchingaslimit} (see also the more recent overview in \cite{legallspatialbranchingbook} as well as \cite[Chapter~10]{kyp2006levy_fluctuations}).
A Markov process $(Y_t, t \geq 0)$ with values in~$\R_+$, whose sample paths are {\cadlag} (right continuous with left limits) is said to be a continuous state branching process (CSBP for short) if the transition kernels $P_t(x,dy)$ of $Y$ satisfy the additivity property:
\begin{equation}
\label{eqn::CSBPproperty}
P_t(x+x', \cdot) = P_t(x, \cdot) * P_t(x',\cdot).
\end{equation}

\begin{remark}
\label{rem::CSBPsubordinator}
Note that~\eqref{eqn::CSBPproperty} implies that the law of a CSBP at a \emph{fixed} time is infinitely divisible.  In particular, this implies that for each fixed $t$ there exists a subordinator (i.e., a non-decreasing process with stationary, independent increments) $A^t$ with $A_0^t = 0$ such that $A_t^t \stackrel{d}{=} Y_t$.  (We emphasize though that $Y$ does not \emph{evolve} as a subordinator in~$t$.)  We will make use of this fact several times.
\end{remark}

The Lamperti representation theorem states that there is a simple time-change procedure that gives a one-to-one correspondence between CSBPs and L\'evy processes without negative jumps starting from a positive value and stopped upon first hitting $0$, where each is a time-change of the other.   The statement of the theorem we present below is lifted from a recent expository treatment of this result \cite{surveyproofoflampertirepresentation}.

Consider the space $\mathcal D$ of {\cadlag} functions $f\colon [0,\infty] \to [0,\infty]$ such that $\lim_{t \to \infty} f(t)$ exists in $[0,\infty]$ and $f(t) = 0$ (resp.\ $f(t) = \infty$) implies $f(t+s) = 0$ (resp.\ $f(t+s)= \infty$) for all $s \geq 0$.  For any $f \in \mathcal D$, let $\theta_t := \int_0^t f(s) ds \in [0,\infty]$, and let $\kappa$ denote the right-continuous inverse of $\theta$, so $\kappa_t := \inf \{u \geq 0 : \theta_u > t \} \in [0,\infty]$, using the convention $\inf \emptyset = \infty$.  The {\em Lamperti transformation} is given by $L(f) = f \circ \kappa$. The following is the Lamperti representation theorem, which applies to $[0,\infty]$-valued processes indexed by $[0,\infty]$.

\begin{theorem}
\label{thm::lamperti}
The Lamperti transformation is a bijection between CSBPs and L\'evy processes with no negative jumps stopped when reaching zero.  In other words, for any CSBP $Y$, $L(Y)$ is a L\'evy process with no negative jumps stopped whenever reaching zero; and for any L\'evy process $X$ with no negative jumps stopped when reaching zero, $L^{-1}(X)$ is a CSBP.
\end{theorem}

Informally, the CSBP is just like the L\'evy process it corresponds to except that its speed (the rate at which jumps appear) is given by its current value (instead of being independent of its current value).  The following is now immediate from Theorem~\ref{thm::lamperti} and the definitions above:

\begin{proposition}
\label{prop::stablecsbp}
Suppose that $X_t$ is a L\'evy process with non-negative jumps that is strictly {\em $\alpha$-stable} in the sense that for each $C>0$, the rescaled process $X_{C^\alpha t}$ agrees in law with $C X_t$ (up to a change of starting point).  Let $Y = L^{-1}(X)$.  Then $Y$ is a CSBP with the property that $Y_{C^{\alpha-1} t}$ agrees in law with $C Y_t$ (up to a change of starting point).  The converse is also true.  Namely, if $Y$ is a CSBP with the property that $Y_{C^{\alpha-1} t}$ agrees in law with $C Y_t$ (up to a change of starting point) then $Y$ is the CSBP obtained as a time-change of the $\alpha$-stable L\'evy process with non-negative jumps.
\end{proposition}

Proposition~\ref{prop::stablecsbp} will be useful on occasions when we want to prove that a given process~$Y$ is the CSBP obtained as a time change of the $\alpha$-stable L\'evy process with non-negative jumps.  (We refer to this CSBP as the {\em $\alpha$-stable CSBP} for short.\footnote{This process is also referred to as a~$\psi$-CSBP with ``branching mechanism'' $\psi(u) = C u^\alpha$, $C > 0$ a constant, in other work in the literature, for example \cite{dlg2002trees_levy}.})  It shows that it suffices in those settings to prove that $Y$ is a CSBP and that it has the scaling symmetry mentioned in the proposition statement.  To avoid dealing with uncountably many points, we will actually often use the following slight strengthening of Proposition~\ref{prop::stablecsbp}:

\begin{proposition}
\label{prop::strongstablecsbp}
Suppose that $Y$ is a Markovian process indexed by the dyadic rationals that satisfies the CSBP property~\eqref{eqn::CSBPproperty} and that $Y_{C^{\alpha-1} t}$ agrees in law with $C Y_t$ (up to a change of starting point) when $C^{\alpha-1}$ is a power of $2$.  Assume that $Y$ is not trivially equal to $0$ for all positive time, or equal to $\infty$ for all positive time.  Then $Y$ is the restriction (to the dyadic rationals) of an $\alpha$-stable CSBP.
\end{proposition}
\begin{proof}
By the CSBP property~\eqref{eqn::CSBPproperty}, the law of $Y_1$, assuming $Y_0 = a > 0$, is infinitely divisible and equivalent to the law of the value $A_a$ where $A$ is a subordinator and $A_0 = 0$ (recall Remark~\ref{rem::CSBPsubordinator}).  Fix $k \in \N$ and pick $C > 0$ such that $C^{1-\alpha} = 2^{-k}$.  Similarly, by scaling, we have that $Y_{C^{1-\alpha}} \stackrel{d}{=} C^{-1}A_{Ca}$.  By the law of large numbers, this law is concentrated on $a \E[A_1]$ when $k$ is large; we observe that $\E[A_1] = 1$ since otherwise (by taking the $k \to \infty$ limit) one could show that $Y$ is equal to $0$ (if $\E[A_1] < 1$) or $\infty$  (if $\E[A_1] > 1$) for all positive time.

From this we deduce that $Y$ is a martingale, and the standard upcrossing lemma allows us to conclude that a.s.\ $Y$ has only finitely many upcrossings across the interval $(x, x+\epsilon)$ for any $x$ and $\epsilon$, and that $Y$ a.s.\ is bounded above.  This in turn guarantees, for all $t \geq 0$, the existence of left and right limits of $Y_{t+s}$ as $s \to 0$.  It implies that $Y$ is a.s.\ the restriction to the dyadic rationals of a {\cadlag} process; and there is a unique way to extend $Y$ to a {\cadlag} process defined for all $t \geq 0$.  Since left limits exist a.s.\ at any fixed time, it is straightforward to verify that the hypotheses of Proposition~\ref{prop::stablecsbp} apply to $Y$.
\end{proof}

CSBPs are often introduced in terms of their Laplace transform \cite{legallspatialbranchingbook}, \cite[Chapter~10]{kyp2006levy_fluctuations} and Proposition~\ref{prop::stablecsbp} is also immediate from this perspective.  We will give a brief review of this here, since this perspective will also be useful in this article.  In the case of an $\alpha$-stable CSBP $Y_t$, this Laplace transform is explicitly given by
\begin{equation}
\label{eqn::csbp_def}
\E[ \exp(-\lambda Y_t) \giv Y_s ] = \exp(-Y_s u_{t-s}(\lambda)) \quad\text{for all}\quad t > s \geq 0
\end{equation}
where, for a constant $c > 0$,
\begin{equation}
\label{eqn::csbp_u_form}
u_t(\lambda) = \left( \lambda^{1-\alpha} + c t\right)^{1/(1-\alpha)}.
\end{equation}
More generally, CSBPs are characterized by the property that they are Markov processes on $\R_+$ such that their Laplace transform has the form given in~\eqref{eqn::csbp_def} where $u_t(\lambda)$, $t \geq 0$, is the non-negative solution to the differential equation
\begin{equation}
\label{eqn::csbp_diffeq}
\frac{\partial u_t}{\partial t}(\lambda) = -\psi(u_t(\lambda))\quad\text{for}\quad u_0(\lambda) = \lambda.
\end{equation}
The function $\psi$ is the so-called \emph{branching mechanism} for the CSBP and corresponds to the Laplace exponent of the L\'evy process associated with the CSBP via the Lamperti transform (Theorem~\ref{thm::lamperti}).  In this language, an $\alpha$-stable CSBP is a called a ``CSBP with branching mechanism $\psi(u) = C u^\alpha$'' (where $C > 0$ is a constant depending on $c > 0$ from~\eqref{eqn::csbp_u_form}).

One of the uses of~\eqref{eqn::csbp_def} is that it provides an easy derivation of the law of the extinction time of a CSBP, which we record in the following lemma.
\begin{lemma}
\label{lem::csbp_extinction_time}
Suppose that $Y$ is an $\alpha$-stable CSBP and let $\zeta = \inf\{t \geq 0 : Y_t = 0\}$ be the extinction time of $Y$.  Then we have for a constant $c_\alpha > 0$ that
\begin{equation}
\label{eqn::csbp_extinction_time}
\p[ \zeta > t]  = 1- \exp\big( -c_\alpha t^{1/(1-\alpha)} Y_0 \big).
\end{equation}
\end{lemma}
\begin{proof}
Note that $\{\zeta > t\} = \{ Y_t > 0\}$.  Consequently,
\begin{align*}
   \p[ \zeta > t]
&= \p[ Y_t > 0 ]
 = 1 - \lim_{\lambda \to \infty} \E[ e^{-\lambda Y_t}]
 = 1- \exp( - c_\alpha t^{1/(1-\alpha)} Y_0),
\end{align*}
which proves~\eqref{eqn::csbp_extinction_time}.
\end{proof}

As we will see in Section~\ref{subsec::levynetbreadthfirst} just below, it turns out that the boundary length of the segment in a ball boundary between two geodesics in the L\'evy net evolves as a CSBP as one decreases the size of the ball.  The merging time for the geodesics corresponds to when this CSBP reaches $0$.  Thus Proposition~\ref{prop::poissonslicestructure} together with Lemma~\ref{lem::csbp_extinction_time} allows us to relate the structure of geodesics in a space which satisfies the hypotheses of Theorem~\ref{thm::markovmapcharacterization} with the L\'evy net.

\subsection{A breadth-first approach to the L\'evy net quotient}
\label{subsec::levynetbreadthfirst}

Now, we would like to consider an alternative approach to the L\'evy net in which we observe loops in the order of their distance from the root of the tree of loops (instead of in the order in which they are completed when tracing the boundary of the stable looptree).  Consider a line at some height $C+s$ as depicted in Figure~\ref{fig::levycb}.  As explained in the figure caption, we would like to define $Z_s$ to be in some sense the ``fractal measure'' of the set of points at which this line intersects the graph of $C+Y_t$ (which should be understood as some sort of local time) and then understand how $Z_s$ evolves as $s$ changes. A detailed account of the construction and properties of $Z_s$, along with Proposition~\ref{prop::Zisalphastable} (stated and proved below), appears in \cite{duquesnelegall2005levytrees}. We give a brief sketch here.

First of all, in what sense is~$Z_s$ defined?  Note that if we fix~$s$, then we may define the set $E_s = \{t : Y_t > s \}$.  Observe that within each open interval of~$E_s$ the process $X_t$ evolves as an $\alpha$-stable L\'evy process, which obtains the same value at its endpoints and is strictly larger than that value in the interim.  In other words, the restriction of~$X_t$ to that interval is (a translation and time-reversal of) an $\alpha$-stable L\'evy excursion.  If we condition on the number~$N_\epsilon$ of excursions of this type that reach height at least~$\epsilon$ above their endpoint height, then it is not hard to see that the conditional law of the set of excursions is that of an i.i.d.\ collection of samples from the L\'evy excursion measure used to generate~$X_t$ (restricted to the positive and finite measure set of excursions which achieve height at least~$\epsilon$).  The ordered collection of L\'evy excursions agree in law with the ordered collection one would obtain by considering the ``reflected $\alpha$-stable L\'evy process'' (with positive jumps) obtained by replacing an $\alpha$-stable L\'evy process~$R_t$ by $\wt{R}_t = R_t - \inf \{R_s: 0 \leq s \leq t \}$.  (See \cite{bertoinlevybook} for a more thorough treatment of local times and reflected processes.)  The process~$\wt{R}_t$ then has a {\em local time} describing the amount of time it spends at zero; this time is given precisely by $\wt{R}_t - R_t$.  For each $Q > 0$, the set of excursions up to the first time that $\wt{R}_t - R_t$ first reaches~$Q$ can be understood as a Poisson point process corresponding to the product of the Lebesgue measure $[0,Q]$ and the (infinite) L\'evy excursion measure.  In particular, one can deduce from this that as $\epsilon$ tends to zero the quantity $\epsilon^\alpha N_\epsilon$ a.s.\ tends to a constant times the local time; this can then be taken as the definition of $Z_s$.

\begin{definition}
\label{def::levycb}
We refer to the process $Z_s$ constructed just above from the height process $Y_t$ associated with $X_t$ as the \emph{boundary length process} associated with a L\'evy net instance generated by $X_t$.
\end{definition}

\begin{figure}[ht!]
\begin{center}
\includegraphics[scale=.77]{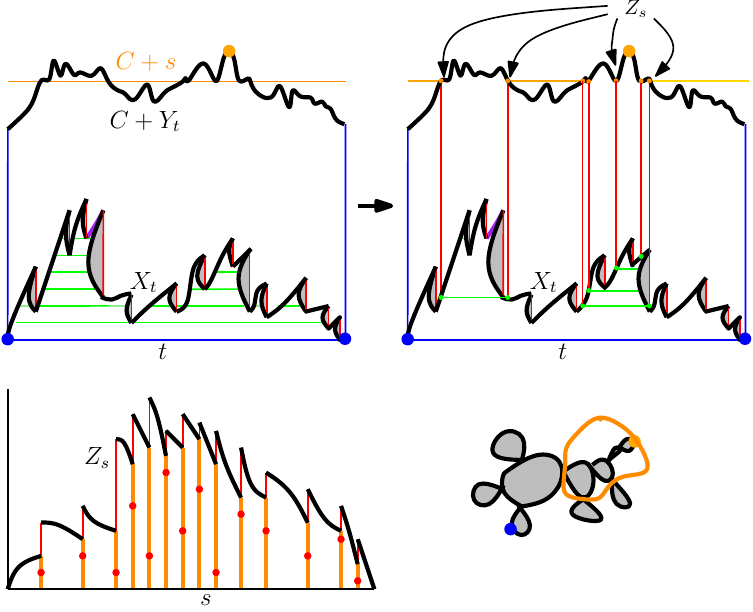}
\vspace{-0.01\textheight}
\caption{\label{fig::levycb} {\small Illustration of the breadth-first construction of the L\'evy net.  The following caption should be read together with the contents of Section~\ref{subsec::levynetbreadthfirst} up to the statement of Proposition~\ref{prop::easyexpectationratiofact}.  {\bf Upper left:}  An orange line is drawn at height $C+s$ for some~$s$.  {\bf Upper right:}  If~$a$ and~$b$ are the endpoints of an excursion of $C+Y_t$ above the orange line, then~$a$ and~$b$ are identified (via a red line) to points on the lower graph that are identified (via a green horizontal line). {\bf Lower left:} As the height of the orange segment in the upper graph increases (i.e., $s$ increases),~$Z_s$ measures the local time of the intersection between that segment and the graph of~$C+Y_t$.  When the rising orange line encounters a point $(t,s)$ on the upper graph such that~$X$ has a jump at time~$t$, there is a corresponding upward jump in~$Z_s$ of the same magnitude.  This is due to the fact (not obvious in this illustration) that all points on the loop corresponding to the jump are identified with points on the upward graph of the same height; the local time of this set of points is the magnitude of the jump.  The amount of this local time in the orange/black intersection which is to the {\em right} of the point $(t,s)$ is a quantity that lies strictly between~$0$ and~$Z_{s-}$ (see \cite[Proposition~1.3.3]{dlg2002trees_levy}); this quantity is encoded by the height of the red dot (one for each of the countably many jumps) shown in the center graph.  Another perspective is that the jumps in~$Z_s$ correspond to loops observed in the tree on the right as one explores them in order of their distance from the root of the tree encoded by $-Y_t$, where the distance is given by their looptree distance.  The orange circle on the right encloses the set of loops explored up until time~$s$.  Each red dot in the middle graph indicates {\em where} along the boundary a new loop is attached to the already-explored looptree structure, as defined relative to the branch in the geodesic tree connecting the root and dual root. Conditioned on the process~$Z$, for each jump time $s$ the vertical location of the red dot is independent and uniform on $[0,Z_{s^-}]$ (see Lemma~\ref{lem::bubble_locations_independent})}.}
\end{center}
\end{figure}

Note that the discussion above in principle only allows us to define $Z_s$ for almost all~$s$, or for a fixed countable dense set of~$s$ values.  We have not ruled out the possibility that there exist exceptional~$s$ values for which the limit that defines $Z_s$ is undefined.  To be concrete, we may use the above definition of $Z_s$ for all dyadic rational times and extend to other times by requiring the process to be {\cadlag} (noting that this definition is a.s.\ equal to the original definition of $Z_s$ for almost all $s$ values, and for any fixed $s$ value; alternatively see \cite{duquesnelegall2005levytrees} for more discussion of the local time definition).  This allows us to use Proposition~\ref{prop::strongstablecsbp} to derive the following, which is referred to in \cite[Theorem~1.4.1]{duquesnelegall2005levytrees} as the Ray-Knight theorem (see also the L\'evy tree level set discussion in \cite{dlg2002trees_levy, duquesnelegall2005levytrees}):

\begin{proposition}
\label{prop::Zisalphastable}
The process $Z$ from Definition~\ref{def::levycb} has the law of an $\alpha$-stable CSBP.
\end{proposition}
\begin{proof}
The CSBP property~\eqref{eqn::CSBPproperty} follows from the derivation above because if the process records $L+L'$ units of local time at height $s$, then the amount of local time it records at height $t > s$ in the first $L$ units of local time at height $s$ is independent of the amount of local time it records at height $t$ in the last $L'$ units of local time.  Moreover, the scaling property required by Proposition~\ref{prop::strongstablecsbp} follows from the scaling properties of $X$ and $Y$.
\end{proof}

Related to Proposition~\ref{prop::Zisalphastable} is the following correspondence between the jumps of the~$Z$ and~$X$ processes shown in Figure~\ref{fig::levycb}.

\begin{proposition}
\label{prop::jump_correspondence}
The (countably many) jumps in the process $Z$ from Definition~\ref{def::levycb} are a.s.\ in one-to-one correspondence with the (countably many) jumps in the process~$X$ used to generate the corresponding L\'evy net instance.  Namely, it is a.s.\ the case that whenever a jump in $Z$ occurs at a time $s$ we have $s = Y_t$ for some $t$ value at which the process $X$ has a jump, and vice-versa; in this case, the corresponding jumps have the same magnitude.
\end{proposition}
\begin{proof}
When a jump occurs in $Z_s$, the line with height of $s$ intersects the graph of $Y_t$ at all points at which $X_t$ (run from right to left) reaches a record minimum following the jump, up until $X_t$ (run from right to left) again reaches the value on the lower side of the jump.  Using the description of local time above (in terms of $\wt R$ and $R$), we see that the amount of local time added due to the appearance of the jump is precisely the height of the $X_t$ jump.
\end{proof}

For each $r > 0$, we let $Z_s^r$ be the local time of the intersection of the graph of $Y$ with the line of height $s$ and width $r$ (i.e., the line connecting $(0,s)$ with $(r,s)$).  Note that $Z_s = Z_s^T$ where $T$ is the length of the L\'evy excursion and $Z_s$ is as in Definition~\ref{def::levycb}.  As in the case of $Z$ itself, $Z_s^r$ is in principle only a.s.\ defined for each $(r,s)$ pair.  In Proposition~\ref{prop::local_time_cadlag_modification} below, we will construct a jointly measurable modification of $(r,s) \mapsto Z_s^r$ which satisfies certain continuity properties.  Throughout, we will assume that we are using this modification so that $Z_s^r$ is defined for all $(r,s)$ simultaneously.  In particular, it makes sense to talk about $Z_s^r$ even at random times.

Let $D = \sup_t Y_t$ so that $[0,D]$ is the interval on which $Z_s$ is defined.  For each $s \in [0,D]$, we let $U_s$ be the set of points in the L\'evy net which have distance at least $D-s$ from the root.  Then $\partial U_s$ is the set of points in the L\'evy net which have distance equal to $D-s$ from the root.  Note that $\partial U_s$ corresponds to a horizontal line in Figure~\ref{fig::bubblegluing}.  In view of Definition~\ref{def::levycb} and Figure~\ref{fig::levycb}, we note that if $x,y \in \partial U_s$ then it makes sense to talk about the clockwise $S_{xy}$ (resp.\ counterclockwise $\wt{S}_{xy}$) segments of~$\partial U_s$ which connect~$x$ and~$y$.  The boundary lengths of $S_{xy}$ and $\wt{S}_{xy}$ are determined by the local time of the intersection of the lines with height $s$ with the graph of~$Y_t$ which correspond to the preimages of $S_{xy}$ and $\wt{S}_{xy}$ under the quotient map.  Fix $r,t > 0$ and assume that we are working on the event that $D > t$ and $Z_t \geq r$.  Let $\gamma$ be the branch of the geodesic tree which connects the root and the dual root.  We can then describe each point $x \in \partial U_s$ in terms of the length of the counterclockwise segment of~$\partial U_s$ which connects~$x$ and the point~$x_s$ on~$\partial U_s$ which is visited by~$\gamma$.

\begin{definition}
\label{def::levy_attachment_points}
For each $s$ which is a jump time for $Z$ and $t$ such that $s = Y_t$, we refer to the amount of local time in the intersection of the line with height $s$ with the graph of $Y$ which lies to the \emph{right} of the point $(t,s)$ (i.e., $Z_s^T - Z_s^t$) as the \emph{attachment point} associated with the jump.
\end{definition}

As explained in the caption of Figure~\ref{fig::levycb}, the attachment point associated with a given jump records the boundary length distance in the counterclockwise direction of the loop in the stable looptree encoded by $X$ from the branch in the geodesic tree that connects the root of the geodesic tree to the root of the looptree.

Next, we make a simple observation:

\begin{proposition}
\label{prop::easyexpectationratiofact}
Suppose that $A_s$ is a subordinator with $A_0 = 0$ and $\p[A_1 > 0] = 1$.  Suppose also that $\wt A_s$ is an independent instance of the same process.  Then for any fixed values $a$ and $b$ we have
\begin{equation}
\label{eqn::a_wt_b_ratio}
\E \left[\frac{A_a}{A_a + \wt A_b} \right] = \frac{a}{a+b}.
\end{equation}
\end{proposition}
\begin{proof}
Let $c = a+b$.  Since $A$ has stationary independent increments, it suffices to show that $\E[ A_a/A_c] = a/c$.  For each $n \in \N$ and $i \in \{1,\ldots,n\}$, we note that $\E[ A_{c/n} \giv A_c] = \E[ A_{ic/n} - A_{(i-1)c/n} \giv A_c]$ by exchangeability of the increments given their sum.  If we now sum over $i$, we see that $\E[ A_{c/n} \giv A_c] = n^{-1} A_c$.  This implies that $\E[ A_a \giv A_c] = \tfrac{a}{c} A_c$ for any $a$ of the form $i c/n$ for $i \in \{0,\ldots,n\}$.  The assertion for general values of $a$ follows because we can find a sequence $(a_k)$ of the form $a_k = i_k c / 2^k$ for $i_k \in \{0,\ldots,2^k\}$ which increases to $a$ and apply the monotone convergence theorem.
\end{proof}

\begin{figure}[ht!]
\begin{center}
\includegraphics [scale=.85, page =1]{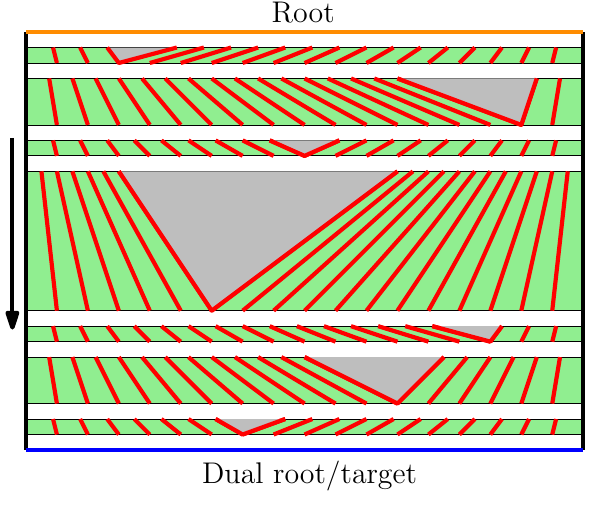}
\caption{\label{fig::bubblegluing} {\bf \it Recovering topological structure from bubbles:} Shown is a representation of a L\'evy net using a width-$1$ rectangle $R$.  The top (resp.\ bottom) line represents the root (resp.\ dual root/target).  The left and right sides of $R$ are identified with each other and represent the branch $\gamma$ in the geodesic tree connecting the root and dual root.  If $r$ is {\em not} one of the countably many values at which a jump in boundary length occurs, then each point $z$ on the L\'evy tree of distance $r$ from the root is mapped to the point in the rectangle whose horizontal location is the length of the counterclockwise radius-$r$-ball boundary segment from $\gamma$ to $z$ divided by the total length of the radius-$r$-ball boundary; the vertical distance from the top of the rectangle is the sum of $r$ and the sum of squares of the boundary-length jumps that occur as the radius varies from $0$ and $r$. Each of the green stripes represents the set of points whose distance from the root is a value $r$ at which a jump {\em does} occur.  Every red line (going from the top to the bottom of a stripe) is an equivalence class that encodes one of these points. The height of each green stripe is equal to the square of the jump in the boundary length corresponding to the grey triangle (the sum of these squares is a.s.\ finite since the sum of the squares of the jumps of an $\alpha$-stable L\'evy process is a.s.\ finite; see, e.g., \cite[Chapter~I]{bertoinlevybook}).
The top (resp.\ bottom) of each green stripe represents the outer boundary of the metric ball infinitesimally before (resp.\ after) the boundary length of the metric ball jumps.  Each red line is a single closed equivalence class (except that when two red lines share an end vertex, their union forms a single closed equivalence class). The uppermost (resp.\ lowermost) horizontal orange (resp.\ blue) line is also a single closed equivalence class.  Also, each pair of left and right boundary points of the rectangle (with the same vertical coordinate) is a closed equivalence class.  Any point that does not belong to one of these classes is its own class.}
\end{center}
\end{figure}

\begin{figure}[ht!]
\begin{center}
\includegraphics [scale=.85,page =2]{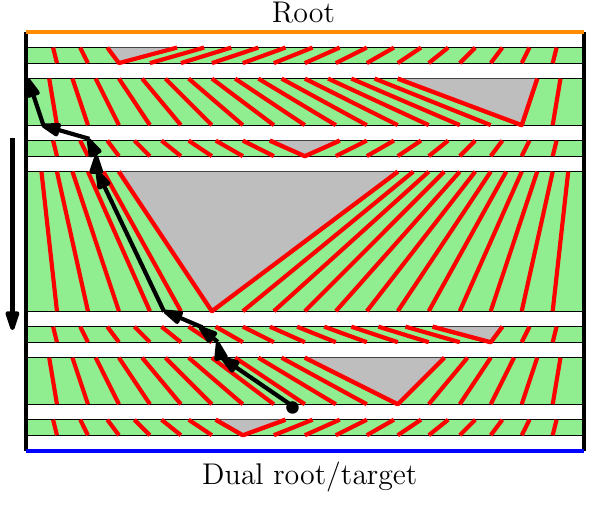}
\caption{\label{fig::bubblegluingplustrajectory} {\bf \it Plotting a geodesic trajectory:} The black sequence of arrows represents a branch $\eta$ in the geodesic tree in the L\'evy net.  We have drawn $\eta$ beginning on one of the horizontal lines of the figure which, as explained in Figure~\ref{fig::bubblegluing}, represents the boundary of the metric ball starting from the root.  As shown in Proposition~\ref{prop::levynetgeodesictrajectoriesdetermined}, $\eta$ eventually merges with the left boundary of the rectangle (both left and right rectangle boundaries correspond to the root-to-target branch in the geodesic tree) just before getting back to the root vertex (represented by the uppermost orange line).  Geodesics started at distinct points can ``merge'' with each other.}
\end{center}
\end{figure}

Proposition~\ref{prop::easyexpectationratiofact} now implies another simple but interesting observation, which we record as Proposition~\ref{prop::levynetgeodesictrajectoriesdetermined} below (and which is related to the standard ``confluence-of-geodesics'' story).  See Figure~\ref{fig::bubblegluing} and Figure~\ref{fig::bubblegluingplustrajectory} for relevant illustrations.  Before we state this result, we now give our third definition of the L\'evy net quotient.

\begin{definition}
\label{def::levynet_bubble_gluing}
(Third definition of the L\'evy net.)  Suppose we are given a realization of the process $Z_s$ from Definition~\ref{def::levycb} as well as the attachment points as defined in Definition~\ref{def::levy_attachment_points}.  Let~$R$ be a rectangle with width~$1$ and height equal to the sum of the length of the interval on which~$Z$ is defined plus the sum of the squares of the jumps of~$Z$.  For each~$s$, we let~$J(s)$ (resp.\ $J_{-}(s)$) be the sum of the squares of the jumps of~$Z$ which have occurred before (resp.\ strictly before) time~$s$.  We define an equivalence relation $\cong$ on~$R$ by declaring points to be equivalent which lie on each line segment connecting points of the form $(s+J_-(s),u/Z_{s-})$ to $(s+J(s),u/Z_s)$ for each $s$ which is a jump time of $Z_s$ and $u \in [0,a_s]$ where $a_s$ is the attachment point corresponding to time $s$ and from $(s+J_-(s),u/Z_{s-})$ to $(s+J(s),(u+\Delta_s)/Z_s)$ for each $u \in [a_s,Z_{s-}]$, where $\Delta_s = Z_s - Z_{s-}$ is the size of the jump at time $s$.
\end{definition}

See Figure~\ref{fig::bubblegluing} for an illustration of $\cong$ as in Definition~\ref{def::levynet_bubble_gluing}.  (As the root in Figure~\ref{fig::bubblegluing} is shown on the top rather than the bottom, one has to vertically reflect the illustration in Figure~\ref{fig::bubblegluing} to correspond exactly to $\cong$.)

Fix $t, r > 0$.  On $D > t$ and $Z_t \geq r$, we let $\eta^{t,r}$ be the geodesic starting from the point on $\partial U_t$ such that the length of the counterclockwise segment of $\partial U_t$ to $x_t$ is equal to $r$.  For each $s \geq t$, we let $A_s^{t,r}$ (resp.\ $B_s^{t,r}$) be the length of the counterclockwise (resp.\ clockwise) segment of $\partial U_s$ which connects $\eta^{t,r} \cap \partial U_s$ to $x_s$.  Note that $A_t^{t,r} = r$, $B_t^{t,r} = Z_t - r$, and $A_s^{t,r} + B_s^{t,r} = Z_s$ for all $s \in [t,D]$.

\begin{proposition}
\label{prop::levynetgeodesictrajectoriesdetermined}
When the processes $A^{t,r}$, $B^{t,r}$, and $Z^{t,r}$ and the values $t$ and $D$ are as defined just above, the following holds for the restrictions of these processes to the interval $s \in [t, D]$.
\begin{enumerate}
\item The processes $A_s^{t,r}$ and $B_s^{t,r}$ are independent $\alpha$-stable CSBPs.
\item The process $A_s^{t,r} / Z_s = A_s^{t,r} / (A_s^{t,r} + B_s^{t,r})$ is a martingale. (This corresponds to the horizontal location in the trajectory illustrated in Figure~\ref{fig::bubblegluingplustrajectory} when parameterized using distance).
\item The process $A_s^{t,r} / Z_s$ a.s.\ hits $0$ or $1$ before time $D$.
\end{enumerate}
\end{proposition}
\begin{proof}
The first point is immediate from the construction; recall the proof of Proposition~\ref{prop::Zisalphastable}.  Given the first point, the second point is immediate from Proposition~\ref{prop::easyexpectationratiofact} (recall Remark~\ref{rem::CSBPsubordinator}).  The fact that the martingale reaches $0$ or $1$ a.s.\ before reaching the upper end of the rectangle is reached simply follows from the fact that two independent CSBPs, both started at positive values, a.s.\ do not reach zero at exactly the same time.
\end{proof}

\begin{lemma}
\label{lem::bubble_locations_independent}
Given the process $Z$, the locations of the attachment points as defined in Definition~\ref{def::levy_attachment_points} are conditionally independent.  If $s$ is a jump time for $Z$, then the corresponding attachment point is uniform in $[0,Z_{s-}]$.
\end{lemma}
In the context of Figure~\ref{fig::levycb}, Lemma~\ref{lem::bubble_locations_independent} states that conditionally on the process $Z_s$, the red dots in the bottom left of Figure~\ref{fig::levycb} are conditionally independent and uniform on each of the vertical orange lines.
\begin{proof}[Proof of Lemma~\ref{lem::bubble_locations_independent}]
This follows because the CSBP property~\eqref{eqn::CSBPproperty} implies that for each fixed $s$ we can write $Z_{s+t}$ for $t \geq 0$ as a sum $n$ independent $\alpha$-stable CSBPs each starting from $Z_s/n$ and the probability that any one of them has a jump in $\epsilon > 0$ units of time is equal.
\end{proof}

\begin{theorem}
\label{thm::levynetrecover}
The $\sigma$-algebra generated by the process $Z$ as in Definition~\ref{def::levycb} and the attachment points defined in Definition~\ref{def::levy_attachment_points} is equal to the $\sigma$-algebra generated by $X$.  (In other words, the information encoded by the graph in the bottom left of Figure~\ref{fig::levycb} a.s.\ determines the information encoded by the first graph.) That is, these definitions yield (as illustrated in Figure~\ref{fig::levycb}) an a.e.-defined one-to-one measure-preserving correspondence between
\begin{enumerate}
\item $\alpha$-stable L\'evy excursions and
\item $\alpha$-stable L\'evy excursions (which are naturally reparameterized and viewed as CSBP excursions) that come equipped with a way of assigning to each jump a distinguished point between zero and the lower endpoint of that jump (as in Definition~\ref{def::levy_attachment_points} and illustrated in the bottom left graph of Figure~\ref{fig::levycb}).
\end{enumerate}
\end{theorem}

To further clarify the statement of Theorem~\ref{thm::levynetrecover}, we recall that an $\alpha$-stable L\'evy excursion is determined by the collection of pairs which give its jump times and jump magnitudes.  Therefore we can think of the infinite measure on $\alpha$-stable L\'evy excursions as an infinite measure on countable subsets of $\R_+^2$ where an element $(t,u) \in \R_+^2$ corresponds to a jump at time $t$ of size $u$.  An $\alpha$-stable L\'evy excursion where each jump is marked by a point between $0$ and the size of the jump can be thought of as a countable subset of $\R_+^2 \times [0,1]$ where an element $(t,u,v) \in \R_+ \times [0,1]$ corresponds to a jump at time $t$ of size $u$ with marked point along the jump at height $uv$.  The measure which will arise in this context in Theorem~\ref{thm::levynetrecover} will be given by the infinite measure on $\alpha$-stable L\'evy excursions where each jump is marked by a conditionally independent uniform random variable which gives the position of the mark corresponding to the jump.

Before we give the proof of Theorem~\ref{thm::levynetrecover}, we first need the following lemmas.

\begin{lemma}
\label{lem::csbp_sup}
Suppose that $W$ is an $\alpha$-stable CSBP with $W_0 > 0$ and let $W^* = \sup_{s \geq 0} W_s$.  Then we have that
\begin{equation}
\label{eqn::csbp_sup_bound}
\p[ W^* \geq u] \leq \frac{W_0}{u} \quad\text{for each}\quad u \geq W_0.
\end{equation}
\end{lemma}
\begin{proof}
Let $\tau = \inf\{t \geq 0: W_t = 0 \quad\text{or}\quad W_t \geq u\}$.  Then we have that
\begin{align*}
  \p[ W^* \geq u ]
&\leq \frac{1}{u} \E[ W_\tau ] \leq \frac{1}{u} \liminf_{t \to \infty} \E[ W_{t \wedge \tau}] \quad\text{(by Fatou's lemma)}\\
&= \frac{W_0}{u} \quad\text{(by the optional stopping theorem)}.	
\end{align*}
\end{proof}

\begin{lemma}
\label{lem::noreflectedcsbp}
Let $W_t$ be a process that starts at $W_0 = \epsilon$, then evolves as an $\alpha$-stable CSBP until it reaches $0$, then jumps to $\epsilon$ and continues to evolve as an $\alpha$-stable CSBP until again reaching zero, and so forth.  For each $T > 0$, the process $W|_{[0,T]}$ converges to zero in probability as $\epsilon \to 0$ with respect to the uniform topology.
\end{lemma}
\begin{proof}
Let $\tau_0 = 0$.  Assuming that $\tau_0,\ldots,\tau_k$ have been defined we let $\tau_{k+1} = \inf\{t > \tau_k : W_t = 0\}$.  Fix $\delta > \epsilon$ and let $N = \min\{k \geq 0 : \sup_{t \in [\tau_k,\tau_{k+1}]} W_t \geq \delta\}$.  Lemma~\ref{lem::csbp_sup} implies that $N$ stochastically dominates a geometric random variable with parameter $\epsilon/\delta$.  Fix $a \in (0,1)$ so that $\alpha-1-a>0$ (recall that $\alpha \in (1,2)$).  Therefore $\p[ N \leq \delta/\epsilon^a] \to 0$ as $\epsilon \to 0$ with $\delta > 0$ fixed.  Note also that there exists a constant $p > 0$ so that $\p[ \tau_1 \geq \epsilon^{\alpha-1}] \geq p$ uniformly in $\epsilon > 0$.  Let $n = \lfloor \delta/\epsilon^a \rfloor$.  Using that the sequence of random variables $(\tau_j-\tau_{j-1})$ is i.i.d., standard concentration results for binomial random variables imply that (with $\delta > 0$ fixed),
\[ \p\!\left[ \tau_n \leq \frac{p}{2} \times \epsilon^{\alpha-1} \times \frac{\delta}{\epsilon^a}  \right] =  \p\!\left[ \sum_{j=1}^n (\tau_j-\tau_{j-1}) \leq \frac{p \delta}{2} \epsilon^{\alpha-1-a} \right] \to 0 \quad\text{as}\quad \epsilon \to 0.\]
Combining, we have that (with $\delta > 0$ fixed),
\begin{align*}
\p[ \tau_N \leq T]
&\leq \p[ \tau_n \leq T] + o(1) = o(1) \quad\text{as}\quad \epsilon \to 0.
\end{align*}
This implies the result.

\end{proof}

\begin{proof}[Proof of Theorem~\ref{thm::levynetrecover}]
Fix $t, r > 0$.  Assume that we are working on the event that $D > t$ and $Z_t \geq r$.  We claim that the trajectory $\eta^{t,r}$ considered in Proposition~\ref{prop::levynetgeodesictrajectoriesdetermined} is a.s.\ uniquely determined by the boundary length process $Z_s$ together with the attachment points (i.e., the information in the decorated graph $Z_s$, as shown in the bottom left graph of Figure~\ref{fig::levycb}). Upon showing this, we will have shown that the geodesic tree is a.s.\ determined by $Z_s$ and the attachment points.  Indeed, then we can recover the process $Y_t$ and from $Y_t$ we can recover the ordered sequence of jumps made by $X_t$ hence we can recover $X_t$ itself.  That is, the entire $\alpha$-stable L\'evy net is a.s.\ determined.  This implies the theorem statement because we know that $X_t$ determines $Z_s$ plus the attachment points.

To prove the claim, we choose two such trajectories $\eta^{t,r}$ and $\wt{\eta}^{t,r}$ conditionally independently, given $Z_s$ and the attachment points, and show that they are a.s.\ equal.

We begin by noting that the length of the segment which is to the left of $\eta^{t,r}$ evolves as an $\alpha$-stable CSBP and the length which is to the right of $\eta^{t,r}$ evolves as an independent $\alpha$-stable CSBP.  The same is also true for $\wt{\eta}^{t,r}$.  It follows from this that in the intervals of time in which $\eta^{t,r}$ is not hitting $\wt{\eta}^{t,r}$ we have that the length $A_s^{t,r}$ (resp.\ $C_s^{t,r}$) of the segment which is to the left (resp.\ right) of both trajectories evolve as independent $\alpha$-stable CSBPs.  Our aim now is to show that the length $B_s^{t,r}$ which lies between $\eta^{t,r}, \wt{\eta}^{t,r}$ also evolves as an independent $\alpha$-stable CSBP in these intervals of time.

Fix an interval of time $I = [a,b]$ in which $\eta^{t,r}$ does not collide with $\wt{\eta}^{t,r}$.  Then we know that both $A^{t,r}|_I$ and $C^{t,r}|_I$ can be a.s.\ deduced from the ordered set of jumps they have experienced in $I$ along with their initial values $A_a^{t,r},C_a^{t,r}$ (since this is true for $\alpha$-stable CSBPs and $\alpha$-stable L\'evy processes).  That is, if we fix $s \in I$ and let $J_s^\epsilon$ be the sum of the jumps made by $A^{t,r}|_{[a,s]}$ with size at least $\epsilon$ then $A_s^{t,r}$ is a.s.\ equal to $A_a^{t,r} + \lim_{\epsilon \to 0} \big( J_s^\epsilon - \E[ J_s^\epsilon] \big)$ and the analogous fact is likewise true for $C^{t,r}|_I$.  Since this is also true for $(A^{t,r} + B^{t,r} + C^{t,r})|_I$ as it is an $\alpha$-stable CSBP (Proposition~\ref{prop::Zisalphastable}), we see that $B^{t,r}|_I$ is a.s.\ determined by the jumps made by $B^{t,r}|_I$ and $B_a^{t,r}$ in the same way.

To finish showing that $B^{t,r}|_I$ evolves as an $\alpha$-stable CSBP, we need to show that the law of the jumps that it has made in $I$ has the correct form.  Lemma~\ref{lem::bubble_locations_independent} implies that each time a new bubble comes along, we may sample which of the three regions it is glued to (with probability of each region proportional to each length).  This implies that the jump law for $B^{t,r}|_I$ is that of an $\alpha$-stable CSBP which implies that $B^{t,r}|_I$ is in fact an $\alpha$-stable CSBP.

The argument is completed by applying Lemma~\ref{lem::noreflectedcsbp} to deduce that since $B_s^{t,r}$ starts at zero and evolves as an $\alpha$-stable CSBP away from time zero, it cannot achieve any positive value in finite time.  We have now shown that it is possible to recover $X$ and $Y$ in the definition of the L\'evy net from $Z$ together with the attachment points.  That is, it is possible to recover the top left graph in Figure~\ref{fig::levycb} from the bottom left graph a.s.  We have already explained how to construct $Z$ and the attachment points from $X$ and $Y$, which completes the proof.
\end{proof}

\begin{figure}[ht!]
\begin{center}
\includegraphics[scale=0.8]{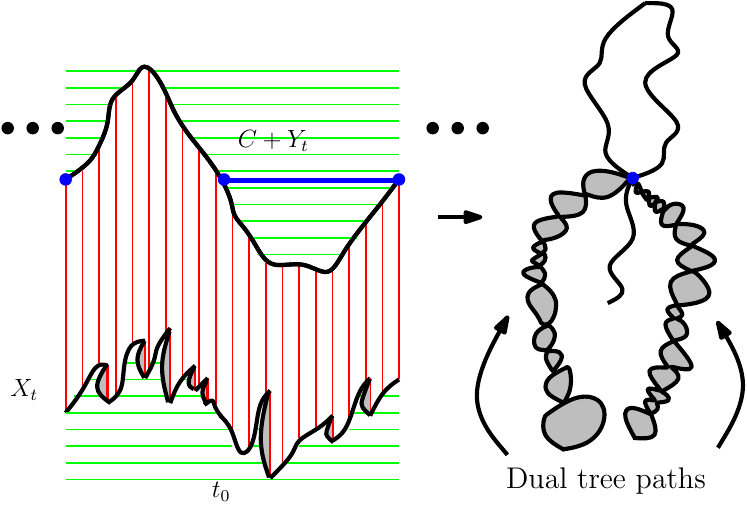}
\end{center}
\vspace{-0.03\textheight}
\caption{\label{fig::nodoublesidedrecord} Illustration of the proof of Proposition~\ref{prop::nodoublesidedrecord}, which states that $Y_t$ cannot have a decrease time, i.e., there cannot be a time $t_0$ and $h > 0$ such that $Y_s \geq Y_{t_0}$ for all $s \in (t_0-h,t_0)$ and $Y_s \leq Y_{t_0}$ for all $s \in (t_0,t_0+h)$.  Shown is the behavior of the geodesic tree and dual tree if $Y$ did have a decrease time $t_0$.  The middle blue line on the graph of $C+Y_t$ corresponds to the decrease time and the blue dots to its left and right are points which are all glued together by the L\'evy net equivalence relation. Observe that every point in the L\'evy net which corresponds to a point in the graph of $C+Y_t$ which lies below the blue line would have more than one geodesic back to the root.  This is a contradiction in view of Lemma~\ref{lem::noreflectedcsbp}, because then we would have a positive measure of points in the geodesic tree from which there is more than one geodesic to the root.  An analogous argument implies that $Y_t$ cannot have an increase time.}
\end{figure}

We now have the tools to give the proof of Proposition~\ref{prop::nodoublesidedrecord}.

\begin{proof}[Proof of Proposition~\ref{prop::nodoublesidedrecord}]
See Figure~\ref{fig::nodoublesidedrecord} for an illustration of the argument.  We will give the proof that $Y_t$ a.s.\ does not have a decrease time; the proof that $Y_t$ a.s.\ does not have an increase time is analogous.  We suppose for contradiction that $Y$ has a decrease time $t_0$.  Then there exists $h > 0$ such that $Y_s \geq Y_{t_0}$ for all $s \in (t_0-h,t_0)$ and $Y_s \leq Y_{t_0}$ for all $s \in (t_0,t_0+h)$.  Let $u_0$ (resp.\ $v_0$) be the supremum (resp.\ infimum) of times $s$ before (resp.\ after) $t_0$ such that $Y_s < Y_{t_0}$ (resp.\ $Y_s > Y_{t_0}$).  As $h > 0$, we have that $u_0 < t_0 < v_0$.  Let $\pi$ be the quotient map as in Definition~\ref{def::levynet}.  By the definition of the geodesic tree in Definition~\ref{def::levynet}, we have that $\pi( (t_0,C+Y_{t_0})) = \pi((v_0,C+Y_{v_0}))$.  Moreover, as $Y_t \geq Y_{u_0} = Y_{v_0}$ for all $t \in [u_0,t_0]$ it follows that $X_t \geq X_{u_0}$ for all $t \in [u_0,t_0]$.  Consequently, it follows that $\pi((u_0,X_{u_0})) = \pi((t_0,X_{t_0}))$.  Since $\pi((t,C+Y_t)) = \pi((t,X_t))$ for all $t$, we conclude that $\pi((u_0,C+Y_{u_0})) = \pi((t_0,C+Y_{t_0}))$.  That is, there are two distinct geodesics from the root of the geodesic tree to $\pi((t_0,X_{t_0})) = \pi((v_0,X_{v_0}))$.  Therefore the projection under~$\pi$ of the line segment $C + [t_0,v_0]$ is a positive measure subset of the geodesic tree from which there are at least two geodesics in the geodesic tree back to the root.

We will now use Lemma~\ref{lem::noreflectedcsbp} to show that the subset of the geodesic tree from which there are multiple geodesics back to the root a.s.\ has measure zero.  It is shown in Proposition~\ref{prop::levynetgeodesictrajectoriesdetermined} that the boundary length between two geodesics in the L\'evy net evolves as an $\alpha$-stable CSBP as the distance from the dual root increases.  Suppose that~$x$ is a fixed point in the L\'evy net and that~$\eta$ is the branch in the geodesic tree from~$x$ back to the root.  Fix $\epsilon > 0$, let $\wt{\tau}_0 = \tau_0 = 0$, and let $\eta_0$ (resp.\ $\wt{\eta}_0$) be the branch in the geodesic tree back to the root which starts from clockwise (resp.\ counterclockwise) boundary length distance $\epsilon$ from $x = \eta(\tau_0)$ back to the root.  We let $\tau_1$ (resp.\ $\wt{\tau}_1$) be the time at which $\eta$ first merges with $\eta_0$ (resp.\ $\wt{\eta}_0$).  Assuming that $\eta_0,\ldots,\eta_j$ and $\wt{\eta}_0,\ldots,\wt{\eta}_j$ as well as $\tau_0,\ldots,\tau_j$ and $\wt{\tau}_0,\ldots,\wt{\tau}_j$ have been defined, we let $\tau_{j+1}$ (resp.\ $\wt{\tau}_{j+1}$) be the first time that $\eta$ merges with $\eta_j$ (resp.\ $\wt{\eta}_j$) and let $\eta_{j+1}$ (resp.\ $\wt{\eta}_{j+1}$) be the branch of the geodesic tree starting from $\epsilon$ units in the clockwise (resp.\ counterclockwise) direction along the boundary relative to $\eta(\tau_{j+1})$ (resp.\ $\wt{\eta}_{j+1}(\wt{\tau}_{j+1})$).  

Suppose that there are at least two geodesics from $x=\eta(0)$ back to the root of the geodesic tree.  Then it would be the case that there exists $\delta > 0$ such that for sufficiently small $\epsilon > 0$ there is a $j$ such that either $\tau_{j+1}-\tau_j \geq \delta$ or $\wt{\tau}_{j+1}-\wt{\tau}_j \geq \delta$.  By Lemma~\ref{lem::noreflectedcsbp}, this a.s.\ does not happen, from which the result follows.
\end{proof}

We will later also need the following lemma, which gives an explicit description of the time-reversal of the L\'evy process whose corresponding CSBP is used to generate a L\'evy net.

\begin{lemma}
\label{lem::levyreversal}
Suppose that $\alpha \in (1,2)$ and $W_t$ is an $\alpha$-stable L\'evy excursion with positive jumps (indexed by $t \in [0,T]$ for some $T$).  That is, $W_t$ is chosen from the natural infinite measure on excursions of this type.  Then the law of $W_{T-t}$ is also an infinite measure, and corresponds to an excursion of a Markov process that has only negative jumps. When the process value is $c$, the jump law for this Markov process is given by a constant times $a^{-\alpha-1} (1-a/c)^{\alpha - 2}$.
\end{lemma}

Lemma~\ref{lem::levyreversal} is a relatively standard sort of calculation about time-reversals of L\'evy excursions.  See, for example, \cite[Theorem~4]{ch1996levy}.  For completeness, we will give a proof just below.

\begin{proof}[Proof of Lemma~\ref{lem::levyreversal}]
Fix $\epsilon > 0$ and let $V_t$ be an $\alpha$-stable L\'evy process with only upward jumps with $V_0 = \epsilon$.  Let $\tau = \inf\{t \geq 0 : V_t = 0\}$.  Then the law of $V_{t-\tau}$ is given by that of an $\alpha$-stable L\'evy process with only downward jumps conditioned to be non-negative stopped at the last time that it hits $\epsilon$ \cite[Chapter~VII, Theorem~18]{bertoinlevybook}.  When starting from a positive value, this process can be constructed explicitly from the law of an $\alpha$-stable L\'evy process with only downward jumps by weighting it by a certain Radon-Nikodym derivative.  To be more precise, recall that the \emph{scale function} \cite[Chapter~VII.2]{bertoinlevybook} $\xi$ for an $\alpha$-stable L\'evy process with only downward jumps is given by $\xi(u) = \alpha u^{\alpha-1}$.  Suppose that $U$ has the law of an $\alpha$-stable L\'evy process with only downward jumps with $U_0 > 0$.  Then the Radon-Nikodym derivative of $U|_{[0,t]}$ conditioned to be positive with respect to the (unconditioned) law of $U|_{[0,t]}$ is given by
\begin{equation}
\label{eqn::rn_cond_form}
	\frac{\xi(U_t)}{\xi(U_0)} \one_{\{t < \zeta\}} = \frac{U_t^{\alpha-1}}{U_0^{\alpha-1}} \one_{\{t < \zeta\}}
\end{equation}
where $\zeta = \inf\{t \geq 0 : U_t \leq 0\}$.  The law of the conditioned process started from $U_0 = 0$ is then given by the limit as of its law when it starts from $U_0 > 0$ as $U_0 \to 0$.

From~\eqref{eqn::rn_cond_form}, it is easy to see that the jump law for the conditioned process is given by a constant times
\begin{equation}
\label{eqn::cond_pos_jump_law}
a^{-\alpha-1}(1-a/c)^{\alpha-1}	
\end{equation}
when the process value is equal to $c$.  To complete the proof, we need to determine the effect on the jump law of further conditioning the process conditioned to be positive to hit $(0,\epsilon)$ in the limit as $\epsilon \to 0$.

We have the following basic fact for the conditioned process.  By \cite[Chapter~VII, Lemma~12]{bertoinlevybook}, the probability that it starting from $y > \epsilon >0$ hits the interval $(0,\epsilon)$ is given by 
\begin{equation}
\label{eqn::prob_hit_form}
p_{y,\epsilon} = 1-\frac{\xi(y-\epsilon)}{\xi(y)} = 1 - \frac{(y-\epsilon)^{\alpha-1}}{y^{\alpha-1}}.
\end{equation}
Using~\eqref{eqn::prob_hit_form}, we see for $y,z > 0$ that
\begin{equation}
\label{eqn::frac_to_zero}
\frac{p_{y,\epsilon}}{p_{z,\epsilon}} \to \frac{z}{y} \quad\text{as}\quad \epsilon \to 0.
\end{equation}
Consider the law of $U$ conditioned to be positive conditioned further on hitting $(0,\epsilon)$.  If the process value is $c$ at a given time, then (by a Bayes' rule calculation) the probability of making a downward jump of size $a \in (0,c)$ is weighted by $p_{c-a,\epsilon}/p_{c,\epsilon}$ in comparison to~\eqref{eqn::cond_pos_jump_law}.  Therefore combining~\eqref{eqn::frac_to_zero} with~\eqref{eqn::cond_pos_jump_law} implies that the jump law for the time-reversed excursion is as desired since the law of the time-reversed excursion can be constructed by taking the limit as $\epsilon \to 0$ of the law of $U$ conditioned to be positive conditioned further on hitting $(0,\epsilon)$.
\end{proof}

\subsection{Topological equivalence of L\'evy net constructions}
\label{subsec::levy_net_top_equiv}

We have so far given three different descriptions of the L\'evy net quotient, namely in Definition~\ref{def::levynet} (illustrated in Figure~\ref{fig::levynet}), Definition~\ref{def::levynet2} (illustrated in Figure~\ref{fig::levynet2}), and Definition~\ref{def::levynet_bubble_gluing} (illustrated in Figure~\ref{fig::bubblegluing}).  Moreover, we explained in Section~\ref{subsec::levynet_second_construction} that the quotients in Definition~\ref{def::levynet} and Definition~\ref{def::levynet2} yield an equivalent topology.  The purpose of this section is show that the topology of the quotient constructed in Definition~\ref{def::levynet_bubble_gluing} is equivalent to the topology constructed in Definition~\ref{def::levynet}.

\begin{proposition}
\label{prop::levy_net_equivalent_topology}
The topology of the L\'evy net quotient constructed as in Definition~\ref{def::levynet} is equivalent to the topology of the quotient constructed in Definition~\ref{def::levynet_bubble_gluing}.  In particular, the quotient constructed in Definition~\ref{def::levynet_bubble_gluing} is a.s.\ homeomorphic to $\S^2$.
\end{proposition}

We remark that it is also possible to give a short, direct proof that the quotient described in Definition~\ref{def::levynet_bubble_gluing} is a.s.\ homeomorphic to $\S^2$ using Moore's theorem (Proposition~\ref{prop::moore}), though we will not do so in view of Proposition~\ref{prop::levy_net_equivalent_topology}.

Recall that for each $r > 0$, $Z_s^r$ is the local time of the intersection of the graph of $Y$ with the line of height $s$ and width $r$ (i.e., the line connecting $(0,s)$ with $(r,s)$) and that $Z_s = Z_s^T$ where $T$ is the length of the L\'evy excursion.  In order to show that the topology of the breadth first construction of the L\'evy net quotient from Definition~\ref{def::levynet_bubble_gluing} (illustrated in Figure~\ref{fig::bubblegluing}) is equivalent to that associated with the constructions from Definition~\ref{def::levynet} (illustrated in Figure~\ref{fig::levynet}) and Definition~\ref{def::levynet2} (illustrated in Figure~\ref{fig::levynet2}), we first need to construct a modification of $Z_s^r$ which has certain continuity properties.  We will then use this modification to construct the map which takes the construction described in Figure~\ref{fig::levynet2} to the breadth first construction.

\begin{proposition}
\label{prop::local_time_cadlag_modification}
The process $(r,s) \mapsto Z_s^r$ has a jointly measurable modification which a.s.\ satisfies the following two properties (for all $r,s$ simultaneously).
\begin{enumerate}
\item The map $r \mapsto Z_\cdot^r$ is continuous with respect to the uniform topology.
\item The map $s \mapsto Z_s^\cdot$ is {\cadlag} with respect to the uniform topology.
\end{enumerate}
\end{proposition}

See \cite[Proposition~1.3.3]{dlg2002trees_levy} for a related result.  We note that the modification obtained in Proposition~\ref{prop::local_time_cadlag_modification} has stronger continuity properties than given in \cite[Proposition~1.3.3]{dlg2002trees_levy}.

We need to collect several intermediate lemmas before we give the proof of Proposition~\ref{prop::local_time_cadlag_modification}.  We begin with an elementary estimate for $\alpha$-stable CSBPs.

\begin{lemma}
\label{lem::csbp_terminal_value}
Suppose that $W$ is an $\alpha$-stable CSBP.  There exists a constant $c_0 > 0$ depending only on $\alpha$ such that
\[ \p[ W_t \leq \delta] \leq \exp( -(\delta- c_0  W_0) t^{1/(1-\alpha)}) \quad\text{for all}\quad \delta  > 0.\]
\end{lemma}
\begin{proof}
Using the representation of the Laplace transform of an $\alpha$-stable CSBP given in~\eqref{eqn::csbp_def}, \eqref{eqn::csbp_u_form}, we have for $\lambda > 0$ that
\begin{align*}
   \p[ W_t \leq \delta]
&= \p[ e^{-\lambda W_t} \geq e^{-\lambda \delta}]
 \leq e^{\lambda \delta} \E[ e^{-\lambda W_t}]
 = e^{\lambda \delta - u_t(\lambda) W_0}.
\end{align*}
where $u_t(\lambda) = (\lambda^{1-\alpha} + c t)^{1/(1-\alpha)}$ and $c > 0$ is a constant.  Taking $\lambda = t^{1/(1-\alpha)}$ yields the result.
\end{proof}

For each $s,u \geq 0$, we let~$T_s^u$ be the smallest value of~$r$ that $Z_s^r \geq u$.  On the event that $T_s^u < \infty$, we note that the same argument used to prove Proposition~\ref{prop::Zisalphastable} implies that~$Z_t^{T_s^u}$ evolves as an $\alpha$-stable CSBP for $t \geq s$ with initial value~$u$.

\begin{lemma}
\label{lem::length_lbd}
There exists a constant $c_0 > 0$ such that the following is true.  Fix $s > 0$. For each $u \geq 0$ and $w,v > 0$ we have that
\begin{equation}
\label{eqn::local_time_euclidean_lbd}
\p[ T_s^{u+v} - T_s^u \leq t,\ T_s^u < \infty \giv Z_s > w] \leq \exp(-c_0 v t^{-1/\alpha}).
\end{equation}
\end{lemma}
\begin{proof}
Let ${\mathbf n}$ be the excursion measure associated with an $\alpha$-stable L\'evy process with only upward jumps from its running infimum.  As explained in \cite[Chapter~VIII.4]{bertoinlevybook}, there exists a constant $c_\alpha > 0$ depending only on~$\alpha$ such that ${\mathbf n}[ \zeta \geq t] = c_\alpha t^{-1/\alpha}$ where~$\zeta$ denotes the length of the excursion.  This implies that in $v$ units of local time, the number $N$ of excursions of $X$ with height at least $s$ and with length at least $t$ is distributed as a Poisson random variable with mean $c_\alpha v t^{-1/\alpha}$.  Note that on the event that we have at least one such excursion, it is necessarily the case that $T_s^{u+v} - T_s^u \geq t$.  Consequently, \eqref{eqn::local_time_euclidean_lbd} follows from the explicit formula for the probability mass distribution for a Poisson random variable evaluated at $0$.
\end{proof}

We turn to describe the setup for the proof of Proposition~\ref{prop::local_time_cadlag_modification}.  We begin by emphasizing that each $Z_t^{T_s^u}$ for $t \geq s$ evolves as an $\alpha$-stable CSBP.

Fix $s_0 > 0$.  Then we know that $Z_t$ for $t \geq s_0$ evolves as an $\alpha$-stable CSBP starting from $Z_{s_0}$.  Fix $\delta > 0$ and assume that $Z_{s_0} \geq \delta/2$.  We inductively define stopping times and a modification of $Z$ as follows.  First, we let $n_1 = \lceil 4\delta^{-1} Z_{s_0} \rceil$, $\delta_1 = Z_{s_0} / n_1$, and let $Z_t^{1,j} = Z_t^{T_{s_0}^{j\delta_1}} - Z_t^{T_{s_0}^{(j-1)\delta_1}}$ so that the $Z^{1,j}$ for $1 \leq j \leq n_1$ are independent $\alpha$-stable CSBPs defined on the time-interval $[s_0,\infty)$ all with initial value $\delta/8 \leq \delta_1 \leq \delta/4$ (unless $n_1 = 1$).  We then take a modification so that $Z_t^{1,j}$ for $t \geq s_0$ and $1 \leq j \leq n_1$ are {\cadlag}. We then let
\[ \tau_1 = \inf\left\{t \geq s_0 : \max_{1 \leq j \leq n_1} Z_t^{1,j} \geq \delta/2 \right\}.\]
Note that if $\tau_1 < \infty$ then $Z_{\tau_1} \geq \delta/2$.
Assume that stopping times $\tau_1,\ldots,\tau_k$ and (\cadlag) CSBPs $Z^{j,1},\ldots,Z^{j,n_j}$ have been defined for $1 \leq j \leq k$.  We then let $n_{k+1} = \lceil 4 \delta^{-1} Z_{\tau_k} \rceil$, $\delta_{k+1} = Z_{\tau_k} / n_{k+1}$, and $Z_t^{k+1,j} = Z_t^{T_{\tau_k}^{j \delta_{k+1}}} - Z_t^{T_{\tau_k}^{(j-1)\delta_{k+1}}}$.  Then the $Z_t^{k+1,j}$ are independent $\alpha$-stable CSBPs defined on the time-interval $[\tau_k,\infty)$ all with initial value $\delta/8 \leq \delta_{k+1} \leq \delta/4$ (unless $n_{k+1} = 1$).  We modify $Z$ again if necessary so that the processes $Z_t^{k+1,j}$ for $t \geq \tau_k$ and $1 \leq j \leq n_{k+1}$ are {\cadlag}.  We then let
\[ \tau_{k+1} = \inf\left\{t \geq \tau_k : \max_{1 \leq j \leq n_{k+1}} Z_t^{k+1,j} \geq \delta/2 \right\}.\]

We note that
\begin{equation}
\label{eqn::splitting_sup_bound}
n^* := \sup_j n_j \leq 1+ \frac{4}{\delta} \sup_{t \geq s_0} Z_t.
\end{equation}
Combining~\eqref{eqn::splitting_sup_bound} and Lemma~\ref{lem::csbp_sup}, we see for a constant $c_0 > 0$ that on the event $\{Z_{s_0} \geq \delta\}$ we have
\begin{equation}
\label{eqn::splitting_bound}
\p[n^* \geq M \giv Z_{s_0}] \leq \frac{c_0 Z_{s_0}}{\delta M}.
\end{equation}

\begin{lemma}
\label{lem::number_of_splittings_bound}
For each $\delta > 0$ and $\delta < a < b < \infty$ there exists a constant $c_0 > 0$ such that on the event $\{ Z_{s_0} \in [a,b]\}$ we have that
\begin{equation}
\label{eqn::tau_n_tail}
\p[ \tau_n \leq 1 \giv Z_{s_0}] \leq c_0 n^{-1/2}.
\end{equation}
\end{lemma}
\begin{proof}
Throughout, we shall assume that we are working on the event $\{ Z_{s_0} \in [a,b]\}$.  By~\eqref{eqn::splitting_bound}, we know that there exists a constant $c_0 > 0$ such that
\begin{align}
      \p[ \tau_n \leq 1 \giv Z_{s_0}]
&\leq \p[ \tau_n \leq 1,\ n^* \leq M \giv Z_{s_0}] + \frac{c_0 Z_{s_0}}{\delta M}. \label{eqn::n_j_bound}
\end{align}
We take $M = n^{1/2} Z_{s_0}/\delta$ so that the error term on the right hand side of~\eqref{eqn::n_j_bound} is at most a constant times $n^{-1/2}$.

Let $\CF_t$ be the $\sigma$-algebra generated by $Z_r^{T_s^u}$ for all $s \leq r \leq t$ with $u,s,r \in \Q_+$.  We claim that, given $\CF_{\tau_k}$, we have that $\tau_{k+1} - \tau_k$ is stochastically dominated from below by a random variable $\xi_k$ such that the probability that $\xi_k$ is at least $1/n_{k+1}$ is at least some constant $p_0 > 0$ (which may depend on $\delta$ but not $k$).  Upon showing this,~\eqref{eqn::tau_n_tail} will follow by combining~\eqref{eqn::n_j_bound} with binomial concentration.  We note that the claim is clear in the case that $n_{k+1}=1$, so we now assume that $n_{k+1} \geq 2$ and we let
\[ \sigma_{k+1} = \inf\left\{ t \geq \tau_k : \min_{1 \leq j \leq n_{k+1}} Z_t^{k+1,j} \leq \delta/16 \right\} \quad\text{and}\quad \wt{\tau}_{k+1} = \tau_{k+1} \wedge \sigma_{k+1}.\]
Since $\wt{\tau}_{k+1} \leq \tau_{k+1}$, it suffices to prove the stochastic domination result for $\wt{\tau}_{k+1}-\tau_k$ in place of $\tau_{k+1}-\tau_k$.

By the Lamperti transform (Theorem~\ref{thm::lamperti}), it suffices to show that the probability that $n_{k+1}$ independent $\alpha$-stable L\'evy processes, each starting from a common value in $[\delta/8,\delta/4]$ and run for time $(16/\delta) \times n_{k+1}^{-1}$, all do not leave the interval $[\delta/16,\delta/2]$ is at least some $p_0 > 0$.  (The factor $16/\delta$ comes from the speedup when transforming to L\'evy process time.)  This, in turn, follows from \cite[Chapter~VII, Corollary~2]{bertoinlevybook} and \cite[Chapter~VIII, Proposition~4]{bertoinlevybook}.
\end{proof}

\begin{proof}[Proof of Proposition~\ref{prop::local_time_cadlag_modification}]
We will prove the result by showing that $r \mapsto Z_\cdot^r$ for $r \in \Q_+$ is a.s.\ uniformly continuous with respect to the uniform topology.  Throughout, we assume that $s_0,\delta_0,\delta > 0$ are fixed and we let $H_{s_0,\delta_0} = \{ Z_{s_0} \in [\delta_0/2,\delta_0]\}$.  Also, $c_j > 0$ will denote a constant (which can depend on $s_0,\delta_0,\delta$).

For each $\ell \in \N$ and $\Delta > 0$ we let
\[ F_{\ell,\Delta}^{\delta} = \bigcap_k \left\{ T_{s_0+ \ell \Delta}^{k\delta^2} - T_{s_0+\ell \Delta}^{(k-1)\delta^2} \geq \Delta^{\alpha} \delta^{3\alpha} \right\}.\]
Lemma~\ref{lem::csbp_sup} and Lemma~\ref{lem::length_lbd} together imply that
\begin{equation}
\label{eqn::f_j_delta_bound0}
\p[ (F_{\ell,\Delta}^{\delta})^c \giv H_{s_0,\delta_0}] \leq c_0 M^{-1} + \frac{M}{\delta^2} \exp(-c_1 \Delta^{-1} \delta^{-1}).
\end{equation}
By optimizing over $M$, it follows from~\eqref{eqn::f_j_delta_bound0} that
\begin{equation}
\label{eqn::f_j_delta_bound}
\p[ (F_{\ell,\Delta}^{\delta})^c \giv H_{s_0,\delta_0}] \leq \exp(-c_2  \Delta^{-1} \delta^{-1}).
\end{equation}

Let $\zeta = \inf\{ s > 0 : Z_s = 0\}$.  By performing a union bound over $\ell$ values, from~\eqref{eqn::f_j_delta_bound} and Lemma~\ref{lem::csbp_extinction_time} we have with $F_\Delta^{\delta} = \cap_\ell F_{\ell,\Delta}^{\delta}$ that
\begin{equation}
\label{eqn::f_delta_delta_bound0}
\p[ (F_\Delta^\delta)^c \giv H_{s_0,\delta_0}] \leq \frac{T}{\Delta} \exp(-c_3 \Delta^{-1} \delta^{-1}) + c_4 T^{1/(1-\alpha)}.
\end{equation}
Optimizing~\eqref{eqn::f_delta_delta_bound0} over $T$ values implies that
\begin{equation}
\label{eqn::f_delta_delta_bound}
\p[ (F_\Delta^\delta)^c \giv H_{s_0,\delta_0}] \leq \exp(-c_5 \Delta^{-1} \delta^{-1}).
\end{equation}
Therefore the Borel-Cantelli lemma implies that with $\Delta = e^{-j}$, for each $\delta > 0$ there a.s.\ exists $j_F^\delta \in \N$ (random) such that $j \geq j_F^\delta$ implies that $F_\Delta^\delta$ occurs.

We also let $G_{\ell,\Delta}^{\delta}$ be the event that for every $s \in \Q$ with $s \in [s_0 + (\ell-1) \Delta ,s_0 + \ell\Delta]$ and $t_1,t_2 \in \Q_+$ with $t_2 \geq t_1$ such that $Z_s^{t_2} - Z_s^{t_1} \geq \delta$ we have that $Z_{s_0 + \ell \Delta}^{t_2} - Z_{s_0 + \ell \Delta}^{t_1} \geq 2\delta^2$.  We claim that it suffices to show that
\begin{equation}
\label{eqn::g_j_delta_bound}
\p[ (G_{\ell,\Delta}^{\delta})^c \giv H_{s_0,\delta_0}] \leq \exp(-c_6 \delta \Delta^{1/(1-\alpha)}).
\end{equation}
Letting $G_\Delta^{\delta} = \cap_\ell G_{\ell,\Delta}^{\delta}$, we have from~\eqref{eqn::g_j_delta_bound} by performing a union bound over $\ell$ values (and applying Lemma~\ref{lem::csbp_extinction_time} as in the argument to prove~\eqref{eqn::f_delta_delta_bound}) that
\[ \p[ (G_\Delta^\delta)^c \giv H_{s_0,\delta_0}]  \leq \exp(-c_7 \delta \Delta^{1/(1-\alpha)}).\]
Thus the Borel-Cantelli lemma implies that with $\Delta = e^{-j}$, for each $\delta > 0$ there a.s.\ exists $j_G^\delta \in \N$ (random) such that $j \geq j_G^\delta$ implies that $G_\Delta^\delta$ occurs.  In particular, this implies that for every $s \geq s_0$ with $s \in \Q$ and $t_1,t_2$ such that $Z_s^{t_2} - Z_s^{t_1} \geq \delta$ we have that $Z_{s_0 + \ell \Delta}^{t_2} - Z_{s_0 + \ell\Delta}^{t_1} \geq 2\delta^2$ where
\begin{equation}
\label{eqn::ell_def}
\ell = \lceil (s-s_0)/\Delta \rceil
\end{equation}
for $\Delta = e^{-j}$ and $j \geq j_G^\delta$.

Assume that $j \geq j_F^\delta \vee j_G^\delta$ so that with $\Delta = e^{-j}$ we have that both $F_\Delta^\delta$ and $G_\Delta^\delta$ occur.  Suppose that $t_1,t_2,s$ are such that $Z_s^{t_2} - Z_s^{t_1} \geq \delta$.  With $\ell$ as in~\eqref{eqn::ell_def}, it must be true that $Z_{s_0 + \ell \Delta}^{t_2} - Z_{s_0 + \ell \Delta}^{t_1} \geq 2\delta^2$.  This implies that there exists $k$ such that 
\begin{equation}
\label{eqn::fit_t_1_t2}
T_{s_0 + \ell \Delta}^{k \delta^2} \leq t_2 \quad\text{and}\quad T_{s_0 + \ell \Delta}^{(k-1) \delta^2} \geq t_1.
\end{equation}
Rearranging~\eqref{eqn::fit_t_1_t2}, we thus have that
\begin{equation}
\label{eqn::t_2_t_1_diff_bound}
t_2 - t_1 \geq T_{s_0 + \ell \Delta}^{k \delta^2} - T_{s_0 + \ell \Delta}^{(k-1) \delta^2} \geq \Delta^\alpha \delta^{3\alpha}.
\end{equation}
This implies that $r \mapsto Z_\cdot^r|_{[s_0,\infty)}$ for $r \in \Q_+$ has a certain modulus of continuity with respect to the uniform topology.  In particular, $r \mapsto Z_\cdot^r|_{[s_0,\infty)}$ for $r \in \Q_+$ is uniformly continuous with respect to the uniform topology hence extends continuously.  The result then follows (assuming~\eqref{eqn::g_j_delta_bound}) since $s_0,\delta_0,\delta > 0$ were arbitrary.

To finish the proof, we need to establish~\eqref{eqn::g_j_delta_bound}.  For each $j$, we let
\[ E_j =  \{\tau_j \geq s_0 + \Delta \} \cup \left( \cap_{k=1}^{n_j} \{ Z_{s_0+\Delta}^{j,k} \geq 2\delta^2\} \right).\]
We first claim that $G_{1,\Delta}^\delta \supseteq \cap_{j=1}^n E_j$.  To see this, fix a value of $s \in [s_0,s_0 + \Delta]$ and suppose that $Z_s^{t_2} - Z_s^{t_1} \geq \delta$.  Let $j$ be such that $\tau_j \leq s < \tau_{j+1}$ and let $k$ be the first index so that $Z_s^{j,1} + \cdots + Z_s^{j,k} \geq Z_s^{t_1}$.  Since $Z_s^{j,i} \leq \delta/2$ for all $i$, it follows that $Z_s^{j,1} + \cdots + Z_s^{j,k+1} \leq Z_s^{t_2}$.  Consequently, $Z_{s_0+\Delta}^{t_2} - Z_{s_0+\Delta}^{t_1} \geq Z_{s_0+\Delta}^{j,k+1}$.  The claim follows because we have that $Z_{s_0+\Delta}^{j,k+1} \geq 2\delta^2$ on $\cap_j E_j$.

Thus to finish the proof, it suffices to show that
\begin{equation}
\label{eqn::e_j_udb}
\p[ \cup_{j=1}^n E_j^c \giv H_{s_0,\delta_0}] \leq \exp(-c_8 \delta \Delta^{1/(1-\alpha)})
\end{equation}
(as the same analysis leads to the same upper bound for $\p[(G_{\ell,\Delta}^\delta)^c \giv H_{s_0,\delta_0}]$ for other $\ell$ values).  To this end, 
Lemma~\ref{lem::csbp_terminal_value} implies that
\begin{equation}
\label{eqn::one_e_j_bound}
\p[ E_j^c,\ Z^* \leq \delta M/4 \giv H_{s_0,\delta_0}]  \leq M \exp(-c_9 \delta \Delta^{1/(1-\alpha)}).
\end{equation}
Thus applying a union bound together with~\eqref{eqn::one_e_j_bound} in the second step below, we have for each $n \in \N$ that
\begin{align}
  & \p[ \cup_j E_j^c,\ Z^* \leq \delta M/4 \giv H_{s_0,\delta_0}] \notag\\
=& \p[ \cup_j E_j^c,\ Z^* \leq \delta M/4,\ \tau_n \geq \Delta \giv H_{s_0,\delta_0}] + \p[ \tau_n \leq \Delta \giv H_{s_0,\delta_0}] \notag\\
\leq& n M \exp(- c_{10} \delta \Delta^{1/(1-\alpha)}) + c_{11} n^{-1/2} \quad\text{(by Lemma~\ref{lem::number_of_splittings_bound})} \label{eqn::e_j_one_step_bound}
\end{align}
Applying Lemma~\ref{lem::csbp_sup}, we therefore have that
\begin{equation}
\label{eqn::e_j_c_bound}
\p[ \cup_j E_j^c \giv H_{s_0,\delta_0}] \leq n M \exp(- c_{10} \delta \Delta^{1/(1-\alpha)}) + c_{11} n^{-1/2} + c_{12} (\delta M)^{-1}.
\end{equation}
Optimizing over $n$ and $M$ values implies~\eqref{eqn::e_j_udb}.
\end{proof}

\begin{proof}[Proof of Proposition~\ref{prop::levy_net_equivalent_topology}]
As we remarked earlier, it suffices to show the equivalence of the quotient topology from Definition~\ref{def::levynet2} (Figure~\ref{fig::levynet2}) with the quotient topology described in Definition~\ref{def::levynet_bubble_gluing} (Figure~\ref{fig::bubblegluing}).  We will show this by arguing that $Z_s^r$ induces a continuous map $\wt{Z}_s^r$ from Figure~\ref{fig::levynet2} to Figure~\ref{fig::bubblegluing} which takes equivalence classes to equivalence classes in a bijective manner.  This will prove the result because this map then induces a bijection which is continuous from the space which arises after quotienting as in Definition~\ref{def::levynet2} (Figure~\ref{fig::levynet2}) to the space which arises after quotienting as in Definition~\ref{def::levynet_bubble_gluing} (Figure~\ref{fig::bubblegluing}) and the fact that bijections which are continuous from one compact space to another are homeomorphisms.

Let $S$ be the vertical height of the rectangle as in the right hand side of Figure~\ref{fig::levynet2} and fix $s \in [0,S]$.  Let $t$ be the vertical height which corresponds to $s$ as in the left side of Figure~\ref{fig::levynet2}.  In other words, $t$ is obtained from $s$ by mapping the right to the left side of Figure~\ref{fig::levynet2} by removing the stripes which correspond to the jumps.  If~$t$ is not a jump time for~$Z$, then we take $\wt{Z}_s^r = Z_t^r / Z_t$.  Suppose that $t$ is a jump time for~$Z$.  If $s$ is the $y$-coordinate of the top (resp.\ bottom) of the corresponding rectangle as in the right side of Figure~\ref{fig::levynet2}, we take $\wt{Z}_s^r = \lim_{q \downarrow t} Z_q^r / Z_q$ (resp.\ $\wt{Z}_s^r = \lim_{q \uparrow t} Z_q^r / Z_q$).  Suppose that $s$ is between the bottom and the top of the corresponding rectangle.  If $(s,r)$ is outside of the interior of the rectangle, then we take $\wt{Z}_s^r = Z_t^r / Z_t$.  Note that in this case we have that the limit $\lim_{q \to t} Z_q^r / Z_q$ exists and is equal to $Z_t^r / Z_t$.  Let $s_1$ (resp.\ $s_2$) be the $y$-coordinate of the bottom (resp.\ top) of the rectangle.  If $(s,r)$ is in the rectangle, then we take $\wt{Z}_s^r$ to be given by linearly interpolating between the values of $\wt{Z}_{s_1}^r$ and $\wt{Z}_{s_2}^r$.  That is,
\[ \wt{Z}_s^r = \frac{s_2-s}{s_2-s_1} \wt{Z}_{s_1}^r + \frac{s-s_1}{s_2-s_1} \wt{Z}_{s_2}^r.\]
By the continuity properties of $Z$ given in Proposition~\ref{prop::local_time_cadlag_modification} and the construction of $\wt{Z}$, we have that the map $(s,r) \mapsto \wt{Z}_s^r$ is continuous.

Observe that $\wt{Z}$ is constant on the equivalence classes as defined in Definition~\ref{def::levynet2} (Figure~\ref{fig::levynet2}).  This implies that~$\wt{Z}$ induces a continuous map from the topological space one obtains after quotienting by the equivalence relation as in Definition~\ref{def::levynet2} (Figure~\ref{fig::levynet2}) into the one from Definition~\ref{def::levynet_bubble_gluing} (Figure~\ref{fig::bubblegluing}, not yet quotiented).  As $\wt{Z}$ bijectively takes equivalence classes as in Definition~\ref{def::levynet2} (Figure~\ref{fig::levynet2}) to equivalence classes as in Definition~\ref{def::levynet_bubble_gluing} (Figure~\ref{fig::bubblegluing}), it follows that~$\wt{Z}$ in fact induces a bijection which is continuous from the quotient space as in Definition~\ref{def::levynet2} (Figure~\ref{fig::levynet2}) to the quotient space as in Definition~\ref{def::levynet_bubble_gluing} (Figure~\ref{fig::bubblegluing}).  The result follows because, as we mentioned earlier, a bijection which is continuous from one compact space to another is a homeomorphism.
\end{proof}

\subsection{Recovering embedding from geodesic tree quotient}
\label{subsec::recovering_embedding}

We now turn to show that the embedding of the L\'evy net into~$\S^2$ is unique up to a homeomorphism of $\S^2$.  Recall that a set is called {\em essentially $3$-connected} if deleting two points always produces either a connected set, a set with two components one of which is an open arc, or a set with three components which are all open arcs.  In particular, every $3$-connected set is essentially $3$-connected.  Suppose that a compact topological space~$K$ can be embedded into $\S^2$ and that $\phi_1 \colon K \to \S^2$ is such an embedding. It is then proved in \cite{3connectedembeduniquely} that $K$ is essentially $3$-connected if and only if for every embedding $\phi \colon K \to \S^2$, there is a homeomorphism $h \colon \S^2 \to \S^2$ such that $\phi = h \circ \phi_1$.\footnote{It is clear from our construction that when $K$ is a L\'evy net there exists at least one embedding of~$K$ into~$\S^2$. More generally, it is shown in \cite{planarityofcompactmetricspaces} that a compact and locally connected set $K$ is homeomorphic to a subset of~$\S^2$ if and only if it contains no homeomorph of~$K_{3,3}$ or~$K_5$.}

\begin{proposition}
\label{prop::levynet_three_connected}
For each $\alpha \in (1,2)$, the topological space associated with the L\'evy net is a.s.\ $3$-connected.  Hence by \cite{3connectedembeduniquely} it can a.s.\ be embedded in $\S^2$ in a unique way (up to a homeomorphism).
\end{proposition}
\begin{proof}
Suppose that~$W$ is an instance of the L\'evy net and assume for contradiction that~$W$ is not $3$-connected.  Then there exists distinct points $x,y \in W$ such that $W \setminus \{x,y\}$ is not connected.  This implies that we can write $W \setminus \{x,y\} = A \cup B$ for $A,B \subseteq W$ disjoint and $A,B \neq \emptyset$.  We assume that~$W$ has been embedded into $\S^2$.  Let~$\wt{A}$ (resp.\ $\wt{B}$) be given by~$A$ (resp.\ $B$) together with all of the components of $\S^2 \setminus W$ whose boundary is entirely contained in $A$ (resp.\ $B$).  Then $\wt{A}$, $\wt{B}$ are disjoint and we can write $\S^2$ as a disjoint union of $\wt{A}$, $\wt{B}$, $\{x\}$, $\{y\}$, and the components of $\S^2 \setminus W$ whose boundary has non-empty intersection with both $A$ and $B$.  Suppose that $C$ is such a component.  Then there exists a point $w \in \partial C$ which is not in $\wt{A}$ or $\wt{B}$.  That is, either $x \in \partial C$ or $y \in \partial C$.

Note that $\S^2 \setminus (\wt{A} \cup \wt{B} \cup \{x,y\})$ must have at least two distinct components $C_1,C_2$ (for otherwise $\wt{A}$, $\wt{B}$ would not be disjoint).  If either $x$ or $y$ is in $\partial C_1 \cap \partial C_2$ then we have a contradiction because the distance of both $\partial C_1$ and $\partial C_2$ to the root of $W$ must be the same but by the breadth-first construction of the L\'evy net (Definition~\ref{def::levynet_bubble_gluing}) and Theorem~\ref{thm::levynetrecover} we know that the metric exploration from the root to the dual root in $W$ does not separate more than one component from the dual root at any given time (as these components correspond to jumps of the boundary length process).  If $\partial C_1 \cap \partial C_2$ does not contain either $x$ or $y$, then there must be a third component $C_3$ of $\S^2 \setminus (\wt{A} \cup \wt{B} \cup \{x,y\})$.  This leads to a contradiction because then (by the pigeon hole principle) either $\partial C_1 \cap \partial C_3$ or $\partial C_2 \cap \partial C_3$ contains either $x$ or $y$.
\end{proof}

We are now going to use that the topological space associated with the L\'evy net a.s.\ has a unique embedding into~$\S^2$ up to homeomorphism to show that it together withe the distance function to the root and an orientation a.s.\ determines the L\'evy excursion~$X$ used to generate it.

\begin{proposition}
\label{prop::levynet_structure_determined}
For each $\alpha \in (1,2)$, the $\alpha$-stable L\'evy excursion $X$ used in the construction of the L\'evy net is a.s.\ determined by the topological space associated with the L\'evy net and distance function to the root together with an orientation.
\end{proposition}
\begin{proof}
By Proposition~\ref{prop::levynet_three_connected}, we know that the embedding of the (topological space associated with the) L\'evy net into $\S^2$ is a.s.\ determined up to homeomorphism; we assume throughout that we have fixed an orientation so that the embedding is determined up to orientation preserving homeomorphism.  Recall that the jumps of $Z_s$ are in correspondence with those made by $X_t$.  Thus, if we can show that the jumps of $Z$ are determined by the L\'evy net, then we will get that the jumps of $X$ are determined by the L\'evy net.  More generally, if we can show that the processes $Z_t^{T_s^u}$ are determined by the L\'evy net, then we will be able to determine the jumps of $X$ and their ordering.  This will imply the result because $X$ is a.s.\ determined by its jumps and the order in which they are made.  For simplicity, we will just show that $Z_s$ is a.s.\ determined by the L\'evy net.  The proof that $Z_t^{T_s^u}$ is a.s.\ determined follows from the same argument.

Let $x$ (resp.\ $y$) denote the root (resp.\ dual root) of the L\'evy net.  Fix $r > 0$ and condition on $R = d(x,y) -r > 0$.  We let $\partial B(x,R)$ be the boundary of the ball of radius~$R$ centered at~$x$ in the geodesic tree in the L\'evy net.  Fix $\epsilon > 0$.  We then fix points $z_1,\ldots,z_{N_\epsilon} \in \partial B(x,R)$ as follows.  We let~$z_1$ be the unique point on $\partial B(x,R)$ which is visited by the unique geodesic from~$x$ to~$y$.  For $j \geq 2$ we inductively let~$z_j$ be the first clockwise point on~$\partial B(x,R)$ (recall that we have assumed that the L\'evy net has an orientation) such that the geodesic from~$z_j$ to~$x$ merges with the geodesic from~$z_{j-1}$ to~$x$ at distance at least~$\epsilon$.  As the embedding of the L\'evy net into $\S^2$ is a.s.\ determined up to (orientation preserving) homeomorphism, it follows that $z_1,\ldots,z_{N_\epsilon}$ is a.s.\ determined by the L\'evy net.

Conditional on the boundary length $L_r$ of $\partial B(x,R)$, we claim that $N_\epsilon$ is distributed as a Poisson random variable $Z_\epsilon$ with mean $m_\epsilon^{-1} L_r$ where $m_\epsilon = (c \epsilon)^{1/(\alpha-1)}$ and $c > 0$ is a constant.  The desired result will follow upon showing this because then
\begin{align*}
	\E[ m_\epsilon Z_\epsilon \giv L_r] =  L_r \quad\text{and}\quad
		\var[ m_\epsilon Z_\epsilon  \giv L_r]  = m_\epsilon L_r \to 0 \quad\text{as}\quad \epsilon \to 0.
\end{align*}

To compute the conditional distribution of $N_\epsilon$ given $L_r$, it suffices to show that the boundary length of the spacings are given by i.i.d.\ exponential random variables with mean $m_\epsilon$ given $L_r$.  We will establish this by using that $L_r$ evolves as an $\alpha$-stable CSBP as $r$ varies.  Fix $\delta > 0$ and let $(Z_j^\delta)$ be a sequence of i.i.d.\ $\alpha$-stable CSBPs, each starting from~$\delta$.  Then the CSBP property~\eqref{eqn::CSBPproperty} implies that the process $s \mapsto L_{r+s}$ is equal in distribution to $Z_1^\delta + \cdots + Z_n^\delta + \wt{Z}^\delta$ where $n = \lfloor L_r/\delta \rfloor$ and $\wt{Z}^\delta$ is an independent $\alpha$-stable CSBP starting from $L_r - \delta n < \delta$.  We then define indices $(j_k^\delta)$ inductively as follows.  We let $j_1^\delta$ be the first index $j$ such that the amount of time it takes the $\alpha$-stable CSBP $Z_1^\delta + \cdots + Z_j^\delta$ (which starts from $j \delta$) to reach $0$ is at least $\epsilon$.  Assuming that $j_1^\delta,\ldots,j_k^\delta$ have been defined, we take $j_{k+1}^\delta$ to be the first index $j$ such that the amount of time that it takes the $\alpha$-stable CSBP $Z_{j_k^\delta+1}^\delta + \cdots + Z_j^\delta$ (which starts from $\delta(j - (j_k^\delta+1))$) to reach $0$ is at least $\epsilon$.

Note that the random variables
\[ Z_k^\delta = Z_{j_{k-1}^\delta+1}^\delta + \cdots + Z_{j_k^\delta}^\delta\]
are i.i.d.  We claim that the law of $Z_1^\delta$ converges in distribution as $\delta \to 0$ to that of an exponential random variable with mean $m_\epsilon$.  To see this, we fix $u > 0$, let $\wt{u} = \delta \lfloor u/\delta \rfloor$, and let $W$ be an $\alpha$-stable CSBP starting from $\wt{u}$.  Then we have that
\begin{align}
    \p[ Z_1^\delta \geq u]
&= \p[ W_\epsilon = 0]
  = \lim_{\lambda \to \infty} \E[ \exp(-\lambda W_\epsilon)]. \label{eqn::z_1_at_least_u}
\end{align}
As in the proof of Lemma~\ref{lem::csbp_terminal_value}, using the representation of the Laplace transform of an $\alpha$-stable CSBP given in~\eqref{eqn::csbp_def}, \eqref{eqn::csbp_u_form}, the Laplace transform on the right hand side of~\eqref{eqn::z_1_at_least_u} is given, for a constant $c > 0$, by
\[ \exp( - (\lambda^{1-\alpha} + (c \epsilon)^{1/(1-\alpha)} \wt{u}).\]
Therefore the limit on the right hand side of~\eqref{eqn::z_1_at_least_u} is given by $\exp(- m_\epsilon^{-1} \wt{u})$.  This, in turn, converges to $\exp(-m_\epsilon^{-1} u)$ as $\delta \to 0$, which proves the result.
\end{proof}

\section{Tree gluing and the Brownian map}
\label{sec::brownianmap}

\subsection{Gluing trees encoded by Brownian-snake-head trajectory}
\label{subsec::spheresassymetric}

We now briefly review the standard construction of the Brownian map (see e.g.\ \cite[Section~3.4]{legall2014icm}).  Our first task is to identify the measure~$\mustwo$ discussed in Section~\ref{subsec::discreteintuition} with a certain Brownian snake excursion measure. In fact, this is the way~$\mustwo$ is formally constructed and defined.

Let~$\snake$ be the set of all finite paths in~$\R$ beginning at~$0$.  An element of $\snake$ is a continuous map $w \colon [0, \zeta] \to \R$ for some value $\zeta = \zeta(w) \geq 0$ that depends on~$w$.   We refer to~$\snake$ as the {\em snake space} and visualize an element of~$\snake$ as the ($y$-to-$x$ coordinate) graph $\{ (w(y), y): y \in [0,\zeta] \}$.  As illustrated in Figure~\ref{fig::browniansnake}, such a graph may be viewed as a ``snake'' with a body beginning at $(0,0)$ and ending at the ``head,'' which is located at $\bigl(w(\zeta), \zeta \bigr)$.  From this perspective, $\zeta = \zeta(w)$ is the {\em height} of the snake, which is also the {\em vertical head coordinate}, and $w(\zeta)$ is the {\em horizontal head coordinate}.

A distance on $\snake$ is given by
\begin{equation}
\label{eqn::snake_metric}
d(w,w') = |\zeta(w) - \zeta(w')| + \sup_{t \geq 0} | w(t \wedge \zeta(w)) - w'(t \wedge \zeta(w'))|.
\end{equation}

\begin{figure}[ht!]
\begin{center}
\includegraphics [width=5.5in]{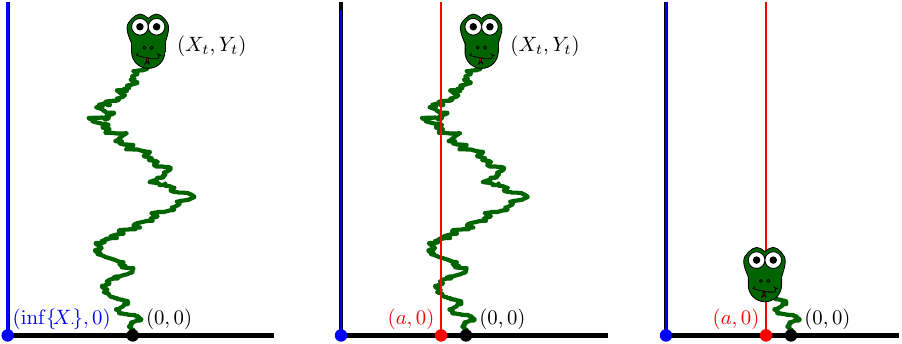}
\caption{\label{fig::browniansnake} {\bf \it Gluing an asymmetric pair of trees.} The doubly marked Brownian map construction is the same as the construction in the left side of Figure~\ref{fig::lamination_treemaking} except that the pair $(X_t, Y_t)$ is produced from a Brownian snake excursion instead of a Brownian excursion.  In this setup $Y_t$ is chosen from the (infinite) Brownian excursion measure and $X_t$ is a Brownian motion indexed by the corresponding CRT.  The process $(X_t, Y_t)$ determines a trajectory in the snake space $\snake$. {\bf Left:} At a given time $t$, the ``snake'' has a body that looks like the graph of a Brownian motion (rotated 90 degrees).  
The blue vertical line represents the leftmost point reached by the process $(X_t, Y_t)$.  It was proved by Duquesne (see \cite[Lemma~4.15]{marckert2006limit}) that there is a single time at which the blue line is hit.  After projection, this corresponds to the Brownian map root.  At all other times, distance from the blue line represents distance from the root in the Brownian map metric. {\bf Middle:} Suppose $ \inf \{X_\cdot \} < a < 0$ and consider the vertical line through $(a,0)$.  This divides the snake space $\snake$ into the subspace $\snake_{>a}$ of snakes not hit by the red line (except at the origin if $a = 0$) and the complementary subspace $\snake_{\leq a} = \snake \setminus \snake_{>a}$ of snakes that {\it are} hit.  {\bf Right:} If a snake is hit by the red line, then it has a unique ``ancestor snake'' whose body lies entirely to the right of the red line and whose head lies on the red line.  A snake lies on the boundary of $\snake_{> a}$ if and only if it has this form.  The distance from a snake in $\snake_{\leq a}$ to $\snake_{>a}$ (in terms of the metric on $\snake$, not the Brownian map metric) is the difference in head height between itself and this ancestor.  This distance evolves as a Brownian motion  (until it first reaches $0$) in the snake space.}
\end{center}
\end{figure}

There is a natural way to create a time-length $T$ excursion into $\snake$ beginning and ending at the zero snake. To do so, let $Y_t$ be a time-length $T$ Brownian excursion into $[0, \infty)$ (starting and ending at zero). Then $Y_t$ encodes a continuum random tree (CRT) $\mathcal T$ \cite{ald1991crt1,ald1991crt2,ald1993crt3}, together with a map $\phi\colon [0,T] \to \mathcal T$ that traces the boundary of $\CT$ in order.  (As we will discuss below, one may also consider a Brownian excursion measure for which the length $T$ is not {\em a priori} determined, and in the most natural way to do this, the Brownian excursion measure is an infinite measure.) Once one is given $Y_t$, one may construct a Brownian process $X_t$ with $X_0 = 0$ and 
\begin{equation}
\label{eqn::brownian_covariance}
\Cov(X_s, X_t) = \inf \left\{Y_r: r \in [s,t] \right\}.
\end{equation}
An application of the Kolmogorov-Centsov theorem implies that $X$ has a H\"older continuous modification; see, e.g.\ \cite[Section~3.4]{legall2014icm}.  The RHS of~\eqref{eqn::brownian_covariance} describes the length of the intersection of the two tree branches that begin at~$\phi(0)$ and end at~$\phi(s)$ or~$\phi(t)$.  In particular, if $\phi(s) = \phi(t)$ then $X_s = X_t$.  Therefore $X$ induces a process $Z$ defined on $\CT$ which satisfies $X_t = Z_{\phi(t)}$.

Given the $(X_t, Y_t)$ process, it is easy to draw the body of the snake in Figure~\ref{fig::browniansnake} for any fixed time $t \in [0,T]$.  To do so, for each value $b < Y_t$, one plots the point $(X_s, b)$ where~$s$ is the last time before~$t$ at which the~$Y$ process reached height~$b$.  Note also that if one takes $s'$ to be the first time after $t$ when the $Y$ process reaches $b$, then we must have $X_{s'} = X_s$. Intuitively speaking, as $Y_t$ goes down, the snake head retraces the snake body; as $Y_t$ goes up, new randomness determines the left-right fluctuations. As discussed in the captions of Figure~\ref{fig::browniansnake}, this evolution can be understood as a diffusion process on $\snake$.

We now consider a natural infinite measure on the space of excursions into $\snake$. It is the measure described informally in the caption to Figure~\ref{fig::browniansnake}.  To construct this, first we define $\excursion$ to be the natural Brownian excursion measure (see \cite[Chapter~XII, Section 4]{ry99martingales} for more detail on the construction of $\excursion$). 
Each such excursion comes with a terminal time $T$ such that $Y_0 = Y_T = 0$, $Y_t > 0$ for $t \in (0,T)$, and $Y_t = 0$ for all $t \geq T$.  We recall that the excursion measure is an infinite measure that can be constructed as follows.  Define $\excursion_\epsilon$ to be $(2\epsilon)^{-1}$ times the probability measure on one-dimensional Brownian paths started at $\epsilon$, stopped the first time they hit zero.  Note that this measure assigns mass 1/2 to the set of paths that reach $1$ before hitting zero.  The measure $\excursion$ is obtained by taking the weak the limit of the $\excursion_\epsilon$ measures as $\epsilon \to 0$ (using the topology of uniform convergence of paths, say).  Note that for each $a>0$ the $\excursion$ measure of the set of paths that reach level $a$ is exactly $(2a)^{-1}$.  Moreover, if one normalizes $\excursion$ to make it a probability on this set of paths, then one finds that the law of the path after the first time it hits $a$ is simply that of an ordinary Brownian motion stopped when it hits zero.  Now that we have defined $\excursion$, we define $\snakeexcursionmeasure$ to be a measure on excursions into $\snake$ such that the induced measure on $Y_t$ trajectories is $\excursion$, and given the $Y_t$ trajectory, the conditional law of $X_t$ is that of the Brownian process indexed by the CRT encoded by $Y_t$ (i.e., with covariance as in~\eqref{eqn::brownian_covariance}).

As we will explain in more detail just below, given a sample from $\snakeexcursionmeasure$, the tree encoded by $X_t$ is the tree of geodesics drawn from all points to a fixed {\bf root}, which is the value of $\phi$ at the point $t$ that minimizes $X_t$.  The tree $\CT$ described by $Y_t$ (the dual tree) has the law of a CRT, and $Y_t$ describes the distance in $\CT$ from the dual root (which corresponds to time $0$ or equivalently time $T$, which is the time when $Y_t$ is minimal).

Note that for any time $t$, we can define the {\em snake} to be the graph of the function from $y\in [0,Y_t]$ to $x$ that sends a point $y$ to the value of the Brownian process at the point on $\CT$ that is $y$ units along the branch in $\mathcal T$ from $\phi(0)$ to $\phi(t)$.

As in Figure~\ref{fig::browniansnake}, for each $a < 0$ we let $\snake_{>a}$ be the subspace of $\snake$ which consists of those snakes $w$ such that $w(t) > a$ for all $t \in [0,\zeta]$.  That is, $w \in \snake_{>a}$ if and only if its body lies to the right of the vertical line through $(a,0)$.  We also let $\snake_{\leq a} = \snake \setminus \snake_{>a}$.

We next proceed to remind the reader how to associate an $(X,Y)$ pair with a metric measure space structure. This will allow us to think of $\snakeexcursionmeasure$ as a measure on $\mmspace$.  Roughly speaking, the procedure described in the left side of Figure~\ref{fig::lamination_treemaking} already tells us how to obtain a sphere from the pair $(X,Y)$. The points on the sphere are the equivalence classes from the left side of Figure~\ref{fig::lamination_treemaking}. The tree described by $X$ alone (the quotient of the graph of $X$ w.r.t.\ the equivalence given by the chords under the graph) can be understood as a geodesic tree (which comes with a metric space structure), and we may construct the overall metric space as a quotient of this metric space (as defined in Section~\ref{subsec::metricsphereobservations}) w.r.t.\ the extra equivalence relations induced by $Y$.

An equivalent way to define the Brownian map is to first consider the CRT $\CT$ described by $Y$, and then define a metric and a quotient using $X$ as the second step. This is the approach usually used in the Brownian map literature (see e.g.\ \cite[Section~3.5]{legall2014icm}) and we give a quick review of that construction here. Consider the function $d^\circ$ on $[0,T]$ defined by:
\begin{equation}
\label{eqn::circ_distance_time}
d^\circ(s,t) = X_s + X_t - 2\max\left( \min_{r \in [s,t]} X_r, \min_{r \in [t,s]} X_r \right).
\end{equation}
Here, we assume without loss of generality that $s < t$ and define $[t,s] = [0,s] \cup [t,T]$.  For $a,b \in \CT$, we then set
\begin{equation}
\label{eqn::circe_distance_tree}
d_\CT^\circ(a,b) = \min\{ d^\circ(s,t) : \phi(s) = a,\ \phi(t) = b\}
\end{equation}
where $\phi \colon [0,T] \to \CT$ is the natural projection map.  Finally, for $a,b \in \CT$, we set
\begin{equation}
\label{eqn::bm_distance_inf}
d(a,b) = \inf\left\{ \sum_{j=1}^k d_\CT^\circ(a_{j-1},a_j) \right\}
\end{equation}
where the infimum is over all $k \in \N$ and $a_0=a,a_1,\ldots,a_k=b$ in $\CT$.  We get a metric space structure by quotienting by the equivalence relation $\cong$ defined by $a \cong b$ if and only if $d(a,b) = 0$ and we get a measure on the quotient space by taking the projection of Lebesgue measure on $[0,T]$.  
As mentioned in the introduction, it was shown by Le Gall and Paulin \cite{le2008scaling} (see also \cite{MR2399286}) that the resulting metric space is a.s.\ homeomorphic to $\S^2$ and that two times $a$ and $b$ are identified if and only if vertical red lines in the left side of Figure~\ref{fig::lamination_treemaking} (where $X_t$ and $Y_t$ are Brownian snake coordinates) belong to the same equivalence class as described in the left side of Figure~\ref{fig::lamination_treemaking}. Thus the topological quotient described in the left side of Figure~\ref{fig::lamination_treemaking} is in natural bijection with the metric space quotient described above.

Given a sample from $\snakeexcursionmeasure$, the corresponding sphere comes with two special points corresponding to a snake whose head is at the leftmost possible value (the root), and the origin snake (the dual root). Indeed, if we let $S$ denote the set of points on the sphere, $\nu$ the measure, $x$ the root, and $y$ the dual root, then we obtain a doubly marked metric measure space $(S,d, \nu,x,y)$ of the sort described in Section~\ref{subsec::mmsigma}.  The dual root~$y$ should be thought of as the target point of a metric exploration starting from the root~$x$.  In what follows, we will always use~$x$ to the denote the root (center point from which a metric ball will grow) and~$y$ to denote the dual root or target point of the metric exploration.

In fact, we claim that $\snakeexcursionmeasure$ induces a measure on $(\gmsspace^{2,O},\mmsigma^{2,O})$.  This measure is precisely the doubly marked grand canonical ensemble of Brownian maps: i.e., it corresponds to the measure~$\mustwo$ discussed in Section~\ref{subsec::theoremstatement}. There is a bit of an exercise involved in showing that the map from Brownian snake instances to $(\mmspace^k, \mmsigma^k)$ is measurable w.r.t.\ the appropriate $\sigma$-algebra on the space of Brownian snakes, so that~$\mustwo$ is a well-defined measure $(\gmsspace^{2,O},\mmsigma^{2,O})$. In particular, one has to check that the distance-function integrals described in Section~\ref{subsec::mmsigma} (the ones used to define the Gromov-weak topology) are in fact measurable functions of the Brownian snake; one can do this by first checking that this is true when the metric is replaced by the function $d^\circ$ discussed above, and then extending this to the approximations of $d$ in which the distance between two points is the infimum of the length taken over paths made up of finitely many segments of the geodesic tree described by the process $X$. This is a straightforward exercise, and we will not include details here.

Given a snake excursion $s$ chosen from $\snakeexcursionmeasure$, we define the snake excursion $\wh s$ so that its associated surface is the surface associated to $s$ {\em rescaled} to have total area $1$.  In other words, $\wh s$ is the snake whose corresponding head process is
\[ (\wh X_t, \wh Y_t) = (\zeta^{-1/4} X_{ \zeta t}, \zeta^{-1/2} Y_{\zeta t}).\]

Here we have scaled $t$ by a factor of $\zeta$, we have scaled $Y_t$ by a factor of $\zeta^{-1/2}$, and we have scaled $X_t$ by a factor of $\zeta^{-1/4}$.  An excursion $s$ can be represented as the pair $(\wh s, \zeta(s))$ where $\zeta(s)$ represents the length of the excursion --- or equivalently, the area of the corresponding surface.  Since a sample from the Brownian excursion measure $\excursion$ is an excursion whose length has law $c \zeta^{-3/2} d\zeta$ \cite[Chapter~XII, Section 4]{ry99martingales}, where $d\zeta$ is Lebesgue measure on $\R_+$ and $c=1/\sqrt{8\pi}$, we have the following:

\begin{proposition}
\label{prop::bm_measures}
If we interpret $\snakeexcursionmeasure$ as a measure on pairs $(\wh s, \zeta)$, then $\snakeexcursionmeasure$ can be written as $\wh \snakeexcursionmeasure \otimes c t^{-3/2} dt$, where $dt$ represents Lebesgue measure on $\R_+$, $c=1/\sqrt{8\pi}$, and $\wh \snakeexcursionmeasure$ is a probability measure on the space of excursions of unit length.
\end{proposition}

\begin{figure}[ht!]
\begin{center}
\includegraphics [width=3.5in]{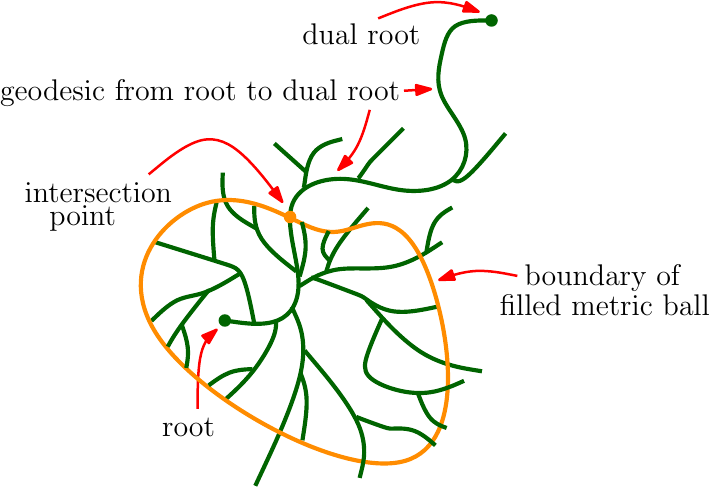}
\caption{\label{fig::geodesictree} The snake trajectory corresponds to a path that traces the boundary of a (space-filling) tree of geodesics in the doubly marked Brownian map $(S,d,\nu,x,y)$.  The figure illustrates several branches of the geodesic tree (the tree itself is space-filling) and along with the outer boundary (as viewed from the dual root) of a radius-$r$ metric ball centered at the root $x$ (i.e., $\partial \fb{x}{r}$). From a generic point $z \in S$, there is a unique path $\gamma$ in the dual tree back to the dual root $y$.  The distances $d(\gamma(t),x)$ vary in $t$; this variation encodes the shape of the body of the Brownian snake $(X,Y)$ associated with $(S,d,\nu,x,y)$.  The total quadratic variation of $t \mapsto d(\gamma(t),x)$ encodes the height of the snake's head (i.e., the value of $Y_s$ if $z$ corresponds to $(X_s,Y_s)$).  During the snake trajectory (as the snake itself changes) the {\em first} and {\em last} times that the horizontal coordinate $X$ of the snake's head reaches $a = \inf\{X_t \} + r$ correspond to the intersection point (shown in orange) of the (a.s.\ unique) dual-root-to-root geodesic and $\partial \fb{x}{r}$.  Intuitively, as one traces the boundary of the space-filling geodesic tree (beginning and ending at the dual root $y$), the orange dot is the first and last point that the path visits within the closed orange disk.}
\end{center}
\end{figure}

\subsection{Brownian maps, disks, and L\'evy nets}
\label{subsec::mapsdisksnets}

The purpose of this subsection is to prove that the unembedded metric net of the doubly marked Brownian map has the law of a $3/2$-stable L\'evy net.  We will refer to the (countably many) components of the complement of the metric net as ``bubbles'' and will describe a one-to-one correspondence between these bubbles and the ``holes'' in the corresponding L\'evy net.  We will also introduce here the measure $\mudonel$ on marked random disks, which give the law of the complement of a filled metric ball in the Brownian map.  (We will later introduce the measure $\mudl$, which gives the law of the complementary components of a metric exploration in the Brownian map, in a more general framework.)  The two jump processes in Figure~\ref{fig::levycb} correspond to different orders in which one might explore these holes.   The first explores holes in a ``depth-first'' order --- i.e., the order in which they are encountered by a path that traces the boundary of the geodesic tree; the second explores holes in a ``breadth-first'' order --- i.e., in order of their distance from a root vertex.  We will see what these two orderings look like within the context of the Brownian map, as constructed from a Brownian snake excursion.

In order to begin understanding the unembedded metric net of the Brownian map, we need a way to make sense of the boundary length measure on a metric ball within the Brownian map.  Observe that for any real number $a < 0$, the snake diffusion process has the property that if the snake lies in~$\snake_{\leq a}$ at time~$t$, then its distance (in the snake space metric as defined in~\eqref{eqn::snake_metric}) from the boundary of $\snake_{>a}$ is given by $Y_t - Y_s$, where $s$ is supremum of the set of times before $t$ at which the snake was in~$\snake_{>a}$; see Figure~\ref{fig::browniansnake} for an illustration.  We in particular emphasize that the distance at time $t$ to the boundary of $\snake_{>a}$ does not depend on $X_t$.  This distance clearly evolves as a Brownian motion until the next time it reaches zero.   Let us define $i_a(t)$ to be the total time before $t$ that the snake process spends inside $\snake_{>a}$, and $o_a(t) = t - i_a(t)$ the total amount of time before $t$ that the snake process spends in~$\snake_{\leq a}$.

We claim that when we parameterize time according to $o_a$ time, i.e.\ by the right-continuous inverse $o_a^{-1}(t) = \inf\{r \geq 0 : o_a(r) > t\}$ of $o_a$, this process is a non-negative, reflected Brownian motion, and hence has a well-defined notion of local time~$\ell_a$ for any given value of $a$ (see \cite[Chapter~VI]{ry99martingales} for more on the construction of Brownian local time).  To see this, it suffices to show that the process $Y_{o_a^{-1}(t)}$ is non-negative, evolves as a Brownian motion in the intervals of time in which it is positive, and is instantaneously reflecting at $0$.  The first two properties are true by the construction as we have explained above, which leaves us to show that $Y_{o_a^{-1}(t)}$ is instantaneously reflecting at $0$.  To prove this, it suffices to show that the Lebesgue measure of the set of times $t$ that $(X_t,Y_t)$ is in the boundary of $\snake_{> a}$ is a.s.\ equal to $0$ which in turn follows from the stronger statement that the Lebesgue measure of the set of times $t$ that $X_t = a$ is a.s.\ equal to $0$.  Fix $\epsilon > 0$ and let $[\tau_\epsilon,\sigma_\epsilon]$ be the interval of time corresponding to the longest excursion that $Y$ makes above $\epsilon$ (breaking ties by taking the one which happens first).  Given $[\tau_\epsilon,\sigma_\epsilon]$ and $X_{\tau_\epsilon} = X_{\sigma_\epsilon}$, we have that $(X_t,Y_t)$ in $[\tau_\epsilon,\sigma_\epsilon]$ is a Brownian snake starting from $X_{\tau_\epsilon} = X_{\sigma_\epsilon}$.  For a given value of $X_{\tau_\epsilon} = X_{\sigma_\epsilon}$, there are at most countably many values of $b$ so that the set of times $t \in [\tau_\epsilon,\sigma_\epsilon]$ such that $X_t = b$ has positive Lebesgue measure and this set depends on $X_{\tau_\epsilon} = X_{\sigma_\epsilon}$ by translation as $X_{\tau_\epsilon} = X_{\sigma_\epsilon}$ gives the initial value of the tree-indexed Brownian motion.  Thus since $X_{\tau_\epsilon} = X_{\sigma_\epsilon}$ has a density with respect to Lebesgue measure, it follows that the Lebesgue measure of the set of times $t \in [\tau_\epsilon,\sigma_\epsilon]$ so that $X_t = a$ is a.s.\ equal to $0$.  This completes the proof of the claim since $\epsilon > 0$ was arbitrary.

We recall that the excursions that a reflected Brownian motion makes from $0$ can be described by a Poisson point process indexed by local time (see, e.g., \cite[Chapter~VI]{ry99martingales}).  We also note that in each such excursion made by $Y$, the initial (and terminal) value of $X$ is given by $a$.

We next claim that a sample from $\snakeexcursionmeasure$ may be obtained in two steps:
\begin{enumerate}
\item First sample the behavior of the snake restricted to $\snake_{>a}$, parameterized according to $i_a$ time, i.e.\ by the right-continuous inverse $i_a^{-1}(t) = \inf\{r \geq 0 : i_a(r) > t\}$ of $i_a$.  That is, we sample the process $(X_{i_a^{-1}(t)}, Y_{i_a^{-1}(t)})$ from its marginal law.

We claim that the process $(X_{i_a^{-1}(t)},Y_{i_a^{-1}(t)})$ determines the local time $\ell_a$.  To see this, let $Y^1_t$ be the difference between $Y_t$ and the height of the ancestor snake head at time $t$ (as in Figure~\ref{fig::browniansnake}), and define $Y^2_t = Y_t - Y^1_t$ so that $Y_t = Y^1_t + Y^2_t$.  As explained just above, we know that $Y_t^1$ evolves as a reflected Brownian motion when we parameterize by $o_a$ time.  Thus it follows that $Y^1_t - \ell_a(t)$ is a continuous martingale (see, e.g., the It\^o-Tanaka formula) when parameterized by $o_a$ time, hence it is a continuous martingale itself.  Consequently, $Y^2_t + \ell_a(t)$ is a continuous martingale.  Moreover, $Y_t^2 + \ell_a(t)$ is a continuous martingale when parameterized by $i_a$ time and thus $Y_{i_a^{-1}(t)}^2$ is a continuous supermartingale.  Hence, one can use the Doob-Meyer decomposition to recover this local time from the process $Y_{i_a^{-1}(t)}^2$.
\item Then, conditioned on the total amount of local time $\ell_a(T)$ sample the set of excursions into $\snake_{\leq a}$ using a Poisson point process on the product of Lebesgue measure on $[0,\ell_a(T)]$ (an interval which is now known, even though $T$ is not itself yet determined) and~$\snakeexcursionmeasure$.  Note that each excursion is translated so that it is ``rooted'' at some point along the vertical line through $(a,0)$, instead of at $(0,0)$.
\end{enumerate}

In what follows, we will write $\ell_a(T)$ for the total amount of local time and condition on its value as just above.  We emphasize that $T$ is a random variable but when we condition on $\ell_a(T)$ we are not conditioning on the value of $T$ unless we explicitly say otherwise.  We extend $(\ell_a(T) : a < 0)$ to a process by taking it so that $s \mapsto \ell_{-s}(T)$ for $s > 0$ is {\cadlag}.

For the Brownian map instance $(S,d,\nu,x,y)$ encoded by $(X,Y)$ and $a < 0$, we define the {\bf boundary length} of $\partial \fb{x}{d(x,y)+a}$ to be the value of $\ell_a(T)$.

We note that the process $\ell_a(t)$ described and constructed just above is a special case of the so-called \emph{exit measure} associated with the Brownian snake.  See, e.g., \cite{legallspatialbranchingbook} for more on exit measures.

From this discussion, the following is easy to derive the following proposition.

\begin{proposition}
\label{prop::brownianmapCSBP}
Suppose that $(X,Y)$ is sampled from $\snakeexcursionmeasure$ and $\ell_a(T)$ is as above.  Then $(\ell_a(T) : a< 0)$ follows the excursion measure of a $3/2$-stable CSBP.
\end{proposition}
\begin{proof}
The proof is nearly the same as the proof of Proposition~\ref{prop::Zisalphastable}.  One has only to verify that the process satisfies the hypotheses Proposition~\ref{prop::strongstablecsbp}.  Again, the scaling factor is obvious (one may rescale time by a factor of $C^2$, the $Y_t$ process values by a factor of $C$ and the $X_t$ process values by a factor of $C^{1/2}$); and the value of the $\ell_a(T)$ process then scales by $C$ and its time to completion scales by $C^{1/2}$, suggesting that the scaling hypothesis of Proposition~\ref{prop::Zisalphastable} is satisfied with $\alpha -1 = 1/2$, so that $\alpha = 3/2$.  The CSBP property~\eqref{eqn::CSBPproperty} is also immediate from the construction.
\end{proof}

\begin{figure}[ht!]
\begin{center}
\includegraphics [width=5.5in]{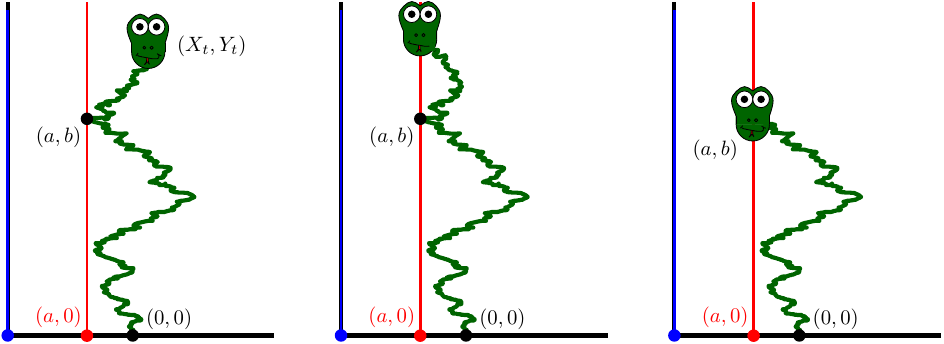}
\caption {\label{fig::browniansnakedisc} {\bf Left:} the snake shown in Figure~\ref{fig::browniansnake} with $a$ shifted to the smallest value for which the red line intersects the snake, and the point of this first intersection marked as $(a,b)$.  For almost all times $t$, the snake will have a unique minimum point of this kind; a.s., only countably many pairs $(a,b)$ arise as minima for any snake in the trajectory.  We let $\mathcal B_{(a,b)}$ denote the collection of all snakes in the trajectory with leftmost point at $(a,b)$, and the head {\em not} at $(a,b)$ itself; e.g., the left snake shown belongs to $\mathcal B_{(a,b)}$.  The set $\mathcal B_{(a,b)}$ represents a ``bubble'' of the corresponding doubly marked Brownian map i.e., an open component of the complement of the metric net between the root and the dual root.  {\bf Middle:} A snake on the bubble boundary $\partial \mathcal B_{(a,b)}$.  {\bf Right:} The ``bubble root'' of $\mathcal B_{(a,b)}$.  (Note: not every snake whose head lies left of its body is a bubble root; there are a.s.\ only countably many such points in the trajectory, one per bubble.)}
\end{center}
\end{figure}

\begin{proposition}
\label{prop::jumpsoneone}
Suppose that $(X,Y)$ is sampled from $\snakeexcursionmeasure$ and $\ell_a(T)$ is as above. The jumps in $\ell_a(T)$ are in one-to-one correspondence with the bubbles of the metric net from the root to the dual root of the Brownian map.  If one keeps track of the location along the boundary (using boundary length) at which each bubble root (see Figure~\ref{fig::browniansnakedisc}) occurs together with the total boundary length process, one obtains an object with the law of the process $Z_s$ as in Definition~\ref{def::levycb} together with the attachment points of Definition~\ref{def::levy_attachment_points} (as shown in Figure~\ref{fig::levycb}).  In particular, conditioned on the process $\ell_a(T)$, the attachment points are independent random variables with law associated with a jump occurring for a given value of $a$ is that of a uniform random variable in $[0,\ell_{a-}(T)]$.
\end{proposition}
\begin{proof}
Recall the two step sampling procedure from the measure $\snakeexcursionmeasure$ described above.  Namely, given the process $(X_{i_a^{-1}(t)},Y_{i_a^{-1}(t)})$ (i.e., the Brownian snake growth within the set $\snake_{>a}$), the conditional law of the process in $\snake_{\leq a}$ is given by a Poisson point process $\Lambda = \{ (u_i,(X^i,Y^i))\}$ with intensity measure given by the product of Lebesgue measure on $[0,\ell_a(T)]$ and $\snakeexcursionmeasure$.  The $u_i$ coordinate gives the location of where the excursion is rooted on the ball boundary, as measured relative to the place on the ball boundary visited by the unique geodesic connecting the root and dual root.  This in particular implies that the following is true.  For each $a < 0$, the law of the snake process is invariant under the operation of replacing the Poisson point process of excursions $\Lambda$ that it makes into $\snake_{\leq a}$ with the process $\{ (u_i + U,(X^i,Y^i))\}$ where $U$ is uniform in $[0,\ell_a(T)]$ independently of everything else and we consider the $u_i + U$ modulo $\ell_a(T)$.  This operation has a simple geometric interpretation.  Namely, it corresponds to the operation of cutting out the filled metric ball of radius $a-\inf\{X\}$ centered at the root of the geodesic tree, then ``rotating it'' by $U$ units of boundary length, and then gluing back together with its complement according to boundary length.  (Recall also the discussion just after Proposition~\ref{prop::recover_sphere}.)

We will now extend the above observations to the setting of stopping times.  More precisely, for each $r > 0$ we let $\CF_r$ be the $\sigma$-algebra generated by the process $t \mapsto (X_{i_{-r}^{-1}(t)},Y_{i_{-r}^{-1}(t)})$.  Then $\CF_r$ is non-decreasing in $r$.  Let $\tau$ be a stopping time for $\CF_r$.  For each~$n$, let~$\tau_n$ be the smallest element of $2^{-n} \Z$ which is at least as large as~$\tau$.  Then $\tau_n$ decreases to~$\tau$ as $n \to \infty$.  By the two step construction of $\Pi$ described above, the conditional law of the snake process given $\CF_{\tau_n}$ within the set $\snake_{\leq -\tau_n}$ is again given by a Poisson point process with intensity measure given by the product of Lebesgue measure on $\ell_{-\tau_n}(T)$ and $\Pi$.  Taking a limit as $n \to \infty$ and using the backward martingale convergence theorem and the continuity of the aforementioned conditional law (recall that $\ell_a(T)$ is left-continuous in $a$), we have that the conditional law of the snake process given $\CF_\tau$ takes exactly the same form.  In particular, it is invariant under the operation of adding to the first coordinate in the Poisson point process an independent random variable which is uniform in $\ell_{-\tau}(T)$ and then working modulo $\ell_{-\tau}(T)$.

The result thus follows by applying the previous paragraph to stopping times which correspond to bubble root times.
\end{proof}

We will now begin to describe the measure $\mudonel$ on random marked disks, which is one of the key actors in what follows.  The key ideas to understanding $\mudonel$ through the perspective of the Brownian snake are illustrated in the caption of Figure~\ref{fig::browniansnake}.  (We will later describe the measure $\mudl$ on random unmarked disks in a more general framework in the process of proving Theorem~\ref{thm::levynetbasedcharacterization} below.  We also remark that a snake-based approach to $\mudl$ is carried out in \cite{alg}.)

Fix $r > 0$ and suppose that $(S,d,\nu,x,y)$ is sampled from $\mustwo$ conditioned on $d(x,y) >r$.  We define $\mudonel$ to be the conditional law of $S \setminus \fb{x}{r}$, viewed as a metric measure space equipped with the interior-internal metric and marked by $y$, given that its boundary length is equal to $L$ as defined just above Proposition~\ref{prop::brownianmapCSBP}.

Write $\CC_a:=\ell_a(T)$ and recall from Proposition~\ref{prop::brownianmapCSBP} that $(\CC_a: a< 0)$ follows the excursion measure of a $3/2$-stable CSBP, indexed in the negative direction. Now suppose the measure on excursions of this form is {\em weighted} by the excursion length --- call it $A=A(\CC_{\cdot})$ --- and that given $(\CC_a)$  the quantity $r$ is chosen uniformly in $[0,A]$ so that $-A + r$ is uniform on the interval $[-A,0]$.  Then we can now think of the triple $\bigl((\CC_a), A, r\bigr)$ as describing a CSBP excursion {\em decorated} by a distinguished time during its length. Set $L = \CC_{-A+r}$.  In this framework, we can then break the excursion into two pieces, corresponding to time before and after $-A+r$ --- with durations $r$ and $A-r$.  We interpret each of these pieces as being defined modulo a horizontal translation of its graph (or equivalently one can imagine that one translates each piece so that it starts or ends at time $0$). Interpreted this way, we now claim that {\em given} $L$, the two pieces of the excursion (before or after the distinguished time) are independent of each other, and in fact we have all three of the following:

\begin{enumerate}
\item {\bf Weighted excursion Markov property:} For the excursion measure {\em weighted}  by excursion length $A$ and decorated as above, we have that given $L =\CC_{-A+r}$ the two sides of the excursion are conditionally independent and are given by forward and reverse CSBPs started at $\CC_{-A+r}$ and stopped upon hitting $0$.

\item {\bf Decreasing time Markov property:} For the unweighted excursion measure, we have that for any fixed $a$, on the event the event $A>a$ the two sides of the excursion (before and after time $a$) are conditionally independent given $L = \CC_{-a}$.

\item {\bf Increasing time Markov property:} For the unweighted excursion measure, we have that for any fixed $r$, on the event that $A>r$, the two sides of the excursion (before and after $-A+r$) are conditonally independent given $L = \CC_{-A+r}$.
\end{enumerate}

The decreasing time Markov property is immediate from the definition of the unweighted excursion measure.  To get the first property, note that for any constant $C$, if we restrict the length-weighted-decorated measure on triples $\bigl((\CC_a), A, r\bigr)$ to the event $A>C$ and $r \in [0,C]$, then on this event the marginal law of $(\CC_a)$ (ignoring $r$) on the time-interval $-a \in [C,A]$ is the same (up to constant factor) as in the unweighted case.  It in particular has the property that given the restriction of $\CC_a$ to $-a \in [0,C]$ the conditional law of the remainder of $\CC_a$ is simply that of a CSBP started at $\CC_a$ and stopped when it hits zero.  Since this holds for any $C$, the first property readily follows, and a similar argument then shows that the first property above implies the third property.

In the proof of Proposition~\ref{prop::metric_net_law} below, we will derive an extension of this equivalence to a setting involving the metric net. In this setting, the boundary length process (for the filled metric ball centered at $x$, as viewed from $y$) is a CSBP excursion, and we will want to say that for a given radius $r$ the inside and the outside of the ball are conditionally independent given the boundary length.  The proof below will be a bit more complicated than the observation above, because we want conditional independence of two random metric spaces (a filled ball and its complement) and defining these metric spaces requires information than just what is encoded in the boundary length process.

\begin{proposition}
\label{prop::metric_net_law}
Fix $r > 0$.  Conditionally on $d(x,y) > r$, we have that $\fb{x}{r}$ and $S \setminus \fb{x}{r}$ (each viewed as elements of $\mmspace^{1,O}$) are conditionally independent given the boundary length of $\partial \fb{x}{r}$ and the conditional law of $S \setminus \fb{x}{r}$ does not depend on $r$.
\end{proposition}
\begin{proof}
We begin by fixing a (deterministic) value of $a < 0$ and let $r = d(x,y)+a$.  Since $d(x,y)$ is random, we emphasize that $r$ is random at this point.  Let $d_r$ be the interior-internal metric on $S \setminus \fb{x}{r}$, which we recall is defined by setting the distance between any two points to be the infimum of lengths of paths connecting the two points which stay in the interior of $S \setminus \fb{x}{r}$.

Let $(X,Y)$ be the Brownian snake process which generates $(S,d,\nu,x,y)$ and let $(X^a,Y^a)$ be the process obtained by truncating the excursions that $(X,Y)$ makes into $\snake_{\leq a}$, as discussed above Proposition~\ref{prop::brownianmapCSBP} (see also Figure~\ref{fig::browniansnake}).  We claim that $(X^a,Y^a)$ determines $d_r$ using the same procedure to construct $d$ from $(X,Y)$; recall~\eqref{eqn::bm_distance_inf}.  We first note that all of the geodesics in $S \setminus \fb{x}{r}$ to $\partial \fb{x}{r}$ are determined by $(X^a,Y^a)$ because such a geodesic is part of a geodesic to $x$.  Fix $z, w \in S \setminus \fb{x}{r}$ and let $\eta$ be a $d_r$-geodesic from $z$ to $w$.  Let $\epsilon > 0$.  By the construction of $d$, it follows that there exists $n \in \N$ and segments $\eta_1,\ldots,\eta_n$ of geodesics to the root $x$ so that the concatenation of $\eta_1,\ldots,\eta_n$ connects $z$ to $w$ and has length at most $\epsilon$ plus the length of $\eta$.  Applying the procedure~\eqref{eqn::bm_distance_inf} to $(X^a,Y^a)$, we see that the successive endpoints of $\eta_1,\ldots,\eta_n$ are identified.  Since $\epsilon > 0$ was arbitrary, we see that $d_r(z,w)$ is determined by $(X^a,Y^a)$.  We similarly have that the interior-internal metric on $\fb{x}{r}$ is determined by the excursions that $(X,Y)$ makes into $\snake_{\leq a}$.

Summarizing, we thus have that the metric measure spaces (with their interior-internal metric) $S \setminus \fb{x}{a+d(x,y)}$ (marked by $y$) and $\fb{x}{a+d(x,y)}$ (marked by $x$) are respectively determined by $(X^a,Y^a)$ and the excursions the snake process makes in $\snake_{\leq a}$.  The discussion just above Proposition~\ref{prop::brownianmapCSBP} therefore implies that the two spaces are conditionally independent given the boundary length~$L$.

Recall from the discussion just before the statement of Proposition~\ref{prop::brownianmapCSBP} that the conditional law of the excursions that $(X,Y)$ makes into $\snake_{\leq a}$ given the boundary length $L = \ell_a(T)$ is given by a Poisson point process with intensity measure given by the product of Lebesgue measure on $[0,\ell_a(T)]$ and $\Pi$.  Recall also that the value of $d(x,y)$ is equal to $-1$ times the minimum value attained by the $x$-coordinate of these excursions.  In particular, given $d(x,y)$ the excursions into $\snake_{\leq a}$ are conditionally independent of $(X^a,Y^a)$ given the boundary length $L$.  Note that given $d(x,y)$, there will be one excursion whose $x$-coordinate has minimum value $-d(x,y)$ and the other excursions have a Poisson law but with minimal $x$-coordinate larger than $-d(x,y)$.

The analysis above applies equally well in the setting of the remark before the beginning of the proof -- where the ``CSBP excursion measure'' (or the corresponding doubly marked Brownian map measure) is weighted by its length (which corresponds to $d(x,y)$ in the corresponding doubly marked Brownian map) and $a$ is uniformly chosen from $[-d(x,y), 0]$ (instead of being deterministic) and again  $r = d(x,y) + a$. That is, even in this setting, we still have conditional independence of $(\fb{x}{r},x)$ and $(S \setminus \fb{x}{r},y)$ given $L$. In particular, in this setting if we condition on $r \in [r_0, r_0 + \epsilon]$ (i.e., we restrict the infinite measure space to the finite-measure subspace on which this is true) we learn nothing about $(S \setminus \fb{x}{r},y)$ from knowing  $(\fb{x}{r},x)$ beyond the value of $L$, and by the scaling symmetry of the overall construction, the value $L$ affects the conditional law of $(S \setminus \fb{x}{r},y)$ only by a scaling (multiplying distances by $L^{1/2}$, measure by $L^2$).  But once we restrict the measure to the set on which $r \in [r_0, r_0 + \epsilon]$, the marginal law (of the doubly marked Brownian surface) is no longer weighted by the length of the interval $[0,d(x,y)]$, but rather by the length of $[0,d(x,y)] \cap  [r_0, r_0 + \epsilon]$, which is simply equal to $\epsilon$ when $d(x,y) > r_0 + \epsilon$ and $0$ if $d(x,y) < r_0$. (The proportion of this measure coming from the case $d(x,y) \in [r_0, r_0+\epsilon]$ tends to zero as $\epsilon \to 0$.) Furthermore, the amount  the CSBP boundary length process changes during $[r_0, r_0+\epsilon]$ tends to zero in probability as $\epsilon \to 0$, and the conditional law of $(S \setminus \fb{x}{r},y)$ is continuous as function of $L$ (w.r.t.\ any natural topology -- e.g., we can use the weak topology induced by the snake space metric).  Thus we can take the $\epsilon \to 0$ limit and conclude that the conditional law of $(S \setminus \fb{x}{r},y)$ given  $(\fb{x}{r},x)$ depends only on $L$.

\end{proof}

\begin{remark}
\label{rem::brownian_map_boundary_length}
In the case that $\alpha = 3/2$, we now have that up to time parameterization, both the process $\ell$ defined for Brownian maps and the process $Z$ defined for the L\'evy net can be understood as descriptions of the natural boundary length measure $L_{r}$ discussed in Section~\ref{sec::introduction}.
\end{remark}

By combining Propositions~\ref{prop::brownianmapCSBP} and~\ref{prop::metric_net_law} together with Theorem~\ref{thm::levynetrecover} we see that a.e.\ instance of the Brownian map \emph{determines} a $3/2$-stable L\'evy net instance.  We will now show that the unembedded metric net of the Brownian map is equal to this instance of the $3/2$-stable L\'evy net as $\treeequivspace$-valued random variables.

\begin{proposition}
\label{prop:metric_net_equivalent}
The unembedded metric net of a sample $(S,d,\nu,x,y)$ from $\mustwo$ has the law of a $3/2$-stable L\'evy net. In this correspondence, the $3/2$-stable CSBP excursion $\ell_a(T)$ described above for a sample from $\mustwo$ agrees with the $3/2$-stable CSBP excursion $Z_s$ of Definition~\ref{def::levycb} (recall also Figure~\ref{fig::levycb}), up to an affine transformation relating $a$ and $s$.
\end{proposition}

\newcommand{\LN}{\mathrm{LN}}
\newcommand{\BM}{\mathrm{BM}}

The rest of this section is aimed at proving Proposition~\ref{prop:metric_net_equivalent}.  Throughout, we will let $(S,d,\nu,x,y)$ be a sample from $\mustwo$, $(X,Y)$ be the corresponding instance of the Brownian snake, and $(X^\LN,Y^\LN)$ the instance of the $3/2$-stable L\'evy net which is determined by $(X,Y)$.  Here, $(X^\LN,Y^\LN)$ are as in Definition~\ref{def::levynet} so that $X^\LN$ is the time-reversal of a $3/2$-stable L\'evy excursion with only upward jumps and $Y^\LN$ is the associated height process.  We let $K^\LN$ be the compact subset of $[0,1]^2$ which describes the equivalence relation in the L\'evy net instance $(X^\LN,Y^\LN)$ so that $(Y^\LN,K^\LN)$ takes values in $\treeequivspace$.  We also let $(Y^\BM,K^\BM)$ be the $\treeequivspace$-valued random variable determined by $(S,d,\nu,x,y)$ as in Proposition~\ref{prop:equiv_measurable}.  We will proceed by first showing (Lemma~\ref{lem:geodesic_tree_is_the_same}) that the leftmost geodesic tree in the metric net of $(S,d,\nu,x,y)$ is the same as in the L\'evy net (i.e., that $Y^\BM$ is equal to $Y^\LN$ up to monotone reparameterization).  We will then show that the associated equivalence relation $K^\BM$ in the Brownian map has \emph{at least as many} identifications as in the L\'evy net (Lemma~\ref{lem:levy_net_jump_equiv_dense}).  That is, we will show that $K^\BM \supseteq K^\LN$.  We will then complete the proof by showing that $K^\BM \subseteq K^\LN$ (Lemma~\ref{lem:net_equiv_implies_levy_equiv}).

For each $a \leq 0$ and $r > 0$, we let $\tau_a^r$ be the first time $t \geq 0$ that $Y$ has accumulated $r$ units of local time on the boundary of $\snake_{>a}$.  For each $b \leq a$, we let $\ell_{a,b}^r(T)$ be the amount of local time that $Y|_{[0,\tau_a^r]}$ accumulates on the boundary of $\snake_{>b}$.  The same argument as in the proof of Proposition~\ref{prop::brownianmapCSBP} implies that $\ell_{a,b}^r(T)$ evolves as a $3/2$-stable CSBP as $b$ decreases.  Assume that $\tau_a^r < \infty$.  For each $b \leq a$, we let $\tau_{a,b}^r$ be the last time that $(X,Y)|_{[0,\tau_a^r]}$ visits $\partial \snake_{> b}$.  Let $\rho \colon [0,T] \to (S,d,\nu,x,y)$ be the projection map from $[0,T]$ associated with the Brownian snake construction of $(S,d,\nu,x,y)$.  Then we observe that $b \mapsto \rho(\tau_{a,b}^r)$ for $b \leq a$ gives the leftmost geodesic from $\rho(\tau_a^r)$ to $x$.  Indeed, fix $b \leq a$.  Then we have by definition that $\rho(\tau_{a,b}^r) \in \partial \fb{x}{d(x,y)+b}$.  We also have by definition that $\tau_{a,b}^r \leq \tau_a^r$ and that $X|_{[\tau_{a,b}^r,\tau_a^r]}$ is at least $X_{\tau_{a,b}^r} = b$.  This implies that the point in the real tree encoded by $X$ (the geodesic tree in $(S,d,\nu,x,y)$ rooted at $x$) corresponding to $\tau_{a,b}^r$ is an ancestor of $\tau_a^r$.  Equivalently, $\rho(\tau_{a,b}^r)$ lies on the leftmost geodesic $\eta$ from $\rho(\tau_a^r)$ to $x$.  This proves the claim since $b \leq a$ was arbitrary.  We have obtained from this discussion that the boundary length on the counterclockwise arc of $\partial \fb{x}{d(x,y)+b}$ from $\eta$ to the unique geodesic from $y$ to $x$ evolves as a $3/2$-stable CSBP.  We can similarly consider the amount of local time that $Y|_{[\tau_a^r,\infty)}$ spends on $\partial \snake_{> b}$.  This process describes the boundary length between on the clockwise arc of $\partial \fb{x}{d(x,y)+b}$ from $\eta$ to the unique geodesic from $y$ to $x$.  This process also evolves as an $3/2$-stable CSBP which is independent of $\ell_{a,b}^r(T)$.  By generalizing these considerations, we obtain the following.

\begin{lemma}
\label{lem:geodesic_boundary_length_csbp}
Fix $a \leq 0$ and $0 < r_1 < \cdots < r_n$.  Suppose that $\tau_{a}^{r_n} < \infty$.  For each $1 \leq j \leq n$, we let $\eta_j$ be the leftmost geodesic in $(S,d,\nu,x,y)$ from $\rho(\tau_a^{r_j})$ to $x$ where we take $\eta_0 = \eta_n$ to be the unique geodesic from $y$ to $x$ starting from when it first hits $\partial \fb{x}{d(x,y)+a}$.  For each $b \leq a$, let $L_{a,b}^{j}$ be the boundary length on the counterclockwise segment of $\partial \fb{x}{d(x,y)+b}$ from $\eta_j$ to $\eta_{j+1}$.  Given the initial values $L_{a,a}^j$, the processes $L_{a,b}^j$ evolve as $b \leq a$ decreases as independent $3/2$-stable CSBPs.
\end{lemma}
\begin{proof}
The boundary length between $\eta_j$ and $\eta_{j+1}$ on $\partial \fb{x}{d(x,y)+b}$ is given by the amount of local time spent by $Y|_{[\tau_{a}^{r_j},\tau_a^{r_{j+1}}]}$ on $\partial \snake_{> b}$.  The same considerations as above therefore imply that $L_{a,b}^j$ evolves as $b \leq a$ decreases as a $3/2$-stable CSBP independently of $L_{a,b}^i$ for $i \neq j$, given the initial values of all of the processes $L_{a,b}^i$.
\end{proof}

We will now deduce from Lemma~\ref{lem:geodesic_boundary_length_csbp} and the inside/outside independence of filled metric balls established in Proposition~\ref{prop::metric_net_law} that the leftmost geodesic tree in the metric net of $(S,d,\nu,x,y)$ is the same as in the L\'evy net.

\begin{lemma}
\label{lem:geodesic_tree_is_the_same}
Up to monotone reparameterization, we have that $Y^\BM$ is equal to $Y^\LN$.
\end{lemma}
\begin{proof}
Fix $a \leq 0$ and $r > 0$.  On the event that the boundary length of $\partial \fb{x}{d(x,y)+a}$ is at least $r$, we let $\eta$ be the leftmost geodesic starting from the point on $\partial \fb{x}{d(x,y)+a}$ whose counterclockwise boundary length from where the unique geodesic from $y$ to $x$ passes through $\partial \fb{x}{d(x,y)+a}$ is equal to $r$.  The same argument used to prove Theorem~\ref{thm::levynetrecover} implies that for each $b \leq a$ the clockwise and counterclockwise boundary lengths on $\partial \fb{x}{d(x,y)+b}$ from $\eta$ to the unique geodesic from $y$ to $x$ are a.s.\ the same as the corresponding boundary lengths for the corresponding geodesic in the L\'evy net.  Since the time at which one of these boundary length processes first hits $0$ gives $1/2$ of the distance in the trees encoded by $Y^\LN$ and $Y^\BM$ between the starting point of $\eta$ and where the unique geodesic from $y$ to $x$ passes through $\partial \fb{x}{d(x,y)+a}$, we see that these distances are the same in the trees encoded by $Y^\LN$ and $Y^\BM$.  Off a set of measure $0$, this in fact holds simultaneously for any fixed countable collection of $a$ and $r$ values.  By allowing $a$ to range in $\Q_- = \Q \cap \R_-$ and $r$ to range in $\Q_+$, we see that the whole geodesic tree in the metric net of the Brownian map agrees with the geodesic tree in the L\'evy net.  That is, $Y^\LN$ and $Y^\BM$ agree up to monotone reparameterization.
\end{proof}

By Lemma~\ref{lem:geodesic_tree_is_the_same}, we can assume (after applying a monotone reparameterization) that $Y^\LN = Y^\BM$.  We now aim to show that $K^\BM = K^\LN$.  We will proceed by first showing that $K^\BM \supseteq K^\LN$ by arguing that $K^\BM$ contains a dense subset of $(s,t)$ pairs in $K^\LN$ and then we will show that $K^\BM \subseteq K^\LN$.  Let $T^\LN$ be such that $[0,T^\LN]$ is the interval on which $X^\LN$ is defined.

\begin{lemma}
\label{lem:levy_net_jump_equiv_dense}
We have that $K^\BM \supseteq K^\LN$.
\end{lemma}
\begin{proof}
Suppose that $(s,t) \in K^\LN$ with $s < t$.  If $t$ is a jump time of $X^\LN$, then Proposition~\ref{prop::jumpsoneone} implies that $(s,t) \in K^\BM$.  If $s$ and $t$ are equivalent in the tree encoded by $Y^\LN = Y^\BM$, then we also have that $(s,t) \in K^\LN$.  Suppose that $t T^\LN$ is not a jump time of $X^\LN$ and that $s$, $t$ are not equivalent in the tree encoded by $Y^\LN$.  By the definition of $K^\LN$, we have that the horizontal chord connecting $(s T^\LN,X_{s T^\LN}^\LN)$ and $(t T^\LN,X_{t T^\LN}^\LN)$ lies below the graph of $X^\LN|_{[sT^\LN,t T^\LN]}$.  Then there exists a sequence of times $t_k$ such that $X^\LN$ has a downward jump at time $t_k T^\LN$ such that if $s_k < t_k$ is such that $X_{s_k T^\LN}^\LN = X_{t_k T^\LN}^\LN$ and the horizontal chord from $(s_k T^\LN,X_{s_k T^\LN}^\LN)$ to $(t_k T^\LN,X_{t_k T^\LN}^\LN)$ lies below the graph of $X^\LN|_{[s_k T^\LN,t_k T^\LN]}$ and $s_k \downarrow s$ and $t_k \uparrow t$ as $k \to \infty$.  Then $(s_k,t_k) \in K^\BM$ by what we explained at the beginning of the proof.  Since $K^\BM$ is closed, it follows that $(s,t) \in K^\BM$.
\end{proof}

\begin{lemma}
\label{lem:net_equiv_implies_levy_equiv}
We have that $K^\BM \subseteq K^\LN$.  In particular, $K^\BM = K^\LN$.
\end{lemma}
\begin{proof}
Suppose that $(s,t) \in K^\BM$ and $s < t$.  If $t T^\LN$ is a jump time of $X^\LN$, then as we explained in the proof of Lemma~\ref{lem:levy_net_jump_equiv_dense} we have that $(s,t) \in K^\LN$.  We may therefore assume that $t T^\LN$ is not a jump time of $X^\LN$.  Then there exists $a < 0$ so that $Y_s^\BM = Y_t^\BM = d(x,y)+a$.  We recall from the breadth-first construction (Definition~\ref{def::levynet_bubble_gluing}) that the boundary length measure in the L\'evy net for the metric ball centered at the root of the geodesic tree is defined for all radii simultaneously.  We consider two possibilities.  Either one of the boundary lengths along the clockwise or counterclockwise segments of the metric ball boundary in the L\'evy net from the point corresponding to $s$ to the point corresponding to $t$ is equal to zero or both boundary lengths are positive.  If one of the boundary lengths is equal to zero, then it follows that $(s,t) \in K^\LN$ by the breadth-first construction of the L\'evy net quotient (Definition~\ref{def::levynet_bubble_gluing}).  Suppose that both boundary lengths are positive.  Let $\rho \colon [0,1] \to S$ be the map which visits the points in (the completion of) the leftmost geodesic tree in the Brownian map in contour order.  We will obtain a contradiction by showing that $\rho(s) \neq \rho (t)$.

By the construction in Definition~\ref{def::levynet_bubble_gluing}, we have that the boundary length measure defined for the L\'evy net is right-continuous.  Fix $\epsilon > 0$ small and rational and let $k = \lceil a/\epsilon\rceil + 1$.  Since $t T^\LN$ does not correspond to a jump in $X^\LN$, it follows that we can find $u,v$ in the L\'evy net with distance $d(x,y)+ k \epsilon$ to $x$ and with boundary length distance from the unique geodesic from $y$ to $x$ given by a multiple of $\epsilon$ so that the following is true.  The geodesics from $u,v$ to $x$ pass through the counterclockwise segment on the boundary of the ball centered at $x$ with radius $d(x,y)+a$ before merging.  By the proof of Lemma~\ref{lem:geodesic_tree_is_the_same}, this implies that the corresponding leftmost geodesics in $(S,d,\nu,x,y)$ also do not merge before passing through the counterclockwise segment of $\partial \fb{x}{d(x,y)+a}$ from $\rho(s)$ to $\rho(t)$.  In particular, this interval has non-empty interior.  We can likewise find a pair of points so that the leftmost geodesics to $x$ do not merge before passing through the clockwise segment of $\partial \fb{x}{d(x,y)+a}$ from $\rho(s)$ to $\rho(t)$.  In particular, this interval also has non-empty interior.  This implies that $\rho(s) \neq \rho(t)$ so that $(s,t) \notin K^\BM$ as desired.
\end{proof}

\begin{proof}[Proof of Proposition~\ref{def::levy_attachment_points}]
As explained above, this follows by combining Lemmas~\ref{lem:levy_net_jump_equiv_dense}--\ref{lem:net_equiv_implies_levy_equiv}.
\end{proof}

\subsection{Axioms that characterize the Brownian map}
\label{subsec::axiomsforbrownianmap}

Most of this subsection will be devoted to a proof of the following L\'evy net based characterization of the Brownian map. At the end of the section, we will explain how to use this result to derive Theorem~\ref{thm::markovmapcharacterization}.

\begin{theorem}
\label{thm::levynetbasedcharacterization}
Up to a positive multiplicative constant, the doubly marked Brownian map measure $\mustwo$ is the unique (infinite) measure on $(\gmsspace^{2,O}, \mmsigmaspho)$ which satisfies the following properties, where an instance is denoted by $(S,d,\nu,x,y)$.  
\begin{enumerate}
\item Given $(S,d,\nu)$, the conditional law of $x$ and $y$ is that of two i.i.d.\ samples from~$\nu$ (normalized to be a probability measure). In other words, the law of the doubly marked surface is invariant under the Markov step in which one ``forgets'' $x$ (or~$y$) and then resamples it from the given measure.
\item The law on $\treeequivspace$ (real trees with an equivalence relation) induced by the unembedded metric net from $x$ to $y$ (whose law is an infinite measure) by the measurable map defined in Proposition~\ref{prop:equiv_measurable} has the law of an $\alpha$-L\'evy net for some $\alpha \in (1,2)$.  In other words, the metric net is a.s.\ strongly coalescent (as defined in Section~\ref{subsec:unembedded_measurability}) and the law of the contour function of the leftmost geodesic tree and set of identified points agrees with that of the L\'evy height process used in the $\alpha$-L\'evy net construction.

\item Fix $r>0$ and consider the circle that forms the boundary $\partial \fb{x}{r}$ (an object that is well-defined a.s.\ on the finite-measure event that the distance from $x$ to $y$ is at least $r$).  Then the inside and outside of $\fb{x}{r}$ (each viewed as an element of $\mmspace^{1,O}$, with the orientation induced by $S$) are conditionally independent, given the boundary length of $\partial \fb{x}{r}$ (as defined from the L\'evy net structure) and the orientation of $S$.  Moreover, the conditional law of the outside of $\fb{x}{r}$ does not depend on $r$.
\end{enumerate}
\end{theorem}

Let us emphasize a few points before we give the proof of Theorem~\ref{thm::levynetbasedcharacterization}.
\begin{itemize}
\item Recalling Proposition~\ref{prop:metric_net_equivalent}, in the case of $\mustwo$ one has $\alpha=3/2$.  Moreover, Proposition~\ref{prop:metric_net_equivalent} implies that $\mustwo$ satisfies the second hypothesis of Theorem~\ref{thm::levynetbasedcharacterization}.  Indeed, we saw in Proposition~\ref{prop:metric_net_equivalent} that the law of the unembedded metric net of a sample $(S,d,\nu,x,y)$ from $\mustwo$ has the law of the $3/2$-stable L\'evy net, this implies that the collection of left and right geodesics geodesics in the metric get and how they are identified has the same law as in the $3/2$-stable L\'evy net.  Proposition~\ref{prop::metric_net_law} implies that $\mustwo$ satisfies the third assumption. 
\item The second assumption together with Proposition~\ref{prop::levynet_structure_determined} implies that the boundary length referenced in the third assumption is a.s.\ well-defined and has the law of a CSBP excursion (just like the CSBP used to encode the L\'evy net). In particular, this implies that for any $r>0$, the measure of the event $d(x,y)>r$ is positive and finite.
\item In the coupling between the metric net and the L\'evy net described above, we have made no assumptions about whether {\em every} geodesic in the metric net, from some point $z$ to the root $x$, corresponds to one of the distinguished left or right geodesics in the L\'evy net.  That is, we allow {\em a priori} for the possibility that the metric net contains many additional geodesics besides these distinguished ones. Each of these additional geodesics would necessarily pass through the filled ball boundaries $\partial \fb{x}{r}$ in decreasing order of $r$, but in principle they could continuously zigzag back and forth in different ways. 
\item The measurability results of Section~\ref{subsec::mmsigma} imply that the objects referred to in the statement of Theorem~\ref{thm::levynetbasedcharacterization} are random variables.  In particular, Proposition~\ref{prop::ball_interior_measurable} implies that the inside and the outside of $\fb{x}{r}$ (viewed as elements of $\mmspace^{1,O}$) are measurable functions of an element of $\gmsspace^{2,O}$ and Proposition~\ref{prop:equiv_measurable} implies that unembedded metric net (i.e., the leftmost geodesic tree together with its identified points viewed as an element of $\treeequivspace$) is a measurable function of an element of $\gmsspace^{2,O}$.  
\item The proof of Theorem~\ref{thm::levynetbasedcharacterization} will make use of a rerooting argument which was used previously by Le Gall; see \cite[Section~8.3]{legalluniqueanduniversal}.
\end{itemize}

Now we proceed to prove Theorem~\ref{thm::levynetbasedcharacterization}. This proof requires several lemmas, beginning with the following.

\begin{lemma}
\label{lem::metricnetmeasurezero}
If $\tmustwo$ satisfies the hypotheses of Theorem~\ref{thm::levynetbasedcharacterization}, and $(S,d,\nu,x,y)$ denotes a sample from $\tmustwo$, then it is a.s.\ the case that the metric net from $x$ to $y$ has $\nu$ measure zero. That is, the set of $(S,d,\nu,x,y)$ for which this is not the case has $\tmustwo$ measure zero.
\end{lemma}
\begin{proof}
Suppose that the metric net does not have $\nu$ measure $0$ with positive $\tmustwo$ measure.  Then if we fix $x$ and resample $y$ from $\nu$ to obtain $\wt{y}$, there is some positive probability that $\wt{y}$ is in the metric net from $x$ to $y$. Let $L_r$ be the process that encodes the boundary length of the complementary component of $B(x,r)$ which contains $\wt{y}$.  Then we have that $L_r$ does not a.s.\ tend to $0$ as $\wt{y}$ is hit.  This is a contradiction as, in the L\'evy net definition, we do have that $L_r$ a.s.\ tends to~$0$ as the target point is reached.
\end{proof}

If~$\tmustwo$ satisfies the hypotheses of Theorem~\ref{thm::levynetbasedcharacterization}, then we let~$\tmudonel$ denote the conditional law of $S \setminus \fb{x}{r}$, together with its interior-internal metric and measure, given that the boundary length of~$\partial \fb{x}{r}$ is equal to~$L$. Once we have shown that~$\tmustwo$ agrees with~$\mustwo$, we will know that~$\tmudonel$ agrees with~$\mudonel$, which will imply in particular that~$\tmudonel$ depends on~$L$ in a scale invariant way. That is, we will know that sampling from~$\tmudonel$ is equivalent to sampling from~$\tmudonelset{1}$ and then rescaling distances and measures by the appropriate powers of~$L$. However, this is not something we can deduce directly from the hypotheses of Theorem~\ref{thm::levynetbasedcharacterization} as stated. We can however deduce a weaker statement directly: namely, that at least the probability measures~$\tmudonel$ in some sense depend on~$L$ in a continuous way.  Note that given our definition in terms of a regular conditional probability, the family of measures~$\tmudonel$ is {\em a priori} defined only up to redefinition on a Lebesgue measure zero set of~$L$ values, so the right statement will be that there is a certain type of a continuous modification.

\begin{lemma}
\label{lem::disklawcontinuity}
Suppose that $\tmustwo$ satisfies the hypotheses of Theorem~\ref{thm::levynetbasedcharacterization}. Let $\tmudonel$ denote the conditional law of $S \setminus \fb{x}{r}$, together with its interior-internal metric and measure, given that the boundary length of~$\partial \fb{x}{r}$ is $L$.  For $L_1,L_2 > 0$, define $\rho(\tmudonelset{L_1}, \tmudonelset{L_2})$ to be the smallest $\epsilon > 0$ such that one can couple a sample from~$\tmudonelset{L_1}$ with a sample from~$\tmudonelset{L_2}$ in such way that with probability at least $1-\epsilon$ the two metric/measure-endowed disks agree when restricted to the $y$-containing component of the complement of the set of all points of distance~$\epsilon$ from the disk boundary (and both such components are nonempty). Then the~$\tmudonel$ (after redefinition on a zero Lebesgue measure set of $L$ values) have the property that as $L_1$ tends to $L_2$ the $\rho$ distance between the $\tmudonelset{L_i}$ tends to zero. In other words, the map from $L$ to $\tmudonel$ has a modification that is continuous w.r.t.\ the metric described by $\rho$.
\end{lemma}
\begin{proof}
We begin by observing that a sample from $\tmudonel$ determines an instance of a time-reversed CSBP starting from $L$ and stopped when it hits $0$.  Indeed, this time-reversed CSBP is simply the continuation of the boundary length process, starting from a point at which it has value $L$, associated with the L\'evy net instance which corresponds to the metric net of $S$.  (Recall Proposition~\ref{prop::levynet_structure_determined}, which gives that the boundary length process can be measurably recovered from the unembedded metric net.)  Note that if $Y$ is a time-reversed CSBP and $t > 0$ is fixed then the law of $Y_t$ is continuous in $Y_0$ in the total variation sense.  By the Markov property of a time-reversed CSBP, this implies that for $t > 0$ fixed the law of $(Y_s : s \geq t)$ is continuous in $Y_0$ in the total variation sense.  In particular, we can couple the corresponding time-reversed CSBPs that arise from $\tmudonelset{L_1}$ and $\tmudonelset{L_2}$ so that they agree starting after time $\epsilon > 0$ with probability tending to $1$ as $L_1 \to L_2$.  Let us define~$\rho'(L_1, L_2)$ to be the smallest~$\epsilon > 0$ so that the two time-reversed CSBPs, started at $L_1$ and $L_2$, can be coupled to agree and are both non-zero after an $\epsilon$ interval of time with probability $1-\epsilon$.  It follows from what we have explained above that $\rho'(L_1, L_2)$ is continuous in $L_1$ and $L_2$ and zero when $L_1 = L_2$.  Now using the Markov property assumed by the hypotheses of Theorem~\ref{thm::levynetbasedcharacterization}, we find $\rho(\tmudonelset{L_1}, \tmudonelset{L_2}) \leq \rho'(L_1, L_2)$ for almost all~$L_1$ and~$L_2$ pairs.  Indeed, running the time-reversed CSBPs $L_1$ and $L_2$ from time $\epsilon$ corresponds to metrically exploring from the disk boundaries for $\epsilon$ distance units.  If the CSBPs have coalesced by time $\epsilon$,  then by the hypotheses of Theorem~\ref{thm::levynetbasedcharacterization}, we know that the conditional law of the unexplored region is the same for both disk instances hence we can couple them to be the same.  Thus, if a countable dense set~$Q$ of~$L$ values is obtained by i.i.d.\ sampling from Lebesgue measure, then this bound a.s.\ holds for all~$L_1$ and~$L_2$ in~$Q$. Then for almost all other $L$ values, we have that with probability one, $\rho(\tmudonelset{L'}, \tmudonelset{L}) \to 0$ as $L'$ approaches $L$ with $L'$ restricted to the set~$Q$. We obtain the desired modification by redefining $\tmudonelset{L}$, on the measure zero set of values for which this is not the case, to be the unique measure for which this limiting statement holds. (It is clear that the limiting statement uniquely determines the law of disk outside of an $\epsilon$-neighborhood of the boundary, and since this holds for any $L$, it determines the law of the overall disk.)
\end{proof}

\begin{lemma}
\label{lem::markovatjumpstop}
Suppose that $\tmustwo$ satisfies the hypotheses of Theorem~\ref{thm::levynetbasedcharacterization}. Let $\tmudonel$ denote the conditional law of $S \setminus \fb{x}{r}$, together with its interior-internal metric and measure, given that the boundary length of $\partial \fb{x}{r}$ is $L$. Then suppose $\tau$ is any stopping time for the process $L_r$ such that a.s.\ $L_r$ has a jump at time $\tau$.  (For example $\tau$ could be the first time at which a jump in a certain size range appears.) Then the conditional law of $S \setminus \fb{x}{\tau}$, given $\fb{x}{\tau}$ and the process $L_r$ up to time $\tau$, is given by $\tmudonel$ with $L= L_\tau$.
\end{lemma}
\begin{proof}
This is simply an extension of the theorem hypothesis from a deterministic stopping time to a specific type of random stopping time. The extension to random stopping times is obvious if one considers stopping times that a.s.\ take one of finitely many values. In particular this is true for the stopping time $\tau_\delta$ obtained by rounding~$\tau$ up to the nearest integer multiple of $\delta$, where $\delta>0$. It is then straightforward to obtain the result by taking the $\delta \to 0$ limit and invoking the continuity described in Lemma~\ref{lem::disklawcontinuity}.  (Recall also the proof of Proposition~\ref{prop::jumpsoneone}.)
\end{proof}

\begin{lemma}
\label{lem::figureeightlaws}
Suppose that we have the same setup as in Lemma~\ref{lem::markovatjumpstop}. Then the union of $\partial\fb{x}{r}$ and the boundary of the ball cut off at time $\tau$ is a.s.\ a topological figure $8$ (of the sort shown in Figure~\ref{fig::exploretopinch}). The boundary length measure along the figure $8$ is a.s.\ well-defined. The total boundary length is $L_{\tau^-}$, while the boundary length of the component surrounding $y$ is $L_\tau$.
\end{lemma}
\begin{proof}
The fact that the union of $\partial\fb{x}{r}$ and the boundary of the ball cut off at time $\tau$ is a.s.\ a topological figure~$8$ is immediate from the the third definition of the L\'evy net quotient (Definition~\ref{def::levynet_bubble_gluing}, see also Figure~\ref{fig::bubblegluing}) and Proposition~\ref{prop::boundariesarecircles}.  Lemma~\ref{lem::markovatjumpstop} implies that the conditional law of $S \setminus \fb{x}{\tau}$ equipped with its interior-internal metric and measure has law $\tmudonelset{L_\tau}$.  In particular, the boundary length along $\partial \fb{x}{\tau}$ is a.s.\ well-defined and the total boundary length is equal to $L_\tau$.  It is left to explain why the boundary length of the other component is well-defined and why the sum of the two boundary lengths is equal to $L_{\tau^-}$.  We note that if we let $\wt{y}$ be an independent sample from $\nu$ given everything else, then there is a positive chance given everything else that it is in the other component as $\nu$ is a good measure.  It thus follows that the other component has law which is absolutely continuous with respect to $\tmudonel$ (for some value of $L$) as it can be obtained as the complement of the filled metric ball at a stopping time on the event that it contains $\wt{y}$.  In particular, it has an a.s.\ well-defined boundary measure.  Finally, that the sum of the boundary lengths is a.s.\ equal to $L_{\tau^-}$ can be seen from the proof of Proposition~\ref{prop::levynet_structure_determined}.
 
\end{proof}

If $\tmustwo$ satisfies the hypotheses of Theorem~\ref{thm::levynetbasedcharacterization}, and $\tau$ is a stopping time as in Lemma~\ref{lem::markovatjumpstop}, then we can now define $\tmudl$ to be the conditional law of the disk cut out at time $\tau$ given that the boundary length of that disk (i.e., the size of the jump in the $L_r$ process that occurs then $r = \tau$) is $L$. The following lemma asserts that this conditional law indeed depends only on $L$ and not on other information about the behavior of the surface outside of this disk.

We define $\mudl$ to be the corresponding law when we start from $\mustwo$.  Recall that the proof that $\mustwo$ satisfies the hypotheses of Theorem~\ref{thm::levynetbasedcharacterization} does not require one to have analyzed any properties of or even to have defined $\mudl$.  Therefore at this point in the paper, we may apply the following lemma in the case of $\mustwo$ in order to give a definition of $\mudl$.  Using this approach, one does not need an argument which is separate from that in the case of $\mudonel$ to construct the boundary length measure for $\mudl$ and to deduce that the disks in the metric net of $\mustwo$ are conditionally independent given their boundary lengths.

We will show in Lemma~\ref{lem::tmudlareaexpectation} just below that $\tmudl$ (resp.\ $\mudl$) can be obtained from $\tmudonel$ (resp.\ $\mudonel$) by unbiasing its law by area.  In other words, we will prove that $\tmudl$ (resp.\ $\mudl$) is obtained from $\tmudonel$ (resp.\ $\mudonel$) by removing a marked point.  This fact is what motivates the notation.

\begin{lemma}
\label{lem::tmudldefined}
Assume that $\tmustwo$ satisfies the hypotheses of Theorem~\ref{thm::levynetbasedcharacterization}. Then the conditional probability $\tmudl$ described above is well-defined and indeed depends only on~$L$.
\end{lemma}
\begin{proof}
Let $L_1,L_2$ be the boundary lengths of the two disks which together form the figure $8$ which arises at the stopping time~$\tau$.  If one explores up until the stopping time~$\tau$, one can resample the target point~$y$ from the restriction of~$\nu$ to the union of the two disks pinched off at time $\tau$.  Indeed, this follows since the conditional law of $y$ given $(S,d,\nu,x)$ is $\nu$.  Since $\nu$ is a.s.\ a good measure, there will be some positive probability that $y$ ends up on each of the two sides. The theorem hypotheses imply that the conditional law of each of the two disks bounded by the figure $8$, on the event that $y$ lies in that disk, is given by $\tmudonel$, independently of any other information about the surface outside of that disk. This implies in particular that the two disks are independent of each other once it has been determined which disk contains $y$. Now, one can resample the location of $y$, resample the disk containing $y$ from $\tmudonel$ (with $L=L_1$ or $L=L_2$ depending on which disk contain $y$), resample the location of $y$, resample the disk containing $y$ again, etc.

The proof will thus be complete upon showing that this Markov chain has a unique invariant distribution which depends only on $L_1$ and $L_2$.  To see this, we can consider the same chain but with the initial distribution consisting of two independent samples from $\tmudonel$, one with $L=L_1$ and the other with $L=L_2$.  We claim that the chain with this initial distribution converges to a limit as the number of resampling steps goes to $\infty$.  Indeed, the reason is that for any pair of initial configurations and $\epsilon > 0$, it is easy to see that there exists $N \in \N$ (depending only on the relative masses of the pairs of disks) so that if one performs the resampling step $n \geq N$ times then the total variation distance between the resulting laws will be at most $\epsilon$.  This fact also implies that the limiting law as the number of resampling steps goes to $\infty$ is the unique invariant distribution for the Markov chain.  Moreover, this limiting law is a function only of~$L_1$ and~$L_2$ because the initial distribution used to define it was a function only of~$L_1$ and~$L_2$.

These assumptions therefore determine the form of $\tmudl$. (The explicit relationship between~$\tmudl$ and $\tmudonel$ will be derived in the proof of Lemma~\ref{lem::tmudlareaexpectation} just below.)
\end{proof}

\begin{lemma}
\label{lem::explore_to_marked_point}
Given the $L_r$ process describing the boundary length of $\partial \fb{x}{r}$, the conditional law of the disks in the complement of the metric net are given by conditionally independent samples from $\tmudlset{L_i}$ where $L_i$ are the lengths of the hole boundaries (which in turn correspond to the jumps of $L_r$).
\end{lemma}
\begin{proof}
We will deduce the result from Lemma~\ref{lem::tmudldefined} as follows.  Fix a value of $k \in \N$ and for each $j \in \N$ we let $\tau_{j,k}$ be the $j$th time that $L_r$ has a downward jump of size at least $2^{-k}$.  Lemma~\ref{lem::tmudldefined} applied with the stopping time $\tau = \tau_{1,k}$ implies that, on $\tau_{1,k} < \infty$, we have that given $L_{\tau_{1,k}}$ and $L_{1,k} = L_{\tau_{1,k}^-} - L_{\tau_{1,k}}$ the conditional laws of $S \setminus \fb{x}{\tau_{1,k}}$ and the component separated from $y$ at time $\tau_{1,k}$ equipped with their interior-internal metrics and the measure given by the restriction of $\nu$ are respectively given by $\tmudonelset{L_{\tau_{1,k}}}$ and $\tmudlset{L_{1,k}}$.  In fact, it is not difficult to see that the same statement holds when we condition on $(L_r : r \leq \tau_{1,k})$, which we note determines $L_{1,k}$ and $L_{\tau_{1,k}}$.  By continuing the exploration and iterating the argument, we see that given the values of $L_{j,k} = L_{\tau_{j,k}^-} - L_{\tau_{j,k}}$ for $j \geq 1$ the components which are separated from $y$ at the times $\tau_{j,k}$ are conditionally independent samples from $\tmudlset{L_{j,k}}$.  The same statement in fact holds when we condition on the entire realization of $L_r$, which in turn determines the $L_{j,k}$.  The result thus follows by taking a limit as $k \to \infty$.
\end{proof}

\begin{lemma}
\label{lem::tmudlareaexpectation}
Assume that~$\tmustwo$ satisfies the hypotheses of Theorem~\ref{thm::levynetbasedcharacterization}, and that~$\tmudl$ and~$\tmudonel$ are defined as above. Let~$A$ be the total area measure of a sample from~$\tmudl$.  Then the~$\tmudl$ expectation of~$A$ is given by a constant times~$L^{2\alpha - 1}$. Moreover, the Radon-Nikodym derivative of~$\tmudonel$ w.r.t.\ $\tmudl$ (where one ignores the marked point, so that the two objects are defined on the same space) is hence given by a constant times~$A/L^{2\alpha - 1}$.
\end{lemma}
\begin{proof}
Fix $L_0 > 0$ and suppose that we start with an instance of $\tmudonelset{L_0}$.  Let $L_r$ be the evolution of the boundary length of the exploration towards the marked point $y$.  Suppose that we evolve~$L_r$ up to $\tau = \inf\{r \geq 0 : (L_{r^-} - L_r)/L_r \geq 1/4\}$.  At time $\tau$, the boundary length $c = L_r$ is divided into two components, of lengths $a$ and $b$ with $a+b=c$.  (That is, $c = L_{r^-}$ and $\{a,b\} = \{L_r,L_{r^-} - L_r\}$.)

Set $L \in \{a,b\}$ to be the boundary length of the component surrounding $y$. By Lemma~\ref{lem::markovatjumpstop}, the conditional law of the disk in this component is given by $\tmudonel$.  Following Lemma~\ref{lem::tmudldefined}, we let $\tmudl$ denote the probability measure that describes the conditional law of the metric disk inside the loop that {\em does not} surround $y$, when $L \in \{a, b \}$ is taken to be the length of {\em that} loop.  (Again, we have not yet proved this is equivalent to the measure $\mudl$ defined from the Brownian map.)

If we condition on the lengths of these two pieces --- i.e., on the pair $\{a,b\}$ (but not the values of $L_r$ and $L_{r^-} - L_r$) --- then what is the conditional probability that $y$ belongs to the $a$ loop versus the $b$ loop?  We will address that question in two different ways.  First of all, if $p$ is that probability, then we can write the overall measure for the pair of surfaces as the following weighted average of probability measures
\[ p \tmudonelset{a} \otimes \tmudlset{b} + (1-p) \tmudlset{a} \otimes \tmudonelset{b}.\]
Now, observe that if we condition on the pair of areas $A_1, A_2$, then the resampling property for $y$ (recall the proof of Lemma~\ref{lem::tmudldefined}) implies that the conditional probability that $y$ is in the first area is $A_1/(A_1 + A_2)$.  This implies the following Radon-Nikodym derivative formula for two (non-probability) measures
\begin{equation}
\label{eqn::radonratio}
\frac{ d \left[ p \tmudonelset{a} \otimes \tmudlset{b} \right] }{d \left[(1-p) \tmudlset{a} \otimes \tmudonelset{b} \right]} = \frac{A_1}{A_2}.
\end{equation}
From this, we may deduce (by holding one of the two disks fixed and letting the other vary) that the Radon-Nikodym derivative of $\tmudonel$ w.r.t.\ $\tmudl$ (ignoring the marked point location) is given by a constant times the area $A$ of the disk; since both objects are probability measures, this Radon-Nikodym derivative must be the ratio $A/\E_{\tmudl}[ A]$. Plugging this back into~\eqref{eqn::radonratio}, we find that
\begin{equation}
\label{eqn::loopprobratio}
\frac{p}{1-p} = \frac{\E_{\tmudlset{a}} [A]}{\E_{\tmudlset{b}} [A]}.
\end{equation}
In other words, the probability that $y$ lies in the disk bounded by the loop of length $L \in \{a,b\}$ (instead of the other disk) is given by a constant times the $\tmudl$-expected area of a disk bounded by that loop.

Next, we note that there is a second way to determine $p$. Namely, we may directly compute the relative likelihood of a jump by $a$ versus a jump by $b$ in the time-reversal of an $\alpha$-stable L\'evy excursion, given that one has a jump of either $a$ or $b$.  By Lemma~\ref{lem::levyreversal}, the ratio of these two probabilities is $a^{2\alpha-1}/b^{2\alpha-1}.$  Plugging this into~\eqref{eqn::loopprobratio} gives
\[\frac{a^{2\alpha-1}}{b^{2\alpha-1}} = \frac{\E_{\tmudlset{a}} [A]}{\E_{\tmudlset{b}} [A]}.\]
Since this is true for generic values of $a$ and $b$, we conclude that $\E_{\tmudl}[A]$ is given by a constant times $L^{2\alpha - 1}$.
\end{proof}

\begin{figure}[ht!]
\begin{center}
\includegraphics[scale=1]{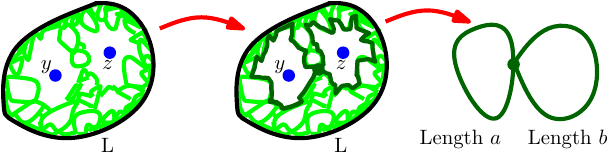}
\caption{\label{fig::exploretopinch} The intersection of the metric net from the boundary to $y$ with the metric net from the boundary to $z$. Intuitively, these are the points one finds as one continually ``explores'' points (in order of distance from the boundary) within the unexplored component containing both $y$ and $z$, stopping at the first time that $y$ and $z$ are separated.  At the time when $y$ and $z$ are separated, the boundaries of the component containing $y$ and the component containing $z$ are disjoint topological circles, each of which comes with a length; we denote the two lengths by $a$ and $b$.
}
\end{center}
\end{figure}

As discussed above, at a time when a point $z$ is disconnected from the target point~$y$, the boundary has the form of a figure $8$ with two loops of distinct lengths~$a$ and~$b$, as shown in Figure~\ref{fig::exploretopinch}.  At this time the process~$L_r$ jumps from some value $c = a+b$ down to~$a$ (if the marked point~$y$ is in the component of boundary length~$a$) or~$b$ (if~$y$ is in the component of boundary length~$b$).  We define a {\em big jump} in the process $L_r$ associated to $\tmudonel$ to be a jump whose lower endpoint is less than half of its upper endpoint. A big jump corresponds to a time when the marked point lies in the disk bounded by the {\em shorter} of the two figure $8$ loops.
  
In what follows, it will sometimes be useful to consider an alternative form of exploration in which the endpoint $y$ is not fixed in advance. We already know that if let $y_1, y_2, \ldots$ be independent samples from $\nu$, then the unembedded metric nets targeted at those points should be in some sense coupled L\'evy nets, which agree up until the first time at which those points are separated. Indeed, there will be countably many times at which one of those points is first disconnected from the other, as illustrated in Figure~\ref{fig::exploretopinch}. This union of all such explorations can be understood as sort of a branching exploration process, where each time the boundary is ``pinched'' into two (forming a figure $8$, as in Figure~\ref{fig::exploretopinch}) the exploration continues on each of the two sides.

In what follows, it will be useful to consider an alternative form of exploration in which, at each such pinch point, the exploration always continues in the {\em longer} of these two loops, rather than continuing in the loop that contains some other predetermined point~$y$. That is, we choose the exploration so that the corresponding boundary length process~$L_r$ has no ``big jumps'' as we defined them above.  It is clear that each~$y_i$ will a.s.\ fail to lie in the bigger loop of a figure~$8$ at some point.  Indeed, the area of the bigger loop a.s.\ tends to $0$ as the exploration continues.  This implies that for each $\epsilon > 0$ and $N$ there exists $R > 0$ so that if $r \geq R$ then the probability that any of $y_1,\ldots,y_N$ are in the region into which the exploration continues after time $r$ is at most $\epsilon$.  Hence a.s.\ all of the points~$y_i$ will lie in disks that are cut off by this exploration process in finite time.

Let~$A_r$ denote the unexplored disk (in which the exploration continues) that remains after~$r$ units of exploration of this process.  Then~$\ol{A}_r$ is a closed set, which is the closure of the set of points~$y_i$ with the property that the L\'evy net explorations targeted at those points have no big jumps before time~$r$.  The intersection of~$\ol{A}_r$, over all $r$, is thus a closed set that we will call the {\em center} of the disk. We do not need to know this {\em a priori} but we expect that the center contains only a single point. Note that the center can be defined if the surface is sampled from either~$\tmudl$ or~$\tmudonel$ (and in the latter case its definition does not depend on the marked point~$y$). We refer to the modified version of the L\'evy net as the {\em center net} corresponding to the surface.  We are now going to prove an analog of Lemma~\ref{lem::explore_to_marked_point} for the center net.

\begin{lemma}
\label{lem::exploretocenter}
Given the $M_r$ process describing the center net corresponding to a sample from $\tmudl$, the conditional law of the disks in the complement of the net are given by conditionally independent samples from $\tmudlset{M_i}$ where $M_i$ are the lengths of the hole boundaries.
\end{lemma}
\begin{proof}
We first suppose that $(S,d,\nu,y)$ is sampled from $\tmudonel$.  Consider the exploration from $\partial S$ towards $y$.  Fix $r > 0$ and condition on the event that $L|_{[0,r]} = M|_{[0,r]}$ where $L$ (resp.\ $M$) denotes the boundary length process associated with the exploration towards $y$ (resp.\ center exploration).  We note that this event is $L|_{[0,r]}$-measurable, so that if we condition on it then Lemma~\ref{lem::explore_to_marked_point} implies that the holes cut out by the exploration up to time $r$ are conditionally independent samples from $\tmudlset{M_i}$ where $M_i$ are the hole boundary lengths.  We also note that conditioning on this event is the same as weighting the law of $(S,d,\nu,y)$ by a normalizing constant times $A_r/A$ where $A_r$ (resp.\ $A$) is the area of the $y$-containing complementary component of the center exploration at time $r$ (resp.\ total area of $S$).  In particular, the marginal law of $(S,d,\nu)$ given this event is the same as the law of a sample from~$\tmudl$ weighted by a normalizing constant times~$A_r$ where here~$A_r$ denotes the area of the unexplored region in the center exploration up to time~$r$.

The above implies that the following is true.  Suppose that $(S,d,\nu)$ is sampled from $\tmudl$, $r > 0$, and $y$ is picked in $S$ from $\nu$.  If we condition on the event that the metric exploration towards $y$ agrees with the center exploration up to time $r$, then the holes cut out are conditionally independent samples from $\tmudlset{M_i}$ where $M_i$ are the hole boundary lengths.  Since $r > 0$ was arbitrary, this implies that the holes cut out by the center exploration up until the first time that the center exploration disagrees with the exploration targeted at $y$ are conditionally independent samples from $\tmudlset{M_i}$ where $M_i$ are the hole boundary lengths.  We can iterate the same procedure at this time (and then keep repeating) to get that the same statement is true when we perform the center exploration until it terminates.

\end{proof}

We now would like to discuss the relationship between the laws of the following processes:
\begin{enumerate} \item The process $L_r$ obtained by exploring the metric net from a sample from $\tmudonel$, starting with $L_0$ equal to some fixed value $L$.
\item The process $M_r$ obtained by exploring a sample from $\tmudl$ toward the center (again starting with $M_0 = L$).
\item The process $M_r^1$ obtained by exploring a sample from $\tmudonel$ toward the center (again starting with $M_0^1 = L$).
\end{enumerate}

We already know that the Radon-Nikodym derivative of~$\tmudonel$ w.r.t.\ $\tmudl$ is given by a constant times the area of the disk. We will use this fact to deduce the following.

\begin{lemma}
\label{lem::explorationprocessrelationships}
The Radon-Nikodym derivative of the process~$M_r^1$ w.r.t.\ the process~$M_r$ is given by a constant times the {\em expected} disk area given the process, which (by Lemma~\ref{lem::tmudlareaexpectation} and Lemma~\ref{lem::exploretocenter}) is given by a constant times $\sum_K K^{2\alpha - 1}$ where $K$ ranges over the jump magnitudes corresponding to the countably many jumps in the process. Moreover, if~$L_r$ and~$M_r^1$ are coupled in the obvious way (i.e., generated from the same instance of $\tmudonel$) then they agree up until a stopping time: namely, the first time that~$L_r$ experiences a big jump.
\end{lemma}
\begin{proof}
By Lemma~\ref{lem::tmudlareaexpectation}, the Radon-Nikodym derivative of $\tmudonel$ with respect to $\tmudl$ is given by a normalizing constant times the amount of area $A$ assigned by $\nu$ to $S$.  Since there is a.s.\ no area in the metric net, we have that $A = \sum_i A_i$ where the $A_i$'s give the areas of the holes cut out by the exploration.  It follows that the Radon-Nikodym derivative of the law of $M_r^1$ with respect to the law of $M_r$ is given by a normalizing constant times the conditional expectation of $A$ given the realization of the entire process $M_r$.  By Lemma~\ref{lem::exploretocenter} and Lemma~\ref{lem::tmudlareaexpectation}, this conditional expectation is equal to a constant times $\sum_K K^{2\alpha-1}$ where $K$ ranges over the jump magnitudes corresponding to the jumps in the process.  This proves the first assertion of the lemma.  The second assertion of the lemma is immediate from the definitions. 
\end{proof}

As a side remark, let us note that the stopping time~$\tau$ of the process~$M_r^1$, as defined in Lemma~\ref{lem::explorationprocessrelationships}, can be constructed in fairly simple way that roughly corresponds to, each time a new figure~$8$ is created, tossing an appropriately weighted coin to decide whether~$y$ is in the smaller or the larger loop, and then stopping when it first lies in the smaller loop. To formulate this slightly more precisely, suppose that for each $r \geq 0$ we let~$\chi_r$ be the product of
\[ \frac{a^{2\alpha-1}}{a^{2\alpha-1}+b^{2\alpha-1}}\]
over all jumps of $M^1|_{[0,r]}$ where $a$ is the size of the jump and~$b$ is equal to the value of~$M^1$ immediately after the jump.  Suppose that we choose~$p$ uniformly in $[0,1]$.  Then we can write $\tau = \inf\{r \geq 0 : \chi_r < p\}$.

We next claim the following:
\begin{lemma}
\label{lem::mtmarkovian}
If one explores the center net of an instance of $\tmudl$ up to some stopping time $\tau$, then the conditional law of the central unexplored disk (i.e., the one in which exploration will continue) is given by an instance of $\tmudlset{L'}$ where $L' = M_\tau$ is the boundary length at that time. In particular, this implies that the process $M_r$ is Markovian.
\end{lemma}
\begin{proof}
This follows by combining Lemma~\ref{lem::markovatjumpstop}, Lemma~\ref{lem::tmudlareaexpectation}, and Lemma~\ref{lem::explorationprocessrelationships}.
\end{proof}

By Lemma~\ref{lem::levyreversal}, the jump density for $\tmudonel$ (for a jump of size $a$ that leaves a loop of size $b = c-a$ in which $y$ is contained) is given by a constant times $a^{-\alpha-1} b^{\alpha - 2}$.

\begin{lemma}
\label{lem::mtjumplaw}
The process $M_r$ agrees in law with the process $L_r$ except that the jump law is different. Instead of having the form
\begin{equation}
\label{eqn::levy_jump_law}
\one_{a \in [0,c]} a^{-\alpha-1} (b/c)^{\alpha-2} da,
\end{equation}
it has the form
\begin{equation}
\label{eqn::firstradon}
\one_{a \in [0,c/2]} a^{-\alpha - 1} (b/c)^{-\alpha-1} da,
\end{equation}
where in both cases $b$ is simply defined via $b  = c-a$, $c$ is defined to the height of the process just before the jump, and $da$ denotes Lebesgue measure.
\end{lemma}

To further clarify the statement of Lemma~\ref{lem::mtjumplaw}, we recall that a L\'evy process is a.s.\ determined by its jumps and jump times in a measurable manner.  Therefore if we observe the jumps and jump times of $L_r$, then we can determine the entire process.  We have that $M_r$ is determined by its jumps using the same measurable function which determines $L_r$ from its jumps.

\begin{proof}[Proof of Lemma~\ref{lem::mtjumplaw}]
Let $L_r$ (resp.\ $M_r^1$) be the boundary length processes associated with an exploration of a sample from $\tmudonel$ explored towards the marked point $y$ (resp.\ the center).  Let $\tau$ be the first time at which the two explorations differ.  Fix $\epsilon \in (0,1/2)$ and let $\tau'$ be the smallest $r \geq 0$ such that the exploration towards $y$ makes a downward jump of size in $[\epsilon L_r,(1-\epsilon) L_r]$.  Recall from Lemma~\ref{lem::levyreversal} that the density for the jump law for $L_r$ is given by a constant times $a^{-\alpha-1} (b/c)^{\alpha - 2}$ where $a$ is the jump size, $c$ is the value of the process at the time of the jump, and $b = c-a$.  Given $\tau' \leq \tau$, the density for the downward jump made by $L_r$ at time $\tau'$ is given by a constant times $a^{-\alpha-1} (b/c)^{\alpha-2} \one_{a \in [c \epsilon,c(1-\epsilon)]}$.  Since a jump of size $a$ in $M_r^1$ can correspond to two kinds of jumps in $L_r$ (one of size $a$ and one of size $b=c-a$), it follows that the density for the downward jump made by $M_r^1$ at the time $\tau'$ given $\tau' \leq \tau$ is given by a constant times
\begin{align*}
    & \bigl(a^{-\alpha-1} (b/c)^{\alpha - 2} + (a/c)^{\alpha-2} b^{-\alpha - 1}\bigr)\one_{a \in [c \epsilon,c/2]}\\
 =& \big( (a/c)^{2\alpha-1} + (b/c)^{2\alpha - 1}\bigr) a^{-\alpha - 1} (b/c)^{-\alpha-1} \one_{a \in [c \epsilon,c/2]}.
 \end{align*}
 Since $\epsilon \in (0,1/2)$ was arbitrary, we find that the jump law for $M_r^1|_{[0,\tau]}$ is given by
\begin{equation}
\label{eqn::mtonejump}
\big( (a/c)^{2\alpha-1} + (b/c)^{2\alpha - 1}\bigr) a^{-\alpha - 1} (b/c)^{-\alpha-1} \one_{a \in [0,c/2]}.
\end{equation}

Let $M_r$ denote the boundary length process associated with an exploration towards the center from a sample from $\tmudl$.  Lemma~\ref{lem::explorationprocessrelationships} implies that the laws of $M_r^1$ and $M_r$ are absolutely continuous.  On the event that $\tau' \leq \tau$, the Radon-Nikodym derivative for the law of the jump made by $M_r^1$ at the time $\tau'$ and the law of the jump made by $M_r$ at the corresponding time is given by a constant times
\begin{equation}
\label{eqn::mtjumpradon}
(a/c)^{2\alpha-1} + (b/c)^{2\alpha - 1}.
\end{equation}
Indeed, this expression gives the expected area in the figure $8$ formed by the two components at the jump time.  This proves the result since the Radon-Nikodym derivative between the laws~\eqref{eqn::mtonejump} and~\eqref{eqn::firstradon} is given by~\eqref{eqn::mtjumpradon}.
\end{proof}

We remark that from the point of view of the discrete models, the jump law for~$M_r$ described in Lemma~\ref{lem::mtjumplaw} is precisely what one would expect if the overall {\em partition function} for a boundary-length $a$ disk were given by a constant times $a^{-\alpha -1}$. Indeed, in this case $a^{-\alpha - 1} b^{\alpha -1}$ would be the weighted sum of all ways to triangulate the loops of a figure~$8$ with loop lengths~$a$ and~$b$, which matches the law described in the lemma statement. It is therefore not too surprising that the jump law for the~$\tmudl$ exploration toward the center has to have this form. Furthermore, we may conclude that the~$M_r$ process can be a.s.\ recovered from the ordered collection of jumps (since this is true for L\'evy processes, hence true for CSBPs, hence true for time-reversals of these processes, hence true for this modified time-reversal that corresponds to~$\tmudl$) and the reconstruction procedure is the same as the one that corresponds to the~$L_r$ process.

As suggested by Figure~\ref{fig::levyslicedecomposition}, now that we have constructed the law of the exploration of a sample from $\tmudl$ toward the center, we can try to iterate this construction within each of the unexplored regions and repeat, so that in the limit, we obtain the joint law of the metric net toward all points, or at least toward all points in some countable dense subset of the metric disk.  The hope is that one can recover the entire law of $\tmudl$ using a branching procedure like this. This idea underlies that the proof below.

\begin{proof}[Proof of Theorem~\ref{thm::levynetbasedcharacterization}]
We will break the proof up into three steps.

\noindent{\it Step 1: Axioms imply $\alpha = 3/2$.}  By Lemma~\ref{lem::metricnetmeasurezero} there is a.s.\ no area in the metric net itself.  This implies that if we explore the center net of a sample from $\tmudl$ up until a given time, then the center net also a.s.\ contains zero area.  Let $M_r$ be the boundary length process associated with the center exploration of a sample from $\tmudl$.  By Lemma~\ref{lem::tmudlareaexpectation}, Lemma~\ref{lem::exploretocenter}, and Lemma~\ref{lem::mtmarkovian} if we perform an exploration towards the center of a sample produced from $\tmudl$ up until a given time $s$ then the conditional expectation of the total area is given by (a constant times)
\begin{equation}
\label{eqn::asdef}
A_s := M_s^{2\alpha-1} + \sum |a_i|^{2\alpha - 1}
\end{equation}
where the $a_i$ are an enumeration of the jumps in the process $M_r$ up to time $s$.   Thus,~\eqref{eqn::asdef} must evolve as a martingale in $s$.  Proposition~\ref{prop::centermartingale} (stated and proved in Section~\ref{subsec::martingale32} below) implies that~\eqref{eqn::asdef} evolves as a martingale if and only if $\alpha = 3/2$. Thus, the fact that $\alpha = 3/2$ is a consequence of the properties listed in the theorem statement. For the remainder of the proof, we may therefore assume that $\alpha = 3/2$.

\noindent{\it Step 2: Conditional law of area given boundary length agrees.}  Recall that the collection $\CU_0$ of complementary components which arise from performing the center exploration each correspond to one of the downward jumps $a_i$ of $M$.  Moreover, $a_i$ gives the boundary length of the corresponding element of $\CU_0$.  We can iterate the process by performing a center exploration into each of the elements of $\CU_0$.  Let $\CG_1$ be the $\sigma$-algebra which is generated by:
\begin{itemize}
\item The initial center exploration $M$ and
\item The same information corresponding to center explorations into each of the elements of $\CU_0$.
\end{itemize}
The iterative step used to define $\CG_1$ yields a collection of components~$\CU_1$, in each of which we can again perform a center exploration.  For $k \in \N$, we inductively let $\CU_k$ (resp.\ $\CG_k$) be the collection of complementary components which arise from (resp.\ $\sigma$-algebra generated by~$\CG_{k-1}$ and by) performing center explorations in all of the components in~$\CU_{k-1}$.

Let $A$ be the overall area measure of a surface sampled from $\tmudl$ and let $A_k = \E[ A \giv \CG_k]$.  We will now show that $A$ is $\CG =\sigma(\CG_k : k \in \N)$ measurable, i.e., $A$ is determined by the information encoded by {\em all} of the countably many exploration iterations.  Upon proving this, we will have by the martingale convergence theorem that $A_k \to \E[ A \giv \CG] = A$ a.s.  Recall from the discussion just after the statement of Theorem~\ref{thm::levynetbasedcharacterization} that all of the hypotheses of Theorem~\ref{thm::levynetbasedcharacterization} apply to $\mustwo$ with $\alpha = 3/2$.  Indeed, Proposition~\ref{prop::metric_net_law} implies that the law of the unembedded metric net in this case is the $3/2$-L\'evy net and one has the conditional independence of the inside and outside of filled metric balls.  We therefore have that all of the lemmas above apply if we use~$\mudl$ and~$\mudonel$ in place of~$\tmudl$ in and~$\tmudonel$, respectively.  Therefore we know that the joint law of the processes encoding the iterations~$A_k$, and the law of the conditional expectation of the area in the unexplored regions, is the same in each case.  Hence, the proof of the step will be complete upon showing that $A$ is $\CG$-measurable.

Fix $\epsilon > 0$ and we let $G_{k,\epsilon}$ be the event that the total amount of area in each of the individual complementary components after performing $k$ iterations of the exploration is at most $\epsilon$.  Under $\tmudl$, we know that $\nu$ is a good measure hence does not have atoms.  Therefore it follows that the $\tmudl$ mass of $G_{k,\epsilon}^c$ tends to $0$ as $k \to \infty$ (with $\epsilon$ fixed).  For each $j$, let $X_j$ denote the area of the $j$th component (according to some ordering) after performing $k$ iterations of the exploration.  Then we have that the total variation distance between the law of $\sum_j X_j \one_{X_j \leq \epsilon}$ and the law of $\sum_j X_j$ under~$\tmudl$ tends to $0$ as $k \to \infty$ (with $\epsilon$ fixed).  As the conditional variance of the former given $\CG_k$ obviously tends to $0$ as $k \to \infty$ and then $\epsilon \to 0$, it thus follows that the latter concentrates around a $\CG$-measurable value as $k \to \infty$.  This proves the claim in the case of $\tmudl$.  The same argument also applies verbatim with $\mudl$ in place of $\tmudl$, hence completes the proof of this step.

\noindent{\it Step 3: Coupled L\'evy net instances.}  Suppose that $(S,d,\nu,x,y)$, $(\wt{S},\wt{d},\wt{\nu},\wt{x},\wt{y})$ are samples from $\mustwo$, $\tmustwo$, respectively.  Let $(z_i)$, $(\wt{z}_i)$ be i.i.d.\ samples from $\nu, \wt{\nu}$, respectively.   The exploration process towards each of the $z_i$ encodes an instance of the L\'evy net, which (recall Definition~\ref{def::levynet_bubble_gluing}) can be encoded by the boundary length process together with the attachment point locations.  By the assumptions of the theorem (and that $\alpha=3/2$), we can couple $\mustwo$ and $\tmustwo$ so that the L\'evy net instances (i.e., the corresponding boundary length process with attachment points) associated with the metric explorations towards $z_1$ and $\wt{z}_1$ are a.s.\ the same.  Each of these L\'evy nets is determined by the encoding information (boundary length process plus attachment points) which in turn a.s.\ fixes the homeomorphism between the two, which by assumption maps geodesics of $\mustwo$ to geodesics of $\tmustwo$.

By the previous step, we can also couple $(S,d,\nu,x,y)$ and $(\wt{S},\wt{d},\wt{\nu},\wt{x},\wt{y})$ so that the masses of all of the holes cut out by these two explorations are the same. Thus, we can then sample $z_2$ and $\wt{z}_2$ coupled in such a way that they a.s.\ lie in the same hole of the L\'evy net complement (i.e., the hole corresponding to the same jump in the boundary length process). At the first time $t$ at which $z_1$ and $z_2$ are separated, there is a.s.\ a figure $8$ in $S$ (hence also in $\wt S$) describing the boundary of the two unexplored regions containing $z_1$ and $z_2$. We can then couple the L\'evy net toward $z_2$  (as started from the time $t$) so that is a.s.\ agrees with the corresponding L\'evy net toward $\wt{z}_2$.  Now one readily sees that the {\em union} of the L\'evy nets in $S$ toward $z_1$ and $z_2$ (which is the union of a figure $8$ and one set contained in each of the three components of its complement) is homeomorphic to its counterpart in $\wt{S}$, and again we may assume that the masses of the holes cut out by the branched exploration are the same.

Note that (since this information is encoded in the L\'evy net) if the two geodesics from $z_1$ and $z_2$ merge at some distance $t$ from the root (with the geodesic from $z_1$ merging from the left, say) then the corresponding paths in $\wt{S}$ a.s.\ exhibit the same behavior.

By iterating this, and taking a limit in the obvious way, we obtain a coupling under which the L\'evy nets associated with the $(z_i)$ (i.e., the corresponding countable collection of boundary length processes with attachment point locations) a.s.\ agree {\em precisely} with those corresponding to the $(\wt{z}_i)$.  Moreover, for each $i$ and $j$, the distance from the root at which the two geodesics merge (and which of the two paths merges from the left) agrees a.s. In other words, the planar tree in $S$ formed by taking the union of the geodesics from the $z_i$ is a.s.\ isomorphic to the corresponding tree in $\wt{S}$. In fact, we know more than that, since we also know that each L\'evy net in $S$ toward one of the $z_i$ is a.s.\ in homeomorphic correspondence with its counterpart in $\wt{S}$.

Now, we would like to argue that in this coupling, the distance between any two of the $\wt{z}_i$ in $\wt{S}$ is a.s.\ {\em at most} the Brownian map distance between the corresponding $z_i \in S$.  By definition of the distance $d$ on the Brownian map side, the distance between any two points in $S$ is the infimum over the lengths of continuous paths between those points made by concatenating arcs, each of which is a segment of a geodesic to the root $x$ (recall~\eqref{eqn::circ_distance_time}--\eqref{eqn::bm_distance_inf}).   Another way to describe this intuitively is to recall Figure~\ref{fig::lamination_treemaking}, where the $X_t$ and $Y_t$ process are the coordinates of a Brownian snake excursion. Let $G$ be the geodesic tree which is the quotient of the graph of $X_t$ that makes two points equivalent if a horizontal green segment connects them. Endow $G$ with the obvious metric structure.

Given a random pair of points $z_1$ and $z_2$ from $\nu$, we can find corresponding points $g_1$ and $g_2$ in $G$. Now recall that $d(z_1, z_2)$ is defined as the minimum length of a path in $G$ from $g_1$ to $g_2$ that is allowed to take finitely many ``shortcuts,'' where a shortcut is a step from one point in $G$ to another point in $G$ that corresponds to an equivalent point in the Brownian map. In Figure~\ref{fig::lamination_treemaking}, a shortcut can be taken by tracing a vertical red line up to the graph of $C- Y_t$, following a horizontal green line back to another point on the graph of $C-Y_t$, and then following a vertical red line back down. Each horizontal segment above the graph of $C-Y_t$ represents a shortcut.

Now, let $\mathcal L \subset S$ be the union of the points in the L\'evy nets targeted at the $(z_i)$.  Let $G_\mathcal L$ be the corresponding subset of $G$.  Note that by construction, if $z \in G_\CL$ then any geodesic from $x$ to $z$ is also in $\mathcal L$, so $G_\CL$ is a.s.\ a dense subtree of $G$.

We claim that in the Brownian map, the distance definition (restricted to points in $(z_i)$) would be equivalent if we required that each of the arcs belong to $\mathcal L$.  To see why, first note that $\mathcal L$ contains every point $z$ with the property that $\{ w : d(x,w) > d(z,x)\}$ has a component with $z$ on its boundary (since then $z$ would be part of the L\'evy net corresponding to any $z_i$ in that component --- recall that $\nu$ is a.s.\ a good measure, so it is a.s.\ the case that any open subset of $S$ contains at least one of the $z_i$).  This would include any point $z$ on the Brownian map dual tree (whose contour function is $Y_t$) which lies in the interior of a branch of the dual tree and (within that branch) is a local minimum of the Brownian process used to define the Brownian map (since this implies that the branch includes a non-trivial path of points in $\{ w : d(x,w) > d(z,x)\}$ that terminates at $z$). So in particular $\mathcal L$ includes a dense set of points along any branch of the dual tree, along with the geodesics connecting these points to the root. In Figure~\ref{fig::lamination_treemaking} this implies that a dense subset of the horizontal segments above $C-Y_t$ correspond to points in $\mathcal L$ --- assuming we encode each segment by its pair of endpoints and use the Euclidean topology on $(\R^2)^2$. 

Now to describe a ``path from $g_1$ to $g_2$ with finitely many shortcuts'' we can simply give the sequence  $(a_1,b_1), \ldots, (a_k,b_k)$ of directed horizontal line segments (above the graph of $C-Y_t$) that describe the shortcuts, where $a_j$ and $b_j$ are the first and last points of the $j$th shortcut.  (We may assume that the arcs between the shortcut endings are minimum length arcs in $G$, so the total length of the path is the sum of the lengths of these arcs). From here it is not hard to see that we can replace each horizontal line segment with an arbitrarily-nearby alternative that corresponds to a point in $\mathcal L$, and we can do so in a way that causes the length of the concatenated path to change by an arbitrarily small amount.  So, as claimed above, the definition of $d$ (restricted to points in $\mathcal L$) does not change if we add the requirement that the geodesic arcs be subsets of $\mathcal L$.

But for every such path in $S$ comprised of geodesic arcs that are subsets of $\mathcal L$, there is a corresponding path in $\wt{S}$ of the same length. This implies that the distance between two of the points $\wt{z}_i$ in $\wt{S}$ is a.s.\ {\em at most} the corresponding distance in $S$, and hence a.s.\ $\wt{d}(\wt{z}_1,\wt{z}_2) \leq d(z_1,z_2)$. 

Recall that the $\musa$ expectation of the diameter is finite.  This combined with the scale invariance of the Brownian map implies that we a.s.\ have
\[ \E[ d(z_1,z_2) \giv \nu(S) ] < \infty .\]
 Moreover, from the above coupling, we a.s.\ have $\nu(S) = \wt{\nu}(\wt{S})$ and
\begin{equation}
\label{eqn::expected_dist_agree}
\E[ \wt{d}(\wt{x},\wt{y}) \giv \wt{\nu}(\wt{S}) ] = \E[ d(x,y) \giv \nu(S) ],
\end{equation}
which also holds if $x,y,\wt{x},\wt{y}$ are replaced by $z_1, z_2, \wt{z}_1, \wt{z}_2$. Recalling that $z_1,z_2$ and $\wt{z}_1,\wt{z}_2$ are independent and uniform samples from $\nu$ and $\wt{\nu}$, respectively, it thus follows from~\eqref{eqn::expected_dist_agree} and the aforementioned one-sided bound on distances that we in fact must have an a.s.\ equality.  Since this holds a.s.\ for any $i$ and $j$, we have that  $(S,d,x,y)$ and $(\wt{S},\wt{d},\wt{x},\wt{y})$ are a.s.\ isomorphic when restricted to a countable dense set, and hence are also isomorphic on the closure of that set (which is the entire Brownian map in the case of $S$, and hence must be an entire sphere homeomorphic surface in the case of $\wt{S}$ as well).  The measures $\nu$ and $\wt{\nu}$ also agree a.s.\ (as they are determined by the sequence of samples $(z_i)$).

\end{proof}

\begin{figure}[ht!]
\begin{center}
\includegraphics[scale=0.85]{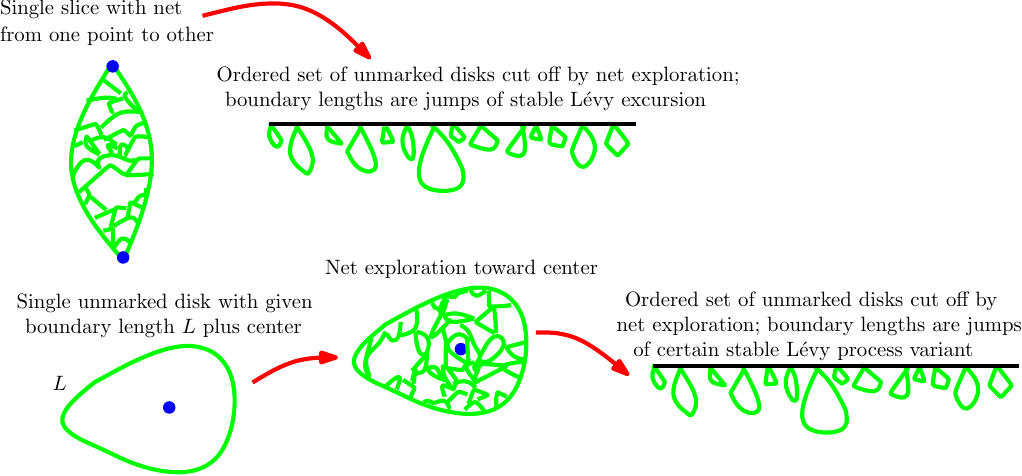}
\caption{\label{fig::levyslicedecomposition} A slice (or doubly marked sphere) comes endowed with a L\'evy net (as explained in Figure~\ref{fig::bubblegluing}) and once the L\'evy net is given, the disks are conditionally independent unmarked Brownian disks with given boundary lengths. As shown below, even an unmarked disk of given boundary length $L$ has a special interior point called the {\em center}. Once one conditions on the exploration net toward that point, the holes are again conditionally independent unmarked Brownian disks with given boundary lengths.
}
\end{center}
\end{figure}

We are now ready to prove Theorem~\ref{thm::markovmapcharacterization}.  The main ideas of the proof already appeared in the proof of Theorem~\ref{thm::levynetbasedcharacterization}.

\begin{proof}[Proof of Theorem~\ref{thm::markovmapcharacterization}]
The beginning of the proof of this result appears in Section~\ref{subsec::slice_independence} with the statement of Proposition~\ref{prop::poissonslicestructure}.  In particular, the combination of Proposition~\ref{prop::poissonslicestructure} and Lemma~\ref{lem::csbp_extinction_time} implies that for each fixed value of $r$ the law of the merging times of the leftmost geodesics of $(S,d,x,y)$ from $\partial \fb{x}{s}$ for $s=d(x,y)-r$ to $x$ have the same law as the geodesic tree in a L\'evy net (when the starting points for the geodesics have the same spacing in both).  Thus in view of the proof of Proposition~\ref{prop::levynet_structure_determined}, we have that $L_r$ is a.s.\ determined by the metric space structure of $(S,d,x,y)$.  This combined with the second assumption in the statement of Theorem~\ref{thm::markovmapcharacterization} implies that $L_r$ is a non-negative Markov process which satisfies the conditions of Proposition~\ref{prop::strongstablecsbp}.  That is, $L_r$ evolves as a CSBP excursion as $r$ increases, stopped when it hits zero.

This discussion almost implies that the hypotheses of Theorem~\ref{thm::levynetbasedcharacterization} are satisfied for some $\alpha \in (1,2)$.  Arguing as in the proof of Proposition~\ref{prop:metric_net_equivalent}, it implies that the portion of the unembedded metric net in $\fb{x}{s}$ looks like a portion of a L\'evy net. However, it does not rule out the possibility that the boundary length process $L_r$ might not tend to zero as $r$ approaches $d(x,y)$. As explained in the proof of Lemma~\ref{lem::metricnetmeasurezero}, this can be ruled out by showing that the metric net from $x$ to $y$ a.s.\ has $\nu$ measure zero.

If the metric net failed to have measure zero, then the expression~\eqref{eqn::atmartingale} from Proposition~\ref{prop::centermartingale} would have to fail to be a martingale, which would imply by Proposition~\ref{prop::centermartingale} that we must have $\alpha \not = 3/2$.

Suppose that the metric net does not have measure zero.  We now suppose that $(S,d,\nu)$ is a sample from the law $\tmudl$.  Let $M_r$ be the boundary length process associated with the center exploration and let $\CJ_r$ be the jumps made by $M$ up to time $r$.  Then the process
\[ A_r = M_r^{2\alpha-1} + \sum_{x \in \CJ_r} |x|^{2\alpha-1}\]
as in~\eqref{eqn::atmartingale} from Proposition~\ref{prop::centermartingale} corresponding to the center exploration of an instance of $\tmudl$ would not be a martingale (implying $\alpha \neq 3/2$).  However, the process $A_r$ would have to be a supermartingale and $A_r + B_r$ is a martingale where $B_r$ is the conditional expectation given $M|_{[0,r]}$ of the amount of area in the metric net disconnected  by the center exploration from the center up to time $r$.  By the Doob-Meyer decomposition, $B_r$ is the unique non-decreasing process so that $A_r + B_r$ is a martingale.  The form of $B_r$ can be determined explicitly from the expression for the drift term associated to~\eqref{eqn::atmartingale}, which is derived in the proof of Proposition~\ref{prop::centermartingale} which is given below.  In particular, it is shown in~\eqref{eqn:a_r_limit} that in the case $L=1$ we have that
\[  \E[ A_r - A_0] = r I_\alpha + o(r)\]
where $I_\alpha$ is a constant which depends only $\alpha$.  By the scaling property of area in terms of boundary length, this implies that for a general value of $L > 0$ given in the assumption of the theorem that
\[  \E[ A_r - A_0] = r I_\alpha L^a + o(r).\]
Lemma~\ref{lem::m_t_martingale} implies that
\[ A_r - I_\alpha \int_0^r M_u^a du\]
is a martingale which (by the uniqueness of the Doob-Meyer decomposition) in turn implies that
\[ B_r = -I_\alpha \int_0^r M_u^a du.\]
Altogether, this implies that
\[ \E[ \nu(S) ] = \E\left[ \sum_{x \in \CJ} |x|^{2\alpha-1} \right] - I_\alpha \E\left[ \int_0^\infty M_u^a du \right]\]
where $\CJ$ is the set of jumps made by $M$ (and we take $M$ to be $0$ after the center is reached).  Since $\E[ \nu(\CS) ]$ is given by a constant times $L^{2\alpha-1}$, we must have that $a = \alpha$ (since the multiplying the boundary length by $C$ changes the time duration by $C^{\alpha-1}$).  On the other hand, the independence of slices assumption implies we must have that $a=1$.  Since $\alpha \in (1,2)$, this cannot be the case and therefore the $\nu$-area of the metric net is zero.

\end{proof}

\subsection{Tail bounds for distance to disk boundary}
\label{subsec::tail_bounds}

It will be important in \cite{qle_continuity} to establish tail bounds for the amount of time that it takes a $\QLE(8/3,0)$ exploration starting from the boundary of a quantum disk to absorb all of the points inside of the quantum disk.  This result will serve as input in the argument in \cite{qle_continuity} to show that the metric space defined by $\QLE(8/3,0)$ satisfies the axioms of Theorem~\ref{thm::levynetbasedcharacterization} (and therefore we cannot immediately apply Theorem~\ref{thm::levynetbasedcharacterization} in the setting we have in mind in \cite{qle_continuity} to transfer the corresponding Brownian map estimates to $\sqrt{8/3}$-LQG).  However, in the results of \cite{qlebm} we already see some of the Brownian map structure derived here appear on the $\sqrt{8/3}$-LQG sphere.  Namely, the evolution of the boundary length of the filled metric ball takes the same form, the two marked points are uniform from the quantum measure, and we have the conditional independence of the surface in the bubbles cut out by the metric exploration given their quantum boundary lengths.  The following proposition will therefore imply that the results of \cite{qlebm} combined with the present work are enough to get that the joint law of the amount of time that it takes for a $\QLE(8/3,0)$ starting from the boundary of a quantum disk to absorb all of the points in the disk and the quantum area of the disk is the same in the case of both the Brownian map and $\sqrt{8/3}$-LQG.

\begin{proposition}
\label{prop::disk_area_distance}
Suppose that we have a probability measure on singly-marked disk-homeomorphic metric measure spaces $(S,d,\nu,x)$ where $\nu$ is an a.s.\ finite, good measure on $S$ such that the following hold.
\begin{enumerate}
\item The conditional law of $x$ given $(S,d,\nu)$ is given by $\nu$ (normalized to be a probability measure).
\item For each $r$ which is smaller than the distance $d(x,\partial S)$ of $x$ to $\partial S$, there is a random variable $L_r$, which we interpret as a boundary length of the $x$-containing component of the complement of the set of points with distance at most~$r$ from~$\partial S$.  As $r$ varies, this boundary length evolves as the time-reversal of a $3/2$-stable CSBP stopped upon hitting $0$.  The time at which the boundary length hits $0$ is equal to $d(x,\partial S)$.
\item The law of the metric measure space inside of such a component given its boundary length is conditionally independent of the outside.
\item There exists a constant $c_0 > 0$ such that the expected $\nu$ mass in such a component given that its boundary length is $\ell$ is $c_0 \ell^2$.
\end{enumerate}
Let $d^* = \sup_{z \in S} \dist(z,\partial S)$.  Then the joint law of $d^*$ and $\nu(S)$ is the same as the corresponding joint law of these quantities under $\mudonel$ where $L$ is equal to the boundary length of $\partial S$ under $(S,d,\nu,x)$.  In particular, for each $0 < a,L_0 < \infty$ there exists a constant $c \geq 1$ such that for all $L \in (0,L_0)$ and $r > 0$ we have
\begin{equation}
\label{eqn::disk_swallowing_time_tail_bound}
\p\left[\nu(S) \leq a \giv d^* \geq r \right] \leq c\exp(- c^{-1} r^{4/3})
\end{equation}
Moreover, the tail bound~\eqref{eqn::disk_swallowing_time_tail_bound} also holds if we use the law with Radon-Nikodym derivative given by $(\nu(S))^{-1}$ with respect to the law of $(d^*,\nu(S))$.
\end{proposition}

We note that the law in the final assertion of Proposition~\ref{prop::disk_area_distance} corresponds to $\mudl$.  We will need to collect two lemmas before we give the proof of Proposition~\ref{prop::disk_area_distance}.

\begin{lemma}
\label{lem::brownian_map_diameter_bound}
For each $0 < a < b < \infty$ there exists a constant $c > 0$ such that the following is true.  For an instance $(S,d,\nu,x,y)$ sampled from $\mustwo$, we let $d^*$ be the diameter of $S$.  Conditionally on $\nu(S) \in [a,b]$, the probability that $d^*$ is larger than $r$ is at most $c \exp(-\tfrac{3}{2}(1+o(1)) b^{-1/3} r^{4/3})$ where the $o(1)$ term tends to $0$ as $r \to \infty$.
\end{lemma}
\begin{proof}
It follows from \cite[Proposition~14]{serlet_ldp} that the probability that the unit area Brownian map has diameter larger than $r$ is at most a constant times $\exp(-\tfrac{3}{2}(1+o(1)) r^{4/3})$ where the $o(1)$ term tends to $0$ as $r \to \infty$.  The assertion of the lemma easily follows from the scaling property of the Brownian map (scaling areas by the factor $a$ scales distances by the factor $a^{1/4}$).
\end{proof}

\begin{lemma}
\label{lem::brownian_disk_area_bound}
Fix $0 < a, L_0 < \infty$.  There exists a constant $c \geq 1$ depending only on $a, L_0$ such that for all $L \in (0,L_0)$ the following is true.  Suppose that we have an instance $(S,d,\nu)$ sampled from $\mudl$ conditioned on $\nu(S) \leq a$.  Let $d^*$ be the supremum over all $z \in S$ of the distance of $z$ to $\partial S$.  The probability that $d^*$ is larger than $r$ is at most $c \exp(-c^{-1} r^{4/3})$.  The same holds with $\mudonel$ in place of $\mudl$.
\end{lemma}
\begin{proof}
Suppose that we have a sample $(S,d,\nu,x,y)$ from $\mustwo$ conditioned on the positive and finite probability event that:
\begin{enumerate}
\item There exists an $r$ and a component $U$ of $S \setminus B(x,r)$ with $y \notin U$ such that the boundary length of $U$ is equal to $L$.
\item $\nu(U) \leq a$ and $1 \leq \nu(S \setminus U) \leq 2$.
\end{enumerate}
Then we know that the law of $U$ (viewed as a metric measure space) is given by $\mudl$ conditioned on having area at most $a$.  The amount of time that it takes the metric exploration starting from $\partial U$ to absorb every point in $U$ is bounded from above by the diameter of $(S,d)$.  Thus the first assertion of the lemma follows from Lemma~\ref{lem::brownian_map_diameter_bound}.

The second assertion follows from the first because the Radon-Nikodym derivative between $\mudonel$ and $\mudl$ is at most $a$ on the event that $\nu(S) \leq a$.
\end{proof}

\begin{proof}[Proof of Proposition~\ref{prop::disk_area_distance}]
The first assertion follows from a simplified version of the argument used to prove Theorem~\ref{thm::levynetbasedcharacterization}.

We now turn to prove the second assertion by combining the first assertion with Lemma~\ref{lem::brownian_disk_area_bound}.  We may assume without loss of generality that $r \geq 1$.  We consider two possibilities depending on whether $L \leq r^{-1/2}$ or $L \in (r^{-1/2}, L_0]$.

Suppose that $L \in (r^{-1/2},L_0]$.  Then we can write
\[ \p[ \nu(S) \leq a \giv d^* \geq r] = \frac{\p[ \nu(S) \leq a, d^* \geq r]}{\p[d^* \geq r]}.\]
Lemma~\ref{lem::brownian_disk_area_bound} implies that the numerator is at c $\exp(-c^{-1} r^{4/3})$ for a constant $c \geq 1$.  As $L \geq r^{-1/2}$, it is easy to see that the denominator is at least a negative power of $r$ as $r \to \infty$.  This proves the desired result in this case.

Now suppose that $L \in (0,r^{-1/2}]$.  Let $(S,d,\nu,x,y)$ be sampled from $\mustwo$.  Let $A_L$ (resp.\ $D_L$) be the area (resp.\ maximal distance to the boundary) of $S \setminus \fb{x}{\tau_L}$ where $\tau_L$ is the smallest $r \geq 0$ so that the boundary length of $\partial \fb{x}{r}$ is equal to $L$.  Then we need to prove an upper bound for $\mustwo(A_L \leq a \giv D_L \geq r)$.  This, in turn is equal to
\[ \frac{\mustwo(A_L \leq a, D_L \geq r)}{\mustwo(D_L \geq r)}.\]
Since $L \leq r^{-1/2}$, the denominator is at least $\mustwo(D_{r^{-1/2}} \geq r)$ which is in turn at least a negative power of $r$ as $r \to \infty$.  Let $D$ be the diameter of $(S,d)$ so that $D \geq D_L$.  Then the numerator is at most a constant times $\mustwo(A \leq 2a, D \geq r)$ where $A = \nu(S)$ as the conditional probability that $\nu(\fb{x}{\tau_L}) \leq a$ given $\tau_L < \infty$ is positive.  We can then write
\begin{equation}
\label{eqn:area_diam_integral}	
\mustwo(A \leq 2a, D \geq r) = \int_0^{2a} \mustwo(D \geq r \giv A = p) \mustwo(A = p) dp
\end{equation}
where $\mustwo( \cdot \giv A = p)$ denotes the conditional law of $\mustwo$ given $A = p$ and $\mustwo(A=p)$ denotes the density of $A$ at $p$.  Recall that the density of $A$ at $p$ is equal to a constant times $p^{-3/2}$.  When $p=1$, we also recall \cite[Proposition~14]{serlet_ldp} implies that $\mustwo(D \geq r \giv A = 1)$ is at most a constant times $\exp(-\tfrac{3}{2}(1+o(1)) r^{4/3})$ where the $o(1)$ term tends to $0$ as $r \to 0$.  Recall that if we scale the unit area Brownian map so that its area becomes $p$ then distances are scaled by the factor $p^{1/4}$.  It therefore follows that $\mustwo(D \geq r \giv A = p) = \mustwo(D \geq r p^{-1/4} \giv A = 1)$ is at most a constant times $\exp(-\tfrac{3}{2}(1+o(1)) p^{-1/3} r^{4/3})$.  Altogether, \eqref{eqn:area_diam_integral} is at most a constant times
\[ \int_0^{2a} \exp(-\tfrac{3}{2}(1+o(1)) p^{-1/3} r^{4/3}) p^{-3/2} dp\]
which in turn is at most $c \exp(-c^{-1} r^{4/3})$ for a constant $c \geq 1$.
\end{proof}

\subsection{Adding a third marked point along the geodesic}
\label{sec::bmintro}

\begin{figure}[ht!]
\begin{center}
\includegraphics[scale=0.85]{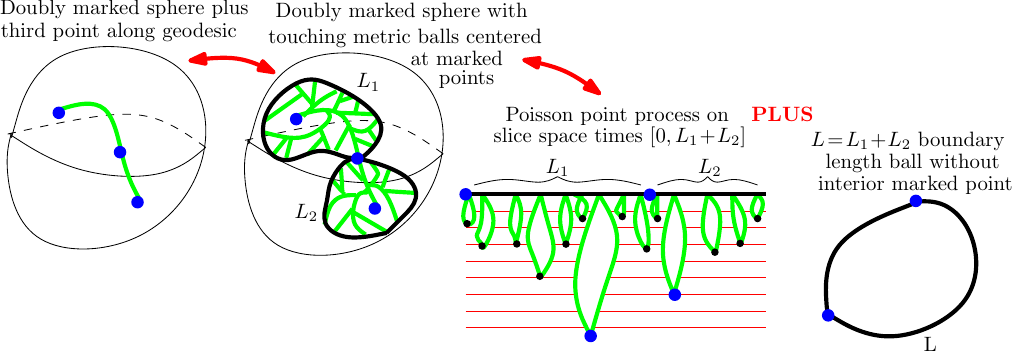}
\caption{\label{fig::eightinsphere}  To sample from the measure  $\mustwoplusone$ on triply marked spheres, one first samples from the measure $\mustwo$ {\em weighted} by the distance $D = d(x,y)$; given a sample from that measure, one then chooses $r$ uniformly in $[0,D]$ and marks the point $r$ units along the (a.s.\ unique) geodesic.  The second figure is a continuum version of Figure~\ref{fig::shards}. Given $L_1$ and $L_2$, one may decompose the metric balls as in Figure~\ref{fig::slicedecomposition} (the first $L_1$ units of time describing the first ball, the second $L_2$ units the second ball). The right figure is an independent unmarked Brownian disk, which represents the surface that lies outside of the two metric balls in the second figure. Given the disk, first blue dot is uniform on the boundary; the second is $L_1$ units clockwise from first. The measure that $\mustwoplusone$ induces on the pair $(L_1, L_2)$ is (up to multiplicative constant) the measure $(L_1 + L_2)^{-5/2}dL_1 dL_2$. This follows from the overall scaling exponent of $L$ and the fact that given $L=L_1 +L_2$ the conditional law of $L_1$ is uniform on $[0, L]$. }
\end{center}
\end{figure}

In this section, we present Figure~\ref{fig::eightinsphere} and use it to informally explain a construction that will be useful in the subsequent works \cite{qlebm, qle_continuity} by the authors to establish the connection between the $\sqrt{8/3}$-Liouville quantum gravity sphere and the Brownian map. This subsection is an ``optional'' component of the current paper and does not contain any detailed proofs; however, the reader who intends to read \cite{qlebm, qle_continuity} will find it helpful to have this picture in mind, and it is easier to introduce this picture here.

Roughly speaking, we want to describe the continuum version of the Boltzmann measure on figures such as the one in Figure~\ref{fig::shards}, where one has a doubly marked sphere together with two filled metric balls (centered at the two marked points) that touch each other on the boundary but do not otherwise overlap. Clearly, the Radon-Nikodym derivative of such a measure w.r.t.\ $\mu_{\mathrm{TRI}}^2$ should be $D+1$ where $D$ is the distance between the two points, since the radius of the first ball can be anything in the interval $[0,D]$. In the discrete version of this story, it is possible for the two metric balls in Figure~\ref{fig::shards} to intersect in more than one point (this can happen if the geodesic between the two marked points is not unique) but in the continuum analog discussed below one would not expect this to be the case (since the geodesic between the marked points is a.s.\ unique).

To describe the continuum version of the story, we need to define a measure  $\mustwoplusone$ on continuum configurations like the one shown in Figure~\ref{fig::eightinsphere}.  To sample from  $\mustwoplusone$, one first chooses a doubly marked sphere from the measure whose Radon-Nikodym derivative w.r.t.\ $\mustwo$ is given by $D$. Then, having done so, one chooses a radius $D_1$ for the first metric ball uniformly in $[0,D]$, and then sets the second ball radius to be $D_2 := D-D_1$. Now $\mustwoplusone$ is a measure on Brownian map surfaces decorated by two marked points and touching two filled metric balls centered at those points. Let $L_1$ and $L_2$ denote the boundary lengths of the two balls and write $L = L_1 + L_2$.

\begin{enumerate}
\item Based on Figure~\ref{fig::shards} and Figure~\ref{fig::eightinsphere}, we would expect that one can first choose the set of slices indexed by time $L$, and then randomly choose $L_1$ uniformly from $[0,L]$. Thus, we expect that given $L$ and $A$, the value $L_1$ is uniform on $[0,L]$.
\item It is possible to verify the following scaling properties (which hold up to a constant multiplicative factor):
\begin{align*}
\mustwo[ A > a] \approx a^{-1/2} \quad\quad&\mathrm{and}\quad\quad \mustwoplusone[ A > a ] \approx a^{-1/4}.\\
\mustwo[ L > a ] \approx a^{-1} \quad\quad&\mathrm{and}\quad\quad \mustwoplusone[ L > a ] \approx a^{-1/2}. \\
\mustwo[ D > a ] \approx a^{-2} \quad\quad&\mathrm{and}\quad\quad \mustwoplusone [ D > a ] \approx a^{-1}.
\end{align*}
\end{enumerate}

The two properties above suggest that $\mustwoplusone$ induces a measure on $(L_1, L_2)$ given (up to constant multiplicative factor) by $(L_1 + L_2)^{-5/2}dL_1 dL_2$.  The measure on $L$ itself is then $L^{-3/2} dL$.

If we condition on the metric ball in Figure~\ref{fig::eightinsphere} of boundary length $L_1$, we expect that conditional law of the complement to be that of a marked disk of boundary length $L_1$, i.e., to be a sample from $\mudonel$ with $L_1$ playing the role of the boundary length. This suggests the following symmetry (which we informally state but will not actually prove here).

\begin{proposition}
\label{prop::twowaystobuildmarkeddisc}
Given $L_1$, the following are equivalent:
\begin{enumerate}
\item Sample a marked disk of boundary length $L_1$ from the probability measure $\mudonel$ (with $L_1$ as the boundary length). One can put a ``boundary-touching circle'' on this disk by drawing the outer boundary of the metric ball whose center is the marked point and whose radius is the metric distance from the marked point to the disk boundary.
\item Sample $L_2$ from the measure $(L_1 + L_2)^{-5/2}dL_2$ (normalized to be a probability measure) and then create a large disk by identifying a length $L_2$ arc of the boundary of a sample from $\mudl$, with the entire boundary of a disk sampled from $\mumlset{L_2}$. The interface between these two samples is the ``boundary-touching circle'' on the larger disk.
\end{enumerate}
\end{proposition}

Interestingly, we do not know how to prove Proposition~\ref{prop::twowaystobuildmarkeddisc} directly from the Brownian snake constructions of these Brownian map measures, or from the breadth-first variant discussed here. Indeed, from direct considerations, we do not even know how to prove the symmetry of $\mustwo$ with respect to swapping the roles of the two marked points $x$ and $y$.  However, both this latter fact and Proposition~\ref{prop::twowaystobuildmarkeddisc} can be derived as consequences of the fact that $\mustwo$ is a scaling limit of discrete models that have similar symmetries (though again we do not give details here). We will see in \cite{qlebm, qle_continuity} that these facts can also be derived in the Liouville quantum gravity setting, where certain symmetries are more readily apparent.

We will also present in \cite{qlebm, qle_continuity} an alternate way to construct Figure~\ref{fig::eightinsphere} in the Liouville quantum gravity setting. In this alternate construction, one begins with a measure $\mu^2_{\mathrm{LQGSPH}}$ on doubly marked LQG spheres. Given such a sphere, one may then decorate it by a whole plane $\SLE_6$ path from one marked point to the other. Such a path will have certain ``cut points'' which divide the trace of the path into two connected components. It is possible to define a quantum measure on the set of cut points. One can then define a measure $\mu^{2+1}_{\mathrm{LQGSPH}}$ on path-decorated doubly marked quantum spheres with a {\em distinguished} cut point along the path. This is obtained by starting with the law of an $\SLE_6$-decorated sample from $\mu^2_{\mathrm{LQGSPH}}$, then {\em weighting} this law by the quantum cut point measure, and then choosing a cut point uniformly from this cut point measure. We will see in \cite{qlebm, qle_continuity} that a certain QLE ``reshuffling'' procedure allows us to convert a sample from $\mu^{2+1}_{\mathrm{LQGSPH}}$ into an object that (once an appropriate metric is defined on it) looks like a sample from $\mustwoplusone$.

\subsection{The martingale property holds if and only if $\alpha=3/2$}
\label{subsec::martingale32}

\begin{proposition}
\label{prop::centermartingale}
Fix $\alpha \in (1,2)$ and suppose that $M_r$ is the process associated with an exploration towards the center of a sample produced from $\tmudl$ where $\tmudl$ is as in Section~\ref{subsec::axiomsforbrownianmap}.  For each $r \geq 0$, we let
\begin{equation}
\label{eqn::atmartingale}
A_r = M_r^{2\alpha - 1} + \sum_{a \in \CJ_r} |a|^{2\alpha-1}
\end{equation}
where $\CJ_r$ is the set of jumps made by $M|_{[0,r]}$.  Then $A_r$ is a martingale if and only if $\alpha  = 3/2$.
\end{proposition}

We will need two intermediate lemmas before we give the proof of Proposition~\ref{prop::centermartingale}.

\begin{lemma}
\label{lem::m_t_martingale}
Suppose that $X_t$ is a non-negative, real-valued, continuous-time {\cadlag} process with $\sup_{t \geq 0} X_t < \infty$ and $X_0 > 0$ a.s.  Let $\tau = \inf\{t \geq 0 : X_t = 0\}$ and let $(\CF_t)$ be the filtration generated by $(X_{t \wedge \tau})$.  Assume that there exists $p > 1$ with
\begin{equation}
\label{eqn::x_u_ui}
\sup_{s \leq t \leq T}\E[ |X_{t \wedge \tau}|^p \giv \CF_s]  < \infty \quad\text{for all}\quad 0 \leq s < T < \infty.
\end{equation}
Suppose that $q \colon \R_+ \to \R_+$ is a non-decreasing function such that $q(\Delta)/\Delta \to 0$ as $\Delta \to 0$.  Assume that $Y_t$ is a {\cadlag} process adapted to $\CF_t$ with $\E|Y_t| < \infty$ for all $t$ and that $a$ is a constant such that
\[ \left| \E[ Y_t - Y_s \giv \CF_s ] - a (t-s) X_{s \wedge \tau} \right| \leq q(t-s)|X_{s\wedge \tau}| \quad\text{for all}\quad t \geq s.\]
Then $Y_t$ is a martingale if and only if $a=0$.
\end{lemma}
\begin{proof}
Fix $\Delta > 0$, $s < t$, and let $t_0 = s < t_1 < \cdots < t_n = t$ be a partition of $[s,t]$ with $\Delta/2 < t_j-t_{j-1} \leq \Delta$ for all $1 \leq j \leq n$.  Then we have that
\begin{align*}
   \E[ Y_t \giv \CF_s ]
&= Y_s + \sum_{j=1}^n \E[ Y_{t_j} - Y_{t_{j-1}} \giv \CF_s]\\
&= Y_s + \sum_{j=1}^n \E[ \E[ Y_{t_j} - Y_{t_{j-1}} \giv \CF_{t_{j-1}}] \giv \CF_s].
\end{align*}
We are going to show that the right hand side above tends to $Y_s + a \int_s^t \E[ X_{u \wedge \tau} \giv \CF_s] du$ in $L^1$ as $\Delta \to 0$.  This, in turn, implies that there exists a positive sequence $(\Delta_k)$ with $\Delta_k \to 0$ as $k \to \infty$ sufficiently quickly so that the convergence is almost sure.  This implies the result because if $s < \tau$ then $a \int_s^t \E[ X_{u \wedge \tau} \giv \CF_s] du =0$ if and only if $a  =0$.

We begin by noting that
\begin{align*}
  & \sum_{j=1}^n \E\left| \E[ (Y_{t_j} - Y_{t_{j-1}}) \giv \CF_{t_{j-1}}] - a(t_j-t_{j-1}) X_{t_{j-1} \wedge \tau} \right|\\
\leq& \sum_{j=1}^n q(t_j-t_{j-1}) \E|X_{t_{j-1} \wedge \tau}|\\
\leq&  \frac{2q(\Delta)}{\Delta} \sup_{s \leq u \leq t} \E|X_u| \to 0  \quad\text{as}\quad \Delta \to 0.
\end{align*}
The {\cadlag} property together with the dominated convergence theorem implies that
\begin{equation*}
\sum_{j=1}^n a(t_j-t_{j-1}) X_{t_{j-1} \wedge \tau} \to a \int_s^t X_{u \wedge \tau} du \quad\text{as}\quad \Delta \to 0.
\end{equation*}
Combining this with the integrability assumption~\eqref{eqn::x_u_ui} implies that
\begin{equation*}
\sum_{j=1}^n a(t_j-t_{j-1}) \E[ X_{t_{j-1} \wedge \tau} \giv \CF_s] \to a \int_s^t \E[ X_{u \wedge \tau} \giv \CF_s] du \quad\text{as}\quad \Delta \to 0,
\end{equation*}
which proves the claim.
\end{proof}

\begin{lemma}
\label{lem::stable_levy_time_reversal}
Fix $\alpha \in (1,2)$ and suppose that $M_r$ is the process associated with an exploration towards the center of a sample produced from $\tmudl$ where $\tmudl$ is as in Section~\ref{subsec::axiomsforbrownianmap}.  There exists constants $c_0,c_1 > 0$ such that
\begin{equation}
\label{eqn::stable_levy_time_reversal_tail}
\p[ M_r \geq u ] \leq c_0 e^{-c_1 r^{-1/\alpha} u} \quad\text{for all}\quad u,r > 0.
\end{equation}
In particular,
\begin{equation}
\label{eqn::stable_levy_time_reversal_moment}
\E |M_r|^p < \infty \quad\text{for all}\quad r,p > 0.
\end{equation}
\end{lemma}
\begin{proof}
We first note that~\eqref{eqn::stable_levy_time_reversal_tail} in the case of an $\alpha$-stable process with only downward jumps follows from \cite[Chapter~VII, Corollary~2]{bertoinlevybook}.  The result in the case of $M_r$ follows by comparing the jump law for $M_r$ as computed in Lemma~\ref{lem::mtjumplaw} with the jump law for an $\alpha$-stable process (which we recall has density $x^{-\alpha-1}$ with respect to Lebesgue measure on $\R_+$).
\end{proof}

\begin{proof}[Proof of Proposition~\ref{prop::centermartingale}]
We assume without loss of generality that $L = 1$.  Let $\CJ_r$ be the set of jumps made by $M|_{[0,r]}$ and, for each $\epsilon,\delta > 0$, let $\CJ_r^\epsilon$ (resp.\ $\CJ_r^{\epsilon,\delta}$) consist of those jumps in $\CJ_r$ with size at least $\epsilon$ (resp.\ size in $[\epsilon,\delta]$).  Let $J_r^\epsilon$ (resp.\ $J_r^{\epsilon,\delta}$) be the sum of the elements in $\CJ_r^\epsilon$ (resp.\ $\CJ_r^{\epsilon,\delta}$) and let
\[ C^\epsilon = \int_\epsilon^\infty x \cdot x^{-\alpha - 1} dx = \int_\epsilon^\infty x^{-\alpha} dx = \frac{1}{\alpha-1} \epsilon^{1-\alpha} \quad\text{and}\quad C^{\epsilon,\delta} = \int_\epsilon^\delta x^{-\alpha} dx.\]
Then we have that
\[ M_r = \lim_{\epsilon \to 0} M_r^\epsilon \quad\text{where}\quad M_r^\epsilon = (1 + J_r^\epsilon + r C^\epsilon)_+.\]

  We also let $A_r^\epsilon$ be given by
\[ A_r^\epsilon = (M_r^\epsilon)^{2\alpha-1} + \sum_{a \in \CJ_r^\epsilon} |a|^{2\alpha-1}.\]
We note that
\begin{align}
\label{eqn::a_t_eps_diff}
A_r - A_r^\epsilon
&= M_r^{2\alpha-1} - (M_r^\epsilon)^{2\alpha-1} - \sum_{a \in \CJ_r \setminus \CJ_r^\epsilon} |a|^{2\alpha-1}
\end{align}
and that the expectation of~\eqref{eqn::a_t_eps_diff} tends to $0$ as $\epsilon \to 0$.

Using that $A_0^\epsilon = M_0 = 1$, we have that
\begin{align*}
      A_r^\epsilon - A_0^\epsilon
&= (M_r^\epsilon)^{2\alpha-1} + \sum_{a \in \CJ_r^\epsilon} |a|^{2\alpha-1}  - 1
  = \left( 1+J_r^\epsilon + r C^\epsilon \right)_+^{2\alpha-1} + \sum_{a \in \CJ_r^\epsilon} |a|^{2\alpha-1} - 1.
\end{align*}

\newcommand{\jumplaw}{\mathsf{M}}

With $\jumplaw$ denoting the jump law of $M_r$ as determined in Lemma~\ref{lem::mtjumplaw}, we let
\begin{align}
I_\alpha^\delta  &= \int_\delta^{1/2} \left( x^{2\alpha - 1} + (1-x)^{2\alpha -1}  - 1 \right) d\jumplaw(x)  +  (2\alpha - 1)C^\delta \quad\text{and} \label{eqn::jumpintegral1}\\
I_\alpha &= \lim_{\delta \to 0} I_\alpha^\delta. \label{eqn::jumpintegral1a}
\end{align}
We will show later in the proof that the limit in~\eqref{eqn::jumpintegral1a} converges, compute its value, and show that $I_\alpha = 0$ precisely for $\alpha=3/2$.

Assuming for now that this is the case, we are going to prove the result by showing that
\begin{equation}
\label{eqn::a_t_change}
\E[A_r - A_0] = \lim_{\epsilon \to 0} \E[ A_r^\epsilon - A_0^\epsilon] = r I_\alpha + o(r) \quad\text{as}\quad r \to 0
\end{equation}
where $I_\alpha$ is as in~\eqref{eqn::jumpintegral1a}.  This suffices because then we can invoke Lemma~\ref{lem::m_t_martingale}.

Let $E_r^{0,\delta}$ (resp.\ $E_r^{1,\delta}$) be the event that $M|_{[0,r]}$ does not make a (resp.\ makes exactly $1$) jump of size at least $\delta$ and let $E_r^{2,\delta}$ be the event that $M|_{[0,r]}$ makes at least two jumps of size at least $\delta$.

Assume $\epsilon \in (0,\delta)$.  We will now establish~\eqref{eqn::a_t_change} by estimating $\E[ (A_r^\epsilon - A_0^\epsilon) \one_{E_r^{j,\delta}}]$ for $j =0,1,2$.

We start with the case $j=0$.  Let
\begin{equation}
\label{eqn::x_def}
X = J_r^{\epsilon,\delta} + r C^\epsilon = J_r^{\epsilon,\delta} + r \big( C^{\epsilon,\delta} + C^\delta \big).
\end{equation}
On $E_r^{0,\delta}$, we have that
\begin{align}
      A_r^\epsilon - A_0^\epsilon
=& (1 +X)_+^{2\alpha-1} + \sum_{a \in \CJ_r^\epsilon} |a|^{2\alpha-1} -1. \label{eqn::e_t_0_delta_a}
\intertext{By performing a Taylor expansion of $u \mapsto (1+u)_+^{2\alpha-1}$ around $u=0$, we see that~\eqref{eqn::e_t_0_delta_a} is equal to}
& (2\alpha-1)X + O(X^2) + O(|X|^3) + \sum_{a \in \CJ_r^\epsilon} |a|^{2\alpha-1}  \label{eqn::e_t_0_delta_b}
\end{align}
where the implicit constants in the $O(X^2)$ and $O(|X|^3)$ terms are non-random.  (The presence of the $O(|X|^3)$ term is so that we have a uniform bound which holds for all $X$ values, not just small $X$ values; we are using that $\alpha \in (1,2)$ so that $2\alpha-1 < 3$.)

The form of the jump law implies that
\begin{align}
\E \sum_{a \in \CJ_r^\epsilon} |a|^{2\alpha-1} &= O( r \delta^{\alpha-1}) \label{eqn::e_t_0_delta_jump}\\
\p[(E_r^{0,\delta})^c] = O(r \delta^{-\alpha}),\quad \p[E_r^{1,\delta}] &= O(r \delta^{-\alpha}),\quad \p[ E_r^{2,\delta}] = O_\delta(r^2) \label{eqn::e_r_0_delta_prob}\\
\E[ |J_r^{\epsilon,\delta} + r C^{\epsilon,\delta}| ] &= O(r \delta^{1-\alpha/2})  \label{eqn::j_r_eps_delta_mean},\\
\E[ (J_r^{\epsilon,\delta} + r C^{\epsilon,\delta})^2] &= O_\delta(r^2), \quad\text{and} \label{eqn::j_r_eps_delta_var}\\
\E[ | J_r^{\epsilon,\delta} + r C^{\epsilon,\delta} |^3] &= O_\delta(r^3). \label{eqn::j_r_eps_delta_third}
\end{align}
In~\eqref{eqn::e_r_0_delta_prob}, \eqref{eqn::j_r_eps_delta_var}, and~\eqref{eqn::j_r_eps_delta_third} the subscript $\delta$ in $O_\delta$ means that the implicit constant depends on~$\delta$.  Thus by the Cauchy-Schwarz inequality and~\eqref{eqn::e_r_0_delta_prob}, \eqref{eqn::j_r_eps_delta_mean}, \eqref{eqn::j_r_eps_delta_var} we have that
\begin{align}
  \E[| J_r^{\epsilon,\delta} + r C^{\epsilon,\delta} | \one_{E_r^{0,\delta}}]
&=  O(r \delta^{1-\alpha/2}) - \E[ | J_r^{\epsilon,\delta} + r C^{\epsilon,\delta}| \one_{(E_r^{0,\delta})^c}] \notag\\
&= O(r \delta^{1-\alpha/2}) + O_\delta(r^{3/2}). \label{eqn::e_t_0_delta_small_jumps_comp}
\end{align}
Moreover, using~\eqref{eqn::j_r_eps_delta_var} we have that
\begin{align}
        \E[ X^2]
&\leq 4\left( \E[ (J_r^\epsilon + r C^{\epsilon,\delta})^2] + (r C^\delta)^2 \right)
  = O_\delta( r^2). \label{eqn::x_square_1_plus_x}
\end{align}
and from~\eqref{eqn::j_r_eps_delta_third} we have
\begin{align}
        \E[|X|^3]
&\leq 8\left( \E[ | J_r^\epsilon + r C^{\epsilon,\delta}|^3] +  (r C^\delta)^3 \right)
 = O_\delta( r^3). \label{eqn::x_third_1_plus_x}
\end{align}
Therefore taking expectations of~\eqref{eqn::e_t_0_delta_b} and using~\eqref{eqn::e_t_0_delta_jump},~\eqref{eqn::e_t_0_delta_small_jumps_comp}, \eqref{eqn::x_square_1_plus_x}, and~\eqref{eqn::x_third_1_plus_x}, we see that
\begin{equation}
 \E[ (A_r^\epsilon - A_0^\epsilon) \one_{E_r^{0,\delta}} ] =  r (2\alpha-1) C^\delta + O(r \delta^{1-\alpha/2}) + O(r \delta^{\alpha-1}) + O_\delta(r^{3/2}). \label{eqn::e_t_0_delta}
\end{equation}

We turn to the case $j=1$.  On $E_r^{1,\delta}$, with $J$ the size of the single jump larger than $\delta$, we have that 
\begin{align}
      A_r^\epsilon - A_0^\epsilon
=& \left( 1 +  J + X \right)_+^{2\alpha-1} + |J|^{2\alpha-1} + \sum_{a \in \CJ_r^\epsilon \setminus \CJ_r^\delta} |a|^{2\alpha-1} - 1. \label{eqn::e_t_1_delta_a}
\intertext{By performing a Taylor expansion of $u \mapsto (1 + J +u)_+^{2\alpha-1}$ about $u=0$, we see that~\eqref{eqn::e_t_1_delta_a} is equal to}
  & (1 + J)_+^{2\alpha-1} + |J|^{2\alpha-1} + \sum_{a \in \CJ_r^\epsilon \setminus \CJ_r^\delta} |a|^{2\alpha-1} + O(X) + O(|X|^3) -1 \notag
\end{align}
where $X$ is as in~\eqref{eqn::x_def} and the implicit constant in the $O(X)$ and $O(|X|^3)$ terms are non-random.  By~\eqref{eqn::e_r_0_delta_prob} and the Cauchy-Schwarz inequality we have $\E[ X \one_{E_r^{1,\delta}}] = O_\delta(r^{3/2})$.  Combining, we have that
\begin{equation}
\label{eqn::e_t_1_delta}
\E[ (A_r^\epsilon - A_0^\epsilon) \one_{E_r^{1,\delta}}] = r ( I_\alpha^\delta - (2\alpha-1)C^\delta) + O(r \delta^{\alpha-1}) + O(r \delta^{1-\alpha/2}) +O_\delta(r^{3/2}).
\end{equation}

We finish with the case $j=2$.  Using Lemma~\ref{lem::stable_levy_time_reversal}, it is easy to see that $A_r^\epsilon$ has finite moments of all order uniformly in $\epsilon$.  Thus using~\eqref{eqn::e_r_0_delta_prob} and H\"older's inequality, we have for any $p > 1$ that
\begin{align}
     \E[ (A_r^\epsilon - A_0^\epsilon) \one_{E_r^{2,\delta}} ]
&= O_{\delta,p}(r^{2/p}) \label{eqn::e_t_2_delta}
\end{align}
where the implicit constant in $O_{\delta,p}(r^{2/p})$ depends on both $\delta$ and $p$.

Combining~\eqref{eqn::e_t_0_delta},~\eqref{eqn::e_t_1_delta}, and~\eqref{eqn::e_t_2_delta} (with $p \in (1,2)$ so that $2/p > 1$), and taking a limit as $\epsilon \to 0$ we see that
\begin{equation}
\label{eqn:a_r_limit}
\E[ A_r - A_0] = r  I_\alpha + o(r) \quad\text{as}\quad  r \to 0.
\end{equation}
Indeed, this follows because each of the error terms which have a factor of $r$ also have a positive power of $\delta$ as a factor, except for the term with $I_\alpha$.  Thus we can make these terms arbitrarily small compared to $r$ by taking $\delta$ small.  The remaining error terms have a factor with a power of $r$ which is strictly larger than $1$, so we can make these terms arbitrarily small compared to $r$ by taking $r$ small.

Therefore to finish the proof we need to show that $I_\alpha = 0$ precisely for $\alpha=3/2$.  The indefinite integral
\begin{equation}
\label{eqn::jumpintegral2}
\int \left( x^{2\alpha - 1} + (1-x)^{2\alpha -1} - 1 \right) d\jumplaw(x)  - (2\alpha - 1) \int x^{-\alpha} dx
\end{equation}
can be directly computed (most easily using a computer algebra package such as Mathematica) to give
\begin{align*}
x^{-\alpha} \bigg(& \frac{_2F_1(1-\alpha,\alpha+1;2-\alpha;x) x }{\alpha-1}+\frac{\,
   _2F_1(-\alpha,\alpha;1-\alpha;x)}{\alpha}+\\
   &\frac{(\alpha-x) x^{2 \alpha-1} (1-x)^{-\alpha}}{(\alpha-1) \alpha}-\frac{(\alpha+x-1)
   (1-x)^{\alpha-1}}{(\alpha-1) \alpha}+  \frac{x-2 \alpha x}{\alpha-1} \bigg)
\end{align*}
where $_2F_1$ is the hypergeometric function.  In particular, the limit in~\eqref{eqn::jumpintegral1} is equal to
\begin{align}
\label{eqn::levyintegral}
-\frac{4^\alpha}{\alpha} - 2 B_{\frac{1}{2}}(-\alpha,1-\alpha) + \frac{2^{\alpha-1} (1-2 \alpha)}{\alpha-1} + (2\alpha - 1) \int_{1/2}^\infty x^{-\alpha }  dx ,
\end{align}
where $B_{x}(a,b) = \int_0^x u^{a-1}(1-u)^{b-1} du$ is the incomplete beta function.

By evaluating the integral in~\eqref{eqn::levyintegral}, we see that~\eqref{eqn::levyintegral} is equal to
\[ -\frac{4^\alpha}{\alpha} - 2 B_{\frac{1}{2}}(-\alpha,1-\alpha).\]
Direct computation shows that this achieves the value $0$ when $\alpha = 3/2$ and (since this is an increasing function of $\alpha$) is non-zero for other values of $\alpha \in (1,2)$. Thus,~\eqref{eqn::jumpintegral1} is equal to zero if and only if $\alpha = 3/2$, and as noted above, the result follows from this.
\end{proof}

\bibliographystyle{hmralphaabbrv}
\addcontentsline{toc}{section}{References}
\bibliography{qlebm}

\bigskip

\filbreak
\begingroup
\small
\parindent=0pt

\bigskip
\vtop{
\hsize=5.3in
Department of Mathematics\\
Massachusetts Institute of Technology\\
Cambridge, MA, USA } \endgroup \filbreak

\end{document}

%% file: acknowledgements.tex
We have benefited from conversations about this work with many people, a partial list of whom includes Omer Angel, Itai Benjamini, Nicolas Curien, Hugo Duminil-Copin, Amir Dembo, Bertrand Duplantier, Ewain Gwynne, Nina Holden, Jean-Fran{\c{c}}ois Le Gall, Gregory Miermont, R\'emi Rhodes, Steffen Rohde, Oded Schramm, Stanislav Smirnov, Xin Sun, Vincent Vargas, Menglu Wang, Samuel Watson, Wendelin Werner, David Wilson, and Hao Wu.

%% file: support_acknowledgements.tex
We would also like to thank the Isaac Newton Institute (INI) for Mathematical Sciences, Cambridge, for its support and hospitality during the program on Random Geometry where part of this work was completed.  J.M.'s work was also partially supported by DMS-1204894 and J.M.\ thanks Institut Henri Poincar\'e for support as a holder of the Poincar\'e chair, during which part of this work was completed.  S.S.'s work was also partially supported by DMS-1209044, DMS-1712862, a fellowship from the Simons Foundation, and EPSRC grants {EP/L018896/1} and {EP/I03372X/1}.